\documentclass[10pt]{amsart}
\usepackage{kotex}
\usepackage[margin=1.1in]{geometry}
\usepackage{color, hyperref}
\usepackage{amsmath,amsfonts,amssymb,amsthm, physics}
\usepackage{mathtools}
\usepackage{commath}
\usepackage{cleveref}
\usepackage{tikz}
\newtheorem{theorem}{Theorem}[section]
\newtheorem{lemma}[theorem]{Lemma}
\newtheorem{proposition}{Proposition}[section]

\theoremstyle{definition}

\theoremstyle{remark}
\newtheorem{remark}[theorem]{Remark}

\numberwithin{equation}{section}

\let\p=\partial

\let\b=\beta
\let\hide\iffalse

\newcommand{\rrr}{\mathbb R^3}

\newcommand{\bq}{\begin{equation}}
\newcommand{\eq}{\end{equation}}

\allowdisplaybreaks

\begin{document}
	
	\title[ ]{Ill-posedness of the Boltzmann-BGK model in the exponential class} 
	
	\author{Donghyun Lee}
	\address{Department of Mathematics, Pohang University of Science and Technology, Republic of Korea}
	\email{donglee@postech.ac.kr}
		\author{Sungbin Park}
	\address{Department of Mathematics, Pohang University of Science and Technology, Republic of Korea}
	\email{parksb2942@postech.ac.kr}
	\author{Seok-Bae Yun}
	\address{Department of Mathematics, Sungkyunkwan University, Suwon 440-746, Republic of Korea}
	\email{sbyun01@skku.edu}

	
	
	
\begin{abstract}
    BGK (Bhatnagar-Gross-Krook) model is a relaxation-type model of the Boltzmann equation, which is popularly used in place of the Boltzmann equation in physics and engineering. In this paper, we address the ill-posedness problem for the BGK model, in which the solution instantly escapes the initial solution space. 
    For this, we propose two ill-posedness scenarios, namely, the homogeneous and the inhomogeneous ill-posedness mechanisms. In the former case, we find a class of spatially homogeneous solutions to the BGK model, where removing the small velocity part of the initial data triggers ill-posedness by increasing temperature. For the latter, we construct a spatially inhomogeneous solution to the BGK model such that the local temperature constructed from the solution has a polynomial growth in spatial variable.
    These ill-posedness properties for the BGK model pose a stark contrast with the Boltzmann equation for which the solution map is, at least for a finite time, stable in the corresponding solution spaces.
\end{abstract}
	
\maketitle
\tableofcontents

\section{Introduction}

\subsection{BGK model of the Boltzmann equation\nopunct}
The Boltzmann equation is a fundamental equation connecting the particle level description and the fluid level description of a gas system. The Boltzmann equation is written by 
\begin{equation} \label{eq:Boltzmann}
    \p_{t}f + v\cdot\nabla_{x}f = Q(f,f) := \iint_{\mathbb{R}^{N}\times  \mathbb{S}^{N-1}} B(v-u, \omega)\big[ f(u')f(v') - f(u)f(v)\big]d\omega du,\quad N \geq 2.
\end{equation}
The velocity distribution function $f(t,x,v)$ is the mass density function of the gas molecules at position $x\in\Omega$ with velocity $v\in\mathbb{R}^N$ at time $t>0$. Post-collisional velocities are defined by
\[
    u' = u - ((u-v)\cdot\omega)\omega,\quad  v' = v + ((u-v)\cdot\omega)\omega.
\]
Throughout this paper, we assume that the collision kernel $B(v-u,\omega)$ satisfies the cutoff hard or soft potential conditions: 
\begin{equation} \label{Collision_kernel}
    B(v-u,\omega) \leq |v-u|^{\kappa} b(\cos\theta),\quad -N< \kappa \leq 1,\quad 0\leq b(\cos\theta) \leq C|\cos\theta|,\quad \cos\theta = \frac{(v-u)\cdot\omega}{|v-u|}.
\end{equation}
Especially for Section 6, we used collision kernel
\begin{align*}
    B(v-u,\omega) = |v-u|^{\kappa} b(\cos\theta),\quad 0< \kappa \leq 1,\quad b(\cos\theta) =  C|\cos\theta|.
\end{align*}

The Boltzmann equation, however, may not be very practical in that it usually involves extremely high computational costs, mostly due to the high dimensional variables and the complicated structure of the collision operator. In this regard, a relaxational model equation was introduced as a numerically amenable model equation of the Boltzmann equation \cite{BGK,Wal}. The equation is now called the BGK (Bhatnagar–Gross–Krook) model and is widely used to simulate various kinetic or hydrodynamic flow problems in place of the Boltzmann equation: 
\begin{align}\label{BGK}
    \begin{split}
        \partial_t f + v\cdot\nabla_x f &= \frac{1}{\tau}\left\{\mathcal{M}(f) - f\right\},\\
        f(0,x,v)&=f_0(x,v),
    \end{split}
\end{align}
where $\tau$ is the relaxation time, which will be normalized to $1$ throughout this paper. It can be done by the Euler scaling $(t,x)\mapsto (\tau t, \tau x)$ for constant $\tau>0$.
The local Maxwellian $\mathcal{M}(f)$ associated to $f$ takes the following form:
\begin{eqnarray*}
    \mathcal{M}(f)=\frac{\rho}{(2\pi T)^{N/2}}\exp\left(-\frac{|v-u|^2}{2T}\right),
\end{eqnarray*}
where the local density $\rho(t,x)$, the bulk velocity $u(t,x)$, and the local temperature $T(t,x)$ are constructed from $f$ by the following relation:
\begin{align*}
    \rho(t,x) &= \int_{\mathbb{R}^N} f(t,x,v)\,dv,\\
    (\rho u)(t,x) &= \int_{\mathbb{R}^N} vf(t,x,v)\,dv,\\
    \rho(|u|^2+N T)(t,x) &= \int_{\mathbb{R}^N} |v|^2 f(t,x,v)\,dv.
\end{align*}
An explicit computation shows that $\mathcal{M}(f)$ shares the first three moments with $f$:
\[
    \int_{\mathbb{R}^N}(1,v,|v|^2)\big\{\mathcal{M}(f)-f\big\}\,dv=0.
\]
This gives the conservation laws of mass, momentum, and energy: if $f$ decays to $0$ fast enough for large $x$ and $v$, we have 
\[
    \frac{d}{dt}\int_{\mathbb{R}_x^N\times\mathbb{R}_v^N}f\,(1,v,|v|^2)\,dxdv=0.
\]
The fact that the local Maxwellian shares the first three moments with the distribution function also leads to the following somewhat unexpected inequality:
\[
    \int_{\mathbb{R}^N}\big\{\mathcal{M}(f)-f\big\}\ln f\,dv\leq 0,
\]
from which the celebrated H-theorem can be derived:
\[
    \frac{d}{dt}\int_{\mathbb{R}_x^N\times\mathbb{R}_v^N}f\ln f\, dxdv \leq 0.
\]

Throughout the paper, we use the following notation:
\begin{align*}
    \|f\|_{L_v^p}=\left\{
        \begin{array}{ll}
            \left(\int_{\mathbb{R}^N}|f(v)|^pdv\right)^p&1\leq p<\infty,\\
            ess\displaystyle\sup_{v\in\mathbb{R}^N}|f(v)|&p=\infty,
        \end{array}
    \right.
\end{align*}
when $f$ is spatially homogeneous. We also use similar notation $\|f\|_{L^{p}_{x,v}}$ if $f$ is spatially inhomogeneous. We define
\begin{align*}
    \|f\|_{L^{q}_{v} L^{p}_{x}}=  \| \| f(\cdot, v)\|_{L^{p}_{x}} \|_{L^{q}_{v}}.
\end{align*}
Definition for $\|f\|_{L^{q}_{x} L^{p}_{v}}$ is also similar.

\subsection{Two instability scenarios for the BGK model\nopunct}

Let us start with the uniform boundedness of the relaxation operator of the BGK model in the space of exponentially decaying functions, which was the main motivation for the current ill-posedness problem. In the existence problem for the BGK model, it is crucial to control the local Maxwellian in a suitable norm. The first such estimate was derived by Perthame and Pulvirenti \cite{P-P} in the space of functions decaying at polynomial orders in velocity ($\beta>0$): 
\begin{align}\label{Mcontrol}
    \big\|(1+|v|^2)^{\beta}\mathcal{M}(f)\big\|_{L^{\infty}_{x,v}}\leq C_{\mathcal{M}}\big\|(1+|v|^2)^{\beta}f\big\|_{L^{\infty}_{x,v}}
\end{align}
through a combination of various highly nontrivial moment estimates of $f$. It was later extended to general $L^p$ space in \cite{Z-H}. This estimate and its variants became a starting point for many of the existence results of the BGK model \cite{C-Z,Mischler,PY1,PY2,Perthame,Son-Yun,Yun2,W-Z,Z-H,Z-VP,Zhang}. The discrete version of such an estimate was also derived to show the convergence of numerical schemes \cite{BCRY,Bello,Issau,RSY,RY}.
When the smallness is involved, there are a couple of ways in which we can avoid the use of such estimates. A few such examples are near-equilibrium problems where the initial perturbation is sufficiently small near the global Maxellian \cite{BKLY,Yun1,Yun22,Yun3}, or stationary value problems in which the equations are posed in a sufficiently small domain \cite{Brull-Yun,BGCY,Bang Y,HY}. For the existence problem with general initial or boundary data without such smallness, however, we are not aware of any methods that can avoid the use of \eqref{Mcontrol}

As mentioned above, the control of the local Maxwellian in \eqref{Mcontrol} has been an extremely useful estimate in the mathematical study of the BGK model, but it is a rough estimate in the sense that it says the Maxwellian, which is an exponential function in $v$, decays at polynomial order in $v$:
\[
    \mathcal{M}(f)\leq \frac{C}{(1+|v|^2)^{\beta}}
\]
with $C=C_{\infty}\|f(1+|v|^2)^{\beta}\|_{L^{\infty}}$.
Therefore, a very natural question arises whether \eqref{Mcontrol} can be
sharpened into an exponential version: 
\begin{equation} \label{exp_est}
    \big\|e^{\alpha|v|^2}\mathcal{M}(f)\big\|_{L^{\infty}}\leq C\big\|e^{\alpha|v|^2}f\big\|_{L^{\infty}}.
\end{equation}
If possible, such an estimate will lead to refined results to many of the existence results reported so far. 

In this paper, however, we prove that such an exponential estimate of the local Maxwellian is not available in general. 
More precisely, we will show that there is a sequence of functions $f^n$ such that $\big\|\mathcal{M}(f^n)e^{\alpha|v|^{\beta}}\big\|_{L^{p}}$ blows up while $\big\|f^n e^{\alpha|v|^{\beta}}\big\|_{L^{p}}$ remain finite as $n$ goes to infinity for the full range of  $\beta > 0$ and $p\in[1,\infty]$. See Theorem \ref{thm:1} and \ref{thm:2}. This failure of the exponential estimate \eqref{exp_est} leads to the ill-posedness result for the BGK model in the same exponential class. In the current work, we develop two different ill-posedness scenarios: spatially homogeneous ill-posedness and spatially inhomogeneous ill-posedness. To the best of our knowledge, such a strong ill-posedness result has never been reported for the Boltzmann equation or the BGK model. Moreover, it is also shown that the Boltzmann equation is well-posed in the corresponding solution spaces, at least for a finite time, which shows a stark discrepancy between the dynamics of the Boltzmann equation and the BGK model. 
\subsubsection{Spatially homogeneous Ill-posedness mechanism:}

We construct a spatially homogeneous solution to the BGK model that instantly escapes the class of exponentially decaying functions as it evolves in time. For this, we remove a small velocity region of the Maxwllian weight used to define the solution space and verify that the solution evolving from such initial data has a slower decay speed for $v$ than that of the Maxwellian weight.

The mechanism behind this instant ill-posedness is as follows. First, removing the small velocity region decreases the density and increases the temperature. As time starts to flow, the solution starts to evolve toward the Maxwellian with the changed density and temperature of the initial data. It turns out that the temperature increase is more dominant than the density decrease in the modified Maxwellian, which leads to a slower decay speed of the solution for $v$ compared to the original Maxwellian. We write these results in weighted norm language as in the following theorem: 

\begin{theorem} [homogeneous ill-posedness for the BGK model] \label{thm:1} 
	For any given $N\geq 1$, $\alpha>0$, $\beta>0$, $\varepsilon > 0$, and $1\leq p\leq \infty$, the following statements hold:
	\begin{enumerate}
		\item If $\beta\geq 2$, then there exists a spatially homogeneous non-negative measurable function $f_{0}$ such that
		\[
			\left\|e^{\alpha|v|^{\b}}f_{0}\,\right\|_{L^{p}_{v}}<\varepsilon,
		\] 
		but the corresponding solution $f$ to the initial value problem \eqref{BGK} satisfies
		\[
			\left\|e^{\alpha|v|^{\b}}f(t)\right\|_{L_{v}^p}=\infty
		\] 
		for any $t>0$.
		\item If $0<\beta< 2$,	there exists a sequence of spatially homogeneous non-negative measurable function $f_{n,0}$ such that
		\[
			\sup_n\left\|e^{\alpha|v|^{\b}}f_{n,0}\,\right\|_{L^{p}_{v}}<\varepsilon,
		\] 
		but the corresponding solution $f_n$ to the initial value problem \eqref{BGK} for each $n$ satisfies
		\[
			\lim_{n\rightarrow \infty}\left\|e^{\alpha|v|^{\b}}f_n(t)\right\|_{L_{v}^p}=\infty
		\] 
		for any $t>0$.
	\end{enumerate}

\end{theorem}

\subsubsection{Ill-posedness from spatial inhomogeneity:}  We suggest another ill-posedness mechanism in which the presence of the spatial variable plays a crucial role. 
More precisely, we construct initial data, which is basically - not exactly -  a homogeneous Maxwellian, but removed from inside in a spatially inhomogeneous manner. The removal of set in the velocity domain varies with respect to the position in the spatial domain. The intuition behind the choice of such inhomogeneously truncated Maxwellian is that it gives rise to a solution for which the local temperature behaves, at least for a short time, like fractional polynomials of the modulus of the spatial variable. It leads to the desired blow-up of the solution in a suitable polynomial-exponential decaying class.

To state the result, we define a weight function
\begin{align}\label{def_w}
    w_{\alpha,\beta,\delta}(v) = (1+|v|^2)^{\delta}e^{\alpha|v|^\beta}
\end{align}
for $\alpha,\beta>0$ and $\delta\geq 0$. In most cases, $\alpha$ and $\beta$ will be fixed, and $\delta$ will play an important role in the detailed analysis. Therefore, we will frequently write the weight function by $w_\delta(v)$ if there is no confusion, especially in Section 5.\\

\begin{theorem} [inhomogeneous Ill-posedness of the BGK model] \label{thm:2}
    Let our spatial domain $\Omega = \mathbb{R}^N_x$. For any given $N\geq 1$, $0 < \alpha^{\prime} \leq \alpha$, $0 < \beta \leq 2$, $\delta \geq 0$, $\varepsilon > 0$, and $1\leq p, q\leq \infty$, there exists a spatially inhomogeneous non-negative measurable function $f_0$ such that
    \[
        \left\|w_{\alpha,\beta,\delta}f_0\right\|_{L^p_xL^q_v} 
        \leq \varepsilon\quad (resp., \left\|w_{\alpha,\beta,\delta}f_0\right\|_{L^q_vL^p_x}\leq \varepsilon),
    \]
    but the solution $f$ to the initial value problem \eqref{BGK} defined in $\mathbb{R}_{x}^{N}\times \mathbb{R}_{v}^{N}$ satisfies
    \[
        \left\|w_{\alpha',\beta,\delta}f(t)\right\|_{L^p_xL^q_v}=\infty\quad (resp., \left\|w_{\alpha',\beta,\delta}f(t)\right\|_{L^q_vL^p_x}=\infty),
    \] 
    for any $t>0$. In fact, we can choose $f_0$ in a more restricted class satisfying
    \[ 
        \left\|w_{\alpha,\beta,\delta}f_0\right\|_{L^p_xL^q_v} 
        + \left\|w_{\alpha,\beta,\delta}f_0\right\|_{L^\infty_{x,v}}
        \leq \varepsilon\quad (resp., \left\|w_{\alpha,\beta,\delta}f_0\right\|_{L^q_vL^p_x} +  \left\|w_{\alpha,\beta,\delta}f_0\right\|_{L^\infty_{x,v}}\leq \varepsilon).
    \] 
    to get the identical result.
\end{theorem}
\begin{remark}
    (1) Above Theorem \ref{thm:2} implies the ill-posedness in $w_{\delta}$-weighted $L^p_xL^q_v\cap L^\infty_{x,v}$ (resp., $w_{\delta}$-weighted $L^q_vL^p_x\cap L^\infty_{x,v}$) as well as $w_{\delta}$-weighted $L^p_xL^q_v$ (resp., ($w_{\delta}$-weighted) $L^q_vL^p_x$).
    (2) We mention that this result is stronger than Theorem \ref{thm:1} in the sense that Theorem \ref{thm:1} implies sequential blow up in a weighted $L^p$ norm for $0 < \beta \leq 2$ case (See (2) of Theorem \ref{thm:2_1}), while Theorem \ref{thm:2} shows a strong ill-posedness in $w_{\alpha,\beta,\delta}$-weighted $L^{p}_{x}L^{q}_{v}$ space even for $0 < \beta \leq 2$.
\end{remark}

We mention that it is a somewhat unexpected result since ill-posed solutions are usually sought in a suitably reduced or symmetric function.
This is because it is sufficient to find a specific example that breaks down the stability property to prove such ill-posedness.  Therefore, for kinetic equations, it is very natural to attempt to look for such examples in the class of spatially homogeneous functions where the dependence of the distribution function on the spatial variables is suppressed. In the current work, however, we were able to find a spatially inhomogeneous ill-posedness mechanism where the presence of spatial variables plays a crucial role in pushing the solution out of the initial solution space.

Before closing this subsection, we mention that our result shows that the relaxation operator cannot be a bounded operator between the space of functions with exponential decay at the functional analysis level. This, however, does not prevent the possibility that such an estimate can be recovered if we restrict our interest to the solutions space of the BGK model, not a general function class. This is especially so since our explicit construction of the sequence of functions $f^n$ is designed such that the velocity distributions are concentrated in the far field, which is not relevant as a solution of collisional or relaxational kinetic equations.\\

\subsection{Comparison with the well-posedness result of the Boltzmann equation\nopunct}

The BGK model is known to produce qualitatively satisfactory simulations at much lower computational cost compared to the Boltzmann equation. Therefore, it is popularly employed in various important flow simulations in physics and engineering. However, since the BGK model somewhat neglects the complicated collision process, the BGK description of the Boltzmann flow cannot be exact. So far, such discrepancy between the BGK model and the Boltzmann has been investigated or measured mostly in terms of transport coefficients, such as the incorrectness of the viscosity and the thermal conductivity obtained from the BGK model in the Navier-Stokes limit (See \cite{A-P,Cercignani,CC,Sone,Sone2}). Our result reveals for the first time the discrepancy between the two equations on a more subtle mathematical level. It shows that the mapping property of the solution operator of the Boltzmann equation and the BGK model is fundamentally different: In the above-mentioned Theorem \ref{thm:1} and Theorem \ref{thm:2}, we showed that the solution map of the BGK model between the space of functions with exponential or polynomial-exponential  decay can instantly blow up: there exists a unique solution $f_{BGK}(t)$ having some initial $f_0$ such that
$$
    \frac{\big\|f_{BGK}(t)\big\|_{X}}{\|f_0 \big\|_{X}}=\infty
$$
for any arbitrarily small $t$. Here $\|\cdot\|_X$ denotes the norm employed either in Theorem \ref{thm:1} or \ref{thm:2}.\\
On the other hand, we show in the following Theorem \ref{thm:boltz} that the solution operator of the Boltzmann equation between the same solution space is bounded local in time:
$$
    \sup_{f_0\in X} \frac{\big\|f_{BE}(t)\big\|_{X}}{\|f_0 \big\|_{X}}<\infty,
$$
Here, $f_{BE}(t)$ is the solution of the Boltzmann equation at time $t$ corresponding to the same initial data $f_0$.
\begin{theorem} [Well-posedness of the Boltzmann] \label{thm:boltz}
    Assume $N$, $\alpha$, $\beta$, and $\delta$ satisfy
    \begin{equation}\label{index:weight}
        \begin{split}
            N\geq 2,\quad \alpha>0,\quad \max\Big\{0, \frac{2}{N-1}\kappa, \frac{2}{3}(1+\kappa)\Big\}<\beta\leq 2,\quad 2\delta>N+1,
        \end{split}
    \end{equation}
    and $1\leq p,q\leq \infty$. If $N=2$, we impose one more condition $1+\kappa\leq \beta$. Let our spatial domain to be $\mathbb{R}^N_x$ or $\mathbb{T}^N_x$. Then, if $f_0$ satisfies  
    \[
        \left\|w_{\alpha,\beta,\delta}f_0\right\|_{L^q_vL^p_x} + \left\|w_{\alpha,\beta,\delta}f_0\right\|_{L^\infty_{x,v}}<\infty,
    \] 
    where $w_{\alpha,\beta,\gamma}(v)$ is defined in \eqref{def_w}, then there exists $T^*>0$ such that the Boltzmann equation \eqref{eq:Boltzmann} has a unique solution $f(t,x,v)$ satisfying
    \begin{align*}
        \left\|w_{\alpha, \beta, \delta}f(t)\right\|_{L^\infty_{x,v}} &\leq 3\left\|w_{\alpha, \beta, \delta}f_0\right\|_{L^\infty_{x,v}}, \\
        \left\|w_{\alpha, \beta, \delta}f(t)\right\|_{L^q_vL^p_x} &\leq \left\|w_{\alpha, \beta, \delta}f_0\right\|_{L^q_vL^p_x}\exp\left(C\left\|w_{\alpha, \beta, \delta}f_0\right\|_{L^\infty_{x,v}}t\right)
    \end{align*}
    for $0\leq t\leq T^*$. In particular, if $1\leq q\leq p\leq \infty$, we have
    \begin{align*}
        \left\|w_{\alpha, \beta, \delta}f(t)\right\|_{L^p_xL^q_v}\leq \left\|w_{\alpha, \beta, \delta}f_0\right\|_{L^q_vL^p_x}\exp\left(C\left\|w_{\alpha, \beta, \delta}f_0\right\|_{L^\infty_{x,v}}t\right)
    \end{align*}
    for $0\leq t\leq T^*$.
\end{theorem}
\begin{remark} \label{contrast2}
	(1) Theorem \ref{thm:2} and Theorem \ref{thm:boltz} provide a stark contrast between the BGK model and the Boltzmann equation in the sense that for any $1 \leq p, q \leq \infty$, the Boltzmann equation \eqref{eq:Boltzmann} is locally well-posed in ($w_{\delta}$-weighted) $L^q_vL^p_x\cap L^{\infty}_{x,v}$ by Theorem \ref{thm:boltz}, while the BGK model \eqref{BGK} is ill-posed in ($w_{\delta}$-weighted) $L^q_vL^p_x\cap L^{\infty}_{x,v}$ by Theorem \ref{thm:2}.\\
	(2) For the existence and uniqueness of the Boltzmann solution, we refer to some classical works \cite{Seiji Ukai, DL, IS, KS, Grad, Carlemann} and a well-known review paper \cite{V}. We also refer to some well-posedness results \cite{DHWY, DuanAdv,DKL2019, KLP} and blow-up results \cite{ISN,ACI}. In particular, these well-posed solutions have a Gaussian tail with extra polynomial decay in velocity, i.e., $\delta > 0$ and $\beta=2$ in \eqref{def_w}. For the existence results in the mixed Strichartz type norm, we refer to \cite{Arsenio,HJ,TRN1,TRN2}. And \cite{BGS}, which demonstrates the non-existence of self-similar type singularities of Landau, Vlasov-Poisson-Landau, and non-cutoff Boltzmann equation is also noteworthy. \\
	(3) Theorem \ref{thm:boltz} extends well-posedness results to $\beta \leq 2$ in (weighted) $L^{q}_{v}L^{p}_{x}\cap L^{\infty}_{x,v}$ for the Boltzmann equation \eqref{eq:Boltzmann}. It covers all cut-off type potential \eqref{Collision_kernel}. The key argument is to apply a new dyadic decomposition type argument in $Q_{\text{gain}}$ estimates and to use Riesz-Thorin interpolation. \\
	(4) The Boltzmann equation can also be ill-posed in other function spaces. We refer to recent interesting results \cite{CH,CSZ} where the authors showed well/ill-posedness results of the Boltzmann equation in $H^s$ Sobolev space. 
\end{remark}

\begin{remark}
    As noted in Remark \ref{contrast2} (3), the condition on $\beta$ in \eqref{index:weight} can be viewed as an extension of the previous works on $\beta = 2$, which is the Gaussian weight case. In Lemma \ref{lem:int_bound4} and \ref{lem:Q_gainL1}, which are the key lemmas in the proof of Theorem \ref{thm:boltz}, a crucial step is, heuristically, to control
    \begin{align*}
        \|w_{\alpha,\beta,\gamma}Q_{gain}(w_{\alpha,\beta,\gamma}^{-1},w_{\alpha,\beta,\gamma}^{-1})\|_{L^\infty_v}.
    \end{align*}
    This naturally leads to a competition between the growth factor $|v|^\kappa$ from $B(v-u,\omega)$ and the decay factor from $\beta$ and $\delta$ in the weight $w^{-1}_{\alpha,\beta,\delta}(v) = (1+|v|^2)^{-\delta} e^{-\alpha|v|^\beta}$. For this reason, we gave enough $\delta$ and the lower bound on $\beta$ depending on $\kappa$ in \eqref{index:weight}. The precise constraints on $\beta$ and $\delta$ are obtained through explicit estimates in the proofs; we refer to Lemma \ref{lem:int_bound4}, Proposition \ref{prop:wellposed}, and Lemma \ref{lem:Q_gainL1}. In particular, when $N=2$, the condition $1+\kappa \leq \beta$ is from the estimate \eqref{5_9:J_12_9}.
\end{remark}

The exponential weight has a strong reminiscence of the moment propagation property of the Boltzmann equation. In spatially homogeneous Boltzmann equation, polynomial or exponential moment propagation has been studied in various collision kernel settings, following the Povzner-type estimates. In particular, the propagation of exponential-type moments in angular non-cutoff settings has been widely researched in recent studies. A typical result can be written as follows: if a solution of the spatially homogeneous Boltzmann equation $f(x,v)$ satisfies $\|f_0 e^{\alpha_0 |v|^\beta}\|_{L^1_v}<\infty$ or $\|f_0 e^{\alpha_0 |v|^\beta}\|_{L^\infty_v}<\infty$, then there exists $0<\alpha$ such that $\|f(t) e^{\alpha |v|^\beta}\|_{L^1_v}<C$ or $\|f(t) e^{\alpha |v|^\beta}\|_{L^\infty_v}<C$ for some constant $C$ for all $t>0$. We refer to \cite{Povzner, Elmroth, Desvillettes, Bobylev, GPV, L1Mouhot, A-G, LM} for an angular cutoff case and \cite{ACGM, LM, TAGP, GPT, Fournier, CHJ} for an angular non-cutoff. For the Landau equation, there are some similar results on the moment propagation; we refer to \cite{V2, Desvillettes2, Desvillettes3, NGLS}. In \cite{NGLS}, the authors obtain the Gaussian upper bound preserving exponent coefficient $\alpha$.

There are fewer results on moment estimates in the spatially inhomogeneous Boltzmann equation. We simply refer to \cite{GPV, GMM, IMS, CS, HF}. For the Landau case, the Gaussian upper bound can be obtained using a maximum principle (See \cite{CSS}). Most studies assume the macroscopic fields are uniformly bounded in time.

The main contribution of Theorem \ref{thm:boltz} from existing exponential moments propagation results can be summarized as follows.
(1) We propagate the exponential weight with some polynomial $(1+|v|)^\delta e^{\alpha |v|^\beta}$ with no loss on the coefficients $\alpha$, $\beta$, and $\delta$ for local in time. We also stress that the propagation remains even when the collision kernel is soft potential; see \eqref{Collision_kernel}.
(2) Thanks to Riesz-Thorin type estimates, we construct exponential moments propagation in $L^q_vL^p_x$ space for all $1\leq p,q\leq \infty$. If we further assume $p\leq q$, the result can be extended to $L^p_xL^q_v$ space using Minkowski's integral inequality.
(3) Our technique can be applied to spatially inhomogeneous settings with no assumption on the macroscopic fields.\\

\noindent \textbf{Organization} The paper is organized as follows. In the following Section 2, we provide the proof of Theorem \ref{thm:1}. In the section, we explicitly construct counterexamples in each case $\beta>2$, $\beta = 2$, and $0<\beta\leq 2$ in general function spaces $1\leq p \leq \infty$. After some technical lemmas in Section 3, we prove the spatially inhomogeneous ill-posedness theory in Section 4, including extra possible polynomial velocity weight. In Section 5, we prove the well-posedness theory of the cutoff hard or soft potential Boltzmann equation \eqref{eq:Boltzmann}, which contrasts sharply with the ill-posedness theory of the BGK as was mentioned in Remark \ref{contrast2}. In the last Section 6, using the well-known $L^1$ convergence theory of the spatially homogeneous Boltzmann equation, we construct a solution in $e^{\alpha|v|^2}$ weighted space that blows up as $t\rightarrow \infty$.

\section{Ill-posedness theory of spatially homogeneous BGK}

In this section, we prove Theorem \ref{thm:1} on the ill-posedness property of the spatially homogeneous BGK model, which is an immediate consequence of the following unboundedness of the relaxation operator in the exponential class:

\begin{theorem}\label{thm:2_1}
    Let $N\geq 1$, $\alpha>0$, $\beta>0$, and $1\leq p\leq \infty$ be given. Then the following statements hold:
    \begin{enumerate}
        \item Let $\beta \geq 2$. Then, for any any fixed $\alpha' > 0$, we can find a non-negative measurable function $f(v)$ such that
        \[
            \|e^{\alpha |v|^\beta}f\,\|_{L_v^p}<\infty,\quad \text{but}\quad \big\|e^{\alpha'|v|^\beta}\mathcal{M}(f)\,\big\|_{L_{v}^p} = \infty.
        \]
        
        \item Let $0 < \beta \leq 2$. Then,  for any fixed $\alpha' > 0$, there exists a sequence of 
        non-negative measurable functions $\{f_n\}_{n=1}^{\infty}$
        such that 
        \[
            \sup_n \|e^{\alpha|v|^\beta}f_n\,\|_{L^p_v}<\infty,\quad \text{but}\quad \lim\limits_{n\rightarrow\infty} \|e^{\alpha'|v|^\beta}\mathcal{M}(f_n)\,\|_{L^p_v} = \infty.
        \]
    \end{enumerate}

\end{theorem}

\begin{remark}
    If $f$ is non-negative and satisfies $\|e^{\alpha|v|^\beta}f\,\|_{L^p}<\infty$, then the local Maxwellian $\mathcal{M}(f)$ constructed from $f$ is well defined as a $L^p$ function. See \cite{P-P,Zhang}.	
\end{remark}

The proof of Theorem \ref{thm:1} is a simple consequence of Theorem \ref{thm:2_1}.
\begin{proof} [\textbf{Proof of Theorem \ref{thm:1}}] 
    Consider the BGK model \eqref{BGK} with initial data $f_0(v)$. Since our choice of the initial data is spatially homogeneous, the time evolution of the distribution function is spatially homogeneous. Therefore, the corresponding solution $f$ satisfies the spatially homogeneous BGK model:
	\begin{align}\label{hom BGK}
		\partial_t f=\mathcal{M}(f)-f,\qquad f(0,v)=f_{0}.
	\end{align}
    We will choose $f_0$ properly to make the ill-posedness for each $\beta$ case.

	For any $C>0$ and a solution $f(t, x, v)$ of BGK model \eqref{BGK}, set $f_C = C^{-1}f(t,v)$. Since $\mathcal{M}(f_C) = C^{-1}\mathcal{M}(f)$, $f_C$ is again a solution of \eqref{BGK}. Therefore, it is enough to assume that $\|e^{\alpha|v|^\beta }f_{0}\|_{L^p_v}<\infty$  or $\sup_n \|e^{\alpha|v|^\beta }f_{0, n}\|_{L^p_v}<\infty$ in Theorem \ref{thm:1} (1) and (2) respectively.
	
	In the spatially homogeneous case, the macroscopic fields are conserved in time so that
	\[
	  \mathcal{M}(f)=\mathcal{M}(f_{0}).
	\]
	Therefore, \eqref{hom BGK} is simplified further as
	\begin{align*}
		\partial_t f=\mathcal{M}(f_{0})-f_n,\qquad f(0,v)=f_{0},
	\end{align*}
	which can be solved explicitly:
	\begin{align*}
		f(v,t)=e^{-t}f_{0}+(1-e^{-t})\mathcal{M}(f_{0}).
	\end{align*}
	Taking the weighted $L^p$ norm on both sides, we have
	\begin{align*}
		\big\|e^{\alpha|v|^\beta}f(t)\big\|_{L^{p}_v}\geq (1-e^{-t})\big\|e^{\alpha|v|^\beta}\mathcal{M}(f_{0})\big\|_{L^p}.
	\end{align*}
    Thus, we conclude Theorem \ref{thm:1} by taking $f_0$ as in Theorem \ref{thm:2_1}.
\end{proof}

\subsection{The proof of Theorem \ref{thm:2_1}\nopunct} 	

We now move on to the proof of Theorem \ref{thm:2_1}, which is divided into the three cases:
$(\beta>2)$, $(\beta = 2)$, and $(0<\beta<2)$. Theorem \ref{thm:2_1} follows from the combination of 
Proposition \ref{prop:2_1} $(\beta>2)$, Proposition \ref{prop:2_2}  $(\beta = 2)$, and Proposition \ref{prop:2_3}  $(0<\beta<2)$ below.\\

\noindent$\bullet$ {\bf The case of $\beta>2$:} 
The proof for the case $\beta>2$ is almost trivial since the exponent in the exponential weight is strictly bigger than that of the local Maxwellian in the far field:
\begin{proposition}\label{prop:2_1}
    Let $N\geq 1$, $\alpha'>0$, $\beta>2$ and $1\leq p\leq \infty$. For any $f(v)$ with nonzero $\rho$ and $T$, we have $\big\|e^{\alpha' |v|^\beta}\mathcal{M}(f)\big\|_{L_v^p} = \infty$.
\end{proposition}
\begin{proof}
    Note that
    \begin{align*}
        e^{\alpha' |v|^\beta}\mathcal{M}(f) = \frac{\rho}{(2\pi T)^{N/2}}\exp\left(-\frac{|v-u|^2}{2T}+\alpha' |v|^\beta\right),
    \end{align*}
    and the exponent
       $ -|v-u|^2/2T+\alpha' |v|^\beta$
    blows up as $v\rightarrow \infty$ for $\beta>2$. This gives the desired result.
\end{proof}

\noindent$\bullet$ {\bf The case of $\beta=2$:}
In this case, there is a competition between the growth of the weight and the decay in the Maxwellian distribution. Therefore, we have to carefully compare the growth of the exponential weight with the decay speed of the local Maxwellian, which is determined by the macroscopic temperature. We start with two technical lemmas. We mention that Lemma \ref{lem:Tri_integral} will be frequently used in Sections 2 and 4.
\begin{lemma}\label{lem:Tri_integral}
    For $\alpha,\beta>0$ and $n\geq 0$, we have
    \begin{align*}
        \lim_{x\rightarrow \infty}\left(\frac{1}{\alpha\beta} e^{-\alpha x^\beta}x^{(n+1)-\beta}\right)^{-1}\int_x^\infty e^{-\alpha r^\beta} r^n \,dr = 1.
    \end{align*}
\end{lemma}

\begin{proof}
    By the change of variable $\alpha r^\beta\mapsto t$, we have
    \begin{align*}
        \int_x^\infty e^{-\alpha r^\beta} r^n \,dr &=\int_{\alpha x^\beta}^\infty e^{-t} \left(\frac{t}{\alpha}\right)^{\frac{n}{\beta}} \frac{1}{\beta\alpha^{1/\beta}} t^{\frac{1-\beta}{\beta}}\,dt\\
        &=\frac{1}{\beta\alpha^{(n+1)/\beta}} \int_{\alpha x^\beta}^\infty e^{-t}t^{\frac{(n+1)-\beta}{\beta}}\,dt.
    \end{align*}
    Using Tricomi's relation \cite{Tricomi}:
    \begin{align*}
        \int_{x}^\infty e^{-t} t^{a}\,dt = \frac{e^{-x}x^{a+1}}{x-a}\left[1-\frac{a}{(x-a)^2}+\frac{2a}{(x-a)^3} + O\left(\frac{a^2}{(x-a)^4}\right)\right]\quad (x\rightarrow \infty),
    \end{align*}
    it completes the proof.
\end{proof}
The next lemma is a consequence of Lemma \ref{lem:Tri_integral}.
\begin{lemma}\label{lem:Tinfty_asym}
    For $N\geq 1$, $\alpha,\beta>0$ and $n\geq 0$, let
    \begin{equation*}
        f_n(v) = C(n)e^{-\alpha |v|^\beta}\mathbf{1}_{n\leq |v|}\quad \textrm{or}\quad C(n)e^{-\alpha |v|^\beta}\mathbf{1}_{n\leq |v|\leq n^2},
    \end{equation*}
    where $C(n)$ is a constant depending only on $n$. Then, the macroscopic fields of $f_n$ satisfy
    \begin{align*}
        \lim_{n\rightarrow \infty} \left(\frac{|\mathbb{S}^{N-1}|C(n) n^{N-\beta}}{\alpha\beta}e^{-\alpha n^\beta}\right)^{-1}\rho_n = 1,
        \quad u_n=0,\quad
        \lim_{n\rightarrow \infty} \left(\frac{n^2}{N}\right)^{-1}T_n = 1. 
    \end{align*}
    where $|\mathbb{S}^{N-1}|$ is the volume of $N-1$ sphere.
\end{lemma}

\begin{proof}
    We only prove for $f_n(v) = C(n)e^{-\alpha |v|^\beta}\mathbf{1}_{n\leq |v|\leq n^2}$, which is slightly harder. Since $f_n$ is isotropic about $v$, $u_n = 0$. For $\rho_n$ and $T_n$, we have
    \begin{align*}
        \rho_n &= C(n)\int_{\mathbb{R}^N}e^{-\alpha |v|^\beta}\mathbf{1}_{n\leq |v|\leq n^2}\,dv = C(n)|\mathbb{S}^{N-1}|\int_n^{n^2}e^{-\alpha r^\beta}r^{N-1}\,dr,\\
        NT_n &= \frac{C(n)}{\rho_n}\int_{\mathbb{R}^N}|v|^2e^{-\alpha |v|^\beta}\mathbf{1}_{n\leq |v|\leq n^2}\,dv=\frac{\int_n^{n^2}e^{-\alpha r^\beta}r^{N+1}\,dr}{\int_n^{n^2}e^{-\alpha r^\beta}r^{N-1}\,dr}.
    \end{align*}
  We then apply Lemma \ref{lem:Tri_integral},
    \begin{align*}
        &\lim_{n\rightarrow \infty} \frac{\rho_n}{\frac{|\mathbb{S}^{N-1}|C(n) n^{N-\beta}}{\alpha\beta}e^{-\alpha n^\beta}} = \lim_{n\rightarrow \infty} \frac{\int_n^\infty e^{-\alpha r^\beta}r^{N-1}\,dr}{\frac{n^{N-\beta}}{\alpha\beta}e^{-\alpha n^\beta}} - \frac{n^{2(N-\beta)}e^{-\alpha n^{2\beta}}}{n^{N-\beta}e^{-\alpha n^\beta}}\frac{\int_{n^2}^\infty e^{-\alpha r^\beta}r^{N-1}\,dr}{\frac{n^{2(N-\beta)}}{\alpha\beta}e^{-\alpha n^{2\beta}}} = 1,\\
        &\lim_{n\rightarrow \infty} \frac{NT_n}{n^2} =\lim_{n\rightarrow \infty} \frac{\int_n^\infty e^{-\alpha r^\beta}r^{N+1}\,dr}{\frac{n^{N-\beta+2}}{\alpha\beta}e^{-\alpha n^\beta}} - \frac{n^{2(N-\beta+2)}e^{-\alpha n^{2\beta}}}{n^{N-\beta+2}e^{-\alpha n^\beta}}\frac{\int_{n^2}^\infty e^{-\alpha r^\beta}r^{N+1}\,dr}{\frac{n^{2(N-\beta+2)}}{\alpha\beta}e^{-\alpha n^{2\beta}}} = 1.
    \end{align*}
    This yields the desired result.
\end{proof}

Now we prove Theorem \ref{thm:2_1} for $\beta = 2$. In fact, the following statements explain more than what is stated in Theorem \ref{thm:2_1} (2).
\begin{proposition}\label{prop:2_2}
    Let $N\geq 1$, $\alpha>0$, $\beta = 2$, and $1\leq p\leq \infty$. \\
    \begin{enumerate}
        \item For any $\alpha'>0$, there exists a function $f$ such that $\left\|e^{\alpha|v|^2}f\right\|_{L^p_v}<\infty$, but $\left\|e^{\alpha' |v|^2}\mathcal{M}(f)\right\|_{L^p_v}=\infty$, satisfying $\alpha' > \frac{1}{2T}$.
        
        \item For any $\alpha'>0$, there exists a sequence of function $f_n$ such that $\sup_n \left\|e^{\alpha|v|^2}f_n\right\|_{L^p_v}<\infty$, but $\lim\limits_{n\rightarrow\infty}\left\|e^{\alpha' |v|^2}\mathcal{M}(f_n)\right\|_{L^p_v}=\infty$, satisfying $\alpha' < \frac{1}{2T_n}$ for all $n$. Moreover, we have $\inf_n \frac{\alpha'}{1-2T_n\alpha'}>\alpha$ and $|u_n|\rightarrow \infty$. 
    \end{enumerate}
\end{proposition}
\begin{proof} (1)
    Define 
    \begin{align*}
        f_n(v) = A_{n,p} e^{-\alpha |v|^2}\mathbf{1}_{n\leq |v|\leq n^2},
    \end{align*}
    where
    \begin{align}\label{2_2:def_Anp}
        A_{n,p} \coloneqq 
        \begin{cases}
            \left(\frac{1}{|B_N(1)|(n^{2N}-n^N)}\right)^{1/p} & 1\leq p<\infty,\\
            1 & p = \infty
        \end{cases}
    \end{align}
    for sufficiently large but fixed $n$, which will be chosen soon. By our choice of $A_{n,p}$, we have $\left\|e^{\alpha |v|^\beta} f_n\right\|_{L^p_v} = 1$. Thanks to Lemma \ref{lem:Tinfty_asym}, we can choose $n$ large enough such that
    \begin{align*}
        \rho_n\geq \frac{1}{2}A_{n,p}\frac{|\mathbb{S}^{N-1}| n^{N-2}}{2\alpha}e^{-\alpha n^2},
        \quad u_n=0,\quad
        \frac{n^2}{2N}\leq T_n\leq \frac{2n^2}{N}.
    \end{align*}
    Using these bounds, we have
    \begin{align*}
        e^{\alpha' |v|^2}\mathcal{M}(f_n) &=\frac{\rho}{(2\pi T_n)^{N/2}}\exp\left\{-\left(\frac{|v-u_n|^2}{2T_n}-\alpha'|v|^2\right)\right\}\\
        &\geq \frac{A_{n,p}|\mathbb{S}^{N-1}|}{4\alpha}\frac{n^{N-2}e^{-\alpha n^2}}{(2\pi \frac{N}{2n^2})^{N/2}}\exp\left\{-\left(\frac{|v|^2}{2\frac{n^2}{2N}}-\alpha'|v|^2\right)\right\}\\
        &=\frac{A_{n,p}|\mathbb{S}^{N-1}|}{4\alpha}\frac{n^{N-2}e^{-\alpha n^2}}{(2\pi \frac{N}{2n^2})^{N/2}}
        \exp\left\{\left(\alpha'-\frac{N}{n^2}\right)|v|^2\right\}.
    \end{align*} 
    \noindent Choosing $n$ sufficiently large to satisfy  
    \begin{align*}
        \alpha'-\frac{N}{n^2}>0,
    \end{align*}
    we get the desired result:
    \begin{align*}
        \left\|e^{\alpha' |v|^2}\mathcal{M}(f_n)\right\|_{L^p_v} = \infty.
    \end{align*}\\
    
    \noindent (2) For a fixed $\alpha'>0$, choose $T$ satisfying $0<\alpha' T<1/2$, but $\frac{\alpha'}{1-2T\alpha'}>\alpha$. Let us define $\vec{n} = (0,\ldots, 0, -n)$ for $n\geq 1$ and
    \begin{align}\label{2_2:f_n}
        f_n(v) = \frac{NT}{n^2|B_N(1)|}\left(A'_{n,p}\int_{\mathbb{R}^N}dw\,e^{-\alpha|w|^2}\mathbf{1}_{|w-\vec{n}|\leq 1/n^2}\right)\mathbf{1}_{|v|\leq 1} + A'_{n,p}e^{-\alpha|v|^2}\mathbf{1}_{|v-\vec{n}|\leq 1/n^2},
    \end{align}
    where 
    \begin{align*}
    A'_{n,p} \coloneqq 
        \begin{cases}
            \left(\frac{1}{|B_N(1)|n^{2N}}\right)^{1/p} & 1\leq p<\infty,\\
            1 & p = \infty.
        \end{cases}
    \end{align*}
    For $n\geq 1$, we obtain
    \begin{align*}
        \left\|e^{\alpha|v|^2}f_n\right\|_{L^p_v}&\leq \frac{NT}{n^2|B_N(1)|}\left(A'_{n,p}\int_{\mathbb{R}^N}dw\,e^{-\alpha|w|^2}\mathbf{1}_{|w-\vec{n}|\leq 1/n^2}\right)\left\|e^{\alpha|v|^2}\mathbf{1}_{|v|\leq 1}\right\|_{L^p_v} + A'_{n,p}\left\|\mathbf{1}_{|v-\vec{n}|\leq 1/n^2}\right\|_{L^p_v}\\
        &\leq \frac{NT}{n^2|B_N(1)|}\left(A'_{n,p}\int_{\mathbb{R}^N}dw\,e^{-\alpha|w|^2}\mathbf{1}_{|w-\vec{n}|\leq 1/n^2}\right)e^{\alpha}|B_N(1)| + A'_{n,p}\left\|\mathbf{1}_{|v-\vec{n}|\leq 1/n^2}\right\|_{L^p_v}\\
        &\leq \frac{e^{\alpha}NT}{n^2} + 1
    \end{align*}
    as $A'_{n,p}\leq 1$ and $\int_{\mathbb{R}^N}e^{-\alpha|v|^2}\mathbf{1}_{|v-\vec{n}|\leq 1/n^2}\,dv\leq 1$.
    We first calculate the macroscopic fields of $f_n$. The macroscopic density $\rho_n$ satisfies
    \begin{equation}\label{2_2:rho_bound}
        \begin{split}
            \rho_n &= \int_{\mathbb{R}^N}dv\,\frac{NT}{n^2|B_N(1)|}\left(A'_{n,p}\int_{\mathbb{R}^N}dw\,e^{-\alpha|w|^2}\mathbf{1}_{|w-\vec{n}|\leq 1/n^2}\right)\mathbf{1}_{|v|\leq 1} + \int_{\mathbb{R}^N} dv\,A'_{n,p}e^{-\alpha|v|^2}\mathbf{1}_{|v-\vec{n}|\leq 1/n^2}\\
            &=\left(\frac{NT}{n^2}\frac{1}{|B_N(1)|}\int_{\mathbb{R}^N} \mathbf{1}_{|v|\leq 1}\,dv + 1\right)A'_{n,p}\int_{\mathbb{R}^N}dv\,e^{-\alpha|v|^2}\mathbf{1}_{|v-\vec{n}|\leq 1/n^2}\\
            &=\left(\frac{NT}{n^2} + 1\right)A'_{n,p}\int_{\mathbb{R}^N}dv\,e^{-\alpha|v|^2}\mathbf{1}_{|v-\vec{n}|\leq 1/n^2}.
        \end{split}
    \end{equation}
    Now, since the integral in \eqref{2_2:rho_bound} is bounded above and below by
    \begin{equation*}
        \begin{split}
            \int_{\mathbb{R}^N}dv\,A'_{n,p}e^{-\alpha|v|^2}\mathbf{1}_{|v-\vec{n}|\leq 1/n^2}&\leq A'_{n,p}e^{-\alpha(n-1/n^2)^2}\int_{\mathbb{R}^N}dv\,\mathbf{1}_{|v-\vec{n}|\leq 1/n^2} \\
            &=\frac{|B_N(1)|}{n^{2N}}A'_{n,p}e^{-\alpha(n^2-2/n + 1/n^4)},\\
            \int_{\mathbb{R}^N}dv\,A'_{n,p}e^{-\alpha|v|^2}\mathbf{1}_{|v-\vec{n}|\leq 1/n^2}&\geq A'_{n,p}e^{-\alpha(n+1/n^2)^2}\int_{\mathbb{R}^N}dv\,\mathbf{1}_{|v-\vec{n}|\leq 1/n^2}\\
            &= \frac{|B_N(1)|}{n^{2N}}A'_{n,p}e^{-\alpha(n^2+2/n + 1/n^4)},
        \end{split}
    \end{equation*}
    we obtain
    \begin{align*}
        \lim_{n\rightarrow\infty}\left(\frac{|B_N(1)|A'_{n,p}}{n^{2N}}\left(\frac{NT}{n^2}+1\right)e^{-\alpha n^2}\right)^{-1}\rho_n = 1.
    \end{align*}
    We then estimate $\rho_n(u_n-\vec{n})$:
    \begin{equation}\label{2_2:u_integral}
        \begin{split}
            |\rho_n (u_n-\vec{n})|&\leq \int_{\mathbb{R}^N}dv\,|v-\vec{n}|f_n(v)\\
            &= \int_{\mathbb{R}^N}dv\,\frac{NT}{n^2|B_N(1)|}\left(A'_{n,p}\int_{\mathbb{R}^N}e^{-\alpha|w|^2}\mathbf{1}_{|w-\vec{n}|\leq 1/n^2}\,dw\right)|v-\vec{n}|\mathbf{1}_{|v|\leq 1}\\
            &\quad + \int_{\mathbb{R}^N}dv\,A'_{n,p}|v-\vec{n}|e^{-\alpha|v|^2}\mathbf{1}_{|v-\vec{n}|\leq 1/n^2}\\
            &\leq \int_{\mathbb{R}^N}dv\,\frac{NT}{n^2|B_N(1)|}\left(A'_{n,p}\int_{\mathbb{R}^N}e^{-\alpha|w|^2}\mathbf{1}_{|w-\vec{n}|\leq 1/n^2}\,dw\right)(|v|+n)\mathbf{1}_{|v|\leq 1}\\
            &\quad + \int_{\mathbb{R}^N}dv\,A'_{n,p}|v-\vec{n}|e^{-\alpha|v|^2}\mathbf{1}_{|v-\vec{n}|\leq 1/n^2}\\
            &\leq \left(\frac{NT}{n^2}(n+1) + \frac{1}{n^2}\right)A'_{n,p}\int_{\mathbb{R}^N}dv\,e^{-\alpha|v|^2}\mathbf{1}_{|v-\vec{n}|\leq 1/n^2}.
        \end{split}
    \end{equation}
    
    Therefore, we have
    \begin{align}\label{2_2:u_bound}
        |u_n-\vec{n}|\leq \frac{NT(n+1)+1}{NT+n^2}
    \end{align}
    from \eqref{2_2:u_integral} and \eqref{2_2:rho_bound}.
    
    We now turn to the estimate of $N\rho_n T_n$. For this, we decompose the integral into
    \begin{align*}
        N\rho_n T_n = \int_{\mathbb{R}^N}dv\,|v-u_n|^2 f_n(v) = \int_{\mathbb{R}^N}dv\,(|v-\vec{n}|^2 + 2(v-\vec{n})\cdot (\vec{n}-u_n) + |\vec{n}-u_n|^2) f_n(v).
    \end{align*}
    The integral involving $|v-\vec{n}|^2$ is bounded from above and below as follows:
    \begin{align*}
        \int_{\mathbb{R}^N}dv\,|v-\vec{n}|^2 f_n(v)
        &\leq\int_{\mathbb{R}^N}dv\, \frac{NT}{n^2|B_N(1)|}\left(A'_{n,p}\int_{\mathbb{R}^N}dw\,e^{-\alpha|w|^2}\mathbf{1}_{|w-\vec{n}|\leq 1/n^2}\right)(|v|+n)^2 \mathbf{1}_{|v|\leq 1}\\
        &\quad +\int_{\mathbb{R}^N}dv\,A'_{n,p}|v-\vec{n}|^2 e^{-\alpha|v|^2}\mathbf{1}_{|v-\vec{n}|\leq 1/n^2}\\
        &\leq \left(NT\frac{(n+1)^2}{n^2}+\frac{1}{n^4}\right)A'_{n,p}\int_{\mathbb{R}^N}dv\,e^{-\alpha|v|^2}\mathbf{1}_{|v-\vec{n}|\leq 1/n^2}
        \end{align*}
    and
    \begin{align*}
        \int_{\mathbb{R}^N}dv\,|v-\vec{n}|^2 f_n(v)
        &\geq\int_{\mathbb{R}^N}dv\, \frac{NT}{n^2|B_N(1)|}\left(A'_{n,p}\int_{\mathbb{R}^N}dw\,e^{-\alpha|w|^2}\mathbf{1}_{|w-\vec{n}|\leq 1/n^2}\right)(n-|v|)^2 \mathbf{1}_{|v|\leq 1}\\
        &\geq NT\frac{(n-1)^2}{n^2}A'_{n,p}\int_{\mathbb{R}^N}dv\,e^{-\alpha|v|^2}\mathbf{1}_{|v-\vec{n}|\leq 1/n^2}.
    \end{align*}
   On other hand, using \eqref{2_2:u_bound} and \eqref{2_2:u_integral}, we bound the term containing $(v-\vec{n})\cdot (\vec{n}-u_n)$ as follows:
    \begin{align*}
        \bigg| \int_{\mathbb{R}^N}dv\, (v-\vec{n})\cdot (\vec{n}-u_n) f_n(v)\bigg|
        &\leq \int_{\mathbb{R}^N}dv\, |v-\vec{n}||\vec{n}-u_n| f_n(v)\\
        &\leq \frac{NT(n+1)+1}{NT+n^2}\int_{\mathbb{R}^N}dv\, |v-\vec{n}|f_n(v)\\
        &\leq \frac{(NT(n+1)+1)^2}{(NT+n^2)n^2}A'_{n,p}\int_{\mathbb{R}^N}dv\,e^{-\alpha|v|^2}\mathbf{1}_{|v-\vec{n}|\leq 1/n^2}.
    \end{align*}
    Lastly, we use \eqref{2_2:u_bound} to bound the last integral containing $|\vec{n}-u_n|^2$ by
    \begin{align*}
        \int_{\mathbb{R}^N}dv\, |\vec{n}-u_n|^2 f_n(v)\leq \left(\frac{NT(n+1)+1}{NT+n^2}\right)^2\int_{\mathbb{R}^N}dv\,f_n(v) = \left(\frac{NT(n+1)+1}{NT+n^2}\right)^2\rho_n.
    \end{align*}
    Combining all three terms and dividing by \eqref{2_2:rho_bound}, we obtain
    \begin{align*}
        NT_n&\leq \left(\frac{NT}{n^2}+1\right)^{-1}\left(NT\frac{(n+1)^2n^2+1}{n^4} + 2\frac{(NT(n+1)+1)^2}{(NT+n^2)n^2} + \frac{(NT(n+1)+1)^2}{(NT+n^2)n^2}\right),\\
        NT_n&\geq \left(\frac{NT}{n^2}+1\right)^{-1}\left(NT\frac{(n-1)^2n^2+1}{n^4} - 2\frac{(NT(n+1)+1)^2}{(NT+n^2)n^2} - \frac{(NT(n+1)+1)^2}{(NT+n^2)n^2}\right),
    \end{align*}
    from which we conclude that
    \begin{align*}
        \lim_{n\rightarrow \infty} NT_n = NT.
    \end{align*}
    
    Finally, we can choose large enough $M$ such that $\alpha'<\frac{1}{2T_n}$ and $\frac{\alpha'}{1-2T_n\alpha'}>\alpha+\epsilon$ for some $\epsilon>0$ for all $n\geq M$.
    Summarising the estimate of the macroscopic fields of $f_n$ so far, we have
    \begin{equation}\label{2_2:macro}
        \begin{split}
            &\lim_{n\rightarrow\infty}\left(\frac{|B_N(1)|A'_{n,p}}{n^{2N}}\left(\frac{NT}{n^2}+1\right)e^{-\alpha n^2}\right)^{-1}\rho_n = 1,\\
            &n - \frac{NT(n+1)+1}{NT+n^2}\leq |u_n|\leq n + \frac{NT(n+1)+1}{NT+n^2},\\
            &\lim_{n\rightarrow \infty} T_n = T,\quad \alpha'<\inf_{n\geq M}\frac{1}{2T_n},\quad \inf_{n\geq M}\frac{\alpha'}{1-2T_n\alpha'}>\alpha+\epsilon
        \end{split}
    \end{equation}
    for some $M$ and $\epsilon>0$.
    
    Now, we first prove the proposition for the $p=\infty$ case. The maximum of $e^{\alpha'|v|^2}\mathcal{M}(f)(v)$ is attained at
    \begin{align*}
        v = \frac{1}{1-2T\alpha'}u,
    \end{align*}
    and the maximum value is
    \begin{align*}
        \left\|e^{\alpha' |v|^2}\mathcal{M}(f)\right\|_{L^\infty_v} = \frac{\rho}{(2\pi T)^{N/2}}\exp\left\{\frac{\alpha'}{(1-2T\alpha')}u^2\right\}.
    \end{align*}
    
    In view of the macroscopic estimate \eqref{2_2:macro}, we obtain
    \begin{equation}\label{2_2:blowup}
        \begin{split}
            &\liminf_{n\rightarrow \infty}\left\|e^{\alpha' |v|^2}\mathcal{M}(f_n)\right\|_{L^\infty_v} = \liminf_{n\rightarrow \infty}\frac{\rho_n}{(2\pi T_n)^{N/2}}\exp\left\{\frac{\alpha'}{(1-2T_n\alpha')}u_n^2\right\}\\
            &\geq \liminf_{n\rightarrow \infty}\frac{|B_N(1)|}{n^{2N}(2\pi T)^{N/2}}\left(\frac{NT}{n^2}+1\right)\exp\left\{\frac{\alpha'}{(1-2T_n\alpha')}\left(n - \frac{NT(n+1)+1}{NT+n^2}\right)^2 -\alpha n^2\right\}\\
            &\geq \frac{C(N, T)}{n^{2N}}\exp\left\{(\alpha+\epsilon)(n^2-3NT) -\alpha n^2\right\}
        \end{split}
    \end{equation}
    for some constant $C(N, T)$. In the final step, we used
    \begin{align*}
        \lim_{n\rightarrow \infty}2n\frac{NT(n+1)+1}{NT+n^2}\leq 3NT.
    \end{align*}
    Since RHS of \eqref{2_2:blowup} blows up as $n\rightarrow \infty$, we get the desired norm blow-up.\\
            
    Let us move on to the case $1\leq p<\infty$. Using cylindrical coordinate aligning $-u_n$ to $z$ axis, we compute $\left\|e^{\alpha' |v|^2}\mathcal{M}(f_n)\right\|^p_{L_v^p}$ with $0<\alpha' T<1/2$ as follows:
    \begin{align*}
        &\left\|e^{\alpha' |v|^2}\mathcal{M}(f_n)\right\|^p_{L^p_v} \\
        &=\frac{\rho_n^p}{(2\pi T_n)^{Np/2}}\int_{\mathbb{R}^N}\exp\left\{-\frac{p|v-\vec{n}|^2}{2T_n} + p\alpha' |v|^2\right\}\,dv\\
        &=\frac{|\mathbb{S}^{N-2}|\rho_n^p}{(2\pi T_n)^{pN/2}}\int_{-\infty}^\infty \int_0^\infty \exp\left\{-p\frac{r^2+(z+u_n)^2}{2T_n} + p\alpha' (r^2+z^2)\right\}r^{N-2}\,dr dz\\
        &=\frac{|\mathbb{S}^{N-2}|\rho_n^p}{(2\pi T_n)^{pN/2}}\int_{-\infty}^\infty\exp\left\{p\left(-\frac{(z+u_n)^2}{2T_n} + \alpha' z^2\right)\right\}\,dz\int_0^\infty \exp\left\{p\left(-\frac{1}{2T_n} + \alpha' \right)r^2\right\}r^{N-2}\,dr\\
        &=\frac{|\mathbb{S}^{N-2}|\rho_n^p}{(2\pi T_n)^{pN/2} p^{N/2}}\exp\left(\frac{p\alpha'}{1-2\alpha' T_n} u_n^2\right)\left(-\alpha'+\frac{1}{2T_n}\right)^{-N/2}\int_{-\infty}^\infty\exp\left(-z^2\right)\,dz\int_0^\infty\exp\left(-r^2\right)r^{N-2}\,dr\\
        &=\frac{|\mathbb{S}^{N-2}|\rho_n^p (2 T_n)^{N/2}}{(2\pi T_n)^{Np/2} \left(p-2p\alpha' T_n\right)^{N/2}}\exp\left(\frac{p\alpha'}{1-2\alpha' T_n} u_n^2\right)\int_{-\infty}^\infty\exp\left(-z^2\right)\,dz\int_0^\infty\exp\left(-r^2\right)r^{N-2}\,dr\\
        &=C(p, \alpha', N) \frac{\rho_n^p}{T_n^{N(p-1)/2}p^{N /2}\left(1-2\alpha' T_n\right)^{N/2}}\exp\left(\frac{p\alpha'}{1-2\alpha' T_n} u_n^2\right)
    \end{align*}
    for some constant $C(p, \alpha' ,N)$. It shows that
    \begin{align*}
        \left\|e^{\alpha' |v|^2}\mathcal{M}(f_n)\right\|_{L^p_v} = C(p, \alpha', N) \frac{\rho_n}{T_n^{N(p-1)/2p}\left(p-2p\alpha' T_n\right)^{N/2p}}\exp\left(\frac{\alpha'}{1-2\alpha' T_n} u_n^2\right).
    \end{align*}
    Using the same sequence of functions \eqref{2_2:f_n} and using \eqref{2_2:macro}, we get a similar blow-up profile of \eqref{2_2:blowup}.
\end{proof}

There is a remark on the proposition.
\begin{remark}
    For $\alpha'=\alpha$, we always have a blow-up result as $\frac{\alpha}{1-2T\alpha}>\alpha$ for $T>0$. For $\frac{\alpha'}{1-2T\alpha'}\leq \alpha$, using the $f_n$ constructed in \eqref{2_2:f_n} and the asymptotic macroscopic fields \eqref{2_2:macro}, we obtain
    \begin{align*}
        &\limsup_{n\rightarrow \infty}\left\|e^{\alpha' |v|^2}\mathcal{M}(f_n)\right\|_{L^\infty_v} = \limsup_{n\rightarrow \infty}\frac{\rho_n}{(2\pi T_n)^{N/2}}\exp\left\{\frac{\alpha'}{(1-2T_n\alpha')}u_n^2\right\}\\
        &\leq \limsup_{n\rightarrow \infty}\frac{|B_N(1)|}{n^{2N}(2\pi T)^{N/2}}\left(\frac{NT}{n^2}+1\right)\exp\left\{\frac{\alpha'}{(1-2T_n\alpha')}\left(n + \frac{NT(n+1)+1}{NT+n^2}\right)^2 -\alpha n^2\right\}\\
        &=0.
    \end{align*}
    From this example, we conjecture that $\frac{\alpha'}{1-2T\alpha'}$ is the critical point so that any solution of the Boltzmann equation can maintain at most exponential weight $e^{\frac{\alpha'}{1-2T\alpha'}}$ globally in time.
\end{remark}

\noindent$\bullet$ {\bf The case of $\beta<2$:}
In contrast to the other cases, the weighted norm $\big\|e^{\alpha|v|^\beta} \mathcal{M}(f)\big\|_{L^p_v}$ is always finite in this case since the decay of the Maxwellian distribution is greater than the growth of exponential weight. As $T\rightarrow \infty$, however, the decay of the Maxwellian distribution becomes weak. So, there is some room to make the exponential norm dominant. We will show that this norm can be, in fact, arbitrarily large compared to $\left\|e^{\alpha|v|^2} f_n(v)\right\|_{L^p_v}$ in isotropic velocity case, i.e., $u=0$ case.

\begin{proposition}\label{prop:2_3}
    Let $N\geq 1$, $\alpha>0$, $\alpha'>0$, $0<\beta<2$, and $1\leq p\leq \infty$. Then, there exists a sequence of functions $f_n$ satisfying
    \begin{align*}
        \sup_n \left\|e^{\alpha|v|^\beta} f_n(v)\right\|_{L^p_v} = 1,\text{ but}\quad \lim_{n\rightarrow \infty} \left\|e^{\alpha'|v|^\beta}\mathcal{M}(f_n)(v)\right\|_{L^p_v} = \infty.
    \end{align*}

\end{proposition}
\begin{proof}
    As in the $\beta = 2$ case, we choose 
    \begin{equation*}
        f_n(v) = A_{n,p}e^{-\alpha |v|^\beta}\mathbf{1}_{n\leq |v|\leq n^2},
    \end{equation*}
    where
    \begin{align*}
        A_{n,p} \coloneqq 
        \begin{cases}
            \left(\frac{1}{|B_N(1)|(n^{2N}-n^N)}\right)^{1/p} & 1\leq p<\infty,\\
            1 & p = \infty.
        \end{cases}
    \end{align*}
    Note that $A_{n,p}$ is chosen to normalize $\left\|e^{\alpha|v|^\beta}f_n(v)\right\|_{L^p_v} = 1$ for all $p$ and $n$. Using Lemma \ref{lem:Tinfty_asym}, we can choose large enough $N(p,\beta)$ such that
    \begin{align}\label{2_3:macro}
        \rho_n\geq \frac{1}{2}A_{n,p}\left(\frac{|\mathbb{S}^{N-1}|n^{N-\beta}}{\alpha\beta}e^{-\alpha n^\beta}\right),\quad  \frac{n^2}{2N}\leq T_n\leq \frac{2n^2}{N}
    \end{align}
    for all $n\geq N(p,\beta)$.
    
    We first prove the $p=\infty$ case, which is easier than the other cases, and then extend the proof to $1\leq p<\infty$.\\
    
    \noindent \textit{(i) The case $p=\infty$.}
    
    We first observe that the function
    \begin{align*}
        e^{\alpha'|v|^\beta}\mathcal{M}(f_n)(v) =\frac{\rho_n}{(2\pi T_n)^{N/2}}\exp\left(\alpha'|v|^\beta-\frac{|v|^2}{2T_n}\right)
    \end{align*} 
    attains its maximum value at $v_n = (\alpha'\beta T_n)^{1/(2-\beta)}$, so
    \begin{align*}
        \begin{split}
            \left\|e^{\alpha'|v|^\beta}\mathcal{M}(f_n)\right\|_{L^\infty_v}&= e^{\alpha'|v|^\beta}\mathcal{M}(f_n)(v_n)\\
            &=\frac{\rho_n}{(2\pi T_n)^{3/2}}\exp\left\{\bigg(\alpha'(\alpha'\beta)^{\beta/(2-\beta)} - \frac{(\alpha'\beta)^{2/(2-\beta)}}{2}\bigg)T_n^{\beta/(2-\beta)}\right\}\\
            & =\frac{\rho_n}{(2\pi T_n)^{3/2}}\exp\left(\frac{\alpha'}{2}(\alpha'\beta)^{\beta/(2-\beta)}\left(2-\beta\right)T_n^{\beta/(2-\beta)}\right).
        \end{split}
    \end{align*}
    
    Therefore, using \eqref{2_3:macro}, we obtain
    \begin{align}
        &\liminf_{n\rightarrow \infty}\left\|e^{\alpha'|v|^\beta}\mathcal{M}(f_n)\right\|_{L^\infty_v}\notag\\
        &\geq \liminf_{n\rightarrow \infty}\frac{\left(\frac{|\mathbb{S}^{N-1}| n^{N-\beta}}{\alpha\beta}e^{-\alpha n^\beta}\right)}{2\left(\frac{4\pi}{N}n^2\right)^{N/2}}\exp\left[{\left(\frac{\alpha'}{2}(\alpha'\beta)^{\beta/(2-\beta)}\left(2-\beta\right)\right)\left(\frac{n^2}{2N}\right)^{\beta/(2-\beta)}}\right]\notag\\
        &= \liminf_{n\rightarrow \infty} \frac{C(\alpha,\beta)}{ n^{\beta}}\exp\left[{\left(\frac{1}{2N}\right)^{\beta/(2-\beta)}
        \left(\frac{\alpha'}{2}(\alpha'\beta)^{\beta/(2-\beta)}\left(2-\beta\right)\right)n^{2\beta/(2-\beta)} - \alpha n^\beta}\right]\label{2_3:blowup},
    \end{align}
    where $C(\alpha,\beta)$ is a constant. Under our assumption $0<\beta<2$, we have
    \begin{align*}
        \frac{\alpha}{2}(\alpha'\beta)^{\beta/(2-\beta)}\left(2-\beta\right)>0,\quad \frac{2\beta}{(2-\beta)}> \beta.
    \end{align*}
    Therefore, 
    \begin{align*}
        \liminf_{n\rightarrow \infty}\left|e^{\alpha|v|^\beta}\mathcal{M}(f_n)(v_n)\right|\geq \liminf_{n\rightarrow \infty}\frac{C(\alpha,\beta)}{n^{\beta}}\exp\left[\frac{1}{2}\frac{(2-\beta)\alpha'(\alpha'\beta)^{\beta/(2-\beta)}}{2(2N)^{\beta/(2-\beta)}}n^{2\beta/(2-\beta)}\right]=\infty.
    \end{align*}
    It completes the proof for the $p=\infty$ case.\newline
   
    \noindent \textit{(ii) The case $1\leq p<\infty$.}
    
    We first bound $L^p_v$ norm of $e^{\alpha'|v|^\beta}\mathcal{M}(f)$ by
    \begin{align}\label{2_3:max1}
        \begin{split}
            \left\|e^{\alpha'|v|^\beta}\mathcal{M}(f^n)\right\|^p_{L^p_v} &= \left(\frac{|\mathbb{S}^{N-1}|\rho_n}{(2\pi T_n)^{N/2}}\right)^p \int_0^\infty e^{p\left(\alpha' r^\beta-\frac{r^2}{2T_n}\right)} r^{N-1} dr\\
            &\geq \left(|\mathbb{S}^{N-1}|\frac{\rho_n}{(2\pi T_n)^{N/2}}\right)^p \int_{(\alpha'\beta T_n)^{1/(2-\beta)}}^{(\alpha'\beta T_n)^{1/(2-\beta)}+1}e^{p\left(\alpha' r^\beta-\frac{r^2}{2T_n}\right)} r^{N-1} dr.
        \end{split}
    \end{align}
    Since $e^{p\left(\alpha' r^\beta-\frac{r^2}{2T_n}\right)}$ decreases monotonically after the maximum point $r = (\alpha\beta T_n)^{1/(2-\beta)}$, we have
    \begin{align}\label{2_3:max2}
        \begin{split}
            &\int_{(\alpha'\beta T_n)^{1/(2-\beta)}}^{(\alpha'\beta T_n)^{1/(2-\beta)}+1}e^{p\left(\alpha' r^\beta-\frac{r^2}{2T_n}\right)} r^{N-1} dr\\
            &\geq (\alpha'\beta T_n)^{(N-1)/(2-\beta)}\exp{p\left(\alpha' \left((\alpha'\beta T_n)^{1/(2-\beta)}+1\right)^\beta-\frac{\left((\alpha'\beta T_n)^{1/(2-\beta)}+1\right)^2}{2T_n}\right)}.
        \end{split}
    \end{align}
    We need the following lemma to proceed.
    \begin{lemma}\label{lem:max3}
        We have
        \begin{align*}
            \lim_{n\rightarrow \infty}\frac{\exp{p\left(\alpha' \left((\alpha'\beta T_n)^{1/(2-\beta)}+1\right)^\beta-\frac{\left((\alpha'\beta T_n)^{1/(2-\beta)}+1\right)^2}{2T_n}\right)}}{\exp{p\left(\alpha' \left((\alpha'\beta T_n)^{1/(2-\beta)}\right)^\beta-\frac{\left((\alpha'\beta T_n)^{1/(2-\beta)}\right)^2}{2T_n}\right)}}= 1
        \end{align*}
        for $\alpha'>0$ and $0<\beta<2$.
    \end{lemma}
    \begin{proof}
        Let $v_n = (\alpha'\beta T_n)^{1/(2-\beta)}$ and define
        \begin{align*}
            g(x) = \alpha' x^\beta - \frac{x^2}{2T_n}.
        \end{align*}
        Since $g'(v_n) = 0$, by the Taylor's theorem, there exists $c$ between $v_n$ and $v_n+1$ such that 
        \begin{align*}
            g\left(v_n+1\right) &= g(v_n) +g^{\prime}(v_n)+\frac{1}{2}g^{\prime\prime}(c)\\ &=g(v_n)+\frac{1}{2}\left(\alpha'\beta(\beta-1)c^{\beta-2} - \frac{1}{T_n}\right).
        \end{align*}
        Since $c$ and $T_n$ go to infinity as $n\rightarrow \infty$, we get
        \begin{align*}
            \lim_{n\rightarrow \infty}\exp\left(p\big\{g(v_n+1)-g(v_n)\big\}\right) = \exp(0)=1.
        \end{align*}
        It gives the desired result.
    \end{proof}
    
    From \eqref{2_3:macro}, \eqref{2_3:max1}, \eqref{2_3:max2}, and Lemma \ref{lem:max3}, we obtain
    \begin{align*}
        &\liminf_{n\rightarrow \infty}\left\|e^{\alpha'|v|^\beta}\mathcal{M}(f^n)\right\|_{L^p_v}\\
        &\geq \liminf_{n\rightarrow \infty}|\mathbb{S}^{N-1}|\frac{\rho_n}{(2\pi T_n)^{N/2}}(\alpha'\beta T_n)^{\frac{N-1}{p(2-\beta)}}\exp{\left(\alpha' \left((\alpha'\beta T_n)^{1/(2-\beta)}\right)^\beta-\frac{(\alpha'\beta T_n)^{2/(2-\beta)}}{2T_n}\right)}\\
        &\geq \liminf_{n\rightarrow \infty} |\mathbb{S}^{N-1}|\frac{C(\alpha,\alpha', \beta, p, N)}{(n^{2N}-n^N)^{1/p}n^{-\frac{2(N-1)}{p(2-\beta)}+\beta}}\exp\left[\frac{(2-\beta)\alpha'(\alpha'\beta)^{\beta/(2-\beta)}}{2(2N)^{\beta/(2-\beta)}}n^{2\beta/(2-\beta)}-\alpha n^\beta\right].
    \end{align*}
    The exponential part is the same as \eqref{2_3:blowup}, so we get the exponential blow up about $n$ overwhelming the front polynomial decays. Finally, we get weighted $L^p$ norm blow-up for $1\leq p<\infty$.
\end{proof}

\section{Technical Lemmas for inhomogeneous Ill-posedness}

In this section, we prove various technical lemmas that will be used to establish the spatially inhomogeneous ill-posedness result in Section 4.

The next lemma is elementary but will be frequently used in Section 4.
\begin{lemma}\label{lem:power0}
    If $a,b\geq 0$ and $0\leq c\leq 2$, then
    \begin{align*}
        (a+b)^c\leq \max\{c, 1\}(a^c+b^c).
    \end{align*}
\end{lemma}
\begin{proof}
    If $b=0$, it is trivially satisfied, so let us assume $b\neq 0$. We take $t = a/b$ to get
    \begin{align}\label{power0:1}
        (t+1)^c - \max\{c, 1\}(t^c+1)\leq 0
    \end{align}
    for $t\geq 0$. For $0\leq c\leq 1$, by the derivative test,
    \begin{align*}
        \left((t+1)^c - (t^c+1)\right)' = c(t+1)^{c-1} - ct^{c-1}\leq 0
    \end{align*}
    for all $t>0$ with $(0+1)^c - 1\leq 0$. So, it holds for $0\leq c\leq 1$.

    For $1<c\leq 2$, we  show that the maximum of \eqref{power0:1} on $t\geq 0$ should be negative. Indeed, the maximum point $t_0$ of \eqref{power0:1} should satisfy
    \begin{align}\label{power0:2}
        \left((t+1)^c - c(t^c+1)\right)'(t_0) = c(t_0+1)^{c-1} - c^2t_0^{c-1} = 0.
    \end{align}
    Therefore,  
    \begin{align*}
        t_0 = \frac{1}{c^{\frac{1}{c-1}}-1}.
    \end{align*}
    Since $c^{\frac{1}{c-1}}$ is a decreasing function on $1< c\leq 2$ with the minimum value $2$ at $c=2$,  we find $t_0\leq1$. From the relation \eqref{power0:2}, we have
    \begin{align*}
        (t_0+1)^c - c(t_0^c+1) = c(t_0+1)t_0^{c-1} - c(t_0^c+1) = c(t_0^{c-1} - 1)\leq 0.
    \end{align*}
    It completes the proof.
\end{proof}
\begin{remark}
    If we set $a=1$, $b = |x|^\gamma$ $(\gamma>0)$, and $c=2$, we have
    \begin{align*}
        (1+|x|^\gamma)^2\leq 2(1+|x|^{2\gamma}).
    \end{align*}
    Conversely, if we set $b = |x|^{2\gamma}$ and $c=1/2$, then
    \begin{align*}
        (1+|x|^{2\gamma})^{1/2}\leq (1+|x|^{\gamma}).
    \end{align*}
    Therefore, we can interchange $(1+|x|^{\gamma})^2$ and $(1+|x|^{2\gamma})$ with some constant multiplication in computation.
\end{remark}

The following lemma will be used to bound a small perturbation in a base; see \eqref{3_8:Lem_3_3_appli} or \eqref{first} for applications.
\begin{lemma}\label{lem:power1}
    For $0\leq x\leq 1$ and $0\leq c\leq 1$,
    \begin{align*}
        1-x\leq (1-x)^c\leq 1-cx.
    \end{align*}
    If $1\leq c\leq 2$, we have
    \begin{align*}
        1-cx\leq (1-x)^c\leq 1-cx+ \frac{c^2}{2}x^2.
    \end{align*}
    Finally, for $a\geq b>0$, we have
    \begin{align*}
        (a-b)^c\geq\begin{cases}
            a^c-a^{c-1}b & 0\leq c\leq 1,\\
            a^c-ca^{c-1}b & 1\leq c\leq 2.
        \end{cases}
    \end{align*}
\end{lemma}
\begin{proof}
    We first consider the $0\leq c\leq 1$ case. Since $0\leq 1-x\leq 1$, the left inequality is trivial. The right inequality can be checked using derivative analysis. For $1\leq c\leq 2$, taking derivative and applying the inequalities for $0\leq c-1\leq 1$, we have
    \begin{align*}
        -c(1-(c-1)x)\leq \left((1-x)^c\right)' = -c(1-x)^{c-1}\leq -c(1-x)\leq -c(1-cx).
    \end{align*}
    Integrating both sides from $0$ to $x$, we obtain
    \begin{align*}
        1-cx\leq 1-cx + \frac{c(c-1)}{2}x^2\leq (1-x)^c\leq 1-cx + \frac{c^2}{2}x^2.
    \end{align*}
    To prove the last inequality, we modify $(a-b)^c = a^c(1-b/a)^c$ and apply the inequalities for $x = b/a$.
\end{proof}

The following three lemmas are intended to precisely define the bounds of the radius of a set of $\{r:r\geq |x\pm rt|\}$, which will be crucially used in the next section.
\begin{lemma}\label{lem:power3}
    If $2\leq a\leq 3$, $x\geq t>0$, then
    \begin{align*}
        \left(x+\frac{2t}{a}\right)^a&\geq x^a + t x^{a-1} + \frac{2t^2}{a} x^{a-2},\\
        \left(x-\frac{t}{a}\right)^a&\leq x^a - t x^{a-1} + \frac{t^2}{a} x^{a-2},\\
        \left(x+\frac{t}{2a}\right)^a&\leq x^a + t x^{a-1} + \frac{t^2}{2a} x^{a-2},\\
        \left(x-\frac{t}{2a}\right)^a&\geq x^a - t x^{a-1} + \frac{t^2}{2a} x^{a-2}.
    \end{align*}
\end{lemma}
\begin{proof}
    We start from elementary inequalities: for $0\leq w\leq \frac{1}{a}$ with $2\leq a\leq 3$. We claim that
    \begin{align*}
        (1+2w)^a&\geq 1+aw + 2aw^2,\\
        (1-w)^a&\leq 1-aw + aw^2,\\
        (1+w/2)^a&\leq 1+aw + aw^2/2,\\
        (1-w/2)^a&\geq 1-aw + aw^2/2,
    \end{align*}
    which can be proved by the derivative test: for $0\leq w\leq 1/a$ and $2\leq a\leq 3$,
    \begin{align*}
        \left((1+2w)^a - 1-aw - 2aw^2\right)' &= 2a(1+2w)^{a-1}-a-4aw\geq 2a(1+2w)-a-4aw = a\geq 0,\\
        \left((1-w)^a - 1+aw - aw^2\right)' &= -a(1-w)^{a-1}+a-2aw\leq -a(1-w)^2+a-2aw=-aw^2\leq 0,\\
        \left((1+w/2)^a - 1-aw - aw^2/2\right)' &= \frac{a}{2}(1+w/2)^{a-1}-a-aw\leq \frac{a}{2}(1+w+w^2/4)-a-aw \\
        &= \frac{a}{2}\left(-1-w + \frac{w^2}{2}\right)\leq 0,\\
        \left((1-w/2)^a - 1+aw - aw^2/2\right)' &= -\frac{a}{2}(1-w/2)^{a-1}+a-aw\geq -\frac{a}{2}(1-w/2)+a-aw=\frac{a}{2} - \frac{3}{4}aw\geq 0.
    \end{align*}
    Now, we substitute $w = \frac{t}{ax}$ and multiply both sides by $x^a$ to finish the proof.
\end{proof}

\begin{lemma}\label{lem:pert_ineq1}
    For $\frac{1}{3}\leq \gamma\leq \frac{1}{2}$ and $t\leq 1$, assume $x\geq t^{\frac{1}{1-\gamma}}$. Then we have
    \begin{align*}
        \left(x^\gamma + 2\gamma t x^{-(1-2\gamma)}\right)^{1/\gamma}&\geq x + \left(x^\gamma + 2\gamma t x^{-(1-2\gamma)}\right)t,\\
        \left(x^\gamma - \gamma t x^{-(1-2\gamma)}\right)^{1/\gamma}&\leq x - \left(x^\gamma - \gamma t x^{-(1-2\gamma)}\right)t,\\
        \left(x^\gamma + \frac{\gamma}{2} t x^{-(1-2\gamma)}\right)^{1/\gamma}&\leq x + \left(x^\gamma + \frac{\gamma}{2}t x^{-(1-2\gamma)}\right)t,\\
        \left(x^\gamma - \frac{\gamma}{2} t x^{-(1-2\gamma)}\right)^{1/\gamma}&\geq x - \left(x^\gamma - \frac{\gamma}{2}t x^{-(1-2\gamma)}\right)t.
    \end{align*}
\end{lemma}
\begin{proof}
    We substitute $a = \frac{1}{\gamma}$ and plug $x^{1-\gamma}$ into $x$ position in Lemma \ref{lem:power3}. Then,
    \begin{align*}
        \left(x^{1-\gamma}+2\gamma t\right)^{1/\gamma}&\geq x^{\frac{1-\gamma}{\gamma}} + t x^{\frac{(1-\gamma)^2}{\gamma}} + 2\gamma t^2 x^{\frac{(1-\gamma)(1-2\gamma)}{\gamma}},\\
        \left(x^{1-\gamma}-\gamma t\right)^{1/\gamma}&\leq x^{\frac{1-\gamma}{\gamma}} - t x^{\frac{(1-\gamma)^2}{\gamma}} + \gamma t^2 x^{\frac{(1-\gamma)(1-2\gamma)}{\gamma}},\\
        \left(x^{1-\gamma}+\frac{\gamma}{2}t\right)^{1/\gamma}&\leq x^{\frac{1-\gamma}{\gamma}} + t x^{\frac{(1-\gamma)^2}{\gamma}} + \frac{\gamma}{2}t^2 x^{\frac{(1-\gamma)(1-2\gamma)}{\gamma}},\\
        \left(x^{1-\gamma}-\frac{\gamma}{2}t\right)^{1/\gamma}&\geq x^{\frac{1-\gamma}{\gamma}} - t x^{\frac{(1-\gamma)^2}{\gamma}} + \frac{\gamma}{2}t^2 x^{\frac{(1-\gamma)(1-2\gamma)}{\gamma}}.
    \end{align*}
    Multiplying both sides by $x^{\frac{-1+2\gamma}{\gamma}}$ yields the lemma.
\end{proof}
    
\begin{lemma}\label{lem:pert_ineq2}
    For $\frac{1}{3}\leq \gamma \leq \frac{1}{2}$ and $t\leq 1$, assume $x\geq t^{\frac{1}{1-\gamma}}$. If $r\geq x^\gamma + 2\gamma t x^{-(1-2\gamma)}$ $($resp., $r\geq x^\gamma - \frac{\gamma}{2} t x^{-(1-2\gamma)}$$)$, then $r\geq |x + r t|^\gamma$ $($resp., $r\geq |x - r t|^\gamma$$)$. Conversely, if $r\leq x^\gamma + \frac{\gamma}{2} t x^{-(1-2\gamma)}$ $($resp., $r\leq x^\gamma - \gamma t x^{-(1-2\gamma)}$$)$, then $r\leq |x + r t|^\gamma$ $($resp., $r\leq |x - r t|^\gamma$$)$.
\end{lemma}
\begin{proof}
    We first prove $r\geq |x+rt|^\gamma$ case. We first note that $r^{1/\gamma} - (x + rt)$ is an increasing function if
    \begin{align*}
        \gamma^{-1}r^{1/\gamma-1} - t\geq 0.
    \end{align*}
    That is, when
    \begin{align}\label{3_7:eq11}
        r\geq (\gamma t)^{\frac{\gamma}{1-\gamma}}.
    \end{align}
    From what we have in Lemma \ref{lem:pert_ineq1}, it is enough to show that $x^\gamma + 2\gamma t x^{-(1-2\gamma)}\geq (\gamma t)^{\frac{\gamma}{1-\gamma}}$. It can be easily checked that
    \begin{align}\label{3_7:eq2}
        x^{\gamma} +2\gamma tx^{-(1-2\gamma)}\geq x^{\gamma}\geq t^{\frac{\gamma}{1-\gamma}}\geq (\gamma t)^{\frac{\gamma}{1-\gamma}}
    \end{align}
    for $\frac{1}{3}\leq \gamma\leq \frac{1}{2}$ and $t\leq 1$. Therefore, we get the desired inequality $r\geq |x+rt|^\gamma$ when
    \begin{align*}
        r\geq x^\gamma + 2\gamma t x^{-(1-2\gamma)}.
    \end{align*}
    
    Next, we check $r\geq |x-rt|^\gamma$. It is enough to examine the inequality for $r\leq \frac{x}{t}$ since the inequality is immediately satisfied when $r\geq \frac{x}{t}$. $r^{1/\gamma} - (x-rt)$ is an increasing function in the region, so we can directly apply Lemma \ref{lem:pert_ineq1} so that $r\geq |x-rt|^\gamma$ if
    \begin{align*}
        r\geq \min\left\{x^\gamma -\frac{\gamma}{2}t x^{-(1-2\gamma)},\frac{x}{t}\right\}.
    \end{align*}
    
    We check $r\leq |x+rt|^\gamma$. Note that $r^{1/\gamma} - (x+rt)$ has minimum at $r = (\gamma t)^{\frac{\gamma}{1-\gamma}}$ by \eqref{3_7:eq11}. Also, by \eqref{3_7:eq2}, it holds that
    \begin{align*}
        x^{\gamma} +\frac{\gamma}{2} tx^{-(1-2\gamma)}\geq (\gamma t)^{\frac{\gamma}{1-\gamma}}.
    \end{align*}
    Since $r\leq |x+rt|^\gamma$ at $r = 0$ and $r = x^{\gamma} +\frac{\gamma}{2} tx^{-(1-2\gamma)}$ by Lemma \ref{lem:pert_ineq1}, the inequality holds for $0\leq r\leq x^{\gamma} +\frac{\gamma}{2} tx^{-(1-2\gamma)}$.
    
    Finally, we prove $r\leq |x-rt|^\gamma$. $r^{1/\gamma}-(x-rt)$ is increasing function for $0\leq r\leq \frac{x}{t}$. As the inequality holds at $r = x^\gamma - \gamma t x^{-(1-2\gamma)}$ by Lemma \ref{lem:pert_ineq1} and $x^\gamma - \gamma t x^{-(1-2\gamma)}\leq \frac{x}{t}$, it holds for $0\leq r\leq x^\gamma - \gamma t x^{-(1-2\gamma)}$.\\
\end{proof}

The following two lemmas are crucially used in the next section to bound the integral $\int_{\mathbb{R}^N} |v|^n\mathcal{M}(f)(\tau, x-v(t-\tau), v)\,dv$. We refer to Theorem \ref{thm:Ext_Uni} to explain why we assume
\begin{align*}
    \mathcal{M}(f)(\tau, x, v)\simeq C\frac{\exp\left(-\alpha |x|^{\beta\gamma}\right)}{(1+|x|^{2\gamma})^{m/2}}\exp\left(-\frac{\left(v-u(\tau, x)\right)^2}{C_T\left(1+|x|^{2\gamma}\right)}\right)
\end{align*}
in the next lemma with $m\in\mathbb{R}$.
    
Although the proof is technical and detailed, the main idea can be summarized as follows: The integral scales $|v|$ by $|x|^\gamma$. As a result, we have
\begin{align*}
    \int_{\mathbb{R}^N} dv\, |v|^n\mathcal{M}(f)(\tau, x-v(t-\tau), v)
    &\leq C(1+|x|^\gamma)^{(n+N)\gamma}\frac{\exp\left(-\alpha |x|^{\beta\gamma}\right)}{(1+|x|^{2\gamma})^{m/2}}\\
    & = C(1+|x|^\gamma)^{(n-m+N)\gamma}\exp\left(-\alpha |x|^{\beta\gamma}\right)
\end{align*}
for some constant $C$. This type of scaling will be the key idea in Section 4.

\begin{lemma}\label{lem:int_bound1}
    Let $\alpha\geq 0$, $0<\beta\leq 2$, $0<C_T<\infty$, $N\geq 1$, and $\gamma$ satisfies $0<\gamma\leq \frac{1}{1+\beta}$. Suppose $u$ satisfies $|u|(t,x)\leq C_u(1+|x|^\gamma)$ for $0\leq t\leq 1$, then it satisfies
    \begin{align*}
        &\sup_{0\leq t,\tau\leq 1}\sup_{x\in\mathbb{R}^N} \frac{1}{(1+|x|^\gamma)^{n-m+N} e^{-\alpha |x|^{\beta\gamma}}}\int_{\mathbb{R}^N} dv\, |v|^n\frac{\exp\left(-\alpha|x-vt|^{\beta\gamma}\right)}{\left(1+|x-vt|^{2\gamma}\right)^{m/2}}\exp\left(-\frac{\left(v-u(\tau, x-vt)\right)^2}{C_T\left(1+|x-vt|^{2\gamma}\right)}\right)\\
        &\leq C(n, m, \alpha, \beta, C_u, C_T, N, \gamma)<\infty,
    \end{align*}
    for all $n \geq 0$ and $m \in \mathbb{R}$.
\end{lemma}
\begin{proof}
    We first rewrite the integral in the spherical coordinates. We denote $v = r\vec{\sigma}$ and use $\left||x|-|v|t\right|\leq|x-vt|\leq \left||x|+|v|t\right|$ to write
    \begin{align}
        &\int_{\mathbb{R}^N} dv\, |v|^n\frac{\exp\left(-\alpha|x-vt|^{\beta\gamma}\right)}{\left(1+|x-vt|^{2\gamma}\right)^{m/2}}\exp\left(-\frac{\left(v-u\right)^2}{C_T\left(1+|x-vt|^{2\gamma}\right)}\right) \notag \\
        &\leq
        \begin{dcases}
            \int_{\mathbb{S}^{N-1}}\int_0^\infty d\sigma dr\, r^{n+N-1}\frac{\exp\left(-\alpha\left||x|-rt\right|^{\beta\gamma}\right)}{\left(1+\left||x|-rt\right|^{2\gamma}\right)^{m/2}}\exp\left(-\frac{(r-|u|(\tau, x-rt\sigma))^2}{C_T\left(1+(|x|+rt)^{2\gamma}\right)}\right) & m\geq 0\\
            \int_{\mathbb{S}^{N-1}}\int_0^\infty d\sigma dr\,  r^{n+N-1}\frac{\exp\left(-\alpha\left||x|-rt\right|^{\beta\gamma}\right)}{\left(1+\left||x|+rt\right|^{2\gamma}\right)^{m/2}}\exp\left(-\frac{(r-|u|(\tau, x-rt\sigma))^2}{C_T\left(1+(|x|+rt)^{2\gamma}\right)}\right) & m< 0
        \end{dcases}\notag\\
        &\leq
        \begin{dcases}
            |\mathbb{S}^{N-1}|\sup_{\sigma\in\mathbb{S}^{N-1}}\int_0^\infty dr\, r^{n+N-1}\frac{\exp\left(-\alpha\left||x|-rt\right|^{\beta\gamma}\right)}{\left(1+\left||x|-rt\right|^{2\gamma}\right)^{m/2}}\exp\left(-\frac{(r-|u|(\tau, x-rt\sigma))^2}{C_T\left(1+(|x|+rt)^{2\gamma}\right)}\right) & m\geq 0,\\
            |\mathbb{S}^{N-1}|\sup_{\sigma\in\mathbb{S}^{N-1}}\int_0^\infty dr\,  r^{n+N-1}\frac{\exp\left(-\alpha\left||x|-rt\right|^{\beta\gamma}\right)}{\left(1+\left||x|+rt\right|^{2\gamma}\right)^{m/2}}\exp\left(-\frac{(r-|u|(\tau, x-rt\sigma))^2}{C_T\left(1+(|x|+rt)^{2\gamma}\right)}\right) & m< 0.
        \end{dcases}\label{3_8:RHS}
    \end{align}
    In the proof, we will properly bound the integral uniformly for $\sigma\in\mathbb{S}^{N-1}$.
    
    We will focus on bounding $m\geq 0$ case of the integral in \eqref{3_8:RHS}; the $m<0$ case can be checked following the similar line of proof. Furthermore, we will replace $n+N-1$ by $n\geq 0$ since $N\geq 1$ and $n\geq 0$. 

    Now, let us initiate the proof by investigating the $t = 0$ case, which is the easiest case in our proof. It is just
    \begin{align}\label{3_8:int_t0}
        \int_0^\infty dr\,r^n\frac{\exp\left(-\alpha|x|^{\beta\gamma}\right)}{\left(|x|^{2\gamma}+1\right)^{m/2}}\exp\left(-\frac{\left(r-|u|(\tau, x)\right)^2}{C_T\left(1+|x|^{2\gamma}\right)}\right).
    \end{align}
    
    We divide the cases $|x|\leq 1$ and $|x|\geq 1$. For $|x|\leq 1$, $|u|(\tau, x)$ and $1+|x|^{2\gamma}$ are bounded by constants, so \eqref{3_8:int_t0} is bounded by a constant not depending on $x$ and $\tau$.
    
    For $|x|\geq 1$, considering a change of variable $r\mapsto |x|^\gamma r$, we get
    \begin{align*}
        \eqref{3_8:int_t0} &= \frac{|x|^{(n+1)\gamma}}{\left(|x|^{2\gamma}+1\right)^{m/2}}\exp\left(-\alpha|x|^{\beta\gamma}\right)\int_0^\infty dr\,r^n\exp\left(-\frac{\left(r-|x|^{-\gamma}|u|(\tau, x)\right)^2}{C_T\left(1+|x|^{-2\gamma}\right)}\right)\\
        &\leq C(1+|x|^\gamma)^{(n+1-m)\gamma}\exp\left(-\alpha|x|^{\beta\gamma}\right)\int_0^\infty dr\,r^n\exp\left(-\frac{\left(r-|x|^{-\gamma}|u|(\tau, x)\right)^2}{C_T\left(1+|x|^{-2\gamma}\right)}\right)
    \end{align*}
    for some constant $C$ for $|x|\geq 1$. Again, $|x|^{-\gamma}|u|(\tau, x)$ and $1+|x|^{-2\gamma}$ are bounded by constants. Therefore, the integral is bounded, not depending on $x$ and $\tau\leq 1$. Combining analysis of $|x|\leq 1$ and $|x|\geq 1$, we bound integral \eqref{3_8:int_t0} in the form of Lemma \ref{lem:int_bound1}.
    
    From now on, we will safely assume $t>0$. From $(x-y)^2\geq \frac{1}{2}x^2 - y^2$ and $|u|(\tau, x-rt\sigma)\leq C_u(1+(|x|+rt)^\gamma)$,
    \begin{align*}
        \exp\left(-\frac{(r-|u|(\tau, x-rt\sigma))^2}{C_T\left((|x|+rt)^{2\gamma}+1\right)}\right)&\leq \exp\left(-\frac{\frac{r^2}{2} - C^2_u(1+(|x|+rt)^{2\gamma})}{C_T\left((|x|+rt)^{2\gamma}+1\right)}\right)\\
        &\leq \exp\left(\frac{C_u^2}{C_T}\right)\exp\left(-\frac{r^2}{2C_T\left((|x|+rt)^{2\gamma}+1\right)}\right).
    \end{align*}
    Therefore, it is enough to bound
    \begin{align*}
        \int_0^\infty dr\, r^n\frac{\exp\left(-\alpha\left||x|-rt\right|^{\beta\gamma}\right)}{\left(\left||x|-rt\right|^{2\gamma}+1\right)^{m/2}}\exp\left(-\frac{r^2}{C_T\left((|x|+rt)^{2\gamma}+1\right)}\right)
    \end{align*}
    by absorbing $2$ into $C_T$. We again divide the cases $|x|\leq \max\{(18 C_T\alpha)^{\frac{1}{1-\gamma}}, 1\}$ and $|x|\geq \max\{(18 C_T\alpha)^{\frac{1}{1-\gamma}}, 1\}$; the splitting constant is intentionally chosen for the later step \eqref{3_8:secondpre}. Before \eqref{3_8:secondpre}, one can just treat $\max\{(18 C_T\alpha)^{\frac{1}{1-\gamma}}, 1\}$ as a constant bigger than $1$.

    \subsubsection{The case of $|x|\leq \max\{(18 C_T\alpha)^{\frac{1}{1-\gamma}}, 1\}$}
    \noindent In this case, we claim that
    \begin{align*}
        \int_0^\infty dr\, r^n\frac{\exp\left(-\alpha\left||x|-rt\right|^{\beta\gamma}\right)}{\left(\left||x|-rt\right|^{2\gamma}+1\right)^{m/2}}\exp\left(-\frac{r^2}{C_T\left((|x|+rt)^{2\gamma}+1\right)}\right)<\infty
    \end{align*}
    uniformly about $x$ and $t$. Indeed, 
    \begin{align*}
        &\int_0^\infty dr\, r^n\frac{\exp\left(-\alpha\left||x|-rt\right|^{\beta\gamma}\right)}{\left(\left||x|-rt\right|^{2\gamma}+1\right)^{m/2}}\exp\left(-\frac{r^2}{C_T\left((|x|+rt)^{2\gamma}+1\right)}\right)\\
        &\leq \int_0^\infty dr\, r^n\exp\left(-\frac{r^2}{C_T\left((M+rt)^{2\gamma}+1\right)}\right),
    \end{align*}
    where $M = \max\{(18 C_T\alpha)^{\frac{1}{1-\gamma}}, 1\}$, and the final integral converges as $2\gamma\leq 1$. (If $m\geq 0$, the polynomial part is bounded by $r^n \left((1+rt)^{2\gamma}+1)\right)^{m/2}$, and it again has a finite integral value thanks to the exponential decay.) It shows
    \begin{align*}
        \sup_{0<t\leq 1}\sup_{|x|\leq \max\{(18 C_T\alpha)^{\frac{1}{1-\gamma}}, 1\}}\int_0^\infty dr\, r^n\frac{\exp\left(-\alpha\left||x|-rt\right|^{\beta\gamma}\right)}{\left(\left||x|-rt\right|^{2\gamma}+1\right)^{m/2}}\exp\left(-\frac{r^2}{C_T\left((|x|+rt)^{2\gamma}+1\right)}\right)<\infty.
    \end{align*}
    \\
    
    We move on to the next case $|x|\geq \max\{(18 C_T\alpha)^{\frac{1}{1-\gamma}}, 1\}$ in bounding \eqref{3_8:RHS}. We divide the integral region by
    \begin{equation}\label{3_8:int_divide}
        \begin{split}
            I_1 &= \int_0^{t^{-1}|x|} dr\, r^n\frac{\exp\left(-\alpha\left||x|-rt\right|^{\beta\gamma}\right)}{\left(\left||x|-rt\right|^{2\gamma}+1\right)^{m/2}}\exp\left(-\frac{r^2}{C_T\left((|x|+rt)^{2\gamma}+1\right)}\right),\\
            I_2 &= \int_{t^{-1}|x|}^\infty dr\, r^n \frac{\exp\left(-\alpha\left||x|-rt\right|^{\beta\gamma}\right)}{\left(\left||x|-rt\right|^{2\gamma}+1\right)^{m/2}}\exp\left(-\frac{r^2}{C_T\left((|x|+rt)^{2\gamma}+1\right)}\right).
        \end{split}
    \end{equation}
    
    \subsubsection{The case of $|x|\geq \max\{(18 C_T\alpha)^{\frac{1}{1-\gamma}}, 1\}$ and $0\leq r\leq t^{-1}|x|$}
    Replacing $r = c|x|$ for some $c\in (0, t^{-1})$,
    \begin{align}\label{3_8:Lem_3_3_appli}
        \begin{split}
            &r^n\frac{\exp\left(-\alpha\left||x|-rt\right|^{\beta\gamma}\right)}{\left(\left||x|-rt\right|^{2\gamma}+1\right)^{m/2}}\exp\left(-\frac{r^2}{C_T\left((|x|+rt)^{2\gamma}+1\right)}\right)\,dr\\
            &=c^n |x|^{n+1}\frac{\exp\left(-\alpha|x|^{\beta\gamma}(1-ct)^{\beta\gamma}\right)}{\left(|x|^{2\gamma}(1-ct)^{2\gamma}+1\right)^{m/2}}\exp\left(-\frac{|x|^2c^2}{C_T\left(|x|^{2\gamma}(1+ct)^{2\gamma}+1\right)}\right)\,dc\\
            &\leq c^n |x|^{n+1}\frac{\exp\left(-\alpha|x|^{\beta\gamma}\right)}{\left(|x|^{2\gamma}(1-ct)^{2\gamma}+1\right)^{m/2}}\exp\left(-\frac{|x|^2c^2}{C_T\left(|x|^{2\gamma}(1+ct)^{2\gamma}+1\right)} + \alpha ct |x|^{\beta\gamma}\right)\,dc,
        \end{split}
    \end{align}
    where we used $(1-ct)^{\beta\gamma}\geq 1-ct$ as $0\leq ct\leq 1$ and $\beta\gamma\leq 1$.
    
    Using $(1+ct)^{2\gamma}\leq 2^{2\gamma}\leq 8$ and $1\leq |x|^{2\gamma}$, we rearrange the terms in exponential by
    \begin{align*}
        &-\frac{|x|^2c^2}{C_T\left(|x|^{2\gamma}(1+ct)^{2\gamma}+1\right)} + \alpha ct |x|^{\beta\gamma}\leq -\frac{|x|^{2-2\gamma}c^2}{C_T\left((1+ct)^{2\gamma} + |x|^{-2\gamma}\right)} + \alpha ct |x|^{\beta\gamma}\\
        &\leq -\frac{|x|^{2-2\gamma}c^2}{9C_T} + \alpha ct |x|^{\beta\gamma}\\
        &\leq -|x|^{2-2\gamma}\left(\frac{c}{3\sqrt{C_T}} - \frac{3\sqrt{C_T}}{2}\alpha t |x|^{\beta\gamma-2+2\gamma}\right)^2 + \frac{9C_T}{4}\alpha^2 t^2 |x|^{2\beta\gamma - 2 + 2\gamma}.
    \end{align*}
    Since $|x|\geq 1$, $\beta\gamma - 1+\gamma\leq 0$, and $0<t\leq 1$, the third term is bounded by
    \begin{align}\label{3_8:bound1}
        \frac{9C_T}{4}\alpha^2 t^2 |x|^{2\beta\gamma - 2 + 2\gamma}\leq \frac{9 C_T}{4}\alpha^2.
    \end{align}
    Finally, we obtain
    \begin{equation} \label{3_8:subsec1}
        \begin{split}   
            I_1\leq e^{C_1}\int_0^{t^{-1}} dc\,c^n \frac{|x|^{n+1}\exp\left(-\alpha|x|^{\beta\gamma}\right)}{\left(|x|^{2\gamma}|1-ct|^{2\gamma}+1\right)^{m/2}}\exp\left(-|x|^{2-2\gamma}\left(\frac{c}{3\sqrt{C_T}} - \frac{3\sqrt{C_T}}{2}\alpha t |x|^{\beta\gamma-2+2\gamma}\right)^2\right)
        \end{split}   
    \end{equation}
    for some constant $C_1$.
    
    Taking change of variable again from $c$ to $|x|^{\gamma-1}c$, we obtain
    \begin{align}
        &\int_0^{t^{-1}}dc\, c^n \frac{|x|^{n+1}\exp\left(-\alpha|x|^{\beta\gamma}\right)}{\left(|x|^{2\gamma}|1-ct|^{2\gamma}+1\right)^{m/2}}\exp\left(-|x|^{2-2\gamma}\left(\frac{c}{3\sqrt{C_T}} - \frac{3\sqrt{C_T}}{2}\alpha t |x|^{\beta\gamma-2+2\gamma}\right)^2\right)\notag\\
        &=\int_0^{|x|^{1-\gamma} t^{-1}}dc\,c^n \frac{|x|^{(n+1)\gamma}\exp\left(-\alpha |x|^{\beta\gamma}\right)}{\left(|x|^{2\gamma}\left|1-|x|^{\gamma-1}ct\right|^{2\gamma}+1\right)^{m/2}}\exp\left(-\left(\frac{c}{3\sqrt{C_T}} - \frac{3\sqrt{C_T}}{2}\alpha t |x|^{\beta\gamma-1+\gamma}\right)^2\right)\label{3_8:subsec1_1}.
    \end{align}
    By \eqref{3_8:bound1}, $\frac{3\sqrt{C_T}}{2}\alpha t |x|^{\beta\gamma-1+\gamma}\leq \frac{3\sqrt{C_T}}{2}\alpha$, so we can treat it as a constant.
    
    We split the integral by
    \begin{align*}
        I_{1,1} &\coloneqq \int_0^{\frac{|x|^{1-\gamma} t^{-1}}{2}} dc\,c^n \frac{|x|^{(n+1)\gamma}\exp\left(-\alpha |x|^{\beta\gamma}\right)}{\left(|x|^{2\gamma}\left|1-|x|^{\gamma-1}ct\right|^{2\gamma}+1\right)^{m/2}}\exp\left(-\left(\frac{c}{3\sqrt{C_T}} - \frac{3\sqrt{C_T}}{2}\alpha t |x|^{\beta\gamma-1+\gamma}\right)^2\right),\\
        I_{1,2} &\coloneqq \int_{\frac{|x|^{1-\gamma} t^{-1}}{2}}^{|x|^{1-\gamma} t^{-1}} dc\,c^n \frac{|x|^{(n+1)\gamma}\exp\left(-\alpha |x|^{\beta\gamma}\right)}{\left(|x|^{2\gamma}\left|1-|x|^{\gamma-1}ct\right|^{2\gamma}+1\right)^{m/2}}\exp\left(-\left(\frac{c}{3\sqrt{C_T}} - \frac{3\sqrt{C_T}}{2}\alpha t |x|^{\beta\gamma-1+\gamma}\right)^2\right).
    \end{align*}
    For $I_{1,1}$, we have
    \begin{align*}
        \frac{1}{\left(|x|^{2\gamma}\left|1-|x|^{\gamma-1}ct\right|^{2\gamma}+1\right)^{m/2}}\leq \frac{1}{\left(\frac{1}{4}|x|^{2\gamma}+1\right)^{m/2}}.
    \end{align*}
    Therefore,
    \begin{equation} \label{3_8:firstpre}
        \begin{split}
            I_{1,1}&\leq \frac{|x|^{(n+1)\gamma}\exp\left(-\alpha|x|^{\beta\gamma}\right)}{\left(\frac{1}{4}|x|^{2\gamma}+1\right)^{m/2}}\int_0^{\frac{|x|^{1-\gamma} t^{-1}}{2}}dc\,c^n \exp\left(-\left(\frac{c}{3\sqrt{C_T}} - \frac{3\sqrt{C_T}}{2}\alpha t |x|^{\beta\gamma-1+\gamma}\right)^2\right)\\
            &\leq C_2(1+|x|^\gamma)^{(n+1-m)}\exp\left(-\alpha|x|^{\beta\gamma}\right) \int_{-\infty}^\infty dc\, \left(|c|+\frac{3\sqrt{C_T}}{2}\alpha\right)^n \exp\left(-\frac{c^2}{9C_T}\right)\\
            &\leq C_3(1+|x|^\gamma)^{(n+1-m)}\exp\left(-\alpha|x|^{\beta\gamma}\right)
        \end{split}
    \end{equation}
    for some constant $C_2$ and $C_3$ independent to $t$ and $|x|$.

    For the other part, as $|x|\geq \max\{(18C_T\alpha)^{\frac{1}{1-\gamma}}, 1\}$, we have
    \begin{align*}
        \frac{c}{3\sqrt{C_T}} - \frac{3\sqrt{C_T}}{2}\alpha t |x|^{\beta\gamma-1+\gamma}\geq \frac{c}{3\sqrt{C_T}} - \frac{3\sqrt{C_T}}{2}\alpha\geq \frac{c}{6\sqrt{C_T}}
    \end{align*}
    for $c\geq \frac{|x|^{1-\gamma} t^{-1}}{2}$. Therefore, just bounding the denominator of the integrand by $1$, we obtain
    \begin{equation}\label{3_8:secondpre}
        \begin{split}
            I_{1,2}&\leq |x|^{(n+1)\gamma}\exp\left(-\alpha|x|^{\beta\gamma}\right)\int_{\frac{|x|^{1-\gamma} t^{-1}}{2}}^\infty dc\,c^n \exp\left(-\left(\frac{c}{3\sqrt{C_T}} - \frac{3\sqrt{C_T}}{2}\alpha t |x|^{\beta\gamma-1+\gamma}\right)^2\right)\\
            &\leq |x|^{(n+1)\gamma}\exp\left(-\alpha|x|^{\beta\gamma}\right)\int_{\frac{|x|^{1-\gamma} t^{-1}}{2}}^\infty dc\,c^n \exp\left(-\frac{c^2}{36C_T}\right).
        \end{split}
     \end{equation}
    Since $\frac{|x|^{1-\gamma} t^{-1}}{2}\geq \frac{1}{2}$, applying Lemma \ref{lem:Tri_integral}, there exists a constant $C$ such that
    \begin{align}\label{3_8:secondpre_1}
        \int_{\frac{|x|^{1-\gamma} t^{-1}}{2}}^\infty dc\,c^n \exp\left(-\frac{c^2}{36C_T}\right)\leq C\left(\frac{36C_T}{2} \left(\frac{|x|^{1-\gamma} t^{-1}}{2}\right)^{n-1}\right)\exp\left(-\frac{\left(\frac{|x|^{1-\gamma} t^{-1}}{2}\right)^2}{36C_T}\right).
    \end{align}
    Using the fact that exponential decay suppresses any polynomial growth, we obtain
    \begin{equation}\label{3_8:secondsub}
        \begin{split}
            (\ref{3_8:secondpre_1}) &\leq C_5(1+|x|)^{-m\gamma}
        \end{split}
    \end{equation}
    for some constant $C_5$ independent to $t$ and $|x|\geq 1$.
    
    Combining \eqref{3_8:subsec1}, \eqref{3_8:subsec1_1}, \eqref{3_8:firstpre}, \eqref{3_8:secondpre}, \eqref{3_8:secondpre_1}, and \eqref{3_8:secondsub}, we finally have
    \begin{equation}\label{3_8:firstres}
        \begin{split}
            \sup_{0<t\leq 1}\sup_{0\leq \tau\leq 1} I_1\leq e^{C_1}(I_{1,1}+I_{1,2} + I_{1,3})\leq e^{C_1}(C_2+C_4 + C_5)(1+|x|)^{(n+1-m)\gamma}\exp\left(-\alpha|x|^{\beta\gamma}\right)
        \end{split}
    \end{equation}
    for $|x|\geq \max\{(18 C_T\alpha)^{\frac{1}{1-\gamma}}, 1\}$.

    \subsubsection{The case of $|x|\geq \max\{(18 C_T\alpha)^{\frac{1}{1-\gamma}}, 1\}$ and $r\geq t^{-1}|x|$\nopunct}
    We start from \eqref{3_8:int_divide}. For $r\geq t^{-1}|x|$,
    \begin{equation}\label{3_8:subsec2}
        \begin{split}
            I_2&=\int_{t^{-1}|x|}^{\infty} dr\,r^{n}\frac{\exp\left(-\alpha\left||x|-rt\right|^{\beta\gamma}\right)}{\left(\left||x|-rt\right|^{2\gamma}+1\right)^{m/2}}\exp\left(-\frac{r^2}{C_T\left((|x|+rt)^{2\gamma}+1\right)}\right)\\
            &\leq \int_{t^{-1}|x|}^{\infty} dr\,r^{n} \exp\left(-\frac{r^2}{C_T\left((|x|+rt)^{2\gamma}+1\right)}\right)\\
            &\leq |x|^{(n+1)\gamma}\int_{t^{-1}|x|^{1-\gamma}}^{\infty}dr\, r^n \exp\left(-\frac{r^2}{C_T\left(\left(1+\frac{rt}{|x|^{1-\gamma}}\right)^{2\gamma}+1\right)}\right).
        \end{split}
    \end{equation}
    In the last step, we used the change of variable $r\mapsto |x|^\gamma r$.
    
    Using Lemma \ref{lem:Tri_integral} again, therefore, we have bound
    \begin{equation}\label{3_8:thirdpre}
        \begin{split}
            &\int_{t^{-1}|x|^{1-\gamma}}^{\infty} dr \,r^n \exp\left(-\frac{r^2}{C_T\left(\left(1+\frac{rt}{|x|^{1-\gamma}}\right)^{2\gamma}+1\right)}\right)\\
            &= \int_{t^{-1}|x|^{1-\gamma}}^{\infty}dr\, r^n \exp\left(-\frac{|x|^{2\gamma-2\gamma^2}r^{2-2\gamma}}{5C_T t^{2\gamma}}\right)\\
            &\leq C\frac{5C_T|x|^{(n-1)(1-\gamma)}}{(2-2\gamma)t^{n-1}} \exp\left(-\frac{|x|^{2-2\gamma}}{5C_T t^2}\right)
        \end{split}
    \end{equation}
    for some constant $C$. For the first to second line, we used
    \begin{align*}
        \left(1+\frac{rt}{|x|^{1-\gamma}}\right)^{2\gamma}+1&\leq 2\left(1^{2\gamma}+\left(\frac{rt}{|x|^{1-\gamma}}\right)^{2\gamma}\right)+1 = 2\left(\frac{rt}{|x|^{1-\gamma}}\right)^{2\gamma}+3\\
        &\leq 5\left(\frac{rt}{|x|^{1-\gamma}}\right)^{2\gamma},
    \end{align*}
    applying Lemma \ref{lem:power0} and $\frac{rt}{|x|^{1-\gamma}}\geq 1$.

    By the choice of $\beta$ and $\gamma$, we have $2-2\gamma>1-\gamma\geq\beta\gamma$. Therefore, the decay speed of $\exp\left(-\frac{|x|^{2-2\gamma}}{10 C_T}\right)$ is greater than $\exp\left(-\alpha|x|^{\beta\gamma}\right)$ for all $0<t\leq 1$, so
    \begin{equation}\label{3_8:thirdsub}
        \begin{split}
            &\frac{5C_T|x|^{(n-1)(1-\gamma)}}{(2-2\gamma)t^{n-1}} \exp\left(-\frac{|x|^{2-2\gamma}}{5C_T t^2}\right) \\
            &\leq \left\{\sup_{0<t\leq 1}\frac{5C_T|x|^{(n-1)(1-\gamma)}}{(2-2\gamma)t^{n-1}} \exp\left(-\frac{|x|^{2-2\gamma}}{10C_T t^2}\right)\right\}\exp\left(-\frac{|x|^{2-2\gamma}}{10C_T t^2}\right)\\
            &\leq C_6 (1+|x|)^{-m\gamma}\exp\left(-\frac{|x|^{2-2\gamma}}{10C_T}\right)\leq C_7(1+|x|)^{-m\gamma} \exp\left(-\alpha|x|^{\beta\gamma}\right)
        \end{split}
    \end{equation}
    for some constant $C_6$ and $C_7$.
    
    Combining \eqref{3_8:subsec2}, \eqref{3_8:thirdpre}, and \eqref{3_8:thirdsub}, we get desired upper bound for $r\geq t^{-1}|x|$ case
    \begin{equation}\label{3_8:secondres}
        \begin{split}
            \sup_{0<t\leq 1}\sup_{0\leq \tau\leq 1} I_2\leq C_7(1+|x|^\gamma)^{(n+1-m)}\exp\left(-\alpha|x|^{\beta\gamma}\right)
        \end{split}
    \end{equation}
    for $|x|\geq \max\{(18 C_T\alpha)^{\frac{1}{1-\gamma}}, 1\}$.
    
    Merging the two cases \eqref{3_8:firstres} and \eqref{3_8:secondres}, we finally prove Lemma \ref{lem:int_bound1}.
\end{proof}

Even though Lemma \ref{lem:int_bound1} covers the whole case, we write down the corollary here to stress the polynomial bound.
\begin{lemma}\label{lem:int_bound2}
    For $|u|(t,x)\leq C_u(1+|x|^\gamma)$ for $0\leq t\leq 1$, we have
    \begin{align*}
        &\sup_{0\leq t,\tau\leq 1}\sup_{x\in\mathbb{R}^N} \frac{1}{\left(1+|x|^\gamma\right)^{n-m+N}}\int_{\mathbb{R}^N} dv\, \frac{|v|^n}{\left(|x-vt|^{2\gamma}+1\right)^{m/2}}\exp\left(-\frac{\left(v-u(\tau, x-vt)\right)^2}{C_T\left(|x-vt|^{2\gamma}+1\right)}\right)\\
        &\leq C(n, m, C_u, C_T, N, \gamma)<\infty
    \end{align*}
    for $n\geq 0$ and $m\in\mathbb{R}$.
\end{lemma}
\begin{proof}
    It is the $\alpha = 0$ case of Lemma \ref{lem:int_bound1}.
\end{proof}

\section{Ill-posedness theory of spatially inhomogeneous BGK}
In this section, we prove Theorem \ref{thm:2}. In particular, we construct an ill-posed solution explicitly for the spatially inhomogeneous BGK equation for a very short time. The initial data is constructed by removing a small velocity part proportional to the spatial coordinates so that the local temperature diverges as $|x|\rightarrow \infty$, incurring the norm blow-up as in the homogeneous case. Considering the ill-posedness scenario in the homogeneous BGK model, it may be reasonable to take a strong cutoff like $e^{-\alpha|v|^\beta}\mathbf{1}_{|x|^{100}}$ to make the temperature blow up fast. However, it causes some problems in showing the well-posedness of the macroscopic fields; see Lemma \ref{lem:int_bound1} and \ref{lem:int_bound2}. Counter-intuitively, therefore, we will use relatively small $\gamma$ depending on $\beta$ for $f_0(x,v)\sim e^{-\alpha |v|^\beta}\mathbf{1}_{|v|\geq |x|^\gamma}$.\\

The local in-time existence and uniqueness of such a solution are given as follows.
\begin{theorem}\label{thm:Ext_Uni}
    Let our spatial domain to be $\mathbb{R}^N_x$. For any given $N\geq 1$, $\alpha>0$, $0<\beta\leq 2$, $\frac{1}{3}\leq\gamma\leq \min\{\frac{1}{\beta+1}, \frac{1}{2}\}$, $\delta\geq 0$, and $\varpi\geq 0$, define 
    \begin{align}\label{4_1:initial}
        f_0(x,v) = (1+|x|)^{-\varpi}(1+|v|^2)^{-\delta} e^{-\alpha |v|^\beta}\mathbf{1}_{|x|^\gamma\leq |v|\leq 2|x|^\gamma + 10}.
    \end{align}
    Then, there exists $t_0(\alpha,\beta,\gamma,\delta,\varpi, N)$ such that there exists a unique solution $f(t,x,v)$ to \eqref{BGK} for $0\leq t\leq t_0$ having initial data $f_0$. Moreover, its macroscopic fields satisfy
    \begin{equation}\label{4_1:inhomo_macro}
        \begin{split}
            &C_1 \leq \frac{\rho(t,x)}{\left(1+|x|\right)^{(N-2\delta - \beta)\gamma-\varpi}\exp\left(-\alpha|x|^{\beta\gamma}\right)}\leq C_2,\\
            &|u(t,x)|\leq C_3(1+|x|^{\gamma})t^{1/6},\\
            &C_4 \leq \frac{T(t,x)}{1+|x|^{2\gamma}}\leq C_5
        \end{split}
    \end{equation}
    for all $0\leq t\leq t_0$ for some constants $0<C_i(\alpha,\beta,\gamma,\delta,\varpi,N)<\infty$.
\end{theorem}

Assuming Theorem \ref{thm:Ext_Uni}, we prove Theorem \ref{thm:2} first. We note that the increase of the macroscopic temperature $T$ in $x$ in \eqref{4_1:inhomo_macro} suggest that we have a weighted norm blow-up as $x\rightarrow \infty$ as in the proof of proposition \ref{prop:2_2} and \ref{prop:2_3}.\\

\begin{proof} [\textbf{Proof of Theorem \ref{thm:2}}] 

    For any $C>0$ and a solution of \eqref{BGK}, set $f_C = C^{-1}f(t,x,v)$. Since $\mathcal{M}(f_C) = C^{-1}\mathcal{M}(f)$, $f_C$ is again a solution of \eqref{BGK}. Therefore, it is enough to show the ill-posedness with the initial data satisfying $\left\|w_{\alpha,\beta,\delta}(v)f_0(x,v)\right\|_{L^p_xL^q_v} + \left\|w_{\alpha,\beta,\delta}(v)f_0(x,v)\right\|_{L^\infty_{x,v}}<\infty$ or $\left\|w_{\alpha,\beta,\delta}(v)f_0(x,v)\right\|_{L^q_vL^p_x} + \left\|w_{\alpha,\beta,\delta}(v)f_0(x,v)\right\|_{L^\infty_{x,v}}<\infty$.\\
    
    Let us first check the ill-posedness for the $L^p_xL^q_v$ case. We choose $\varpi$ in \eqref{4_1:initial} as follows:
    \begin{align}\label{omega_range}
        \begin{cases}
            \varpi\geq 0 & p=q=\infty,\\
            \varpi\geq N\gamma q^{-1} & p=\infty,q<\infty,\\
            \varpi>Np^{-1} & 1\leq p<\infty, q= \infty,\\
            \varpi>N(\gamma q^{-1}+p^{-1}) & 1\leq p,q<\infty.
        \end{cases}
    \end{align}
    Then, it satisfies
    \begin{align*}
        \left\|w_{\alpha,\beta,\delta}(v)f_0(x,v)\right\|_{L^p_xL^q_v} + \left\|w_{\alpha,\beta,\delta}(v)f_0(x,v)\right\|_{L^\infty_{x,v}} <\infty.
    \end{align*}
    By Theorem \ref{thm:Ext_Uni}, there exists a unique solution $f(t,x,v)$ for $0\leq t\leq t_0$.
    
    For any $0<\alpha'\leq \alpha$, using \eqref{4_1:inhomo_macro}, we have
    \begin{align}\label{4_1:finalform}
        \begin{split}
            &(1+|v|^2)^{\delta}e^{\alpha' |v|^\beta} f(t,x,v)\geq\int_0^t d\tau\,e^{-(t-\tau)}(1+|v|^2)^{\delta}e^{\alpha' |v|^\beta} \mathcal{M}(f)(\tau, x-v(t-\tau), v)\\
            &\geq C\int_0^t d\tau\,(1+|v|^2)^{\delta}e^{\alpha' |v|^\beta} \frac{(1+|x-v(t-\tau)|)^{(N-2\delta-\beta)\gamma-\varpi}e^{-\alpha|x-v(t-\tau)|^{\beta\gamma}}}{(1+|x-v(t-\tau)|^{2\gamma})^{N/2}}\\
            &\qquad\times\exp\left(-\frac{(v-u(\tau, x-v(t-\tau))^2}{2C_4\left(1+|x-v(t-\tau)|^{2\gamma}\right)}\right)
        \end{split}
    \end{align}
    for some constant $C>0$.
    We focus on the exponential part of the integral in RHS:
    \begin{align}\label{4_1:exponential}
        \exp\left(\alpha'|v|^\beta - \alpha|x-v(t-\tau)|^{\beta\gamma}-\frac{(v-u)^2}{2C_4\left(1+|x-v(t-\tau)|^{2\gamma}\right)}\right).
    \end{align}
    
    From now on, we will restrict $|x|\geq 1$. 
    
    \noindent \textit{(i) The case $\beta\leq 1$.}
    
    Set $v_{x,z} = v_x+z$, where $v_x\in\mathbb{R}^N$ is a vector satisfying 
    \begin{align*}
        |v_x| = \left(\frac{\alpha'\beta C_4}{8}|x|^{2\gamma}\right)^{1/(2-\beta)},
    \end{align*}
    and $z\in\mathbb{R}^N$ satisfies $|z|\leq 1$. Let us choose $|x|$ large enough so that $|v_x|\geq 2$, then $\frac{1}{2}|v_x|\leq |v_{x,z}|\leq 2|v_x|$. Using Lemma \ref{lem:power0}, we bound
    \begin{align}\label{4_1:betagammabound}
        \begin{split}
            &\alpha\frac{\left|x-v_{x,z}(t-\tau)\right|^{\beta\gamma}}{|v_x|^\beta} \\
            &\leq \alpha\frac{|x|^{\beta\gamma}+\left(2|v_x|(t-\tau)\right)^{\beta\gamma}}{|v_x|^\beta}\\
            &\leq \alpha\left(\left(\frac{8}{\alpha'\beta C_4}\right)^{\frac{\beta}{2-\beta}}|x|^{-\frac{\beta^2\gamma}{2-\beta}}+2^{\beta\gamma}\left(\frac{8}{\alpha'\beta C_4}\right)^{\frac{\beta(1-\gamma)}{2-\beta}}(t-\tau)^{\beta\gamma}|x|^{-\frac{2\beta\gamma}{2-\beta}(1-\gamma)}\right).
        \end{split}
    \end{align}
    The RHS goes to $0$ as $|x|\rightarrow\infty$, so we can choose large enough $x$ so that
    \begin{align*}
        \alpha|x-v_{x,z}(t-\tau)|^{\beta\gamma}\leq \frac{\alpha'}{2}|v_x|^\beta
    \end{align*}
    for all $0\leq |z|\leq 1$.
    
    As $\frac{2\gamma}{2-\beta}\leq 2\gamma\leq 1$, if we choose $t_0$ small enough so that $2\left(\frac{\alpha'\beta C_4}{8}\right)^{1/(2-\beta)}(t-\tau)\leq \frac{1}{2}$ for all $\tau\leq t\leq t_0$ and $|x|\geq 1$, we can bound
    \begin{equation}\label{4_1:v_xzbound}
        \begin{split}
            \left|x-v_{x,z}(t-\tau)\right|&\leq |x|+2(t-\tau)|v_x| = |x|+2\left(\frac{\alpha'\beta C_4}{8}\right)^{1/(2-\beta)}(t-\tau)|x|^{\frac{2\gamma}{2-\beta}}\leq 2|x|,\\
            \left|x-v_{x,z}(t-\tau)\right|&\geq \left||x|-2(t-\tau)|v_x|\right|=\left||x|-2\left(\frac{\alpha'\beta C_4}{8}\right)^{1/(2-\beta)}(t-\tau)|x|^{\frac{2\gamma}{2-\beta}}\right|\geq \frac{1}{2}|x|.
        \end{split}
    \end{equation}
    Finally, as $\frac{2\gamma}{2-\beta}> \gamma$, we have
    \begin{align*}
        \frac{|u|(\tau, x-v_{x,z}(t-\tau))}{|x|^{\frac{2\gamma}{2-\beta}}} &\leq C_3(1+\left|x-v_{x,z}(t-\tau)\right|^\gamma)|x|^{-\frac{2\gamma}{2-\beta}}\tau^{1/6}\leq C_3(1+2^\gamma |x|^\gamma)|x|^{-\frac{2\gamma}{2-\beta}}\tau^{1/6}\\
        &\leq 3C_3|x|^{\gamma - \frac{2\gamma}{2-\beta}}\tau^{1/6}.
    \end{align*}
    Thus, we can make $(v_{x,z}-u)^2\leq 2 v_x^2$ by choosing large enough $x$.
    
    Therefore, we get
    \begin{equation}\label{4_1:lowerbound1}
        \begin{split}
            &\liminf_{x\rightarrow\infty} \exp\left(\alpha'|v_{x,z}|^\beta - \alpha\left|x-v_{x,z}(t-\tau)\right|^{\beta\gamma}-\frac{(v_{x,z}-u)^2}{2C_4\left(1+\left|x-v_{x,z}(t-\tau)\right|^{2\gamma}\right)}\right)\\
            &\geq \liminf_{x\rightarrow\infty}\exp\left(\frac{1}{2}\alpha'|v_x|^\beta-\frac{v_x^2}{C_4\left(1+\frac{1}{2^{2\gamma}}|x|^{2\gamma}\right)}\right)\\
            &\geq \liminf_{x\rightarrow\infty}\exp\left(\frac{1}{2}\alpha'|v_x|^\beta-\frac{2v_x^2}{C_4|x|^{2\gamma}}\right)\geq \liminf_{x\rightarrow\infty}\left(\exp\left(\alpha'|v_x|^\beta-\frac{4|v_x|^2}{C_4|x|^{2\gamma}}\right)\right)^{1/2},
        \end{split}
    \end{equation}
    where we used \eqref{4_1:betagammabound} and \eqref{4_1:v_xzbound} in the middle line. It resembles the homogeneous BGK model case with temperature $T = \frac{C_4}{8}|x|^{2\gamma}$, so we get the exponential blow-up at $v = v_{x,z}$ for $|z|\leq 1$ as $x\rightarrow \infty$ from the proof of proposition \ref{prop:2_3}, overwhelming all polynomial decay about $x$ in \eqref{4_1:finalform}. Therefore, it shows $L^\infty_{x,v}$ norm blows up of \eqref{4_1:finalform}.

    To compute the $L^q_v$ lower bound, using Fatou's lemma, we bound it for $0<t\leq t_0$ by
    \begin{align}
        &\liminf_{x\rightarrow\infty}\left\|(1+|v|^2)^{\delta}e^{\alpha' |v|^\beta} f(t,x,v)\right\|^q_{L^q_v}\notag\\
        &\geq \liminf_{x\rightarrow\infty}\int_{\mathbb{R}^N}dv\,\Bigg(\int_0^t d\tau\, e^{-(t-\tau)}e^{\alpha' |v|^\beta} \mathcal{M}(f)(\tau, x-v(t-\tau), v)\Bigg)^q\notag\\
        &\geq C\liminf_{x\rightarrow\infty}\int_{\mathbb{R}^N}dv\,\Bigg(\int_0^t d\tau\,e^{\alpha' |v|^\beta} \frac{(1+|x-v(t-\tau)|)^{(N-2\delta-\beta)\gamma-\varpi}e^{-\alpha|x-v(t-\tau)|^{\beta\gamma}}}{(1+|x-v(t-\tau)|^{2\gamma})^{N/2}}\notag\\
        &\qquad \times \exp\left(-\frac{(v-u(\tau, x-v(t-\tau)))^2}{2C_4\left(1+|x-v(t-\tau)|^{2\gamma}\right)}\right)\Bigg)^q\notag\\
        &\geq C\int_{\mathbb{R}^N}dv\,\left(\int_0^t d\tau\,\liminf_{x\rightarrow\infty}\frac{e^{\alpha' |v|^\beta}e^{-\alpha|x-v(t-\tau)|^{\beta\gamma}}}{(1+|x-v(t-\tau)|)^{(2\delta + \beta)\gamma + \varpi}}\exp\left(-\frac{(v-u(\tau, x-v(t-\tau)))^2}{2C_4\left(1+|x-v(t-\tau)|^{2\gamma}\right)}\right)\right)^q \notag\\
        &\geq C\int_{|z|\leq 1}dv\,\left(\int_0^t d\tau\,\liminf_{x\rightarrow\infty}\frac{e^{\alpha' |v_{x,z}|^\beta}e^{-\alpha|x-v_{x,z}(t-\tau)|^{\beta\gamma}}}{(1+|x-v_{x,z}(t-\tau)|)^{(2\delta + \beta)\gamma + \varpi}}\exp\left(-\frac{(v_{x,z}-u(\tau, x-v_{x,z}(t-\tau)))^2}{2C_4\left(1+|x-v_{x,z}(t-\tau)|^{2\gamma}\right)}\right)\right)^q\label{4_1:Lqblowup_inter}.
    \end{align}
    From the last line, we restricted the domain of integration to $v_x+z$, with $|z|\leq 1$. By \eqref{4_1:v_xzbound} and \eqref{4_1:lowerbound1}, we can write
    \begin{align}\label{4_1:Lqblowup}
        \begin{split}
            \eqref{4_1:Lqblowup_inter}&\geq C\int_{|z|\leq 1}dz\,\left(\int_0^t d\tau\,\liminf_{x\rightarrow\infty}\left[\frac{1}{(1+2|x|)^{(4\delta + 2\beta)\gamma + 2\varpi}}\exp\left(\alpha'|v_x|^\beta-\frac{4|v_x|^2}{C_4|x|^{2\gamma}}\right)\right]^{1/2}\right)^q.
        \end{split}
    \end{align}
    The final integrand is independent of $\tau$ and $v$ and blows up as $x\rightarrow\infty$. It proves
    \begin{align*}
        \left\|w_{\alpha',\beta,\delta}(v)f(t,x,v)\right\|_{L^p_xL^q_v} = \infty
    \end{align*}
    for any $1\leq p,q\leq \infty$ for any $0<t\leq t_0$.\\

    \noindent \textit{(ii) The case $1<\beta\leq 2$.}
    
    We again write $v_{x,z} = v_x+z$, where $v_x\in\mathbb{R}^N$ is a vector satisfying 
    \begin{align*}
        |v_x| = |x|^{2\gamma},
    \end{align*}
    and $z\in\mathbb{R}^N$ satisfies $|z|\leq 1$. Choosing $|x|\geq 2^{-2\gamma}$, we can assume $\frac{1}{2}|x|^{2\gamma}\leq |v_{x,z}|\leq 2|x|^{2\gamma}$ for all $|z|\leq 1$. Following a similar line of proof, we can obtain the similar bounds
    \begin{align*}
        &\alpha\left|x-v_{x,z}(t-\tau)\right|^{\beta\gamma}\leq \frac{\alpha'}{2}|v_x|^\beta,\\
        &\frac{1}{2}|x|\leq \left|x-v_{x,z}(t-\tau)\right|\leq 2|x|,\\
        &\frac{|u|(\tau, x-v_{x,z}(t-\tau))}{|x|^{2\gamma}} \leq 3C_3|x|^{-\gamma}\tau^{1/4}\Rightarrow (v_{x,z}-u)^2\leq 2 v_x^2
    \end{align*}
    for all $|z|\leq 1$, large $x$, and small $t_0$. Therefore, we obtain
    \begin{align*}
        &\liminf_{x\rightarrow\infty} \exp\left(\alpha'|v_{x,z}|^\beta - \alpha\left|x-v_{x,z}(t-\tau)\right|^{\beta\gamma}-\frac{(v_{x,z}-u)^2}{2C_4\left(1+\left|x-v_{x,z}(t-\tau)\right|^{2\gamma}\right)}\right)\\
        &\geq \liminf_{x\rightarrow\infty}\exp\left(\frac{1}{2}\alpha'|v_x|^\beta-\frac{v_x^2}{C_4\left(1+\frac{1}{2^{2\gamma}}|x|^{2\gamma}\right)}\right)\\
        &\geq \liminf_{x\rightarrow\infty}\exp\left(\frac{1}{2}\alpha'|x|^{2\beta\gamma}-\frac{2|x|^{4\gamma}}{C_4|x|^{2\gamma}}\right) = \liminf_{x\rightarrow\infty}\left(\exp\left(\alpha'|x|^{2\beta\gamma}-\frac{4}{C_4}|x|^{2\gamma}\right)\right)^{1/2}
    \end{align*}
    for all $|z|\leq 1$ and small enough $t_0$. Finally, as $\beta>1$, $|x|^{2\beta\gamma}$ is the dominating term and we get blow-up at $v_{x,z}$ for all $|z|\leq 1$. It shows
    \begin{align*}
        \left\|w_{\alpha',\beta,\delta}f(t)\right\|_{L^p_xL^q_v} = \infty
    \end{align*}
    for any $1\leq p,q\leq \infty$ for any $0<t\leq t_0$. This ends the proof of Theorem \ref{thm:2} for the $L^p_xL^q_v$ norm.\\
    
    Now, we prove the same argument for $L^q_vL^p_x$ space. We first state that $\|w_{\alpha,\beta,\delta}(v)f_0(x,v)\|_{L^q_vL^p_x}<\infty$ if and only if $\varpi$ satisfies \eqref{omega_range}.
    
    We have the following:
    \begin{align*}
        \{(x,v):|x|^\gamma\leq |v|\leq 2|x|^\gamma+10\} = \left\{(x,v):\frac{(\max\{|v|-10, 0\})^{1/\gamma}}{2}\leq |x|\leq |v|^{1/\gamma}\right\}.
    \end{align*}

    We first assume $p<\infty$. Using $\frac{1}{1+r}\leq \frac{1}{(1+r^N)^{1/N}}$ for $r\geq 0$,
    \begin{align*}
        \int_\Omega \left(f_0(x,v)\right)^p\,dx &=(1+|v|^2)^{-p\delta} e^{-p\alpha |v|^\beta}\int_\Omega (1+|x|)^{-p\varpi}\mathbf{1}_{\frac{(\max\{|v|-10, 0\})^{1/\gamma}}{2}\leq |x|\leq |v|^{1/\gamma}}\,dx \\
        &\leq C_N(1+|v|^2)^{-p\delta} e^{-p\alpha |v|^\beta}\int_{\frac{(\max\{|v|-10, 0\})^{1/\gamma}}{2}}^{|v|^{1/\gamma}} \frac{r^{N-1}}{(1+r^N)^{p\varpi/N}}\,dr.
    \end{align*}

    If $p\varpi\leq N$, then the integral blows up as $v\rightarrow \infty$. If $p\varpi>N$, using $\frac{1}{1+r^N}\leq C'_N\frac{1}{(1+r)^N}$ for $r\geq 0$ for some constant $C'_N$,
    \begin{align*}
        &(1+|v|^2)^{-p\delta} e^{-p\alpha |v|^\beta}\int_0^\infty \frac{r^{N-1}}{(1+r^N)^{p\varpi/N}}\mathbf{1}_{\frac{(\max\{|v|-10, 0\})^{1/\gamma}}{2}\leq |x|\leq |v|^{1/\gamma}}\\
        &\leq C_{p,\varpi, N}(1+|v|^2)^{-p\delta} e^{-p\alpha |v|^\beta}\left[-(1+r^N)^{1-p\varpi/N}\right]_{\frac{(\max\{|v|-10, 0\})^{1/\gamma}}{2}}^{|v|^{1/\gamma}}\\
        &\leq C_{p,\varpi, N}(1+|v|^2)^{-p\delta} e^{-p\alpha |v|^\beta}\left(1+\frac{(\max\{|v|-10, 0\})^{1/\gamma}}{2}\right)^{N-p\varpi}.
    \end{align*}
    
    Now, if $q<\infty$, we have
    \begin{align*}
        &\int_{\mathbb{R}^N} \left[w_{\alpha,\beta,\delta}(v)\left(\int_\Omega \left(f_0(x,v)\right)^p\,dx\right)^{1/p}\right]^q\,dv\\
        &\leq C_{p,\varpi, N}\int_{\mathbb{R}^N}(1+\frac{(\max\{|v|-10, 0\})^{1/\gamma}}{2})^{qN/p-q\varpi}\,dv\\
        &\leq C_{p,\varpi, N}\left(\int_{|v|\leq 20}\,dv + \int_{|v|> 20} \left(1+\frac{(|v|/2)^{1/\gamma}}{2}\right)^{qN/p-q\varpi}\,dv\right),
    \end{align*}
    and the integral converges if and only if $\varpi$ satisfies \eqref{omega_range}. For $q=\infty$, we can again check that $\varpi$ should satisfies \eqref{omega_range}.
    
    If $p=\infty$, for $\varpi\geq 0$,
    \begin{align*}
        \|f_0(x,v)\|_{L^\infty_x} &= (1+|v|^2)^{-\delta} e^{-\alpha |v|^\beta} \|(1+|x|)^{-\varpi}\mathbf{1}_{\frac{(\max\{|v|-10, 0\})^{1/\gamma}}{2}\leq |x|\leq |v|^{1/\gamma}}\|_{L^\infty_x}\\
        &=(1+|v|^2)^{-\delta} e^{-\alpha |v|^\beta} \left(1+\frac{(\max\{|v|-10, 0\})^{1/\gamma}}{2}\right)^{-\varpi},
    \end{align*}
    and it is finite on $w_{\alpha,\beta,\delta}(v)$ weighted $L^q_v$ norm if and only if $\varpi$ satisfies \eqref{omega_range}. Finally, we checked that
    \begin{align*}
        \left\|w_{\alpha,\beta,\delta}(v)f_0(x,v)\right\|_{L^q_vL^p_x} + \left\|w_{\alpha,\beta,\delta}(v)f_0(x,v)\right\|_{L^\infty_{x,v}} <\infty.
    \end{align*}
    
    To show the ill-posedness, let us consider $\mathcal{M}(f)(\tau, x-v(t-\tau), v)$. As before, we will focus on the blow up of exponential term \eqref{4_1:exponential}:
    \begin{align*}
        \exp\left(\alpha'|v|^\beta - \alpha|x-v(t-\tau)|^{\beta\gamma}-\frac{(v-u)^2}{2C_4\left(1+|x-v(t-\tau)|^{2\gamma}\right)}\right).
    \end{align*}
    
    Set $x_{v,z} = x_v+z$, where $x_v\in\mathbb{R}^N$ is a vector satisfying 
    \begin{align*}
        |x_v| = |v|^{\frac{4-\beta}{4\gamma}},
    \end{align*}
    and $z\in\mathbb{R}^N$ satisfies $|z|\leq 1$. Let us choose $|v|$ large enough so that $|x_v|\geq 2$, then $\frac{1}{2}|x_v|\leq |x_{v,z}|\leq 2|x_v|$. Using Lemma \ref{lem:power0}, we bound
    \begin{align}\label{4_1:betagammabound2}
        \alpha\frac{\left|x_{v,z}-v(t-\tau)\right|^{\beta\gamma}}{|v|^\beta}\leq \alpha\frac{|2x_{v}|^{\beta\gamma}+|v(t-\tau)|^{\beta\gamma}}{|v|^\beta}= \alpha \left(2^{\beta\gamma}|v|^{\frac{\beta(4-\beta)}{4} - \beta} + (t-\tau)^{\beta\gamma}|v|^{\beta(\gamma-1)}\right).
    \end{align}
    The RHS goes to $0$ as $|v|\rightarrow\infty$, so we can choose large enough $v$ so that
    \begin{align*}
        \alpha|x_{v,z}-v(t-\tau)|^{\beta\gamma}\leq \frac{\alpha'}{2}|v|^\beta
    \end{align*}
    for all $0\leq |z|\leq 1$.
    
    As $\frac{4-\beta}{4\gamma}\geq \frac{4-2}{4\frac{1}{2}}\geq 1$, if we choose $t_0$ small enough so that $t-\tau\leq 1/4$ for all $t\leq\tau\leq t_0$ and $|v|\geq 1$, we can assume
    \begin{equation}\label{4_1:v_xzbound2}
        \begin{split}
            \left|x_{v,z}-v(t-\tau)\right|&\leq 2|x_v|+(t-\tau)|v| = 2|v|^{\frac{4-\beta}{4\gamma}}+(t-\tau)|v|\leq 3|v|^{\frac{4-\beta}{4\gamma}},\\
            \left|x_{v,z}-v(t-\tau)\right|&\geq \left|\frac{1}{2}|x_v|-(t-\tau)|v|\right|=\left|\frac{1}{2}|v|^{\frac{4-\beta}{4\gamma}} - (t-\tau)|v|\right|\geq \frac{1}{4}|v|^{\frac{4-\beta}{4\gamma}}.
        \end{split}
    \end{equation}
    Finally, we have
    \begin{align*}
        |u|(\tau, x_{v,z}-v(t-\tau)) &\leq C_3(1+\left|x_{v,z}-v(t-\tau)\right|^\gamma)\tau^{1/4}\leq C_3(1+3^\gamma |v|^{\frac{4-\beta}{4}})\tau^{1/4}.
    \end{align*}
    Thus, we can make $|u|(\tau, x_{v,z}-v(t-\tau))\leq \frac{|v|}{2}$ by choosing large enough $v$ as $\frac{4-\beta}{4}<1$. It implies $(v-u)^2\leq 2|v|^2$.
    
    Putting all these bounds to \eqref{4_1:exponential}, we get
    \begin{equation*}
        \begin{split}
            &\liminf_{v\rightarrow\infty} \exp\left(\alpha'|v|^\beta - \alpha\left|x_{v,z}-v(t-\tau)\right|^{\beta\gamma}-\frac{(v-u)^2}{2C_4\left(1+\left|x_{v,z}-v(t-\tau)\right|^{2\gamma}\right)}\right)\\
            &\geq \liminf_{v\rightarrow\infty}\exp\left(\frac{1}{2}\alpha'|v|^\beta-\frac{|v|^2}{C_4\left(1+\frac{1}{4^{2\gamma}}|v|^{\frac{4-\beta}{2}}\right)}\right)\\
            &\geq \liminf_{v\rightarrow\infty}\exp\left(\frac{1}{2}\alpha'|v|^\beta-\frac{4|v|^2}{C_4|v|^{\frac{4-\beta}{2}}}\right)= \liminf_{v\rightarrow\infty}\exp\left(\frac{1}{2}\alpha'|v|^\beta-\frac{4}{C_4}|v|^{\beta/2}\right),
        \end{split}
    \end{equation*}
   where we used \eqref{4_1:betagammabound2} and \eqref{4_1:v_xzbound2} in the middle line. Comparing the exponent of $|v|$, we conclude that \eqref{4_1:exponential} blows up at $x = x_{v,z}$ for $|z|\leq 1$ as $v\rightarrow \infty$. Following similar calculation in \eqref{4_1:Lqblowup}, it shows that the
   \begin{align*}
       \left\|\int_0^t w_{\alpha',\beta,\delta}(v)\mathcal{M}(f)(\tau, x-v(t-\tau), v)\,d\tau\right\|_{L^q_vL^p_x} = \infty
   \end{align*}
   for any small enough $t>0$ and $\alpha'>0$. Eventually, we obtain
   \begin{align*}
       \|w_{\alpha',\beta,\delta}f(t)\|_{L^q_vL^p_x} = \infty
   \end{align*}
   for any $t\leq t_0$ following the similar proof step of the $L^p_xL^q_v$ case.\\
\end{proof}

The remainder of the section is devoted to the proof of the local existence result. Before diving into the complicated calculation, let us briefly sketch the proof of Theorem \ref{thm:Ext_Uni}. 
To prove the existence and uniqueness of the solution, we first define a solution space $X$ to be the space of functions satisfying
\begin{equation}\label{4_1:spacedef}
    \begin{split}
        &~\sup_{t\leq t_0}\left\|\int_{\mathbb{R}_{v}^{N}} (1+|x|^{2\gamma}+v^2)f(t,x,v)\,dv\right\|_{L^\infty_{x}}<\infty,\\
        &C_1 \leq \frac{\rho(t,x)}{\left(1+|x|\right)^{(N-2\delta - \beta)\gamma-\varpi}\exp\left(-\alpha|x|^{\beta\gamma}\right)}\leq C_2,\\
        &\hspace{-0.3cm}|u(t,x)|\leq C_3(1+|x|^{\gamma})t^{1/6},\qquad
        C_4 \leq \frac{T(t,x)}{1+|x|^{2\gamma}}\leq C_5
    \end{split}
\end{equation}
for some constants $0<C_i<\infty$ and $0\leq t\leq t_0$, where $t_0>0$ is a fixed small constant. The constants will be chosen in the proof. It is a Banach space with the norm given by
\begin{align*}
    \|f\| \coloneqq\sup_{t\leq t_0}\left\|\int_{\mathbb{R}_{v}^{N}} (1+|x|^{2\gamma}+v^2)f(t,x,v)\,dv\right\|_{L^\infty_{x}}.
\end{align*}
Now, define the solution operator $\mathcal{T}:X\rightarrow X$ by
\begin{align}\label{4_1:Tmap}
    \mathcal{T}f = e^{-t}f_0(x-vt, v) + \int_0^t e^{-(t-\tau)}\mathcal{M}(f)(\tau, x-v(t-\tau), v)\,d\tau.
\end{align}

First, we need to check that $\mathcal{T}$ is well-defined. Precisely speaking, we need to prove that there exists $t_0$ such that the density $\rho$, bulk velocity $u$, and local temperature $T$ of $\mathcal{T}f$ is well-defined throughout $t\leq t_0$. The idea is to show that the macroscopic fields do not vary much from the initial data for a very short time in the sense that
\begin{align*}
    C\leq \sup_{t,x}\frac{\left|\int (1,v,v^2) f(t,x,v)\,dv\right|}{\left|\int (1,v,v^2) f_0(x,v)\,dv\right|}\leq C'
\end{align*}
for some constant $C$ and $C'$. For this, we will prove that
\begin{align*}
    \sup_{x}\frac{\left|\int_{\mathbb{R}^N} \int_0^t (1,v,v^2)\mathcal{M}(f)(\tau, x-v(t-\tau), v)\,d\tau dv\right|}{\left|\int (1,v,v^2) f_0(x,v)\,dv\right|} = O(t^\sigma)
\end{align*}
for some $\sigma>0$.
    
To explain why it is possible, let us replace $\rho$, $u$, and $T$ in $\mathcal{M}(f)$ by the estimates in \eqref{4_1:spacedef}:
\begin{align*}
    &\mathcal{M}(f)(\tau, x-v(t-\tau), v) \\
    &= \frac{\rho(\tau, x-v(t-\tau))}{(2\pi T(\tau, x-v(t-\tau)))^{N/2}}\exp\left(-\frac{(v-u(\tau, x-v(t-\tau)))^2}{2T(\tau, x-v(t-\tau))}\right)\\
    &\simeq C\frac{\exp\left(-\alpha|x-v(t-\tau)|^{\beta\gamma}\right)}{\left(1+|x-v(t-\tau)|^{2\gamma}\right)^{m/2}}\exp\left(-\frac{\left(v-u(\tau, x-v(t-\tau))\right)^2}{2C_4\left(1+|x-v(t-\tau)|^{2\gamma}\right)}\right)
\end{align*}
with $|u|(t, x)\leq C_3(1+|x|^\gamma)t^{1/6}$. Therefore, we can apply Lemma \ref{lem:int_bound1} and bound the integral uniformly about $t$. 

After checking the well-definedness of $\mathcal{T}$, we check that $\mathcal{T}$ satisfies contraction mapping property and finally apply the Banach fixed point theorem to conclude the existence and uniqueness of the solution.\\
\begin{proof}[Proof of Theorem \ref{thm:Ext_Uni}]
The solution $f$ of \eqref{BGK} should satisfy
    \begin{align*}
        f(t,x,v) = e^{-t}f_0(x-vt, v) + \int_0^t e^{-(t-s)}\mathcal{M}(f)(\tau, x-v(t-\tau), v)\,d\tau.
    \end{align*}
    To apply a fixed point argument to prove the existence and uniqueness of the solution, we define an operator $\mathcal{T}:X\rightarrow X$ as \eqref{4_1:Tmap} and $X$ by \eqref{4_1:spacedef}. Recall that $t_0$ will be determined in the proof; tentatively, we will restrict $t_0 < 1$. 
    We prove the theorem step by step.\\

    \noindent \textit{(i)} In this step, we prove $g(t,x,v) = f_0(x-vt ,v)\in X$.

    We define
    \begin{align}\label{4_1:macro0_def}
        (\tilde{\rho}_0, \tilde{\rho}_0 \tilde{u}_0, \tilde{E}_0)(t,x)=\int_{\mathbb{R}^N} (1, v, |v|^2) f_0(x-vt, v)\,dv.
    \end{align}
    Note that $\tilde{E}_0$ is computed using the weight $|v|^2$. To show that $g(t,x,v) = f_0(x-vt, v)\in X$, we need to verify $f_0$ satisfies \eqref{4_1:spacedef}; the condition for temperature is replaced by
    \begin{align*}
        C_4\leq \frac{1}{1+|x|^{2\gamma}}\frac{\tilde{E}_0 - \tilde{\rho}_0|\tilde{u}_0|^2}{N \tilde{\rho}_0}\leq C_5.
    \end{align*}

    The density $\tilde{\rho}_0(t,x)$ is bounded as follows:
    \begin{align*}
        \tilde{\rho}_0(t,x) &= \int_{\mathbb{R}^N} (1+|x-vt|)^{-\varpi}(1+|v|^2)^{-\delta} e^{-\alpha |v|^\beta}\mathbf{1}_{|x-vt|^\gamma\leq |v|\leq 2|x-vt|^\gamma + 10}\,dv\\
        &\leq \int_{\mathbb{R}^N} (1+\max\{|v|/2-5, 0\}^{1/\gamma})^{-\varpi}(1+|v|^2)^{-\delta}e^{-\alpha |v|^\beta}\mathbf{1}_{\left||x|-|v|t\right|^\gamma\leq |v|}\,dv\\
        &=|\mathbb{S}^{N-1}|\int_0^\infty (1+\max\{r/2-5, 0\}^{1/\gamma})^{-\varpi}(1+r^2)^{-\delta}e^{-\alpha r^\beta} \mathbf{1}_{\left||x|-rt\right|^\gamma\leq r} r^{N-1}\,dr,\\
        \intertext{and}
        \tilde{\rho}_0(t,x) &\geq \int_{\mathbb{R}^N} (1+|v|^{1/\gamma})^{-\varpi}(1+|v|^2)^{-\delta} e^{-\alpha |v|^\beta}\mathbf{1}_{\left||x|+|v|t\right|^\gamma\leq |v|\leq 2\left||x|-|v|t\right|^\gamma + 10}\,dv\\
        &=|\mathbb{S}^{N-1}|\int_0^\infty (1+r^{1/\gamma})^{-\varpi}(1+r^2)^{-\delta} e^{-\alpha r^\beta} \mathbf{1}_{\left||x| + rt\right|^\gamma\leq r\leq 2\left||x| - rt\right|^\gamma + 10} r^{N-1}\,dr.
    \end{align*}
    For $|x|\geq 1$, $\frac{1}{3}\leq \gamma \leq \min\left\{\frac{1}{1+\beta}, \frac{1}{2}\right\}$, and $t<\frac{1}{2}$, from Lemma \ref{lem:Tri_integral} and Lemma \ref{lem:pert_ineq2}, there exist constants $0<C_1<C_2<\infty$ such that
    \begin{align}\label{4_1:pre_rho0_bound}
        \begin{split}
            \tilde{\rho}_0(t,x)&\leq |\mathbb{S}^{N-1}|\int_{|x|^\gamma - \gamma t|x|^{-(1-2\gamma)}}^\infty (1+\max\{r/2-5, 0\}^{1/\gamma})^{-\varpi}(1+r^2)^{-\delta} e^{-\alpha r^\beta} r^{N-1}\,dr\\
            &\leq \frac{C_2}{2}\left(1+|x|\right)^{(N-2\delta -\beta)\gamma - \varpi} e^{-\alpha |x|^{\beta\gamma}},\\
            \tilde{\rho}_0(t,x)&\geq |\mathbb{S}^{N-1}|\int_{|x|^\gamma + 2\gamma t|x|^{-(1-2\gamma)}}^{2|x|^\gamma -4\gamma t|x|^{-(1-2\gamma)}} (1+r^{1/\gamma})^{-\varpi} (1+r^2)^{-\delta}e^{-\alpha r^\beta} r^{N-1}\,dr\\
            &\geq 2C_1\left(1+|x|\right)^{(N-2\delta -\beta)\gamma - \varpi} e^{-\alpha |x|^{\beta\gamma}}.
        \end{split}
    \end{align}
    On the other hand, when $0\leq |x|\leq 1$ and $t<1$, we have
    \begin{align*}
        [2, 10]\subset \{r:\left||x| + rt\right|^\gamma\leq r\leq 2\left||x| - rt\right|^\gamma + 10\}
    \end{align*}
    as $|1+2\cdot 1|^{1/2}\leq 2$ and $|1+10\cdot 1|^{1/2}\leq 10$; we remind that $\gamma\leq 1/2$. Therefore, $\tilde{\rho}_0(t,x)$ satisfies
    \begin{align*}
        |\mathbb{S}^{N-1}|\int_{2}^{10} (1+r^{1/\gamma})^{-\varpi}(1+r^2)^{-\delta}e^{-\alpha r^\beta} r^{N-1}\,dr\leq \tilde{\rho}_0(t,x)\leq |\mathbb{S}^{N-1}|\int_{0}^\infty (1+r^2)^{-\delta} e^{-\alpha r^\beta} r^{N-1}\,dr
    \end{align*}
    so that it is uniformly bounded from above and below. Combining this bound with \eqref{4_1:pre_rho0_bound}, we get
    \begin{align}\label{4_1:rho0_bound}
        2C_1 \leq \frac{\tilde{\rho}_0(t,x)}{\left(1+|x|\right)^{(N-2\delta - \beta)\gamma-\varpi}\exp\left(-\alpha|x|^{\beta\gamma}\right)}\leq \frac{C_2}{2}
    \end{align}
    for some $C_1$ and $C_2$.\\

    For energy $\tilde{E}_0$, following the same calculation, we can show that
    \begin{equation*}
        \begin{split}
           \tilde{E}_0(t,x)&\leq |\mathbb{S}^{N-1}|\int_{|x|^\gamma - \gamma t|x|^{-(1-2\gamma)}}^\infty (1+\max\{r/2-5, 0\}^{1/\gamma})^{-\varpi}(1+r^2)^{-\delta}e^{-\alpha r^\beta} r^{N+1}\,dr\\
            &\leq \frac{C_5}{2}\left(1+|x|\right)^{(N+2-2\delta-\beta)\gamma - \varpi} e^{-\alpha |x|^{\beta\gamma}},\\
            \tilde{E}_0(t,x)&\geq |\mathbb{S}^{N-1}|\int_{|x|^\gamma + 2\gamma t|x|^{-(1-2\gamma)}}^{2|x|^\gamma -4\gamma t|x|^{-(1-2\gamma)}}(1+r^{1/\gamma})^{-\varpi}(1+r^2)^{-\delta}e^{-\alpha r^\beta} r^{N+1}\,dr\\
            &\geq 2C_4\left(1+|x|\right)^{(N+2-2\delta-\beta)\gamma - \varpi} e^{-\alpha |x|^{\beta\gamma}}
        \end{split}
    \end{equation*}
    for $|x|\geq 1$ and
    \begin{align*}
        |\mathbb{S}^{N-1}|\int_{2}^{10} (1+r^{1/\gamma})^{-\varpi}(1+r^2)^{-\delta}e^{-\alpha r^\beta} r^{N+1}\,dr\leq \tilde{E}_0(t,x)\leq |\mathbb{S}^{N-1}|\int_{0}^\infty (1+r^2)^{-\delta}e^{-\alpha r^\beta} r^{N+1}\,dr
    \end{align*}
    for $|x|\leq 1$. Therefore, dividing these bounds by $\tilde{\rho}_0$ bound \eqref{4_1:rho0_bound}, we get
    \begin{align}\label{4_1:E0_bound}
        2C_4 \leq \frac{1}{1+|x|^{2\gamma}}\frac{\tilde{E}_0(t,x)}{N\tilde{\rho}_0}\leq \frac{C_5}{2}
    \end{align}
    for some $C_4$ and $C_5$.\\

    We now turn to the estimate of $|\tilde{u}_0|$. We assume that $t< \frac{1}{2}$. We write $\tilde{\rho}_0 \tilde{u}_0(t, x)$ as follows.
    \begin{align*}
        2\tilde{\rho}_0 \tilde{u}_0(t, x)& = 2\int_{\mathbb{R}^N} v (1+|x-vt|)^{-\varpi}(1+|v|^2)^{-\delta}e^{-\alpha |v|^\beta}\mathbf{1}_{|x-vt|^\gamma\leq |v|\leq 2|x-vt|^\gamma + 10}\,dv\\
        & \leq \int_{\mathbb{R}^N} v (1+|v|^2)^{-\delta}e^{-\alpha |v|^\beta}\\
        &\qquad \times \left((1+|x-vt|)^{-\varpi}\mathbf{1}_{|x-vt|^\gamma\leq |v|\leq 2|x-vt|^\gamma + 10} - (1+|x+vt|)^{-\varpi}\mathbf{1}_{|x+vt|^\gamma\leq |v|\leq 2|x+vt|^\gamma + 10}\right)\,dv.
    \end{align*}
    We decompose the difference term by
    \begin{align*}
        &(1+|x-vt|)^{-\varpi}\mathbf{1}_{|x-vt|^\gamma\leq |v|\leq 2|x-vt|^\gamma + 10} - (1+|x+vt|)^{-\varpi}\mathbf{1}_{|x+vt|^\gamma\leq |v|\leq 2|x+vt|^\gamma + 10}\\
        &=\left((1+|x-vt|)^{-\varpi}-(1+|x+vt|)^{-\varpi}\right)\mathbf{1}_{|x-vt|^\gamma\leq |v|\leq 2|x-vt|^\gamma + 10}\\
        &\quad + (1+|x+vt|)^{-\varpi}\left(\mathbf{1}_{|x-vt|^\gamma\leq |v|}\mathbf{1}_{|v|\leq 2|x-vt|^\gamma + 10}-\mathbf{1}_{|x+vt|^\gamma\leq |v|}\mathbf{1}_{|v|\leq 2|x+vt|^\gamma + 10}\right)\\
        &=\left((1+|x-vt|)^{-\varpi}-(1+|x+vt|)^{-\varpi}\right)\mathbf{1}_{|x-vt|^\gamma\leq |v|\leq 2|x-vt|^\gamma + 10}\\
        &\quad + (1+|x+vt|)^{-\varpi}\left(\mathbf{1}_{|x-vt|^\gamma\leq |v|}-\mathbf{1}_{|x+vt|^\gamma\leq |v|}\right)1_{|v|\leq 2|x-vt|^\gamma + 10}\\
        &\quad + (1+|x+vt|)^{-\varpi}\left(1_{|v|\leq 2|x-vt|^\gamma + 10}-1_{|v|\leq 2|x+vt|^\gamma + 10}\right)\mathbf{1}_{|x+vt|^\gamma\leq |v|}\\
        &\leq \left|(1+|x-vt|)^{-\varpi}-(1+|x+vt|)^{-\varpi}\right|\mathbf{1}_{||x|-|v|t|^\gamma\leq |v|\leq 2(|x|+|v|t)^\gamma + 10}\\
        &\quad + (1+|x+vt|)^{-\varpi}\mathbf{1}_{||x|-|v|t|^\gamma\leq |v|\leq (|x|+|v|t)^\gamma}\\
        &\quad + (1+|x+vt|)^{-\varpi}\mathbf{1}_{2||x|-|v|t|^\gamma\leq |v|-10\leq 2(|x|+|v|t)^\gamma}.
    \end{align*}
    We split the integral by
    \begin{align*}
        D_1&\coloneqq \int_{\mathbb{R}^N} |v| (1+|v|^2)^{-\delta}e^{-\alpha |v|^\beta}\left|(1+|x-vt|)^{-\varpi}-(1+|x+vt|)^{-\varpi}\right|\mathbf{1}_{||x|-|v|t|^\gamma \leq |v|\leq 2(|x|+|v|t)^\gamma + 10}\,dv,\\
        D_2&\coloneqq \int_{\mathbb{R}^N} |v| (1+|v|^2)^{-\delta}e^{-\alpha |v|^\beta}(1+|x+vt|)^{-\varpi}\mathbf{1}_{||x|-|v|t|^\gamma\leq |v|\leq (|x|+|v|t)^\gamma}\,dv,\\
        D_3&\coloneqq \int_{\mathbb{R}^N} |v| (1+|v|^2)^{-\delta}e^{-\alpha |v|^\beta}(1+|x+vt|)^{-\varpi}\mathbf{1}_{2||x|-|v|t|^\gamma\leq |v|-10\leq 2(|x|+|v|t)^\gamma}\,dv.
    \end{align*}
    To bound the integrals, we divide the cases: $|x|\geq 20 t^{\frac{\gamma}{1-\gamma}}$ and the other case.\\

    \noindent \textit{(i) The case $|x|\geq 20 t^{\frac{\gamma}{1-\gamma}}$.}
    
    We first assume $|x|\geq 20 t^{\frac{\gamma}{1-\gamma}}$. For further analysis, we bound the size of the set of the step functions. We claim that
    \begin{align}
        \{||x|-|v|t|^\gamma\leq |v|\}&\subset \{|x|^\gamma - \gamma t |x|^{-(1-2\gamma)}\leq |v|\}\subset \{(1-\frac{1}{2\sqrt{2}})|x|^\gamma\leq |v|\},\label{4_1:set_bound_1}\\
        \{|v|\leq (|x|+|v|t)^\gamma\}&\subset \{|v|\leq |x|^\gamma + 2\gamma t |x|^{-(1-2\gamma)}\}\subset \{|v|\leq (1+\frac{1}{\sqrt{2}})|x|^\gamma\},\label{4_1:set_bound_2}\\
        \{2||x|-|v|t|^\gamma\leq |v|-10\}&\subset \{(2|x|-20t)^\gamma - 2\gamma t (2|x|-20t)^{-(1-2\gamma)}\leq |v|-10\}\notag\\
        &\subset \{|x|^\gamma+5\leq |v|\},\label{4_1:set_bound_3}\\
        \{|v|-10\leq 2(|x|+|v|t)^\gamma\}&\subset \{|v|-10\leq (2|x|+20t)^\gamma + 4\gamma t (2|x|+20t)^{-(1-2\gamma)}\}\notag\\
        &\subset \{|v|\leq 3|x|^\gamma+15\}\label{4_1:set_bound_4}.
    \end{align}
    Except for the last inequalities in each line, all the inequalities are derived from Lemma \ref{lem:pert_ineq2} with
    \begin{align*}
        \{2||x|-|v|t|^\gamma\leq |v|-10\} &= \{|(2|x|-20t)-(|v|-10)(2t)|^\gamma\leq (|v|-10)\},\\
        \{|v|-10\leq 2(|x|+|v|t)^\gamma\} &= \{(|v|-10)\leq ((2|x|+20t)+(|v|-10)(2t))^\gamma\}.
    \end{align*}
    To check the remaining inequalities, we bound $\gamma t |x|^{-(1-2\gamma)}$, $\gamma t (2|x|-20t)^{-(1-2\gamma)}$, and $2|x|\pm 20t$ as follows. For $\frac{1}{3}\leq \gamma\leq \frac{1}{2}$, $|x|\geq 20 t^{\frac{\gamma}{1-\gamma}}\geq t^{\frac{\gamma}{1-\gamma}}$, and $t<\frac{1}{2}$,
    \begin{align}
        &\gamma t |x|^{-(1-2\gamma)}\leq \gamma t t^{-\frac{\gamma(1-2\gamma)}{1-\gamma}}\leq \frac{1}{2}t^{\frac{\gamma^2 + (1-\gamma)^2}{1-\gamma}}\leq \frac{1}{2}t^{1-\gamma}|x|^\gamma\leq \frac{1}{2\sqrt{2}}|x|^\gamma,\notag\\
        &\gamma t (2|x|-20t)^{-(1-2\gamma)}\leq \gamma t (40 t^{\frac{\gamma}{1-\gamma}}-20t^{\frac{\gamma}{1-\gamma}})^{-(1-2\gamma)}= \frac{1}{2} t (20 t^{\frac{\gamma}{1-\gamma}})^{-(1-2\gamma)}\leq \frac{1}{2},\label{4_1:gammat_bound}\\
        &2|x|-20t\geq 2|x|-20t^{\frac{\gamma}{1-\gamma}}\geq 2|x| - |x|\geq |x|,\notag\\
        &2|x|+20t\leq 2|x|+20t^{\frac{\gamma}{1-\gamma}}\leq 2|x| + |x|\geq 3|x|.\notag
    \end{align}
    The last inequalities in each line are deduced from the above estimates. Also, one can readily check that $|v|t\leq \frac{|x|}{2}$ from $|v|\leq (1+\frac{1}{\sqrt{2}})|x|^\gamma$ and $|x|\geq 20 t^{\frac{\gamma}{1-\gamma}}$.
    
    Now, we bound each $D_i$ under $|x|\geq 20 t^{\frac{\gamma}{1-\gamma}}$. By the mean value theorem with \eqref{4_1:set_bound_1} and \eqref{4_1:set_bound_2}, the difference term in $D_1$ is bounded by
    \begin{align}\label{4_1:u_0_MVT}
        \left|(1+|x-vt|)^{-\varpi}-(1+|x+vt|)^{-\varpi}\right|\leq \frac{\varpi}{(1+||x|-|v|t|)^{\varpi+1}}(2|v|t).
    \end{align}

    As $|v|t\leq \frac{|x|}{2}$, it is bounded by $\frac{Ct}{(1+|x|)^\varpi}$ for some constant $C$. Therefore,
    \begin{align*}
        D_1&=\int_{\mathbb{R}^N} |v| (1+|v|^2)^{-\delta}e^{-\alpha |v|^\beta}\left|(1+|x-vt|)^{-\varpi}-(1+|x+vt|)^{-\varpi}\right|\mathbf{1}_{||x|-|v|t|^\gamma \leq |v|\leq 2(|x|+|v|t)^\gamma + 10}\,dv\\
        &\leq \frac{Ct}{(1+|x|)^\varpi}\int_{\mathbb{R}^N} (1+|v|^2)^{\frac{1}{2}-\delta}e^{-\alpha |v|^\beta} \mathbf{1}_{||x|-|v|t|^\gamma\leq |v|}\,dv.
    \end{align*}
    From the analysis in \eqref{4_1:pre_rho0_bound}, it is bounded by
    \begin{align}\label{4_1:D_1_bound1}
        D_1\leq C \left(1+|x|\right)^{(N+1-2\delta -\beta)\gamma - \varpi} e^{-\alpha |x|^{\beta\gamma}} t.
    \end{align}
    Next, we consider $D_2$. Since $(1-\frac{1}{2\sqrt{2}})|x|^\gamma\leq |v|\leq (1+\frac{1}{\sqrt{2}})|x|^\gamma$ and $|v|t\leq \frac{|x|}{2}$ from \eqref{4_1:set_bound_1} and \eqref{4_1:set_bound_2}, using the sphrical coordinate, we have
    \begin{align}\label{4_1:D_2_bound1_1}
        \begin{split}
            D_2&=\int_{\mathbb{R}^N} |v| (1+|v|^2)^{-\delta}e^{-\alpha |v|^\beta}(1+|x+vt|)^{-\varpi}\mathbf{1}_{||x|-|v|t|^\gamma\leq |v|\leq (|x|+|v|t)^\gamma}\,dv\\
            &\leq C|\mathbb{S}^{N-1}|(1+|x|)^{(N-2\delta)\gamma-\varpi}\int_0^\infty e^{-\alpha r^\beta}\mathbf{1}_{|x|^\gamma - \gamma t|x|^{-(1-2\gamma)} \leq r\leq |x|^\gamma + 2\gamma t|x|^{-(1-2\gamma)})}\,dr\\
            &\leq C|\mathbb{S}^{N-1}|(1+|x|)^{(N+1-2\delta-\beta)\gamma-\varpi}\int_0^\infty e^{-\alpha r}\mathbf{1}_{(|x|^\gamma - \gamma t|x|^{-(1-2\gamma)})^\beta \leq r\leq (|x|^\gamma + 2\gamma t|x|^{-(1-2\gamma)}))^\beta}\,dr
        \end{split}
    \end{align}
    for some constant $C$. In the last step, we used a change of variable $r^\beta\mapsto r$. The last integral is bounded by
    \begin{align*}
        \left(\left(|x|^\gamma + 2\gamma t|x|^{-(1-2\gamma)}\right)^\beta - \left(|x|^\gamma - \gamma t|x|^{-(1-2\gamma)}\right)^\beta\right)e^{-\alpha (|x|^\gamma - \gamma t|x|^{-(1-2\gamma)})^\beta}.
    \end{align*}
    
    Now, we need to carefully bound the difference term:
    \begin{align*}
        \left(|x|^\gamma + 2\gamma t|x|^{-(1-2\gamma)}\right)^\beta - \left(|x|^\gamma - \gamma t|x|^{-(1-2\gamma)}\right)^\beta.
    \end{align*}
    Since $|x|\geq t^{\frac{\gamma}{1-\gamma}}$, we have $2\gamma t|x|^{-(1-\gamma)}\leq 2\gamma t^{1-\gamma}\leq 1$. On the other hand, from Lemma \ref{lem:power1}, we have
    \begin{align}\label{4_1:c_beta_bound}
        (1+c)^\beta\leq 1+3c,\quad
        \left(1-\frac{c}{2}\right)^\beta\geq 1-c
    \end{align}
    for $0<\beta\leq 2$ and $0\leq c\leq 1$. Therefore,
    \begin{align}
        \left(|x|^\gamma + 2\gamma t|x|^{-(1-2\gamma)}\right)^\beta - |x|^{\beta\gamma} &= |x|^{\beta\gamma}\left(1 + 2\gamma t|x|^{-(1-\gamma)}\right)^\beta - |x|^{\beta\gamma}\notag\\
        &\leq 6\gamma t |x|^{\beta\gamma - (1-\gamma)},\notag\\
        |x|^{\beta\gamma} - \left(|x|^\gamma -\gamma t|x|^{-(1-2\gamma)}\right)^\beta &= |x|^{\beta\gamma} - |x|^{\beta\gamma}\left(1 -\gamma t|x|^{-(1-\gamma)}\right)^\beta\notag\\
        &\leq 2\gamma t|x|^{\beta\gamma - (1-\gamma)}\label{4_1:D_2_bound1_2}
    \end{align}
    so that
    \begin{align}\label{first}
    	\left(|x|^\gamma + 2\gamma t|x|^{-(1-2\gamma)}\right)^\beta - \left(|x|^\gamma - \gamma t|x|^{-(1-2\gamma)}\right)^\beta
    	\leq 8\gamma t|x|^{\beta\gamma - (1-\gamma)}.
    \end{align}
    Also, by \eqref{4_1:D_2_bound1_2}, we also have
    \begin{align}\label{second}
        e^{-\alpha (|x|^\gamma - \gamma t|x|^{-(1-2\gamma)})^\beta}\leq e^{-\alpha |x|^{\beta\gamma} + \alpha (2\gamma t|x|^{\beta\gamma - (1-\gamma)})}.
    \end{align}
    Since $-1\leq \beta\gamma - 1 + \gamma\leq 0$, $\frac{\gamma}{1-\gamma}\leq 1$, and $|x|\geq t^{\frac{\gamma}{1-\gamma}}$,
    \begin{align*}
        2\gamma t|x|^{\beta\gamma+\gamma -1}\leq 2\gamma t^{\frac{\gamma}{1-\gamma}(\beta\gamma + \gamma - 1) + 1}\leq 2\gamma\leq 1,
    \end{align*}
    and the last term is bounded by a constant.

    We now insert \eqref{first} and \eqref{second} into the integral of \eqref{4_1:D_2_bound1_1} to conclude that
    \begin{align*}
         D_2\leq  C\left(1+|x|\right)^{(N+1-2\delta-\beta)\gamma-\varpi}e^{-\alpha |x|^{\beta\gamma}}\gamma t|x|^{\beta\gamma-(1-\gamma)}.
    \end{align*}
    Since $\beta\gamma-(1-\gamma)\leq 0$, and we are considering the region $|x|\geq t^{\frac{\gamma}{1-\gamma}}$, we have 
    \begin{align*}
        |x|^{\beta\gamma-(1-\gamma)}\leq t^{\frac{\beta\gamma^2}{1-\gamma}-\gamma}\leq t^{-1/2}.
    \end{align*}
    Therefore, we obtain
    \begin{align}\label{4_1:D_2_bound1}
        D_2&\leq C(1+|x|)^{(N+1-2\delta-\beta)\gamma}e^{-\alpha|x|^{\beta\gamma}}t^{1/2}.
    \end{align}

    For $D_3$, we follow a similar line of proof with \eqref{4_1:set_bound_3} and \eqref{4_1:set_bound_4}. As $|x|^\gamma + 5\leq |v|\leq 3|x|^\gamma + 15$, we have
    \begin{align}\label{4_1:D_3_bound1_1}
        \begin{split}
            D_3&=\int_{\mathbb{R}^N} |v| (1+|v|^2)^{-\delta}e^{-\alpha |v|^\beta}(1+|x+vt|)^{-\varpi}\mathbf{1}_{2||x|-|v|t|^\gamma\leq |v|-10\leq 2(|x|+|v|t)^\gamma}\,dv\\
            &\leq C(1+(10+|x|))^{(N-2\delta)\gamma} \int_0^\infty e^{-\alpha r^\beta}\mathbf{1}_{(2|x|-20t)^\gamma - 2\gamma t (2|x|-20t)^{-(1-2\gamma)}\leq r-10\leq (2|x|+20t)^\gamma + 4\gamma t (2|x|+20t)^{-(1-2\gamma)}}\,dr\\
            &\leq C(1+(10+|x|))^{(N+1-2\delta-\beta)\gamma}\\
            &\qquad \times \int_0^\infty e^{-\alpha r}\mathbf{1}_{((2|x|-20t)^\gamma - 2\gamma t (2|x|-20t)^{-(1-2\gamma)}+10)^\beta\leq r\leq ((2|x|+20t)^\gamma + 4\gamma t (2|x|+20t)^{-(1-2\gamma)} + 10)^\beta}\,dr.
        \end{split}
    \end{align}
    Now, we need to compute the last integral. To measure the size of the intergration interval, we estimate the following quantities. From \eqref{4_1:gammat_bound}, we know that $4\gamma t (2|x|\pm 20t)^{-(1-2\gamma)}\leq 2$. Using \eqref{4_1:c_beta_bound}, 
    \begin{align}\label{4_1:2|x|_lower_bound1}
        \begin{split}
            ((2|x|+20t)^\gamma + 4\gamma t (2|x|+20t)^{-(1-2\gamma)} + 10)^\beta&=((2|x|+20t)^\gamma + 10)^\beta\left(1+\frac{4\gamma t (2|x|+20t)^{-(1-2\gamma)}}{(2|x|+20t)^\gamma + 10}\right)^\beta\\
            &\leq ((2|x|+20t)^\gamma + 10)^\beta + \frac{12\gamma t (2|x|+20t)^{-(1-2\gamma)}}{((2|x|+20t)^\gamma + 10)^{1-\beta}},\\
            ((2|x|-20t)^\gamma - 2\gamma t (2|x|-20t)^{-(1-2\gamma)} + 10)^\beta&=((2|x|-20t)^\gamma + 10)^\beta\left(1-\frac{2\gamma t (2|x|-20t)^{-(1-2\gamma)}}{(2|x|-20t)^\gamma + 10}\right)^\beta\\
            &\geq ((2|x|-20t)^\gamma + 10)^\beta - \frac{4\gamma t (2|x|-20t)^{-(1-2\gamma)}}{((2|x|-20t)^\gamma + 10)^{1-\beta}}.
        \end{split}
    \end{align}
    We also need to compute the difference between $(2|x|\pm 20 t)^\gamma$. For $|x|\geq 20 t^{\frac{\gamma}{1-\gamma}}$, it satisfies $10\frac{t}{|x|}\leq \frac{1}{2}$ and $2^\gamma 10t|x|^{\gamma-1}\leq \sqrt{10}$. Therefore, applying \eqref{4_1:c_beta_bound} and $(1+c)^\gamma\leq 1+c$ (resp., $(1-c)^\gamma\geq 1-c$) for $\frac{1}{3}\leq \gamma\leq \frac{1}{2}$ and $0\leq c\leq 1$, we get
    \begin{align}\label{4_1:2|x|_lower_bound2}
        \begin{split}
            ((2|x|+20t)^\gamma + 10)^\beta &= \left(2^\gamma |x|^\gamma\left(1+10\frac{t}{|x|}\right)^\gamma + 10\right)^\beta\\
            &\leq \left(2^\gamma \left(|x|^\gamma + 10t|x|^{\gamma-1}\right) + 10\right)^\beta\\
            &=\left(2^\gamma |x|^\gamma + 10\right)^\beta\left(1+\frac{2^\gamma 10t|x|^{\gamma-1}}{2^\gamma |x|^\gamma + 10}\right)^\beta\\
            &\leq \left(2^\gamma |x|^\gamma + 10\right)^\beta + 60t|x|^{\gamma-1}(2^\gamma |x|^\gamma + 10)^{\beta-1},
        \end{split}
    \end{align}
    and by a similar computation,
    \begin{align}\label{4_1:2|x|_lower_bound3}
        \begin{split}
            ((2|x|-20t)^\gamma + 10)^\beta &\geq \left(2^\gamma |x|^\gamma + 10\right)^\beta\left(1-\frac{2^\gamma 10t|x|^{\gamma-1}}{2^\gamma |x|^\gamma + 10}\right)^\beta\\
            &\geq \left(2^\gamma |x|^\gamma + 10\right)^\beta - 40t|x|^{\gamma-1}(2^\gamma |x|^\gamma + 10)^{\beta-1}.
        \end{split}
    \end{align}

    Thus, from \eqref{4_1:2|x|_lower_bound1}, \eqref{4_1:2|x|_lower_bound2}, and \eqref{4_1:2|x|_lower_bound3}, the total difference is given by
    \begin{align*}
        &((2|x|+20t)^\gamma + 4\gamma t (2|x|+20t)^{-(1-2\gamma)} + 10)^\beta - ((2|x|-20t)^\gamma - 2\gamma t (2|x|-20t)^{-(1-2\gamma)} + 10)^\beta\\
        &\leq C\left(\frac{\gamma t (2|x|+20t)^{-(1-2\gamma)}}{(2|x|-20t)^\gamma + 10)^{1-\beta}} + t|x|^{\gamma-1}(2^\gamma |x|^\gamma + 10)^{\beta-1}\right).
    \end{align*}
    Now, we bound the final term. When $20 t^{\frac{\gamma}{1-\gamma}}\leq |x|\leq 40$, then $2^\gamma |x|^\gamma + 10$ and $(2|x|- 20t)^\gamma + 10$ are uniformly lower and upper bounded by a constant, so regardless of $\beta$, we have
    \begin{align}\label{4_1:2|x|_diff_bound1}
        \begin{split}
            \frac{\gamma t (2|x|+20t)^{-(1-2\gamma)}}{(2|x|-20t)^\gamma + 10)^{1-\beta}} + t|x|^{\gamma-1}(2^\gamma |x|^\gamma + 10)^{\beta-1}&\leq C\left(\gamma t (60t^{\frac{\gamma}{1-\gamma}})^{-(1-2\gamma)} + t(20 t^{\frac{\gamma}{1-\gamma}})^{\gamma-1}\right)\\
            &\leq C\left(\gamma t^{1-\gamma + \frac{\gamma^2}{1-\gamma}} + t^{1-\gamma}\right)\\
            &\leq C t^{1/2}.
        \end{split}
    \end{align}
    If $|x|>40$, then
    \begin{align*}
        |x|^\gamma\leq (2|x|-20t)^\gamma + 10\leq 10|x|^\gamma,\quad  2|x|^\gamma\leq (2|x|)^\gamma + 10\leq 10|x|^\gamma.
    \end{align*}
    Therefore,
    \begin{align}\label{4_1:2|x|_diff_bound2}
        \begin{split}
            \frac{\gamma t (2|x|+20t)^{-(1-2\gamma)}}{(2|x|-20t)^\gamma + 10)^{1-\beta}} + t|x|^{\gamma-1}(2^\gamma |x|^\gamma + 10)^{\beta-1}&\leq C\left(\gamma t (3|x|)^{-(1-2\gamma)}|x|^{\gamma(\beta-1)} + t|x|^{\gamma-1}|x|^{\gamma(\beta-1)}\right)\\
            &\leq C t\left(\gamma|x|^{\gamma \beta + \gamma - 1} + |x|^{\gamma\beta - 1}\right).
        \end{split}
    \end{align}
    Since $\gamma\beta + \gamma - 1\leq 0$, it is bounded by $Ct$. Combining \eqref{4_1:2|x|_diff_bound1} and \eqref{4_1:2|x|_diff_bound2}, we obtain
    \begin{align}\label{4_1:2|x|_diff_bound}
        \begin{split}
            &((2|x|+20t)^\gamma + 4\gamma t (2|x|+20t)^{-(1-2\gamma)} + 10)^\beta - ((2|x|-20t)^\gamma - 2\gamma t (2|x|-20t)^{-(1-2\gamma)} + 10)^\beta\\
            &\leq Ct^{1/2}.
        \end{split}
    \end{align}
    As a corollary of \eqref{4_1:2|x|_diff_bound} applied to \eqref{4_1:2|x|_lower_bound1} and \eqref{4_1:2|x|_lower_bound3}, we have
    \begin{align}\label{4_1:2|x|_exp_bound}
        e^{-\alpha ((2|x|-20t)^\gamma - 2\gamma t (2|x|-20t)^{-(1-2\gamma)}+10)^\beta}\leq C e^{(2^\gamma |x|^\gamma + 10)^\beta}
    \end{align}
    for some constant $C$. Multiplying \eqref{4_1:2|x|_diff_bound} and \eqref{4_1:2|x|_exp_bound} for \eqref{4_1:D_3_bound1_1}, we have
    \begin{align}\label{4_1:D_3_bound1_2}
        D_3\leq C(1+(10+|x|))^{(N+1-2\delta-\beta)\gamma}e^{(2^\gamma |x|^\gamma + 10)^\beta} t^{1/2}.
    \end{align}
    Even though we ignored $(1+|x|)^{-\varpi}$ decay in $D_3$, the exponential decay is much greater than $D_1$ and $D_2$. Therefore, when $|x|$ is large, $D_3$ is much smaller than $D_1$ and $D_2$.
    
    Combining \eqref{4_1:D_1_bound1}, \eqref{4_1:D_2_bound1}, and \eqref{4_1:D_3_bound1_2}, we have
    \begin{align}\label{4_1:rho0u0_bound1}
        \begin{split}
            2|\tilde{\rho}_0\tilde{u}_0|(t,x)&\leq C\left(1+|x|\right)^{(N+1-2\delta-\beta)\gamma-\varpi}e^{-\alpha |x|^{\beta\gamma}}t^{\frac{1}{2}}
        \end{split}
    \end{align}
    when $|x|\geq 20t^{\frac{\gamma}{1-\gamma}}$.\\

    \noindent \textit{(ii) The case $|x|\leq 20 t^{\frac{\gamma}{1-\gamma}}$.}

    We now move on to the region $|x|\leq 20t^{\frac{\gamma}{1-\gamma}}$ for the $2|\tilde{\rho}_0\tilde{u}_0|$ estimate. We first consider $D_1$. The condition $|v|\leq 2(|x|+|v|t)^\gamma + 10$ implies (say) $|v|\leq 20$. Using \eqref{4_1:u_0_MVT} again, we easily check
    \begin{align}\label{4_1:D_1_bound2}
        \begin{split}
            D_1&\leq\int_{\mathbb{R}^N} |v| (1+|v|^2)^{-\delta}e^{-\alpha |v|^\beta}\left|(1+|x-vt|)^{-\varpi}-(1+|x+vt|)^{-\varpi}\right|\mathbf{1}_{|v|\leq 20}\,dv\\
            &\leq Ct.
        \end{split}
    \end{align}
    
    To bound $D_2$, we observe that
    \begin{align*}
        &\left\{v:|v|\leq (|x|+|v|t)^\gamma\right\}\subset \left\{v: |v|\leq 20 t^{\frac{\gamma^2}{1-\gamma}}\right\}.
    \end{align*}
    Indeed, if $|v|> 20t^{\frac{\gamma^2}{1-\gamma}}$, then $(|v|t)^\gamma\leq \frac{1}{20^{1-\gamma}}|v|$ as
    \begin{align*}
         \frac{1}{20^{1-\gamma}}|v|^{1-\gamma}\geq t^{\gamma^2}\geq t^\gamma.
    \end{align*}
    So,
    \begin{align*}
        |x + vt|^\gamma \leq |x|^\gamma + (|v|t)^\gamma\leq \left(20 t^{\frac{\gamma}{1-\gamma}}\right)^\gamma + (|v|t)^\gamma\leq \frac{20^\gamma}{20}|v| + (|v|t)^\gamma\leq \left(\frac{20^\gamma}{20}+\frac{1}{20^{1-\gamma}}\right)|v|< |v|.
    \end{align*}
    Therefore, the integral is bounded by
    \begin{align}\label{4_1:D_2_bound2}
        \begin{split}
            D_2&\leq \int_{\mathbb{R}^N} |v| (1+|v|^2)^{-\delta}e^{-\alpha |v|^\beta}(1+|x+vt|)^{-\varpi}\mathbf{1}_{|v|\leq (|x|+|v|t)^\gamma}\,dv\\
            &\leq C|\mathbb{S}^{N-1}|\int_0^{20 t^{\frac{\gamma^2}{1-\gamma}}} r^N \,dr\leq C t^{\frac{2\gamma^2}{1-\gamma}}\leq C t^{\frac{1}{3}}
        \end{split}
    \end{align}
    since $N\geq 1$ and $\frac{1}{3}\leq \gamma\leq \frac{1}{2}$.
    
    By a similar computation,
    \begin{align*}
        &\left\{v:2||x|-|v|t|^\gamma\leq |v|-10\leq 2(|x|+|v|t)^\gamma\right\}\\
        &\subset \left\{v:0\leq |v|-10\leq (2(|x|+10t)+(|v|-10)(2t))^\gamma\right\}\\
        &\subset \left\{v: 0\leq |v|-10\leq 40 t^{\frac{\gamma^2}{1-\gamma}}\right\}.
    \end{align*}
    Using this bounds, we have
    \begin{align}\label{4_1:D_3_bound2}
        \begin{split}
            D_3&\leq \int_{\mathbb{R}^N} |v| (1+|v|^2)^{-\delta}e^{-\alpha |v|^\beta}(1+|x+vt|)^{-\varpi}\mathbf{1}_{0\leq |v|-10\leq 2(|x|+|v|t)^\gamma}\,dv\\
            &\leq C|\mathbb{S}^{N-1}|\int_{10}^{10 + 40 t^{\frac{\gamma^2}{1-\gamma}}} r^N \,dr\leq C t^{\frac{\gamma^2}{1-\gamma}}\leq C t^{\frac{1}{6}}
        \end{split}
    \end{align}
    as $N\geq 1$ and $\frac{1}{3}\leq \gamma\leq \frac{1}{2}$.\\

    Combining \eqref{4_1:D_1_bound2}, \eqref{4_1:D_2_bound2}, and \eqref{4_1:D_3_bound2}, we obtain
    \begin{align}\label{4_1:rho0u0_bound2}
        \begin{split}
            2|\tilde{\rho}_0\tilde{u}_0|(t,x)&\leq C\left(1+|x|\right)^{(N+1-2\delta-\beta)\gamma-\varpi}e^{-\alpha |x|^{\beta\gamma}}t^{\frac{1}{6}}
        \end{split}
    \end{align}
    for $|x|\leq 20 t^{\frac{\gamma}{1-\gamma}}$ and $t<\frac{1}{2}$. Merging \eqref{4_1:rho0u0_bound1} and \eqref{4_1:rho0u0_bound2}, we finally conclude that \eqref{4_1:rho0u0_bound2} holds for all $x\in\mathbb{R}^N$.

    Dividing \eqref{4_1:rho0u0_bound2} by \eqref{4_1:rhobound}, we finally get
    \begin{align}\label{4_1:u0_bound}
        |\tilde{u}_0|\leq \frac{C_3}{2}(1+|x|^\gamma) t^{1/6}
    \end{align}
    for some constant $C_3$.

    Recall that we define $\tilde{E}_0$ for the weight $|v|^2$ in \eqref{4_1:macro0_def}. From \eqref{4_1:E0_bound} and \eqref{4_1:u0_bound}, the temperature is bounded by
    \begin{align*}
        2C_4 - \frac{C^2_3}{2N}t_0^{1/3}\leq \frac{1}{1+|x|^{2\gamma}}\frac{\tilde{E}_0 - \tilde{\rho}_0|\tilde{u}_0|^2}{N \tilde{\rho}_0} \leq \frac{C_5}{2}.
    \end{align*}
    We can choose small enough $t_0$ and redefining $2C_4$ by $2C_4 - \frac{C^2_3}{2N}t_0^{1/3}$, so the temperature satisfies \eqref{4_1:spacedef}.

    Finally, we can easily check the first condition of \eqref{4_1:spacedef}: from \eqref{4_1:rho0_bound} and \eqref{4_1:E0_bound},
    \begin{align*}
        \int_{\mathbb{R}^{N}} (1+|x|^{2\gamma}+v^2)f_0(x-vt,v)\,dv&\leq (1+|x|^{2\gamma})\tilde{\rho}_0(t,x) + \tilde{E}_0(t,x)\\
        &\leq C\left(1+|x|\right)^{(N-2\delta - \beta+2)\gamma-\varpi}\exp\left(-\alpha|x|^{\beta\gamma}\right).
    \end{align*}
    Since it has exponential decay about $x$, it is uniformly bounded about $x$ and $0\leq t\leq t_0$.\\

    \noindent \textit{(ii)} Now, we show that $\mathcal{T}$ maps $X$ to $X$. In the previous step, we showed that $f_0(x-vt, v)\in X$, so $X$ is not an empty set.
    
    To check that $\mathcal{T}f(t,x,v)$ satisfies the first condition of \eqref{4_1:spacedef}, we compute 
    \begin{align*}
        &\sup_{t\leq t_0}\int_{\mathbb{R}^{N}} (1+|x|^{2\gamma}+v^2)\mathcal{T}f(t,x,v)\,dv\\
        &\leq \sup_{t\leq t_0}e^{-t}\int_{\mathbb{R}^{N}} (1+|x|^{2\gamma}+v^2)f_0(x-vt,v)\,dv \\
        &\quad + \sup_{t\leq t_0}\int_0^t e^{-(t-\tau)}\int_{\mathbb{R}^{N}} (1+|x|^{2\gamma}+v^2)\mathcal{M}(f)(\tau,x-v(t-\tau),v)\,dv d\tau\\
        &\leq \sup_{t\leq t_0}\int_{\mathbb{R}^{N}} (1+|x|^{2\gamma}+v^2)f_0(x-vt,v)\,dv \\
        &\quad + (1-e^{-t_0})\sup_{\tau, t\leq t_0}\int_{\mathbb{R}^{N}} (1+|x|^{2\gamma}+v^2)\mathcal{M}(f)(\tau,x-v(t-\tau),v)\,dv.
    \end{align*}
    The first term is bounded since $f_0(x-vt, v)\in X$. The second term is bounded by applying \eqref{4_1:spacedef} to Lemma \ref{lem:int_bound1}:
    \begin{equation}\label{4_1:Mfbound}
        \begin{split}
            &\int_{\mathbb{R}^N} (1+|x|^{2\gamma}+ v^2)\mathcal{M}(f)(\tau, x-v(t-\tau), v)\,dv\\
            &\leq \int_{\mathbb{R}^N}dv\,(1+\left(|x-vt|+|vt|\right)^{2\gamma}+v^2)\frac{C_2(1+|x-v(t-\tau)|)^{(N-2\delta-\beta)\gamma-\varpi} e^{-\alpha |x-v(t-\tau)|^{\beta\gamma}}}{C_3\left(1+|x-v(t-\tau)|^{2\gamma}\right)^{N/2}}\\
            &\qquad\times \exp\left(-\frac{(v-u)^2}{2C_4(1+|x-v(t-\tau)|^{2\gamma})}\right)\\
            &\leq \int_{\mathbb{R}^N}dv\,(1+|x-vt|^{2\gamma}+|vt|^{2\gamma}+v^2)\frac{C_2(1+|x-v(t-\tau)|)^{(N-2\delta-\beta)\gamma-\varpi} e^{-\alpha |x-v(t-\tau)|^{\beta\gamma}}}{C_3\left(1+|x-v(t-\tau)|^{2\gamma}\right)^{N/2}}\\
            &\qquad\times \exp\left(-\frac{(v-u)^2}{2C_4(1+|x-v(t-\tau)|^{2\gamma})}\right)\\
            &\leq C(1+|x|)^{(N+2-2\delta-\beta)\gamma - \varpi}e^{-\alpha |x|^{\beta\gamma}}
        \end{split}
    \end{equation}
    for some constant $C(t_0)$ by choosing small enough $t_0\leq 1$. In the middle, we used Lemma \ref{lem:power0} to get
    \begin{align*}
        \left(|x-vt|+|vt|\right)^{2\gamma}\leq |x-vt|^{2\gamma}+|vt|^{2\gamma}.
    \end{align*}

    To check the bounds of the macroscopic fields of $\mathcal{T}f$, we will show that the macroscopic fields of $\mathcal{T}f$ do not change rapidly for small enough time $t$ from the ones of initial data $f_0$. We write the macroscopic fields of $\mathcal{T}f$ by
    \begin{align*}
        (\rho_f, \rho_f u_f, N\rho_f T_f)(t,x) &= \int_{\mathbb{R}^N} (1, v, |v-u_f(t,x)|^2) f(t, x, v)\,dv,\\
        (\rho, \rho u, N\rho T)(t,x) &= \int_{\mathbb{R}^N} (1, v, |v-u(t,x)|^2) \mathcal{T}f(t, x, v)\,dv.
    \end{align*}

    \noindent We deal with the density $\rho(t,x)$.  From the definition of $\mathcal{T}$, 
    \begin{align*}
        \rho(t,x) = e^{-t}\tilde{\rho}_0(t,x) + \int_0^t e^{-(t-\tau)}\int_{\mathbb{R}^N} \mathcal{M}(f)(\tau, x-v(t-\tau), v)\,dv d\tau.
    \end{align*}
    
    By the computation in \eqref{4_1:Mfbound}, we have
    \begin{align*}
        \int_0^t e^{-(t-\tau)}\int_{\mathbb{R}^N} \mathcal{M}(f)(\tau, x-v(t-\tau), v)\,dv d\tau\leq C(1-e^{-t})(1+|x|)^{(N-2\delta-\beta)\gamma - \varpi}e^{-\alpha |x|^{\beta\gamma}}.
    \end{align*}
    It shows that
    \begin{align}\label{4_1:rhobound}
        e^{-t}\leq \frac{\rho(t,x)}{\tilde{\rho}_0(t,x)} \leq e^{-t} + C(1-e^{-t})\frac{\left(1+|x|\right)^{(N-2\delta-\beta)\gamma - \varpi}\exp\left(-\alpha|x|^{\beta\gamma}\right)}{\tilde{\rho}_0(t,x)}\leq e^{-t} + C(2C_1)^{-1}(1-e^{-t}),
    \end{align}
    where $C_1$ and $C_2$ are constants in \eqref{4_1:rho0_bound}.
    
    We restrict $e^{-t_0}\geq \frac{1}{2}$ and $C(2C_1)^{-1}(1-e^{-t_0})\leq \frac{3}{2}$ by choosing small enough $t_0$. From \eqref{4_1:rho0_bound}, we have
    \begin{align}\label{4_1:rhobound1}
        C_1\leq \frac{\rho(t,x)}{\left(1+|x|\right)^{(N-2\delta-\beta)\gamma - \varpi}\exp\left(-\alpha|x|^{\beta\gamma}\right)}\leq C_2
    \end{align}
    for all $0\leq t\leq t_0$.

    We consider $|u|(t,x)$. To compute the second $v\mathcal{M}(f)$ integral term, let us use Lemma \ref{lem:int_bound1} to write
    \begin{align*}
        &\int_{\mathbb{R}^N} \int_0^t e^{-(t-\tau)}|v|\mathcal{M}(f)(\tau, x-v(t-\tau), v)\,d\tau dv\\
        &\leq C\int_0^t d\tau\, e^{-(t-\tau)}\int_{\mathbb{R}^N} dv\, |v|\frac{C_2(1+|x-v(t-\tau)|)^{(N-2\delta-\beta)\gamma-\varpi} e^{-\alpha |x-v(t-\tau)|^{\beta\gamma}}}{C_3\left(1+|x-v(t-\tau)|^{2\gamma}\right)^{N/2}}\\
        &\qquad \times \exp\left(-\frac{\left(v-u_f\right)^2}{2C_4\left(1+|x-v(t-\tau)|^{2\gamma}\right)}\right)\\
        &\leq C'(1-e^{t})\left(1+|x|\right)^{(N+1-2\delta-\beta)\gamma-\varpi}e^{-\alpha |x|^{\beta\gamma}}
    \end{align*}
    for some $C'>1$ as $|u_f|(t,x)\leq C_3\left(1 + |x|^{\gamma}\right)t^{1/6}$.

    We choose $t_0$ satisfying
    \begin{align*}
        C'C_1^{-1}(1-e^{-t})(1+|x|)^\gamma\leq \frac{C_3}{2}(1+|x|^\gamma)t^{1/6}
    \end{align*}
    for all $0<t\leq t_0$ and for all $x$. From \eqref{4_1:rhobound}, \eqref{4_1:rhobound1}, and \eqref{4_1:u0_bound}
    \begin{equation}\label{4_1:ubound}
        \begin{split}            
            |u(t,x)| = \frac{\left|\int_{\mathbb{R}^N} v \mathcal{T}f\,dv\right|}{\rho(t,x)} &\leq e^{-t}\frac{\tilde{\rho}_0}{\rho(t,x)}|\tilde{u}_0| + C'(1-e^{-t})(1+|x|)^\gamma\frac{\left(1+|x|\right)^{(N-2\delta-\beta)\gamma-\varpi}e^{-\alpha |x|^{\beta\gamma}}}{\rho(t,x)}\\
            &\leq \frac{C_3}{2}(1+|x|)^\gamma t^{1/6} + C'C_1^{-1}(1-e^{-t})(1+|x|)^\gamma\\
            &\leq C_3(1+|x|^\gamma)t^{1/6}
        \end{split}
    \end{equation}
    for $0\leq t\leq t_0$.

    To bound temperature $T(t,x)$, we start from the computation in \eqref{4_1:Mfbound}:
    \begin{align*}
        &\int_0^t e^{-(t-\tau)}\int_{\mathbb{R}^N} |v|^2\mathcal{M}(f)(\tau, x-v(t-\tau), v)\,d\tau\\
        &\leq C''(1-e^{t})(1+|x|)^{(N+2-2\delta-\beta)\gamma - \varpi}e^{-\alpha |x|^{\beta\gamma}}.
    \end{align*}
    Repeating similar analysis as in \eqref{4_1:rhobound} and \eqref{4_1:rhobound1}, and considering \eqref{4_1:E0_bound}, we can choose small enough $t_0$ satisfying
    \begin{align*}
        \frac{3}{2}C_4(1+|x|^{2\gamma})\leq \frac{\int_{\mathbb{R}^N}|v|^2 \mathcal{T}f\,dv}{N \rho(t,x)}\leq \frac{2}{3}C_5(1+|x|^{2\gamma}).
    \end{align*}
    Finally, since $|u|^2(t,x)\leq C^2_3 (1+|x|^{\gamma})^2t^{1/3}$, we have for some constant $C'''$, 
    \begin{align*}
        \frac{\rho(t,x)|u|^2}{\int_{\mathbb{R}^N}|v|^2 \mathcal{T}f\,dv}\leq \frac{2}{3}\frac{C_3^2(1+|x|^{\gamma})^2t^{1/3}}{N C_4(1+|x|^{2\gamma})} \leq \frac{C'''}{N C_4}t^{1/3}.
    \end{align*}
    Here, we used the remark of Lemma \ref{lem:power0}. By choosing $t_0$ small enough again, we can force $\rho(t,x)|u|^2\leq \frac{1}{3}\int_{\mathbb{R}^N}|v|^2 \mathcal{T}f\,dv$. Thus, we obtain
    \begin{align}\label{4_1:Tbound}
        C_4\leq \frac{T(t,x)}{(1+|x|^{2\gamma})}\leq C_5.
    \end{align}
    
    Combining \eqref{4_1:rhobound1}, \eqref{4_1:ubound}, and \eqref{4_1:Tbound}, we get the well-definedness of $\mathcal{T}$.\\
    
    \noindent \textit{(iii)} In this final step, we show that the solution map $\mathcal{T}$ is a contraction map. For this, we will show that
    \begin{equation}\label{4_1:differenceintegral}
        \begin{split}
            &\sup_{t\leq t_0}\left\|\int_{\mathbb{R}^N} (1+|x|^{2\gamma} + v^2)|\mathcal{T}f_2-\mathcal{T}f_1|(t,x,v)\,dv\right\|_{L^\infty_x}\\
            &\leq \sup_{t\leq t_0}\left\|\int_{\mathbb{R}^N}\int_0^t  e^{-(t-\tau)}(1+|x|^{2\gamma} + v^2)|\mathcal{M}(f_2)-\mathcal{M}(f_1)|(\tau,x-v(t-\tau),v)\,d\tau dv\right\|_{L^\infty_x}\\
            &\leq C\sup_{t\leq t_0}\left\|\int_{\mathbb{R}^N}(1+|x|^{2\gamma} + v^2)|f_2-f_1|(t,x,v)dv\right\|_{L^\infty_x}
        \end{split}
    \end{equation}
    for some constant $C<1$ by choosing small enough $t_0$.
    
    For $f_1,f_2\in X$, we define $f_\lambda = \lambda f_1+(1-\lambda)f_2$ for $0\leq \lambda \leq 1$. Then, we have by the fundamental theorem of calculus,
    \begin{align}\label{4_1:difference}
        \left(\mathcal{M}(f_2) - \mathcal{M}(f_1)\right)(\tau, x-v(t-\tau), v)d\tau &= \int_0^1 \dv{\lambda} \mathcal{M}(f_\lambda)(\tau, x-v(t-\tau), v)\,d\lambda,
    \end{align}
    where
    \begin{equation*}
    	\begin{split}
    		\mathcal{M}(f_\lambda) &= \frac{\rho_\lambda}{(2\pi T_\lambda)^{N/2}} \exp\left(-\frac{|v-u_\lambda|^2}{2T_\lambda}\right)
    	\end{split}
    \end{equation*}
    with
    \begin{equation*}
        \begin{split}
            \rho_\lambda &= \int_{\mathbb{R}^N} \lambda f_1 + (1-\lambda)f_2\,dv = \lambda \rho_1 + (1-\lambda)\rho_2,\\
            \rho_\lambda u_\lambda &= \int_{\mathbb{R}^N} v\left(\lambda f_1 + (1-\lambda)f_2\right)\,dv = \lambda \rho_1u_1 + (1-\lambda)\rho_2u_2,\\
            \rho_\lambda (u^2_\lambda + NT_\lambda) &= \int_{\mathbb{R}^N} v^2\left(\lambda f_1 + (1-\lambda)f_2\right)\,dv = \lambda \rho_1(u^2_1+NT_1) + (1-\lambda)\rho_2(u^2_2+NT_2).
        \end{split}
    \end{equation*}
    We claim that $\rho_\lambda$, $u_\lambda$, and $T_\lambda$ almost satisfy the relation \eqref{4_1:spacedef} by choosing small enough $t_0$:
    \begin{equation}\label{spacedeflambda}
        \begin{split}
            &C_1 \leq \frac{\rho_\lambda(t,x)}{\left(1+|x|\right)^{(N-2\delta - \beta)\gamma}\exp\left(-\alpha|x|^{\beta\gamma}\right)}\leq C_2,\\
            &\hspace{1cm}|u_\lambda(t,x)|\leq C_3(1+|x|^{\gamma})t^{1/6},\\
            &\hspace{1.5cm}\frac{C_4}{2} \leq \frac{T_\lambda(t,x)}{1+|x|^{2\gamma}}\leq 2C_5
        \end{split}
    \end{equation}
    for all $0\leq \lambda\leq 1$. Indeed, $\rho_\lambda$ is a linear combination of $\rho_1$ and $\rho_2$, by \eqref{4_1:spacedef}, it satisfies
    \begin{align*}
        C_1 \leq \frac{\rho_\lambda(t,x)}{\left(1+|x|\right)^{(N-2\delta - \beta)\gamma-\varpi}\exp\left(-\alpha|x|^{\beta\gamma}\right)}\leq C_2.
    \end{align*}
    The bound for $|u_{\lambda}|$ is obtained by
    \begin{align*}
        |u_\lambda|(t,x)\leq \frac{\lambda \rho_1 + (1-\lambda)\rho_2}{\rho_\lambda}\max\{|u_1|,|u_2|\}\leq C_3(1+|x|^\gamma)t^{1/6}.
    \end{align*}
    For $T_\lambda$, we note that
    \begin{align*}
        \frac{1}{N}\left(\min\{u_1^2,u_2^2\}-u^2_\lambda\right) + \min\{T_1, T_2\}\leq T_\lambda\leq \frac{1}{N}\left(\max\{u_1^2,u_2^2\}-u^2_\lambda\right) + \max\{T_1, T_2\}
    \end{align*}
    with
    \begin{align*}
        \max\{u_1^2,u_2^2\}-u^2_\lambda&\leq C_3^2(1+|x|^\gamma)^2t^{1/3},\\
        \min\{u_1^2,u_2^2\}-u^2_\lambda&\geq -C_3^2(1+|x|^\gamma)^2t^{1/3}.
    \end{align*}
    Choosing small enough $t_0$ so that $\frac{C_3^2}{N}t^{1/3}(1+|x|^\gamma)^2\leq \frac{C_4}{2}(1+|x|^{2\gamma})$ for $0\leq t\leq t_0$, we can bound $T_\lambda$ by
    \begin{align*}
        \frac{C_4}{2} \leq \frac{T_\lambda(t,x)}{1+|x|^{2\gamma}}\leq 2C_5.
    \end{align*}
    It proves the claim.
    
    To bound the integral \eqref{4_1:differenceintegral}, let us calculate $\dv{\lambda}\mathcal{M}(f_\lambda)$:
    \begin{equation}\label{4_1:derivlambda}
        \begin{split}
            &\dv{\lambda} \mathcal{M}(f_\lambda) \\
            &= \frac{1}{(2\pi T_\lambda)^{N/2}}(\rho_1-\rho_2)\exp\left(-\frac{|v-u_\lambda|^2}{2T_\lambda}\right)  + \frac{(v-u_\lambda)\cdot (u_1-u_2)}{(2\pi T_\lambda)^{N/2}T_\lambda}\frac{\rho_1\rho_2}{\rho_\lambda}\exp\left(-\frac{|v-u_\lambda|^2}{2T_\lambda}\right)\\
            &\quad -\frac{1}{(2\pi T_\lambda)^{N/2}}\left(\frac{N}{2 T_\lambda} - \frac{|v-u_\lambda|^2}{2T^2_\lambda} \right)\left(\rho_1\rho_2(u_1-u_2)^2\frac{-\lambda\rho_1 + (1-\lambda)\rho_2}{\rho_\lambda^2} + \frac{\rho_1\rho_2(T_1-T_2)}{\rho_\lambda}\right)\\
            &\qquad\quad \times\exp\left(-\frac{|v-u_\lambda|^2}{2T_\lambda}\right)\\
            &\eqqcolon D_1+D_2+D_3.
        \end{split}
    \end{equation}
    
    The difference between the macroscopic field of $f_1$ and $f_2$ is given by
    \begin{equation}\label{4_1:macrodiff}
        \begin{split}
            \rho_1-\rho_2 &= \int_{\mathbb{R}^N} (f_1-f_2)\,dv,\\
            u_1-u_2 &= \rho_1^{-1}\left(\int_{\mathbb{R}^N} v(f_1-f_2)\,dv - u_2\int_{\mathbb{R}^N} (f_1-f_2)\,dv\right),\\
            T_1-T_2
            &=\rho_1^{-1}\left(\int_{\mathbb{R}^N} v^2(f_1-f_2)\,dv - T_2\int_{\mathbb{R}^N} (f_1-f_2)\,dv + u_1\cdot u_2\int_{\mathbb{R}^N} (f_1-f_2)\,dv\right) \\
            &\qquad- \rho_1^{-1}(u_1+u_2)\cdot\left(\int_{\mathbb{R}^N} v(f_1-f_2)\,dv\right).
        \end{split}
    \end{equation}
    From now on, the constant $C$ denotes a general constant depending on the context. The constant will only depend on $\gamma$, $C_i$ in \eqref{spacedeflambda}, and the constants in Lemma \ref{lem:int_bound1} and \ref{lem:int_bound2}. Also, we restrict $t_0\leq 1$.

    Combining \eqref{4_1:differenceintegral}, \eqref{4_1:difference}, and \eqref{4_1:derivlambda}, let us bound the integral term by term. For the first term of \eqref{4_1:derivlambda}, it is bounded as follows:
    \begin{align*}
        &\int_0^1d\lambda \int_{\mathbb{R}^N}dv\,(1+|x|^{2\gamma}+v^2)|D_1|\\
        & \leq \sup_{\lambda}\int_{\mathbb{R}^N}dv\,(1+|x|^{2\gamma}+v^2)\frac{|\rho_1-\rho_2|}{(2\pi T_\lambda)^{N/2}}\exp\left(-\frac{|v-u_\lambda|^2}{2 T_\lambda}\right)\\
        & =\sup_{\lambda}\int_{\mathbb{R}^N}dv\, \frac{1+|x|^{2\gamma}+v^2}{(2\pi T_\lambda)^{N/2}}(\tau, x-v(t-\tau))\exp\left(-\frac{|v-u_\lambda|^2}{2 T_\lambda}\right)\int_{\mathbb{R}^N} dw\, |f_1-f_2|(\tau, x-v(t-\tau), w)\\
        & \leq \sup_{\lambda}\int_{\mathbb{R}^N}dv\, \frac{1+|x-v(t-\tau)|^{2\gamma}+\left(|v|(t-\tau)\right)^{2\gamma}+v^2}{(2\pi T_\lambda)^{N/2}}\exp\left(-\frac{|v-u_\lambda|^2}{2 T_\lambda}\right)\int_{\mathbb{R}^N}dw\,|f_1-f_2|(\tau, x-v(t-\tau), w)\\
        & \leq \sup_{\lambda}\int_{\mathbb{R}^N}dv\, \frac{1}{(2\pi T_\lambda)^{N/2}}\exp\left(-\frac{|v-u_\lambda|^2}{2 T_\lambda}\right)\times \int_{\mathbb{R}^N}dw\,(1+|x-v(t-\tau)|^{2\gamma})|f_1-f_2|(\tau, x-v(t-\tau), w)\\
        &\quad+\sup_{\lambda}\int_{\mathbb{R}^N}dv\, \frac{|v|^{2\gamma}(t-\tau)^{2\gamma}+v^2}{(2\pi T_\lambda)^{N/2}(1+|x-v(t-\tau)|^{2\gamma})}\exp\left(-\frac{|v-u_\lambda|^2}{2 T_\lambda}\right)\\
        &\quad\quad \times \int_{\mathbb{R}^N}dw\,(1+|x-v(t-\tau)|^{2\gamma})|f_1-f_2|(\tau, x-v(t-\tau), w)\\
        & \leq \sup_{\lambda, x} \int_{\mathbb{R}^N}dv\, \frac{1}{(2\pi T_\lambda)^{N/2}}\exp\left(-\frac{|v-u_\lambda|^2}{2 T_\lambda}\right)\times \sup_x \int_{\mathbb{R}^N}dw\,(1+|x|^{2\gamma})|f_1-f_2|(\tau, x, w) \\
        &\quad + \sup_{\lambda, x} \int_{\mathbb{R}^N}dv\, \frac{|v|^{2\gamma}(t-\tau)^{2\gamma}+v^2}{(1+|x-v(t-\tau)|^{2\gamma})(2\pi T_\lambda)^{N/2}}\exp\left(-\frac{|v-u_\lambda|^2}{2 T_\lambda}\right)\times \sup_x \int_{\mathbb{R}^N}dw\,(1+|x|^{2\gamma})|f_1-f_2|(\tau, x, w) \\
        & \leq C\sup_x \int_{\mathbb{R}^N} dw\, (1+|x|^{2\gamma})|f_1-f_2|(\tau, x, w),
    \end{align*}
    where we used Lemma \ref{lem:power0} from the third to fourth line so that
    \begin{align}\label{4_1:simpleineq}
        |x-v(t-\tau)+v(t-\tau)|^{2\gamma}\leq |x-v(t-\tau)|^{2\gamma}+|v(t-\tau)|^{2\gamma}
    \end{align}
    and Lemma \ref{lem:int_bound2} with \eqref{spacedeflambda} for the last step. It shows that the weighted integral of $D_1$ is well bounded.
    
    For $D_2$, using \eqref{4_1:macrodiff} and \eqref{4_1:simpleineq}, we get

    \begin{align*}
        &\int_0^1d\lambda \int_{\mathbb{R}^N}dv\,(1+|x|^{2\gamma}+v^2)|D_2|\\
        &\leq \sup_{\lambda}\int_{\mathbb{R}^N} (1+|x|^{2\gamma}+v^2)\frac{(v-u_\lambda)\cdot (u_1-u_2)}{T_\lambda(2\pi T_\lambda)^{N/2}}\frac{\rho_1\rho_2}{\rho_\lambda}\exp\left(-\frac{|v-u_\lambda|^2}{2 T_\lambda}\right)\,dv\\
        &\leq C\sup_{\lambda, x} \frac{\rho_2}{\rho_\lambda} \int_{\mathbb{R}^N} dv\,\left((1+|x-v(t-\tau)|^\gamma)^2+|v|^{2\gamma}(t-\tau)^{2\gamma}+v^2\right)\frac{|v-u_\lambda|}{T_\lambda(2\pi T_\lambda)^{N/2}}\exp\left(-\frac{|v-u_\lambda|^2}{2 T_\lambda}\right) \\
        &\qquad\times \int_{\mathbb{R}^N}dw\, (1+|w|)|f_1-f_2|(\tau, x-v(t-\tau), w)\\
        &\quad+ C\sup_{\lambda, x} \frac{\rho_2}{\rho_\lambda} \int_{\mathbb{R}^N} dv\,\left(1+|x-v(t-\tau)|^{2\gamma}+|v|^{2\gamma}(t-\tau)^{2\gamma}+v^2\right)\frac{|v-u_\lambda|}{T_\lambda(2\pi T_\lambda)^{N/2}}\exp\left(-\frac{|v-u_\lambda|^2}{2 T_\lambda}\right) \\
        &\qquad\times |u_2|\int_{\mathbb{R}^N}dw\, |f_1-f_2|(\tau, x-v(t-\tau), w)\\
        &\eqqcolon I_1+I_2.
    \end{align*}

    Since
    \begin{equation}\label{4_1:inftyeq}
        \begin{split}
            &\sup_{x}\left\|\int_{\mathbb{R}^N} dw\, (1+|x-v(t-\tau)|^{\gamma})(1+|w|)|f_1-f_2|(\tau, x-v(t-\tau), w)\right\|_{L^\infty_v}\\
            &=\sup_{x}\int_{\mathbb{R}^N} dw\, (1+|x|^{\gamma})(1+|w|)|f_1-f_2|(\tau, x, w),
        \end{split}
    \end{equation}
    $I_1$ is bounded as follows:
    \begin{align}\label{4_1:I_1}
        \begin{split}
            I_1 &\leq C \sup_{\lambda, x} \int_{\mathbb{R}^N} dv\, (1+|x-v(t-\tau)|^{\gamma})\frac{|v-u_\lambda|}{T_\lambda(2\pi T_\lambda)^{N/2}}\exp\left(-\frac{|v-u_\lambda|^2}{2 T_\lambda}\right)\\
            &\qquad\times\sup_x \int_{\mathbb{R}^N} dw\, (1+|x|^{\gamma})(1+|w|)|f_1-f_2|(\tau, x, w)\\
            &\quad + C\sup_{\lambda, x}\int_{\mathbb{R}^N} dv\, \frac{1}{1+|x-v(t-\tau)|^{\gamma}}\frac{(|v|^{2\gamma}(t-\tau)^{2\gamma}+v^2)(|v|+|u_\lambda|)}{T_\lambda(2\pi T_\lambda)^{N/2}}\exp\left(-\frac{|v-u_\lambda|^2}{2 T_\lambda}\right)\\
            &\qquad \times \sup_x \int_{\mathbb{R}^N} dw\, (1+|x|^{\gamma})(1+|w|)|f_1-f_2|(\tau, x, w)\\
            &\leq C \sup_x \int_{\mathbb{R}^N} dw\,(1+|x|^{2\gamma}+w^2)|f_1-f_2|(\tau, x, w).
        \end{split}
    \end{align}
    In the last inequality, we used the remark of Lemma \ref{lem:power0} so that
    \begin{align}\label{4_1:simpleineq2}
        (1+|x-v(t-\tau)|^{2\gamma})^{1/2}\leq (1+|x-v(t-\tau)|^{\gamma})\leq \sqrt{2}(1+|x-v(t-\tau)|^{2\gamma})^{1/2}
    \end{align}
    and then Lemma \ref{lem:int_bound2} with \eqref{spacedeflambda}. For the weight term, we used Cauchy-Schwartz inequality $|x|^\gamma |w|\leq \frac{|x|^{2\gamma}+w^2}{2}$.
    
    Using \eqref{4_1:inftyeq}, $I_2$ is bounded as follows:
    \begin{align}\label{4_1:I_2}
        \begin{split}
            I_2&\leq C\sup_x \int_{\mathbb{R}^N}dv\, \frac{(|v|+|u_\lambda|)|u_2|}{T_\lambda(2\pi T_\lambda)^{N/2}}\exp\left(-\frac{|v-u_\lambda|^2}{2 T_\lambda}\right)\\
            &\qquad \times \sup_{\lambda, x} \int_{\mathbb{R}^N} dw\, (1+|x|^{2\gamma})|f_1-f_2|(\tau, x, w)\\
            &\quad + C\sup_{\lambda, x} \int_{\mathbb{R}^N}dv\, \frac{1}{1+|x-v(t-\tau)|^{2\gamma}}\frac{(|v|+|u_\lambda|)(|v|^{2\gamma}(t-\tau)^{2\gamma}+v^2)|u_2|}{T_\lambda(2\pi T_\lambda)^{N/2}}\exp\left(-\frac{|v-u_\lambda|^2}{2 T_\lambda}\right)\\
            &\qquad \times \sup_x \int_{\mathbb{R}^N} dw\, (1+|x|^{2\gamma})|f_1-f_2|(\tau, x, w)\\
            &\leq C\sup_x \int_{\mathbb{R}^N} dw\,(1+|x|^{2\gamma})|f_1-f_2|(\tau, x, w).
        \end{split}
    \end{align}
    From second to third equation, we again used Lemma \ref{lem:int_bound2}. It shows that the weighted integral of $D_2$ is well bounded.

    For $D_3$, let us first bound it by
    \begin{align*}
        &\frac{1}{\sqrt{2\pi  T_\lambda}}\left(\frac{1}{2 T_\lambda} - \frac{|v-u_\lambda|^2}{2T^2_\lambda} \right)\left(\rho_1\rho_2(u_1-u_2)^2\frac{-\lambda\rho_1 + (1-\lambda)\rho_2}{\rho_\lambda^2} + \frac{\rho_1\rho_2(T_1-T_2)}{\rho_\lambda}\right)\\
        & \leq \frac{\rho_1\rho_2}{\rho_\lambda \sqrt{2\pi  T_\lambda}}\left|\frac{1}{2 T_\lambda} - \frac{|v-u_\lambda|^2}{2T^2_\lambda} \right|\left((u_1-u_2)^2 + |T_1-T_2|\right)\\
        & \leq \frac{\rho_1\rho_2}{\rho_\lambda \sqrt{2\pi  T_\lambda}}\left|\frac{1}{2 T_\lambda} - \frac{|v-u_\lambda|^2}{2T^2_\lambda} \right|\left((|u_1|+|u_2|)|u_1-u_2| + |T_1-T_2|\right)\\
        & \eqqcolon D_{3,1}+D_{3,2}.
    \end{align*}
    For $D_{3,1}$ term, using \eqref{4_1:macrodiff} and \eqref{4_1:simpleineq}, we have
    \begin{align*}
        &\sup_x\int_0^1 d\lambda \int_{\mathbb{R}^N}dv\,(1+|x|^{2\gamma}+v^2)|D_{3,1}|\\
        & \leq \sup_{\lambda, x}\int_{\mathbb{R}^N} \,dv(1+|x|^{2\gamma}+v^2)\frac{\rho_1\rho_2}{\rho_\lambda (2\pi T_\lambda)^{N/2}}\left|\frac{1}{2 T_\lambda} - \frac{|v-u_\lambda|^2}{2T^2_\lambda}\right|(|u_1|+|u_2|)|u_1-u_2|\\
        & \leq C\sup_{\lambda, x} \frac{\rho_2}{\rho_\lambda} \int_{\mathbb{R}^N}dv\, \left((1+|x-v(t-\tau)|^\gamma)^2+|v|^{2\gamma}(t-\tau)^{2\gamma} + v^2\right)\frac{\left(T_\lambda + |v-u_\lambda|^2\right)}{2 T^2_\lambda(2\pi T_\lambda)^{N/2}}\exp\left(-\frac{|v-u_\lambda|^2}{2 T_\lambda}\right)\\
        &\quad\quad \times (|u_1|+|u_2|)\int_{\mathbb{R}^N}dw\, (1+|w|)|f_1-f_2|(\tau, x-v(t-\tau), w)\\
        &\quad + C\sup_{\lambda, x} \frac{\rho_2}{\rho_\lambda} \int_{\mathbb{R}^N}dv\,\left(1+|x-v(t-\tau)|^{2\gamma}+|v|^{2\gamma}(t-\tau)^{2\gamma} + v^2\right)\frac{\left(T_\lambda + |v-u_\lambda|^2\right)}{2T^2_\lambda(2\pi T_\lambda)^{N/2}}\exp\left(-\frac{|v-u_\lambda|^2}{2 T_\lambda}\right)\\
        &\quad\quad \times |u_2|(|u_1|+|u_2|)\int_{\mathbb{R}^N}dw\, |f_1-f_2|(\tau, x-v(t-\tau), w)\\
        &\eqqcolon I_3 + I_4.
    \end{align*}

    Using \eqref{4_1:inftyeq}, \eqref{4_1:simpleineq2}, and Lemma \ref{lem:int_bound2} with \eqref{spacedeflambda}, $I_3$ is bounded as follows:
    \begin{align}\label{4_1:I_3}
        \begin{split}
            I_3 &\leq C\sup_{\lambda, x} \int_{\mathbb{R}^N}dv\,(1+|x-v(t-\tau)|^{\gamma})\frac{\left(T_\lambda + 2(|v|^2+|u_\lambda|^2)\right)(|u_1|+|u_2|)}{2T^2_\lambda(2\pi T_\lambda)^{N/2}}\exp\left(-\frac{|v-u_\lambda|^2}{2 T_\lambda}\right)\\
            &\qquad \times \sup_x \int_{\mathbb{R}^N}dw\, (1+|x|^\gamma)(1+|w|)|f_1-f_2|(\tau, x, w)\\
            &\quad + C\sup_{\lambda, x} \int_{\mathbb{R}^N}dv\,\left(1+|v|^{2\gamma}(t-\tau)^{2\gamma} + v^2\right)\frac{\left(T_\lambda +2(|v|^2+|u_\lambda|^2)\right)(|u_1|+|u_2|)}{2T^2_\lambda(2\pi T_\lambda)^{N/2}(1+|x-v(t-\tau)|^\gamma)}\exp\left(-\frac{|v-u_\lambda|^2}{2 T_\lambda}\right)\\
            &\qquad \times \sup_x \int_{\mathbb{R}^N}dw\, (1+|x|^\gamma)(1+|w|)|f_1-f_2|(\tau, x, w)\\
            &\leq C\sup_x \int_{\mathbb{R}^N}dw\, (1+|x|^{2\gamma}+w^2)|f_1-f_2|(\tau, x, w).
        \end{split}
    \end{align}

    For the same reason for $I_3$, $I_4$ is bounded as follows:
    \begin{align}\label{4_1:I_4}
        \begin{split}
            I_4&\leq C\sup_{\lambda, x} \int_{\mathbb{R}^N}dv\,\frac{\left(T_\lambda + 2(|v|^2+|u_\lambda|^2)\right)(|u_1|+|u_2|)|u_2|}{2T^2_\lambda(2\pi T_\lambda)^{N/2}}\exp\left(-\frac{|v-u_\lambda|^2}{2 T_\lambda}\right)\\
            &\qquad \times \sup_x \int_{\mathbb{R}^N}dw\, (1+|x|^{2\gamma})|f_1-f_2|(\tau, x, w)\\
            &\quad + C\sup_{\lambda, x} \int_{\mathbb{R}^N}dv\,\left(1+|v|^{2\gamma}(t-\tau)^{2\gamma} + v^2\right)\frac{\left(T_\lambda + 2(|v|^2+|u_\lambda|^2)\right)(|u_1|+|u_2|)|u_2|}{2T^2_\lambda(2\pi T_\lambda)^{N/2}(1+|x-v(t-\tau)|^{2\gamma})}\exp\left(-\frac{|v-u_\lambda|^2}{2 T_\lambda}\right)\\
            &\qquad \times \sup_x \int_{\mathbb{R}^N}dw\, (1+|x|^{2\gamma})|f_1-f_2|(\tau, x, w)\\
            &\leq C\sup_x \int_{\mathbb{R}^N}dw\, (1+|x|^{2\gamma})|f_1-f_2|(\tau, x, w).
        \end{split}
    \end{align}
    
    For $D_{3,2}$ term, using again \eqref{4_1:macrodiff} and \eqref{4_1:simpleineq}, we obtain
    \begin{align*}
        &\sup_x\int_0^1 d\lambda \int_{\mathbb{R}^N}dv\,(1+|x|^{2\gamma}+v^2)|D_{3,1}|\\
        & \leq \int_{\mathbb{R}^N} (1+|x|^{2\gamma}+v^2)\frac{\rho_1\rho_2}{\rho_\lambda (2\pi T_\lambda)^{N/2}}\left|\frac{1}{2 T_\lambda} - \frac{|v-u_\lambda|^2}{2T^2_\lambda}\right|\left(T_1-T_2\right)\,dv\\
        & \leq C\sup_x \frac{\rho_2}{\rho_\lambda} \int_{\mathbb{R}^N} dv\, \left(1+|x-v(t-\tau)|^{2\gamma} + |v|^{2\gamma}(t-\tau)^{2\gamma}+v^2\right)\frac{T_\lambda + 2(|v|^2+|u_\lambda|^2)}{2T^2_\lambda(2\pi T_\lambda)^{N/2}}\exp\left(-\frac{|v-u_\lambda|^2}{2T_\lambda}\right)\\
        &\quad\quad \times \int_{\mathbb{R}^N}dw\, (1+w^2)|f_1-f_2|(\tau, x-v(t-\tau), w)\\
        &\quad + C\sup_x \frac{\rho_2}{\rho_\lambda} \int_{\mathbb{R}^N} dv\, \left(1+|x-v(t-\tau)|^{2\gamma} + |v|^{2\gamma}(t-\tau)^{2\gamma}+v^2\right)\frac{T_\lambda + 2(|v|^2+|u_\lambda|^2)}{2T^2_\lambda(2\pi T_\lambda)^{N/2}}\exp\left(-\frac{|v-u_\lambda|^2}{2T_\lambda}\right)\\
        &\quad \quad \times \left(|T_2|+|u_1||u_2|\right) \int_{\mathbb{R}^N}dw\, |f_1-f_2|(\tau, x-v(t-\tau), w)\\
        &\quad  + C\sup_x \frac{\rho_2}{\rho_\lambda} \int_{\mathbb{R}^N} dv\,\left(1+|x-v(t-\tau)|^{2\gamma} + |v|^{2\gamma}(t-\tau)^{2\gamma}+v^2\right)\frac{T_\lambda + 2(|v|^2+|u_\lambda|^2)}{2T^2_\lambda(2\pi T_\lambda)^{N/2}}\exp\left(-\frac{|v-u_\lambda|^2}{2 T_\lambda}\right)\\
        &\quad\quad \times (|u_1|+|u_2|)\int_{\mathbb{R}^N} dw\, (1+|w|)|f_1-f_2|(\tau, x-v(t-\tau), w)\\
        &\eqqcolon I_5+I_6+I_7.
    \end{align*}

    Using \eqref{4_1:inftyeq} and Lemma \ref{lem:int_bound2} with \eqref{spacedeflambda}, $I_5$ is bounded as follows:
    \begin{align}\label{4_1:I_5}
        \begin{split}
            I_5 &\leq C\sup_x \frac{\rho_2}{\rho_\lambda} \int_{\mathbb{R}^N} dv\, (1+|x-v(t-\tau)|^{2\gamma})\frac{T_\lambda + 2(|v|^2+|u_\lambda|^2)}{2T^2_\lambda(2\pi T_\lambda)^{N/2}}\exp\left(-\frac{|v-u_\lambda|^2}{2T_\lambda}\right)\\
            &\qquad \times \sup_x \int_{\mathbb{R}^N}dw\, (1+w^2)|f_1-f_2|(\tau, x, w)\\
            &\quad +C\sup_x \frac{\rho_2}{\rho_\lambda} \int_{\mathbb{R}^N} dv\, \left(|v|^{2\gamma}(t-\tau)^{2\gamma}+v^2\right)\frac{T_\lambda + 2(|v|^2+|u_\lambda|^2)}{2T^2_\lambda(2\pi T_\lambda)^{N/2}}\exp\left(-\frac{|v-u_\lambda|^2}{2T_\lambda}\right)\\
            &\qquad \times \sup_x \int_{\mathbb{R}^N}dw\, (1+w^2)|f_1-f_2|(\tau, x, w)\\
            &\leq C\sup_x \int_{\mathbb{R}^N}dw\, (1+w^2)|f_1-f_2|(\tau, x, w).
        \end{split}
    \end{align}

    For the same reason for $I_5$, $I_6$ is bounded as follows:
    \begin{align}\label{4_1:I_6}
        \begin{split}
            I_6 &\leq C\sup_x \frac{\rho_2}{\rho_\lambda} \int_{\mathbb{R}^N} dv\, \frac{(T_\lambda + 2(|v|^2+|u_\lambda|^2))(|T_2|+|u_1||u_2|)}{2T^2_\lambda(2\pi T_\lambda)^{N/2}}\exp\left(-\frac{|v-u_\lambda|^2}{2T_\lambda}\right)\\
            &\qquad \times \sup_x \int_{\mathbb{R}^N}dw\, (1+|x|^{2\gamma})|f_1-f_2|(\tau, x, w)\\
            &\quad +C\sup_x \frac{\rho_2}{\rho_\lambda} \int_{\mathbb{R}^N} dv\, \left(|v|^{2\gamma}(t-\tau)^{2\gamma}+v^2\right)\frac{(T_\lambda + 2(|v|^2+|u_\lambda|^2))(|T_2|+|u_1||u_2|)}{2T^2_\lambda(2\pi T_\lambda)^{N/2}(1+|x-v(t-\tau)|^{2\gamma})}\exp\left(-\frac{|v-u_\lambda|^2}{2T_\lambda}\right)\\
            &\qquad \times \sup_x \int_{\mathbb{R}^N}dw\, (1+|x|^{2\gamma})|f_1-f_2|(\tau, x, w)\\
            &\leq C\sup_x \int_{\mathbb{R}^N}dw\, (1+|x|^{2\gamma})|f_1-f_2|(\tau, x, w).
        \end{split}
    \end{align}

    For the same reason for $I_5$, Cauchy-Schwartz inequality, and \eqref{4_1:simpleineq2}, $I_7$ is bounded as follows:
    \begin{align}\label{4_1:I_7}
        \begin{split}
            I_7 &\leq C\sup_x \int_{\mathbb{R}^N}dw\, (1+|x|^{\gamma})(1+|w|)|f_1-f_2|(\tau, x, w)\\
            &\qquad \times \sup_x \frac{\rho_2}{\rho_\lambda} \int_{\mathbb{R}^N} dv\, (1+|x-v(t-\tau)|^\gamma)\frac{(T_\lambda + 2(|v|^2+|u_\lambda|^2))(|u_1| + |u_2|)}{2T^2_\lambda(2\pi T_\lambda)^{N/2}}\exp\left(-\frac{|v-u_\lambda|^2}{2T_\lambda}\right)\\
            &\quad +C\sup_x \int_{\mathbb{R}^N}dw\, (1+|x|^{\gamma})(1+|w|)|f_1-f_2|(\tau, x, w)\\
            &\qquad \times \sup_x \frac{\rho_2}{\rho_\lambda} \int_{\mathbb{R}^N} dv\, \left(|v|^{2\gamma}(t-\tau)^{2\gamma}+v^2\right)\frac{(T_\lambda + 2(|v|^2+|u_\lambda|^2))(|u_1| + |u_2|)}{2T^2_\lambda(2\pi T_\lambda)^{N/2}(1+|x-v(t-\tau)|^{\gamma})}\exp\left(-\frac{|v-u_\lambda|^2}{2T_\lambda}\right)\\
            &\leq C\sup_x \int_{\mathbb{R}^N}dw\, (1+|x|^{2\gamma}+w^2)|f_1-f_2|(\tau, x, w).
        \end{split}
    \end{align}
    It shows that the weighted integral of $D_3$ is well bounded.
    
    Combining \eqref{4_1:I_1}, \eqref{4_1:I_2}, \eqref{4_1:I_3}, \eqref{4_1:I_4}, \eqref{4_1:I_5}, \eqref{4_1:I_6}, and \eqref{4_1:I_7}, we finally obtain
    \begin{align*}
        &\int_{\mathbb{R}^N}\int_0^t  e^{-(t-\tau)}(1+|x|^{2\gamma} + v^2)|\mathcal{M}(f_2)-\mathcal{M}(f_1)|(\tau,x-v(t-\tau),v)\,d\tau dv\\
        &\leq \int_0^t  \int_{\mathbb{R}^N} e^{-(t-\tau)}(1+|x|^{2\gamma} + v^2)(|D_1|+|D_2|+|D_3|)\,dv d\tau\\
        &\leq \int_0^t e^{-(t-\tau)}\sum_{i=1}^7 I_i\,d\tau\\
        &\leq C\int_0^t e^{-(t-\tau)}\sup_x \int_{\mathbb{R}^N}(1+|x|^{2\gamma}+v^2)|f_1-f_2|(\tau, x, v)\,dvd\tau\\
        &\leq C(1-e^{-t_0})\sup_{t\leq t_0,x}\int_{\mathbb{R}^N}(1+|x|^{2\gamma}+v^2)|f_1-f_2|(t, x, v)\,dv
    \end{align*}
    and prove \eqref{4_1:differenceintegral} by choosing $t_0$ small enough so that $C(1-e^{-t_0})<1$.
    
    Now, we showed that $\mathcal{T}$ is a contraction mapping. Finally, we employ the Banach fixed point theorem to get the existence and uniqueness of the solution with initial $f_0(x,v)$.
\end{proof}

\section{Well-Posedness theory of the Boltzmann equation}
In this section, we explain the well-posedness theory of the Boltzmann equation with cutoff collision kernel \eqref{Collision_kernel}, which will show stark contrasts with the ill-posedness theory of the previous sections. Throughout this section, spatial domain $\Omega$ is $\mathbb{R}^{N}_{x}$ or $\mathbb{T}^{N}_{x}$.

The following lemma extends the energy conservation $|v|^2+|u|^2 = |v'|^2+|u'|^2$ in collision variable for exponent $0<\beta<2$; see \eqref{w-exp} for an application.
\begin{lemma} \label{lem:ineq}
    Let $x,y\geq 0$ and $0<\alpha\leq 1$. If $y\geq x$, then 
    \begin{equation} \notag
        x^\alpha+y^\alpha\geq (x+y)^\alpha\geq x^\alpha + (2^\alpha - 1)y^\alpha.
    \end{equation}
    If $2^{-n-1}x\leq y\leq 2^{-n}x$, then 
    \begin{equation} \notag
        x^\alpha+y^\alpha\geq (x+y)^\alpha\geq x^\alpha + \frac{\alpha/2}{2^{(1-\alpha)(n+1)}}y^\alpha.
    \end{equation}
\end{lemma}

\begin{proof}
    Homogenising the inequality for $x/y$, it is enough to specify the $(t+1)^\alpha - t^\alpha$ for $t>0$. For $0<t\leq 1$,
    \begin{align*}
        1\geq (t+1)^\alpha -t^\alpha\geq 2^{\alpha}-1.
    \end{align*}
    For $2^n\leq t<2^{n+1}$,
    \begin{align*}
        (t+1)^\alpha = t^{\alpha}(1+1/t)^{\alpha}\geq t^{\alpha} + \frac{\alpha}{2} t^{\alpha - 1}\geq t^{\alpha} + \frac{\alpha/2}{2^{(1-\alpha)(n+1)}}.
    \end{align*}
    Taking $t = x/y$ and multiplying $y^\alpha$ on both sides, we get the inequality.
\end{proof}

The following lemma extracts some decaying terms from a geometric factor.
\begin{lemma}\label{lem:cone}
    Let $x,y,a\in\mathbb{R}^N$ satisfies
    \begin{align*}
        x\perp y,\quad |x|\leq |a|/4,\quad R\leq |a|/4,\quad |a+x+y|\leq R.
    \end{align*}
    Then the set
    \begin{align*}
        \{x:\{y:y\perp x\}\cap \{y:|a+x+y|\leq R\}\neq \emptyset,\quad |x|\leq |a|/4\}
    \end{align*}
    is contained a set $C_{a, R}$ satisfying the following: for any isotropic function $f(x) = f(|x|)$,
    \begin{align*}
        \int_{\{|x|\leq |a|/4\}\cap C_{a, R}} f(x)\,dx \leq C_N\frac{R}{|a|}\int_{0}^{|a|/4}|x|^{N-1}f(|x|)\,d|x|
    \end{align*}
    for some constant $C_N$ only depending on $N$. See figure \ref{fig:fig1} for the shape of $C_{a, R}$.
\end{lemma}

\begin{figure}[htbp]
    \centering
    \begin{tikzpicture}
        \draw[thick, ->] (0,-3) -- (0,3);
        \draw[thick, ->] (-3.5,0) -- (3.5,0);
        \draw[] (3, 0) node[above, xshift = -0.2cm]{$\frac{|a|}{4}$};
        \draw[thick](0,0) circle (2.5);
        \draw[fill=red!30, rotate=-40] (0,0) --  (0:2.5) arc(0:20:2.5) -- cycle;
        \draw[color = red, rotate=-110] (0,0) -- (0,2.5);
        \draw[color = red, rotate=-130] (0,0) -- (0,2.5);
        \draw[color = red, rotate=-120, dotted, thick] (0,0) -- (0,2.5);
        \draw[fill=red!30, rotate=200] (0,0) --  (0:2.5) arc(0:20:2.5) -- cycle;
        \draw[color = red, rotate=110] (0,0) -- (0,2.5);
        \draw[color = red, rotate=130] (0,0) -- (0,2.5);
        \draw[color = red, rotate=120, dotted, thick] (0,0) -- (0,2.5);
        \draw[latex-latex]  (221:2.5) arc(221:199:2.5) node[midway,left]{$2\delta\theta$};
        \draw[latex-latex]  (-41:2.5) arc(-41:-19:2.5) node[midway,right]{$2\delta\theta$};
        \draw[latex-latex]  (90:0.5) arc(90:-30:0.5) node[midway,right]{$\theta_0$};
    \end{tikzpicture}
    \caption{Shape of $C_a$.}
    \label{fig:fig1}
\end{figure}
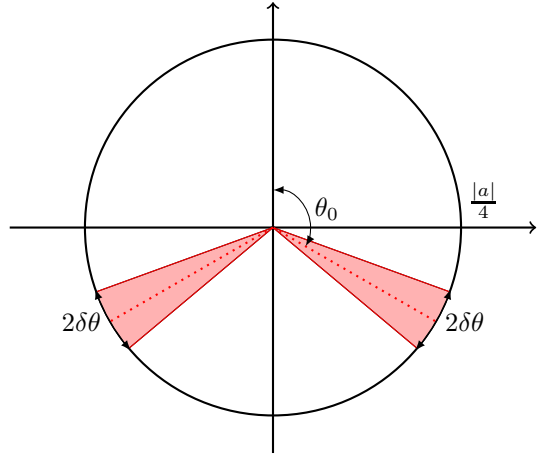
\begin{proof}
    Fixing $x$, $|a+x+y|\leq R$ is a closed ball centered at $-(a+x)$ with radius $R$ about $y$. When the distance between $-(a+x)$ and the plane $y\perp x$ is smaller than $R$, the plane intersects with the ball. The distance is given by
    \begin{align*}
        \left|(a+x)\cdot \hat{x}\right| = \left|a\cdot \hat{x} + |x|\right|\leq R.
    \end{align*}
    It shows that the sufficient condition to make an intersection is 
    \begin{align*}
        -R-|x|\leq a\cdot \hat{x}\leq -|x|+R.
    \end{align*}
    Let $a = (0, \cdots, 0, |a|)$ and $a\cdot \hat{x} = |a|\cos\theta$. Let us define $\cos\theta_0 = -\frac{|x|}{|a|}$ and $\delta\theta= \frac{4}{\sqrt{15}}\frac{|R|}{|a|}$. Since $\frac{|x|}{|a|}\leq \frac{1}{4}$,
    \begin{align*}
        \cos(\theta_0+\delta\theta)-\cos\theta_0&\leq -(\sin \theta_0)\delta\theta \leq -\frac{|R|}{|a|},\\
        \cos(\theta_0-\delta\theta)-\cos\theta_0&\geq (\sin \theta_0)\delta\theta \leq \frac{|R|}{|a|}.
    \end{align*}
    Therefore, we have
    \begin{align*}
        \{x:R-|x|\leq a\cdot \hat{x}\leq |x|+R\}\subset \{x:\theta\in [\theta_0-\delta\theta, \theta_0+\delta\theta]\}\eqqcolon C_{a, R}.
    \end{align*}

    Finally, for any isotropic function $f(x) = f(|x|)$,
    \begin{align*}
        \int_{\{|x|\leq |a|/4\}\cap C_{a, R}} f(x)\,dx &= |\mathbb{S}^{N-2}|\int_{0}^{|a|/4}f(|x|)\int_{\theta_0-\delta\theta}^{\theta_0+\delta\theta}|x|^{N-1}\sin^{N-2}\theta\,d\theta d|x|\\
        &\leq C_N\frac{R}{|a|}\int_{0}^{|a|/4}|x|^{N-1}f(|x|)\,d|x|.
    \end{align*}
\end{proof}

\begin{lemma}\label{lem:cone2}
    Let $x,y,a\in\mathbb{R}^N$ satisfies
    \begin{align*}
        x\perp y,\quad R\leq |a|/4,\quad |a+y|\leq R.
    \end{align*}
    There exists a set $C_{a, R}$ which satisfies
    \begin{align*}
        \{x:\{y:y\perp x\}\cap \{y:|a+y|\leq R\}\neq \emptyset\} \subset C_{a,R}
    \end{align*}
    and for any isotropic function $f(x) = f(|x|)$,
    \begin{align*}
        \int_{\{|x|\leq |a|/4\}\cap C_{a, R}} f(x)\,dx \leq C\frac{R}{|a|}\int_{0}^{|a|/4}|x|^{N-1}f(|x|)\,d|x|,
    \end{align*}
    for some constant $C$ only depending on $N$.
\end{lemma}
\begin{proof}
    The proof is almost identical to Lemma \ref{lem:cone}.
\end{proof}

The next lemma is designed to bound a specific integral appearing in proof of Lemma \ref{lem:Q_gain}.
\begin{lemma}\label{lem:int_bound3}
    For $a<N$ and $2\delta + a>N$, we have
    \begin{align*}
        \int_{|y|\leq 2|x|} dy\,\frac{1}{|y|^a\left(1+|y+x|^2\right)^\delta} &\leq C_{N, a, \delta}\min\{|x|^{N-a}, (1+|x|)^{-a} + (1+|x|)^{-2\delta - a + N}\}\\
        \int_{\mathbb{R}^N} dy\,\frac{1}{|y|^a\left(1+|y+x|^2\right)^\delta} &\leq C'_{N, a, \delta}
    \end{align*}
    for some constant $C_{N,a,\delta}$ and $C'_{N,a,\delta}$.
\end{lemma}
\begin{proof}
    First, we assume $|x|\leq 2$. If $|y|\geq 4$, then $|y+x|\geq |y|/2$. Using this inequality, the integral is bounded by
    \begin{equation}\label{5_4:xleq2}
        \begin{split}
            &\mathbf{1}_{|x|\leq 2}\int_{|y|\leq 2|x|} dy\,\frac{1}{|y|^a \left(1+|y+x|^2\right)^\delta}\leq \mathbf{1}_{|x|\leq 2}\int_{|y|\leq 2|x|} dy\,\frac{1}{|y|^a}\leq C_{N, a} |x|^{N-a},\\
            &\mathbf{1}_{|x|\leq 2}\int_{\mathbb{R}^N} dy\,\frac{1}{|y|^a \left(1+|y+x|^2\right)^\delta}\leq \int_{|y|\leq 4} dy\,\frac{1}{|y|^a} + \int_{|y|\geq 4} dy\,\frac{1}{|y|^a(1+|y|^2/4)^\delta}<\infty
        \end{split}
    \end{equation}
    for $a<N$ and $2\delta + a > N$.\\

    For $|x|\geq 2$, let us decompose the integral set by
    \begin{align*}
        A_1&= \{y:|y|\leq |x|/2\},\\
        A_2&= \{y:|x+y|\leq |x|/2\},\\
        A_3&= \{y:|y|\leq 2|x|\}\setminus(A_1\cup A_2),\\
        A_4&= \{y:|y|> 2|x|\}.
    \end{align*}
    Each integral on $A_i$ is bounded as follows. For $A_1$,
    \begin{align*}
        \int_{A_1}dy\,\frac{1}{|y|^a\left(1+|x+y|^2\right)^{\delta}}&\leq\int_{A_1}dy\,\frac{1}{|y|^a(1+(|x| - |y|)^2)^{\delta}}\\
        &\leq \frac{|\mathbb{S}^{N-1}|}{(1+(|x|/2)^2)^{\delta}}\int_0^{|x|/2}\frac{r^{N-1}}{r^a}\,dr\\
        &\leq C_{N, a, \delta}\frac{1}{|x|^{a + 2\delta - N}}.
    \end{align*}
    
    For $A_2$, if $a\geq 0$, 
    \begin{align*}
        \int_{A_2}dy\,\frac{1}{|y|^a\left(1+|x+y|^2\right)^{\delta}}&\leq\int_{A_2}dy\,\frac{1}{(|x| - |x+y|)^a\left(1+|x+y|^2\right)^{\delta}}\\
        &\leq \frac{|\mathbb{S}^{N-1}|}{(|x|/2)^a}\int_0^{|x|/2}\frac{r^{N-1}}{(1+r^2)^{\delta}}\,dr\\
        &\leq \frac{|\mathbb{S}^{N-1}|}{(|x|/2)^a}\left(\int_0^1\frac{r^{N-1}}{(1+r^2)^{\delta}}\,dr + \int_1^{|x|/2}\frac{r^{N-1}}{(1+r^2)^{\delta}}\,dr\right)\\
        &\leq \frac{|\mathbb{S}^{N-1}|}{(|x|/2)^a}\left(\int_0^1 r^{N-1}\,dr +C_\delta\int_1^{|x|/2}r^{N-1-2\delta}\,dr\right)\\
        &\leq C_{N, a, \delta}\left(\frac{1}{|x|^a} + \frac{1}{|x|^{a+2\delta-N}}\right).
    \end{align*}
    If $a<0$, as $|y|\leq |x| + |x+y|\leq 3|x|/2$,
    \begin{align*}
        \int_{A_2}dy\,\frac{1}{|y|^a\left(1+|x+y|^2\right)^{\delta}}&\leq \frac{|\mathbb{S}^{N-1}|}{(3|x|/2)^a}\int_0^{|x|/2}\frac{r^{N-1}}{(1+r^2)^{\delta}}\,dr\\
        &\leq C_{N, a, \delta}\left(\frac{1}{|x|^a} + \frac{1}{|x|^{a+2\delta-N}}\right).
    \end{align*}
    
    For $A_3$, since $|x+y|\geq |x|/2$ and $|x|/2\leq |y|\leq 2|x|$, we obtain
    \begin{align*}
        \int_{A_3}dy\,\frac{1}{|y|^a\left(1+|x+y|^2\right)^{\delta}}\leq C_{a, \delta}\frac{|\mathbb{S}^{N-1}|}{|x|^{a+2\delta}} \int_0^{2|x|}r^{N-1}\,dr\leq C_{N, a, \delta}\frac{1}{|x|^{a+2\delta - N}}.
    \end{align*}
    
    Therefore, collecting above estimates for $A_i$, for $2\delta + a >N$, we get
    \begin{equation}\label{5_4:xgeq2}
        \begin{split}
            \mathbf{1}_{|x|\geq 2}
            \int_{|y|\leq 2|x|}dy\,\frac{1}{|y|^a\left(1+|x+y|^2\right)^{\delta}}&=\mathbf{1}_{|x|\geq 2} \sum_{i=1}^3\int_{A_i}dy\,\frac{1}{|y|^a\left(1+|x+y|^2\right)^{\delta}}\\
            &\leq C_{N, a, \delta}\mathbf{1}_{|x|\geq 2}\left(\frac{1}{|x|^a} + \frac{1}{|x|^{a+2\delta - N}}\right).
        \end{split}
    \end{equation}
    Combining \eqref{5_4:xleq2} and \eqref{5_4:xgeq2}, conclude the first inequality of Lemma \ref{lem:int_bound3}.

    If $|y|>2|x|$, then $|x+y|\geq |y|/2$, so we have
    \begin{align*}
        \int_{A_4}dy\,\frac{1}{|y|^a\left(1+|x+y|^2\right)^{\delta}}\leq C_{\delta}\int_{A_4}dy\,\frac{1}{|y|^a\left(1+|y|^2\right)^{\delta}}\leq C_{\delta}\int_{\mathbb{R}^N}dy\,\frac{1}{|y|^a\left(1+|y|^2\right)^{\delta}}<\infty
    \end{align*}
    as $2\delta+a>N$. Since the integral on $A_4$ is uniformly bounded regardless of the value of $|x|$, combining with the first inequality of the lemma, we get
    \begin{align*}
        \int_{\mathbb{R}^N} dy\,\frac{1}{|y|^a\left(1+|y+x|^2\right)^\delta} &\leq C'_{N, a, \delta}
    \end{align*}
    for some constant $C'_{N,a,\delta}$.
    \end{proof}

The following lemma is one of the main lemmas used throughout this section.
\begin{lemma}\label{lem:int_bound4}
    For the choice of $\alpha$, $\beta$, $N$ in \eqref{index:weight}, $-N<\kappa\leq 1$, and $2\delta>N$, we have
    \begin{align*}
        \sup_{v\in\mathbb{R}^N} \iint_{\mathbb{R}^{N}\times   \mathbb{S}^{N-1}}dud\omega B(v-u, \omega) \frac{w_{\delta}(v)}{w_{\delta}(v')w_{\delta}(u')}\leq C_{\alpha,\beta,N,\delta, \kappa}\left(\frac{1}{(1+|v|)^{\frac{N-1}{2}\beta - \kappa}}+\frac{1}{(1+|v|)^{\frac{3}{2}\beta - 1 - \kappa}}\right)
    \end{align*}
    for some constant $C_{\alpha,\beta,N,\delta, \kappa}$.
\end{lemma}
\begin{proof}
    We first compute upper bound of $\frac{w_{\delta}(v)}{w_{\delta}(v')w_{\delta}(u')}$. From energy conservation $|v|^{2} + |u|^{2} = |v'|^{2}+|u'|^{2}$, either $|v|^{2} \leq 2|v'|^{2}$ or $|v|^{2} \leq 2|u'|^{2}$ must hold. Hence for $\delta > 0$,
    \begin{equation} \label{w-polyv}
        \frac{(1+|v|^2)^{\delta}}{(1+|v'|^2)^\delta (1+|u'|^2)^{\delta}} \leq C_{{\delta}} \Big( \frac{1}{(1+|v'|^2)^{\delta}} + \frac{1}{(1+|u'|^2)^{\delta}} \Big).
    \end{equation}    
    
    For the exponential term, using Lemma \ref{lem:ineq} and energy conservation, for $\beta/2 \leq 1$, we obtain
    \begin{equation} \label{w-exp}
        \begin{split}
            &e^{-\alpha \left(\left(|v'|^2\right)^{\beta/2}+\left(|u'|^2\right)^{\beta/2} - \left(|v|^2\right)^{\beta/2}\right)} \\
            & \leq e^{-\alpha \left(\left(|v'|^2+|u'|^2\right)^{\beta/2} - \left(|v|^2\right)^{\beta/2}\right)} \\
            & \leq e^{-\alpha \left(\left(|v|^2+|u|^2\right)^{\beta/2} - \left(|v|^2\right)^{\beta/2}\right)} \\
            & \leq e^{-\alpha (2^{\beta/2}-1)|u|^\beta}\mathbf{1}_{|u|\geq |v|} + \sum_{n=0}^\infty e^{-\frac{\alpha \beta/4}{2^{(1-\beta/2)(n+1)}}|u|^\beta}\mathbf{1}_{2^{-n-1}|v|^2\leq |u|^2\leq 2^{-n}|v|^2},
        \end{split}
    \end{equation}
    where we used 
    \begin{align*}
        x^\beta+y^\beta\geq (x^2+y^2)^{\beta/2}\geq x^\beta + \frac{\beta/4}{2^{(1-\beta/2)(n+1)}}y^\beta
    \end{align*}
    for $2^{-n-1}x^2\leq y^2\leq 2^{-n}x^2$. From \eqref{w-polyv} and \eqref{w-exp}, we get
    \begin{align*}
        \frac{w_{\delta}(v)}{w_{\delta}(v')w_{\delta}(u')} &\leq C_{\delta} \Big(\frac{1}{(1+|v'|^2)^{\delta}} + \frac{1}{(1+|u'|^2)^{\delta}} \Big)  \\
        &\quad \times \left(e^{-\alpha (2^{\beta/2}-1)|u|^\beta}\mathbf{1}_{|u|\geq |v|} + \sum_{n=0}^\infty e^{-\frac{\alpha \beta/4}{2^{(1-\beta/2)(n+1)}}|u|^\beta}\mathbf{1}_{2^{-n-1}|v|^2\leq |u|^2\leq 2^{-n}|v|^2}\right).
    \end{align*}
    
    Now, we apply Carlemann type estmate with definitions $\tilde{u}:= u-v$, $\tilde{u}_{\parallel}:= (\tilde{u}\cdot\omega)\omega$, and $\tilde{u}_{\perp}:= \tilde{u} - \tilde{u}_{\parallel}$. Since the definitions yields $u = v+\tilde{u}_\parallel+\tilde{u}_\perp$, $v'=v+\tilde{u}_{\parallel}$, and $u'=v+\tilde{u}_{\perp}$, we obtain
    \begin{equation*}
        \begin{split}
            &\iint_{\mathbb{R}^{N}\times   \mathbb{S}^{N-1}}dud\omega \,B(v-u, \omega)\frac{w_{\delta}(v)}{w_{\delta}(v')w_{\delta}(u')}\\
            &\leq C\iint_{\mathbb{R}^{N}\times   \mathbb{S}^{N-1}}dud\omega \,|v-u|^{\kappa-1}|(v-u)\cdot\omega| \frac{w_{\delta}(v)}{w_{\delta}(v')w_{\delta}(u')}\\
            &\leq 2C\int_{\mathbb{R}^N}d\tilde{u}_\parallel\,\int_{E_{\tilde{u}_\parallel}}d\tilde{u}_\perp\, \frac{\left(|\tilde{u}_\parallel|^2+|\tilde{u}_\perp|^2\right)^{\frac{\kappa-1}{2}}}{|\tilde{u}_\parallel|^{N-2}}
            \Big[ \frac{1}{(1+|v+\tilde{u}_{\parallel}|^2)^{\delta} } + \frac{1}{(1+|v+\tilde{u}_{\perp}|^2)^{\delta}} \Big]\\
            &\quad\times 
            \left\{
            e^{-\alpha (2^{\beta/2}-1)|v+\tilde{u}_\parallel+\tilde{u}_\perp|^\beta}\mathbf{1}_{|v+\tilde{u}_\parallel+\tilde{u}_\perp|\geq |v|}
            +
            \sum_{n=0}^\infty e^{-\frac{\alpha \beta/4}{2^{(1-\beta/2)(n+1)}}|v+\tilde{u}_\parallel+\tilde{u}_\perp|^\beta}\mathbf{1}_{2^{-n-1}|v|^2\leq |v+\tilde{u}_\parallel+\tilde{u}_\perp|^2\leq 2^{-n}|v|^2}
            \right\},
        \end{split}
    \end{equation*}
    where $E_{\tilde{u}_\parallel}$ is the hyperplane passing through $0$ and orthogonal to $\tilde{u}_\parallel$.

    Now, let us define
    \begin{align*}
        I_{1,1} &=\int_{\mathbb{R}^N}d\tilde{u}_\parallel\,\int_{E_{\tilde{u}_\parallel}}d\tilde{u}_\perp\,\frac{\left(|\tilde{u}_\parallel|^2+|\tilde{u}_\perp|^2\right)^{\frac{\kappa-1}{2}}}{|\tilde{u}_\parallel|^{N-2}}
        \frac{1}{(1+|v+\tilde{u}_{\parallel}|^2)^{\delta}} e^{-\frac{\alpha\beta/4}{2^{2-\beta}}|v+\tilde{u}_\parallel+\tilde{u}_\perp|^\beta}\mathbf{1}_{|v+\tilde{u}_\parallel+\tilde{u}_\perp|\geq |v|/2},\\
        I_{1,2} &=\int_{\mathbb{R}^N}d\tilde{u}_\parallel\,\int_{E_{\tilde{u}_\parallel}}d\tilde{u}_\perp\,\frac{\left(|\tilde{u}_\parallel|^2+|\tilde{u}_\perp|^2\right)^{\frac{\kappa-1}{2}}}{|\tilde{u}_\parallel|^{N-2}}
        \frac{1}{(1+|v+\tilde{u}_{\parallel}|^2)^{\delta}}\\
        &\qquad \times \sum_{n=2}^\infty e^{-\frac{\alpha \beta/4}{2^{(1-\beta/2)(n+1)}}|v+\tilde{u}_\parallel+\tilde{u}_\perp|^\beta}\mathbf{1}_{2^{-n-1}|v|^2\leq |v+\tilde{u}_\parallel+\tilde{u}_\perp|^2\leq 2^{-n}|v|^2},\\
        I_{2,1} &=\int_{\mathbb{R}^N}d\tilde{u}_\parallel\,\int_{E_{\tilde{u}_\parallel}}d\tilde{u}_\perp\,\frac{\left(|\tilde{u}_\parallel|^2+|\tilde{u}_\perp|^2\right)^{\frac{\kappa-1}{2}}}{|\tilde{u}_\parallel|^{N-2}}
        \frac{1}{(1+|v+\tilde{u}_{\perp}|^2)^{\delta}} e^{-\frac{\alpha \beta/4}{2^{8-4\beta}}|v+\tilde{u}_\parallel+\tilde{u}_\perp|^\beta}\mathbf{1}_{|v+\tilde{u}_\parallel+\tilde{u}_\perp|\geq |v|/16},\\
        I_{2,2} &=\int_{\mathbb{R}^N}d\tilde{u}_\parallel\,\int_{E_{\tilde{u}_\parallel}}d\tilde{u}_\perp\,\frac{\left(|\tilde{u}_\parallel|^2+|\tilde{u}_\perp|^2\right)^{\frac{\kappa-1}{2}}}{|\tilde{u}_\parallel|^{N-2}}\frac{1}{(1+|v+\tilde{u}_{\perp}|^2)^{\delta}}\\
        &\qquad \times \sum_{n=8}^\infty e^{-\frac{\alpha \beta/4}{2^{(1-\beta/2)(n+1)}}|v+\tilde{u}_\parallel+\tilde{u}_\perp|^\beta}\mathbf{1}_{2^{-n-1}|v|^2\leq |v+\tilde{u}_\parallel+\tilde{u}_\perp|^2\leq 2^{-n}|v|^2},
    \end{align*}
    and $I_1 = I_{1,1}+I_{1,2}$ and $I_2 = I_{2,1}+I_{2,2}$. When $n\geq 2$, then $|v+\tilde{u}_\parallel + \tilde{u}_\perp|\leq \frac{|v|}{2}$, so $|\tilde{u}_\parallel + \tilde{u}_\perp|\geq \frac{|v|}{2}$. Therefore, as $\kappa\leq 1$, we have
    \begin{align*}
        &\left(|\tilde{u}_\parallel|^2+|\tilde{u}_\perp|^2\right)^{\frac{\kappa-1}{2}} e^{-\frac{\alpha \beta/4}{2^{(1-\beta/2)(n+1)}}|v+\tilde{u}_\parallel+\tilde{u}_\perp|^\beta}\mathbf{1}_{2^{-n-1}|v|^2\leq |v+\tilde{u}_\parallel+\tilde{u}_\perp|^2\leq 2^{-n}|v|^2}\\
        &\leq C_\kappa|v|^{\frac{\kappa-1}{2}} e^{-\frac{\alpha \beta/4}{2^{(1-\beta/2)(n+1)}}|v+\tilde{u}_\parallel+\tilde{u}_\perp|^\beta}\mathbf{1}_{2^{-n-1}|v|^2\leq |v+\tilde{u}_\parallel+\tilde{u}_\perp|^2\leq 2^{-n}|v|^2}\\
        &\leq C_\kappa|v|^{\frac{\kappa-1}{2}} e^{-\frac{\alpha \beta/4}{2^{n+1}}|v|^\beta}\mathbf{1}_{2^{-n-1}|v|^2\leq |v+\tilde{u}_\parallel+\tilde{u}_\perp|^2\leq 2^{-n}|v|^2}.
    \end{align*}
    Also, by the definition of $\tilde{u}_\parallel$,
    \begin{align*}
        |\tilde{u}_\parallel| = |((u-v)\cdot \omega)\omega|\leq 2|v|.
    \end{align*}
    (Or you can directly show this by investigating possible $\tilde{u}_\parallel$ satisfying $2^{-n-1}|v|^2\leq |v+\tilde{u}_\perp+\tilde{u}_\parallel|^2\leq 2^{-n}|v|^2$ and $\tilde{u}_\parallel\perp\tilde{u}_\perp$.) From these observations, we redefine $I_{1,2}$ and $I_{2,2}$ by
    \begin{align*}
        I_{1,2} &=\sum_{n=2}^\infty \frac{e^{-\frac{\alpha \beta/4}{2^{n+1}}|v|^\beta}}{|v|^{1-\kappa}}\int_{|\tilde{u}_\parallel|\leq 2|v|}d\tilde{u}_\parallel\,\int_{E_{\tilde{u}_\parallel}}d\tilde{u}_\perp\,\frac{1}{|\tilde{u}_\parallel|^{N-2}}
        \frac{1}{(1+|v+\tilde{u}_{\parallel}|^2)^{\delta}}\mathbf{1}_{2^{-n-1}|v|^2\leq |v+\tilde{u}_\parallel+\tilde{u}_\perp|^2\leq 2^{-n}|v|^2},\\
        I_{2,2} &=\sum_{n=8}^\infty \frac{e^{-\frac{\alpha \beta/4}{2^{n+1}}|v|^\beta}}{|v|^{1-\kappa}}\int_{|\tilde{u}_\parallel|\leq 2|v|}d\tilde{u}_\parallel\,\int_{E_{\tilde{u}_\parallel}}d\tilde{u}_\perp\,\frac{1}{|\tilde{u}_\parallel|^{N-2}}\frac{1}{(1+|v+\tilde{u}_{\perp}|^2)^{\delta}}\mathbf{1}_{2^{-n-1}|v|^2\leq |v+\tilde{u}_\parallel+\tilde{u}_\perp|^2\leq 2^{-n}|v|^2}.
    \end{align*}
    \\

    We bound $I_{1,1}$ first. Using
    \begin{align}\label{5_5:kappa_split}
        \begin{split}
            \frac{\left(|\tilde{u}_\parallel|^2+|\tilde{u}_\perp|^2\right)^{\frac{\kappa-1}{2}}}{|\tilde{u}_\parallel|^{N-2}}\leq \frac{|\tilde{u}_\parallel|^{\frac{2}{N+1}(\kappa-1)}}{|\tilde{u}_\parallel|^{N-2}}|\tilde{u}_\perp|^{\frac{N-1}{N+1}(\kappa-1)} = \frac{1}{|\tilde{u}_\parallel|^{(N-2) + \frac{2}{N+1}(1-\kappa)}}\frac{1}{|\tilde{u}_\perp|^{\frac{N-1}{N+1}(1-\kappa)}}
        \end{split}
    \end{align}
    for $\kappa-1\leq 0$, and
    \begin{align*}
        e^{-\frac{\alpha\beta/4}{2^{2-\beta}}|v+\tilde{u}_\parallel+\tilde{u}_\perp|^\beta}\mathbf{1}_{|v+\tilde{u}_\parallel+\tilde{u}_\perp|\geq |v|/2}& = e^{-\frac{\alpha\beta/4}{2^{3-\beta}}|v+\tilde{u}_\parallel+\tilde{u}_\perp|^\beta}e^{-\frac{\alpha\beta/4}{2^{3-\beta}}|v+\tilde{u}_\parallel+\tilde{u}_\perp|^\beta}\mathbf{1}_{|v+\tilde{u}_\parallel+\tilde{u}_\perp|\geq |v|/2}\\
        &\leq e^{-\frac{\alpha\beta/4}{2^3}|v|^\beta}e^{-\frac{\alpha\beta/4}{2^{3-\beta}}|v+\tilde{u}_\parallel+\tilde{u}_\perp|^\beta},
    \end{align*}
    we bound $I_{1,1}$ as follows:
    \begin{equation} \label{5_5:I_11}
        \begin{split}
            I_{1,1} &= \int_{\mathbb{R}^N}d\tilde{u}_\parallel\,\int_{E_{\tilde{u}_\parallel}} d\tilde{u}_\perp\,\frac{\left(|\tilde{u}_\parallel|^2+|\tilde{u}_\perp|^2\right)^{\frac{\kappa-1}{2}}}{|\tilde{u}_\parallel|^{N-2}\left(1+|v+\tilde{u}_\parallel|^2\right)^{\delta}} e^{-\frac{\alpha\beta/4}{2^{2-\beta}}|v+\tilde{u}_\parallel+\tilde{u}_\perp|^\beta}\mathbf{1}_{|v+\tilde{u}_\parallel+\tilde{u}_\perp|\geq |v|/2} \\
            &\leq e^{-\frac{\alpha\beta/4}{2^3}|v|^\beta}\int_{\mathbb{R}^N}d\tilde{u}_\parallel\, \frac{1}{|\tilde{u}_\parallel|^{(N-2) + \frac{2}{N+1}(1-\kappa)}\left(1+|v+\tilde{u}_\parallel|^2\right)^{\delta}} \\
            &\qquad \times \int_{E_{\tilde{u}_\parallel}}d\tilde{u}_\perp\, \frac{1}{|\tilde{u}_\perp|^{\frac{N-1}{N+1}(1-\kappa)}}e^{-\frac{\alpha\beta/4}{2^{3-\beta}}|v+\tilde{u}_{\parallel} + \tilde{u}_{\perp}|^\beta} \\
            &\leq e^{-\frac{\alpha\beta}{2^5}|v|^\beta}\int_{\mathbb{R}^N}d\tilde{u}_\parallel\, \frac{1}{|\tilde{u}_\parallel|^{(N-2) + \frac{2}{N+1}(1-\kappa)}\left(1+|v+\tilde{u}_\parallel|^2\right)^{\delta}}\\
            &\qquad \times \left(\int_{|\tilde{u}_\perp|\leq 1, E_{\tilde{u}_\parallel}}d\tilde{u}_\perp\, \frac{1}{|\tilde{u}_\perp|^{\frac{N-1}{N+1}(1-\kappa)}} + \int_{|\tilde{u}_\perp|\geq 1, E_{\tilde{u}_\parallel}}d\tilde{u}_\perp\, e^{-\frac{\alpha\beta}{2^{5-\beta}}|v+\tilde{u}_\parallel+\tilde{u}_{\perp}|^\beta}\right)\\
            &\leq C_{\alpha, \beta, N, \delta}e^{-\frac{\alpha\beta}{32}|v|^\beta}.
        \end{split}
    \end{equation}
    In the last step, we used 
    \begin{align}\label{5_5:indexes}
        \begin{split}
            &0\leq (N-2) + \frac{2}{N+1}(1-\kappa)< N,\\
            &2\delta + \left((N-2) + \frac{2}{N+1}(1-\kappa)\right)>N,\\
            &0\leq \frac{N-1}{N+1}(1-\kappa)< N-1,\\
            &\sup_{a\in\mathbb{R}^N}\int_E d\tilde{u}_\perp\, e^{-\frac{\alpha\beta}{2^{5-\beta}}|\tilde{u}_{\perp} + a|^\beta}\leq \int_E d\tilde{u}_\perp\, e^{-\frac{\alpha\beta}{2^{5-\beta}}|\tilde{u}_{\perp}|^\beta}<\infty,
        \end{split}
    \end{align}
    where $E$ denotes any $N-1$ dimensional plane through $0$, and Lemma \ref{lem:int_bound3}. Therefore, $I_{1,1}$ is bounded with decay $e^{-\frac{\alpha\beta}{32}|v|^\beta}$.\\
    
    Now, we consider the $I_{1,2}$ part. We need to bound
    \begin{align*}
        I_{1,2} &= \sum_{n=2}^\infty \frac{e^{-\frac{\alpha \beta/4}{2^{n+1}}|v|^\beta}}{|v|^{1-\kappa}} \int_{|\tilde{u}_\parallel|\leq 2|v|}d\tilde{u}_\parallel\,\int_{E_{\tilde{u}_\parallel}}d\tilde{u}_\perp\,\frac{1}{|\tilde{u}_\parallel|^{N-2}}
        \frac{1}{(1+|v+\tilde{u}_{\parallel}|^2)^{\delta}}\mathbf{1}_{2^{-n-1}|v|^2\leq |v+\tilde{u}_\parallel+\tilde{u}_\perp|^2\leq 2^{-n}|v|^2}.
    \end{align*}

    The set
    \begin{align*}
        \{\tilde{u}_\perp:{2^{-n-1}|v|^2\leq |\tilde{u}_\perp + (\tilde{u}_\parallel + v)|^2\leq 2^{-n}|v|^2}\}
    \end{align*}
    is difference between two concentric balls at $\tilde{u}_\parallel + v$ with radius $2^{-n/2}|v|$ and $2^{-(n+1)/2}|v|$. It implies that the integration of $\tilde{u}_\perp$ is the volume of intersection between a plane and the difference between two concentric balls. Therefore,
    \begin{align}\label{int_perp}
        \begin{split}
            &\int_{E_{\tilde{u}_\parallel}}d\tilde{u}_\perp\,\mathbf{1}_{2^{-n-1}|v|^2\leq |v+\tilde{u}_\parallel+\tilde{u}_\perp|^2\leq 2^{-n}|v|^2}\\
            &=\left|\{\tilde{u}_\perp:{2^{-n-1}|v|^2\leq |\tilde{u}_\perp + (\tilde{u}_\parallel + v)|^2\leq 2^{-n}|v|^2}\}\cap E_{\tilde{u}_\parallel}\right|\\
            &\leq C_N 2^{-(N-1)n/2}|v|^{N-1}
        \end{split}
    \end{align}
    for some constant $C_N$.

    Applying Lemma \ref{lem:int_bound3} and \eqref{int_perp} with $2\delta>N$, we get
    \begin{align}\label{5_5:I_1_2_int}
        \begin{split}
            &\sum_{n=2}^\infty \frac{e^{-\frac{\alpha \beta/4}{2^{n+1}}|v|^\beta}}{|v|^{1-\kappa}}\int_{|\tilde{u}_\parallel|\leq 2|v|}d\tilde{u}_\parallel\, \frac{1}{|\tilde{u}_\parallel|^{N-2}\left(1+|v+\tilde{u}_\parallel|^2\right)^{\delta}} \int_{E_{\tilde{u}_\parallel}}d\tilde{u}_\perp\,\mathbf{1}_{2^{-n-1}|v|^2\leq |v+\tilde{u}_\parallel+\tilde{u}_\perp|^2\leq 2^{-n}|v|^2}\\
            &\leq C_N	 
            \sum_{n=2}^\infty \frac{e^{-\frac{\alpha \beta/4}{2^{n+1}}|v|^\beta}}{|v|^{1-\kappa}}\int_{|\tilde{u}_\parallel|\leq 2|v|}d\tilde{u}_\parallel\,\frac{1}{|\tilde{u}_\parallel|^{N-2}\left(1+|v+\tilde{u}_\parallel|^2\right)^{\delta}} 2^{-(N-1)n/2}|v|^{N-1}\\
            &\leq C_{N, \delta} 
            \sum_{n=2}^\infty e^{-\frac{\alpha \beta/4}{2^{n+1}}|v|^\beta} 2^{-(N-1)n/2} |v|^{N-2+\kappa}\min\{|v|^2, (1+|v|)^{-(N-2)}\}\\
            &\leq C_{N, \delta, \kappa} 
            \sum_{n=2}^\infty e^{-\frac{\alpha \beta/4}{2^{n+1}}|v|^\beta} 2^{-(N-1)n/2}\min\{|v|^{N+\kappa}, (1+|v|)^\kappa\}.
        \end{split}
    \end{align}

    For $|v| \leq 1$, the summations are obviously well-bounded, so we focus on $|v| \geq 1$. For given $|v|\geq 1$, choose an integer $M\in \mathbb{N}\cup \{0\}$ such that $2^M\leq |v|< 2^{M+1}$.
    
    Using the fact that $x^N e^{-x}\leq C_N$ for $x\geq 0$ for some constant $C_N$, the summation $\sum_{n=0}^{\lfloor \beta M\rfloor}$ is bounded as follows:
    \begin{align*}
        \begin{split}
            \sum_{n=0}^{\lfloor\beta M\rfloor} e^{-\frac{\alpha \beta/4}{2^{n+1}}|v|^\beta} 2^{-(N-1)n/2} &\leq C_N\sum_{n=0}^{\lfloor\beta M\rfloor} \left(\frac{2^n}{(\alpha \beta/8)|v|^\beta}\right)^{1/2}\left(\frac{2^n}{(\alpha \beta/8)|v|^\beta}\right)^{(N-1)/2} 2^{-(N-1)n/2}\\
            &\leq C_{\alpha,\beta, N}\sum_{n=0}^{\lfloor\beta M\rfloor} 2^{n/2 - \beta M/2}|v|^{-\beta (N-1)/2}\\
            &\leq C_{\alpha,\beta, N}|v|^{-(N-1)\beta/2}.
        \end{split}
    \end{align*}

    The remaining summation is bounded as follows:
    \begin{align*}
        \begin{split}
            \sum_{n=\lfloor\beta M\rfloor+1}^\infty e^{-\frac{\alpha \beta/4}{2^{n+1}}|v|^\beta} 2^{-(N-1)n/2}&\leq \sum_{n=\lfloor\beta M\rfloor+1}^\infty 2^{-(N-1)n/2}\leq 2^{-\frac{N-1}{2}(\lfloor\beta M\rfloor+1)}\frac{1}{1 - 2^{-\frac{N-1}{2}}}\\
            &\leq C_N 2^{-\frac{N-1}{2}\beta M}\leq C_N |v|^{-(N-1)\beta/2}.
        \end{split}
    \end{align*}
    
    By the choice of $\beta$ in \eqref{index:weight}, we have $\kappa - \frac{N-1}{2}\beta < 0$. Therefore, we have
    \begin{align*}
         \sum_{n=2}^\infty e^{-\frac{\alpha \beta/4}{2^{n+1}}|v|^\beta} 2^{-(N-1)n/2}(1+|v|)^\kappa\leq C_{\alpha,\beta, N}(1+|v|)^{\kappa-\frac{N-1}{2}\beta}
    \end{align*}
    for $|v|\geq 1$.
    
    Finally, we get an estimate for $I_{1,2}$:
    \begin{align}\label{5_5:I_12}
        I_{1,2} \leq C_{\alpha,\beta, N}(1+|v|)^{\kappa-\frac{N-1}{2}\beta}.
    \end{align}
    
    We move on to $I_{2,1}$. We split $I_{2,1}$ by
    \begin{align*}
        \begin{split}
            I_{2,1} &=\int_{\mathbb{R}^N}d\tilde{u}_\parallel\,\int_{E_{\tilde{u}_\parallel}}d\tilde{u}_\perp\,\frac{\left(|\tilde{u}_\parallel|^2+|\tilde{u}_\perp|^2\right)^{\frac{\kappa-1}{2}}}{|\tilde{u}_\parallel|^{N-2}}
            \frac{1}{(1+|v+\tilde{u}_{\perp}|^2)^{\delta}} e^{-\frac{\alpha \beta/4}{2^{8-4\beta}}|v+\tilde{u}_\parallel+\tilde{u}_\perp|^\beta}\mathbf{1}_{|v+\tilde{u}_\parallel+\tilde{u}_\perp|\geq |v|/16}\\
            &\leq \int_{\mathbb{R}^N}d\tilde{u}_\parallel\,\int_{E_{\tilde{u}_\parallel}}d\tilde{u}_\perp\, \mathbf{1}_{|\tilde{u}_\parallel|\leq 2|v+\tilde{u}_\parallel+\tilde{u}_\perp|}\frac{1}{|\tilde{u}_\parallel|^{(N-2) + \frac{2}{N+1}(1-\kappa)}}
            \frac{1}{(1+|(v+\tilde{u}_\parallel + \tilde{u}_\perp) - \tilde{u}_\parallel|^2)^{\delta}}\\
            &\qquad \times \frac{1}{|\tilde{u}_\perp|^{\frac{N-1}{N+1}(1-\kappa)}} e^{-\frac{\alpha \beta/4}{2^{8-4\beta}}|v+\tilde{u}_\parallel+\tilde{u}_\perp|^\beta}\mathbf{1}_{|v+\tilde{u}_\parallel+\tilde{u}_\perp|\geq |v|/16}\\
            &\quad + \int_{\mathbb{R}^N}d\tilde{u}_\parallel\,\int_{E_{\tilde{u}_\parallel}}d\tilde{u}_\perp\, \mathbf{1}_{|\tilde{u}_\parallel|\geq 2|v+\tilde{u}_\parallel+\tilde{u}_\perp|}\frac{1}{|\tilde{u}_\parallel|^{(N-2) + \frac{2}{N+1}(1-\kappa)}}
            \frac{1}{(1+|(v+\tilde{u}_\parallel + \tilde{u}_\perp) - \tilde{u}_\parallel|^2)^{\delta}}\\
            &\qquad \times \frac{1}{|\tilde{u}_\perp|^{\frac{N-1}{N+1}(1-\kappa)}} e^{-\frac{\alpha \beta/4}{2^{8-4\beta}}|v+\tilde{u}_\parallel+\tilde{u}_\perp|^\beta}\mathbf{1}_{|v+\tilde{u}_\parallel+\tilde{u}_\perp|\geq |v|/16}\\
            &\eqqcolon I_{2,1,1} + I_{2,1,2}.
        \end{split}
    \end{align*}
    In the middle, we used \eqref{5_5:kappa_split} and divided into two cases $|\tilde{u}_\parallel|\leq 2|u|$ and $|\tilde{u}_\parallel|\geq 2|u|$. To use the exponential decay, we split the exponential function by
    \begin{align}
        e^{-\frac{\alpha \beta/4}{2^{8-4\beta}}|u|^\beta} \mathbf{1}_{|\tilde{u}_\parallel|\leq 2|u|}\mathbf{1}_{|u|\geq |v|/16}&\leq e^{-\frac{1}{4}\frac{\alpha \beta/4}{2^{8-4\beta}}16^{-\beta}|v|^\beta}e^{-\frac{1}{2}\frac{\alpha \beta/4}{2^{8-4\beta}}(|\tilde{u}_\parallel|/2)^\beta}e^{-\frac{1}{4}\frac{\alpha \beta/4}{2^{8-4\beta}}|u|^\beta},\label{5_5:I_211}\\
        e^{-\frac{\alpha \beta/4}{2^{8-4\beta}}|u|^\beta} \mathbf{1}_{|u|\geq |v|/16}&\leq e^{-\frac{1}{2}\frac{\alpha \beta/4}{2^{8-4\beta}}16^{-\beta}|v|^\beta}e^{-\frac{1}{2}\frac{\alpha \beta/4}{2^{8-4\beta}}|u|^\beta}\label{5_5:I_212}
    \end{align}
    with $u = v+\tilde{u}_\parallel+\tilde{u}_\perp$.

    From \eqref{5_5:I_211}, we have
    \begin{align*}
        I_{2,1,1}&\leq e^{-\frac{1}{4}\frac{\alpha \beta/4}{2^8}|v|^\beta}\int_{\mathbb{R}^N}d\tilde{u}_\parallel\,\frac{e^{-\frac{1}{2}\frac{\alpha \beta/4}{2^{8-4\beta}}2^{-\beta}|\tilde{u}_\parallel|^\beta}}{|\tilde{u}_\parallel|^{(N-2) + \frac{2}{N+1}(1-\kappa)}}
        \int_{E_{\tilde{u}_\parallel}}d\tilde{u}_\perp\,\frac{e^{-\frac{1}{4}\frac{\alpha \beta/4}{2^{8-4\beta}}|v+\tilde{u}_\parallel+\tilde{u}_\perp|^\beta}}{|\tilde{u}_\perp|^{\frac{N-1}{N+1}(1-\kappa)}}\\
        &\leq e^{-\frac{\alpha \beta}{2^{12}}|v|^\beta}\int_{\mathbb{R}^N}d\tilde{u}_\parallel\,\frac{e^{-\frac{\alpha \beta}{2^{11-3\beta}}|\tilde{u}_\parallel|^\beta}}{|\tilde{u}_\parallel|^{(N-2) + \frac{2}{N+1}(1-\kappa)}}\\
        &\qquad \times \left(\int_{|\tilde{u}_\perp|\leq 1, E_{\tilde{u}_\parallel}}d\tilde{u}_\perp\,\frac{1}{|\tilde{u}_\perp|^{\frac{N-1}{N+1}(1-\kappa)}} + \int_{|\tilde{u}_\perp|\geq 1, E_{\tilde{u}_\parallel}}d\tilde{u}_\perp\,e^{-\frac{\alpha \beta}{2^{12-4\beta}}|v+\tilde{u}_\parallel+\tilde{u}_\perp|^\beta}\right)\\
        &\leq C_{\alpha,\beta, N, \delta} e^{-\frac{\alpha \beta}{2^{12}}|v|^\beta}.
    \end{align*}
    In the last step, we used \eqref{5_5:indexes} and Lemma \ref{lem:int_bound3} to bound the integrals by constants.
    
    By a similar computation, from \eqref{5_5:I_212}, we get
    \begin{align*}
            I_{2,1,2}&\leq e^{-\frac{1}{2}\frac{\alpha \beta/4}{2^8}|v|^\beta}\int_{\mathbb{R}^N}d\tilde{u}_\parallel\,\frac{1}{|\tilde{u}_\parallel|^{(N-2) + \frac{2}{N+1}(1-\kappa)}}\frac{1}{(1+|\tilde{u}_{\parallel}|^2)^{\delta}} 
            \int_{E_{\tilde{u}_\parallel}}d\tilde{u}_\perp\,\frac{e^{-\frac{1}{2}\frac{\alpha \beta/4}{2^{8-4\beta}}|v+\tilde{u}_\parallel+\tilde{u}_\perp|^\beta}}{|\tilde{u}_\perp|^{\frac{N-1}{N+1}(1-\kappa)}}\\
            &\leq e^{-\frac{\alpha \beta}{2^{11}}|v|^\beta}\int_{\mathbb{R}^N}d\tilde{u}_\parallel\,\frac{1}{|\tilde{u}_\parallel|^{(N-2) + \frac{2}{N+1}(1-\kappa)}}\frac{1}{(1+|\tilde{u}_{\parallel}|^2)^{\delta}}\\
            &\qquad \times \left(\int_{|\tilde{u}_\perp|\leq 1, E_{\tilde{u}_\parallel}}d\tilde{u}_\perp\,\frac{1}{|\tilde{u}_\perp|^{\frac{N-1}{N+1}(1-\kappa)}} + \int_{|\tilde{u}_\perp|\geq 1, E_{\tilde{u}_\parallel}}d\tilde{u}_\perp\,e^{-\frac{\alpha \beta}{2^{11-4\beta}}|v+\tilde{u}_\parallel+\tilde{u}_\perp|^\beta}\right)\\
            &\leq C_{\alpha,\beta, N, \delta} e^{-\frac{\alpha \beta}{2^{11}}|v|^\beta}.
    \end{align*}
    It shows
    \begin{align}\label{5_5:I_21}
        I_{2,1}\leq C_{\alpha,\beta, N, \delta} e^{-\frac{\alpha \beta}{2^{12}}|v|^\beta}.
    \end{align}

    Now, we need to deal with the remaining $I_{2,2}$. We recall the definition of $I_{2,2}$.
    \begin{align}\label{5_5:I_22_0}
        \begin{split}
            I_{2,2}&=\sum_{n=8}^\infty \frac{e^{-\frac{\alpha \beta/4}{2^{n+1}}|v|^\beta}}{|v|^{1-\kappa}}\int_{|\tilde{u}_\parallel|\leq 2|v|}d\tilde{u}_\parallel\,\int_{E_{\tilde{u}_\parallel}}d\tilde{u}_\perp\,\frac{1}{|\tilde{u}_\parallel|^{N-2}}\frac{1}{(1+|v+\tilde{u}_{\perp}|^2)^{\delta}}\mathbf{1}_{2^{-n-1}|v|^2\leq |v+\tilde{u}_\parallel+\tilde{u}_\perp|^2\leq 2^{-n}|v|^2}.
        \end{split}
    \end{align}    
    To bound $I_{2,2}$, we split into three cases: $|\tilde{u}_\parallel|>|v|/4$, $2^{-n/2+2}|v|<|\tilde{u}_\parallel|\leq |v|/4$, and $|\tilde{u}_\parallel|\leq 2^{-n/2+2}|v|$.\\

    \noindent \textit{(i) The case $|\tilde{u}_\parallel|> |v|/4$.}
    
    Since $|\tilde{u}_\parallel|>|v|/4$, by $|v+\tilde{u}_\perp+\tilde{u}_\parallel|\leq 2^{-n/2}|v|\leq |v|/16$ for $n\geq 8$, we have $|v+\tilde{u}_\perp|\geq |v|/8$. Using these upper bounds and \eqref{int_perp}, we get
    \begin{align}\label{5_5:I_22_1}
        \begin{split}
            &\sum_{n=8}^\infty \frac{e^{-\frac{\alpha \beta/4}{2^{n+1}}|v|^\beta}}{|v|^{1-\kappa}} \int_{|v|/4<|\tilde{u}_\parallel|\leq 2|v|}d\tilde{u}_\parallel\,\int_{E_{\tilde{u}_\parallel}}d\tilde{u}_\perp\,\frac{1}{|\tilde{u}_\parallel|^{N-2}}\frac{1}{(1+|v+\tilde{u}_{\perp}|^2)^{\delta}}
            \mathbf{1}_{2^{-n-1}|v|^2\leq |v+\tilde{u}_\perp+\tilde{u}_\parallel|^2\leq 2^{-n}|v|^2}\\
            &\leq C_{N, \delta}\sum_{n=8}^\infty \frac{e^{-\frac{\alpha \beta/4}{2^{n+1}}|v|^\beta}}{|v|^{N-1-\kappa}}\frac{1}{(1+|v|^2)^{\delta}}\int_{|v|/4<|\tilde{u}_\parallel|\leq 2|v|}d\tilde{u}_\parallel\,\int_{E_{\tilde{u}_\parallel}}d\tilde{u}_\perp\,
            \mathbf{1}_{2^{-n-1}|v|^2\leq |v+\tilde{u}_\perp+\tilde{u}_\parallel|^2\leq 2^{-n}|v|^2}\\
            &\leq C_{N, \delta}\sum_{n=8}^\infty e^{-\frac{\alpha \beta/4}{2^{n+1}}|v|^\beta}\frac{2^{-n(N-1)/2}|v|^\kappa}{(1+|v|^2)^{\delta}}\int_{|v|/4\leq |\tilde{u}_\parallel|\leq 2|v|}d\tilde{u}_\parallel\,\\
            & = C_{N, \delta}\frac{|v|^{N+\kappa}}{(1+|v|^2)^{\delta}}\sum_{n=8}^\infty e^{-\frac{\alpha \beta/4}{2^{n+1}}|v|^\beta} 2^{-n(N-1)/2}.
        \end{split}
    \end{align}
    Note that the summation has the same form as \eqref{5_5:I_1_2_int}. Since $2\delta>N$, by \eqref{5_5:I_12}, it is bounded by
    \begin{align*}
        C_{N, \delta}\frac{|v|^{N+\kappa}}{(1+|v|^2)^{\delta}}\sum_{n=8}^\infty e^{-\frac{\alpha \beta/4}{2^{n+1}}|v|^\beta} 2^{-n(N-1)/2}\leq C_{\alpha,\beta, N, \delta}(1+|v|)^{\kappa-\frac{N-1}{2}\beta}.
    \end{align*}
    \\
    
    \noindent \textit{(ii) The case $2^{-n/2+2}|v|< |\tilde{u}_\parallel|\leq |v|/4$.}
    
    On this set, we can apply Lemma \ref{lem:cone}. For $n\geq 8$, replacing $x = \tilde{u}_\parallel$, $y = \tilde{u}_\perp$, $a = v$, and $R = 2^{-n/2}|v|$, we can easily check that these variables satisfy the condition of Lemma \ref{lem:cone}. Therefore, there exists a set $C_{v, 2^{-n/2}|v|}$ satisfying the property in Lemma \ref{lem:cone}.
    
    As $|\tilde{u}_\parallel|>2^{-n/2+2}|v|\geq 4|u|$ as $2^{-n/2}|v|\geq |u|$, using $u = \tilde{u}_\parallel+\tilde{u}_\perp + v$, we have
    \begin{align*}
        |v+\tilde{u}_\perp| = |u-\tilde{u}_\parallel|\geq |\tilde{u}_\parallel| - |u|\geq \frac{3}{4}|\tilde{u}_\parallel|.
    \end{align*}
    Therefore, applying \eqref{int_perp}, we have
    \begin{align*}
        \begin{split}
            &\mathbf{1}_{2^{-n/2+2}|v|< |\tilde{u}_\parallel|\leq |v|/4}\int_{E_{\tilde{u}_\parallel}}d\tilde{u}_\perp\,\frac{1}{(1+|v+\tilde{u}_{\perp}|^2)^{\delta}}
            \mathbf{1}_{2^{-n-1}|v|^2\leq |v+\tilde{u}_\perp+\tilde{u}_\parallel|^2\leq 2^{-n}|v|^2}\\
            &\leq C_\delta\frac{1}{(1+|\tilde{u}_\parallel|^2)^{\delta}}\int_{E_{\tilde{u}_\parallel}}d\tilde{u}_\perp\,
            \mathbf{1}_{2^{-n-1}|v|^2\leq |v+\tilde{u}_\perp+\tilde{u}_\parallel|^2\leq 2^{-n}|v|^2}\\
            &\leq C_{N, \delta}\frac{1}{(1+|\tilde{u}_\parallel|^2)^{\delta}}(2^{-\frac{n}{2}}|v|)^{N-1}.
        \end{split}
    \end{align*}
    
    Applying these bounds to \eqref{5_5:I_22_0}, we get
    \begin{align}\label{5_5:I_22_2}
        \begin{split}
            &\frac{C_{N, \delta}}{|v|^{1-\kappa}}e^{-\frac{\alpha \beta/4}{2^{n+1}}|v|^\beta}\int_{2^{-n/2+2}|v|< |\tilde{u}_\parallel|\leq |v|/4}d\tilde{u}_\parallel\,\frac{1}{|\tilde{u}_\parallel|^{N-2}}\int_{E_{\tilde{u}_\parallel}}d\tilde{u}_\perp\,\frac{1}{(1+|v+\tilde{u}_{\perp}|^2)^{\delta}}
            \mathbf{1}_{2^{-n-1}|v|^2\leq |v+\tilde{u}_\perp+\tilde{u}_\parallel|^2\leq 2^{-n}|v|^2}\\
            &\leq \frac{C_{N, \delta}}{|v|^{1-\kappa}}e^{-\frac{\alpha \beta/4}{2^{n+1}}|v|^\beta}\int_{\{2^{-n/2+2}|v|\leq |\tilde{u}_\parallel|\leq |v|/4\}\cap C_{v, 2^{-n/2}|v|}}d\tilde{u}_\parallel\,\frac{1}{|\tilde{u}_\parallel|^{N-2}}\frac{1}{(1+|\tilde{u}_\parallel|^2)^{\delta}}
            (2^{-\frac{n}{2}}|v|)^{N-1}\\
            &\leq \frac{C_{N, \delta}}{|v|^{2-\kappa}}e^{-\frac{\alpha \beta/4}{2^{n+1}}|v|^\beta}(2^{-\frac{n}{2}}|v|)^{N}\int_{2^{-n/2+2}|v|}^{|v|/4}dr\,\frac{r}{(1+r^2)^{\delta}}.
        \end{split}
    \end{align}
    The final $\frac{2^{-\frac{n}{2}}|v|}{|v|}$ decay is from $C_{v, 2^{-n/2}|v|}$; see Lemma \ref{lem:cone}.
    
    Now, let us assume $2^M\leq |v|\leq 2^{M+1}$ for $M\in\mathbb{N}\cup\{0\}$. Indeed, if $|v|\leq 1$, then
    \begin{align}\label{5_5:I_22_3}
        \begin{split}
            &\sum_{n=8}^\infty \frac{1}{|v|^{2-\kappa}}(2^{-\frac{n}{2}}|v|)^{N}\int_{2^{-n/2+2}|v|}^{|v|/4}dr\,\frac{r}{(1+r^2)^{\delta}}\\
            &\leq\sum_{n=8}^\infty C_\delta |v|^{N-2+\kappa} (2^{-\frac{n}{2}})^{N}\int_0^{|v|/4}dr\,r\leq \sum_{n=8}^\infty C_\delta |v|^{N+\kappa} (2^{-\frac{n}{2}})^{N}<\infty.
        \end{split}
    \end{align}

    We divide the summation by $\sum_{n=8}^{\lfloor \beta M\rfloor}$ and $\sum_{n=\lfloor \beta M\rfloor + 1}^{2M}$, and $\sum_{n=2M+1}^\infty$. Let us start from $\sum_{n=8}^{\lfloor \beta M\rfloor}$ so that $2^{-n/2}|v|\geq 1$. In this case, we bound
    \begin{align*}
        \int_{2^{-n/2+2}|v|}^{|v|/4}dr\,\frac{r}{(1+r^2)^{\delta}}\leq \int_{2^{-n/2+2}|v|}^{|v|/4}dr\,r^{1-2\delta}\leq C_\delta(2^{-n/2}|v|)^{2-2\delta}
    \end{align*}
    as $\delta>1$. Thus, we get
    \begin{align*}
        &\sum_{n=8}^{\lfloor \beta M\rfloor}\frac{e^{-\frac{\alpha \beta/4}{2^{n+1}}|v|^\beta}}{|v|^{2-\kappa}}(2^{-\frac{n}{2}}|v|)^{N}\int_{2^{-n/2+2}|v|}^{|v|/4}dr\,\frac{r}{(1+r^2)^{\delta}}\\
        &\leq C_\delta |v|^{N+\kappa-2\delta}\sum_{n=8}^{\lfloor \beta M\rfloor}e^{-\frac{\alpha \beta/4}{2^{n+1}}|v|^\beta} 2^{-\frac{N+2-2\delta}{2}n}\\
        &\leq C_{N,\delta} |v|^{N+\kappa-2\delta}\sum_{n=8}^{\lfloor \beta M\rfloor}\left(\frac{2^{n+1}}{(\alpha\beta/4)|v|^\beta}\right)\left(\frac{2^{n+1}}{(\alpha\beta/4)|v|^\beta}\right)^{\frac{N+2}{2}} 2^{-\frac{N+2-2\delta}{2}n}\\
        &\leq C_{\alpha,\beta, N,\delta} |v|^{N+\kappa-2\delta-\frac{N+4}{2}\beta}\sum_{n=8}^{\lfloor \beta M\rfloor} 2^{(\delta+1)n}.
    \end{align*}
    In the middle, we used $x^{\frac{N+4}{2}}e^{-x}\leq C_N$ for $x>0$.
    
    Since $\delta> 1$,
    \begin{align*}
        |v|^{N+\kappa-2\delta-\frac{N+4}{2}\beta}\sum_{n=8}^{\lfloor \beta M\rfloor} 2^{(\delta+1)n}&\leq |v|^{N+\kappa-2\delta-\frac{N+4}{2}\beta}2^{8(\delta+1)}\frac{2^{(\lfloor \beta M\rfloor-7)(\delta+1)} - 1}{2^{\delta+1}-1}\\
        &\leq C_{\delta} |v|^{N+\kappa-2\delta-\frac{N+4}{2}\beta}2^{\beta M(\delta+1)}\\
        &\leq C_{\delta} |v|^{N+\kappa-2\delta-\frac{N+4}{2}\beta}|v|^{\beta(\delta+1)}\\
        &=C_{\delta}|v|^{(2-\beta)(\frac{N}{2}-\delta) + \kappa-\beta}.
    \end{align*}
    Using the condition $2\delta>N$, we finally get
    \begin{align}\label{5_5:I_22_4}
        \sum_{n=8}^{\lfloor \beta M\rfloor}\frac{e^{-\frac{\alpha \beta/4}{2^{n+1}}|v|^\beta}}{|v|^{2-\kappa}}(2^{-\frac{n}{2}}|v|)^{N}\int_{2^{-n/2+2}|v|}^{|v|/4}dr\,\frac{r}{(1+r^2)^{\delta}}\leq \frac{C_{\alpha,\beta, N,\delta}}{|v|^{\beta-\kappa}}
    \end{align}
    for some constant $C_{N,\delta}$.

    Next, we move on to $\sum_{n=\lfloor \beta M\rfloor+1}^{2M}$. In this case, we can ignore the exponential term. Also, still we have $2^{-n/2}|v|\geq 1$. Therefore,
    \begin{align*}
        \sum_{n=\lfloor \beta M\rfloor+1}^{2M}\frac{e^{-\frac{\alpha \beta/4}{2^{n+1}}|v|^\beta}}{|v|^{2-\kappa}}(2^{-\frac{n}{2}}|v|)^{N}\int_{2^{-n/2+2}|v|}^{|v|/4}dr\,\frac{r}{(1+r^2)^{\delta}}&\leq C_\delta \sum_{n=\lfloor \beta M\rfloor+1}^{2M}\frac{1}{|v|^{2-\kappa}}(2^{-\frac{n}{2}}|v|)^{N}(2^{-n/2}|v|)^{2-2\delta}\\
        &=C_\delta |v|^{N+\kappa -2\delta}\sum_{n=\lfloor \beta M\rfloor+1}^{2M}2^{-(\frac{N}{2}+1-\delta)n}.
    \end{align*}
    When $\frac{N}{2}+1-\delta>0$, then the summation is given by $C_{N,\delta} |v|^{-\beta(\frac{N}{2}+1-\delta)}$, and if $\frac{N}{2}+1-\delta< 0$, then the summation is given by $C_{N,\delta}|v|^{-(N+2-2\delta)}$. Summing up the two cases, we have
    \begin{align}\label{5_5:I_22_4_1}
        \sum_{n=\lfloor \beta M\rfloor+1}^{2M}\frac{e^{-\frac{\alpha \beta/4}{2^{n+1}}|v|^\beta}}{|v|^{2-\kappa}}(2^{-\frac{n}{2}}|v|)^{N}\int_{2^{-n/2+2}|v|}^{|v|/4}dr\,\frac{r}{(1+r^2)^{\delta}}\leq C_{N,\delta}|v|^{(2-\beta)(\frac{N}{2}-\delta) +\kappa - \beta}
    \end{align}
    when $\frac{N}{2}+1-\delta\neq 0$. As $2\delta>N$, it is bounded by $|v|^{\kappa-\beta}$ order.
    
    When $\frac{N}{2}+1-\delta = 0$, then the summation is bounded by $2M - \lfloor \beta M\rfloor\leq C(2-\beta)\ln(|v|+1)\leq 2C |v|^{2-\beta}$ for $|v|\geq 1$; we used $\epsilon \ln(x+1)\leq 2 x^\epsilon$ for $x\geq 1$. Therefore,
    \begin{align}\label{5_5:I_22_4_2}
        \begin{split}
            \sum_{n=\lfloor \beta M\rfloor+1}^{2M}\frac{e^{-\frac{\alpha \beta/4}{2^{n+1}}|v|^\beta}}{|v|^{2-\kappa}}(2^{-\frac{n}{2}}|v|)^{N}\int_{2^{-n/2+2}|v|}^{|v|/4}dr\,\frac{r}{(1+r^2)^{\delta}}&\leq C_{N,\delta}|v|^{\kappa-2} (2-\beta)\ln(|v|+1)\\
            &\leq C_{N,\delta}|v|^{\kappa-\beta}.
        \end{split}
    \end{align}
    It concludes that the summation is bounded by $C_{N,\delta}|v|^{\kappa-\beta}$ when $\delta>\frac{N}{2}$.

    Now, we consider the remaining $\sum_{n=2M+1}^\infty$ summation. We bound it by
    \begin{align}\label{5_5:I_22_5}
        \begin{split}
            \sum_{n=2M+1}^\infty \frac{1}{|v|^{2-\kappa}}(2^{-\frac{n}{2}}|v|)^{N}\int_{2^{-n/2+2}|v|}^{|v|/4}dr\,\frac{r}{(1+r^2)^{\delta}}\leq \sum_{n=2M+1}^\infty \frac{1}{|v|^{2-\kappa}}(2^{-\frac{n}{2} + M + 1})^{N}\int_0^\infty dr\,\frac{r}{(1+r^2)^{\delta}}.
        \end{split}
    \end{align}
    Since $\delta>1$, the final integral and summation are bounded, so we have $|v|^{-2+\kappa}$ decay.
    
    Combining \eqref{5_5:I_22_3}, \eqref{5_5:I_22_4}, \eqref{5_5:I_22_4_1}, \eqref{5_5:I_22_4_2}, and \eqref{5_5:I_22_5}, we have
    \begin{align}\label{5_5:I_22_6}
        \begin{split}
            &\sum_{n=8}^\infty \frac{e^{-\frac{\alpha \beta/4}{2^{n+1}}|v|^\beta}}{|v|^{1-\kappa}}\int_{2^{-n/2+2}|v|< |\tilde{u}_\parallel|\leq |v|/4}d\tilde{u}_\parallel\,\frac{1}{|\tilde{u}_\parallel|^{N-2}}\int_{E_{\tilde{u}_\parallel}}d\tilde{u}_\perp\,\frac{1}{(1+|v+\tilde{u}_{\perp}|^2)^{\delta}}
            \mathbf{1}_{2^{-n-1}|v|^2\leq |v+\tilde{u}_\perp+\tilde{u}_\parallel|^2\leq 2^{-n}|v|^2}\\
            &\leq \frac{C_{\alpha,\beta, N,\delta}}{(1+|v|)^{\beta-\kappa}}
        \end{split}
    \end{align}
    for some constant $C_{\alpha,\beta, N,\delta}$.\\

    \noindent \textit{(iii) The case $|\tilde{u}_\parallel|\leq 2^{-n/2+2}|v|$.}
    
    When $|\tilde{u}_\parallel|\leq 2^{-n/2+2}|v|$, $|v+\tilde{u}_\perp+\tilde{u}_\parallel|\leq 2^{-n/2}|v|$ implies $|v+\tilde{u}_\perp|\leq 2^{-\frac{n-5}{2}}|v|$. Therefore, we have
    \begin{align*}
        \begin{split}
            &\int_{E_{\tilde{u}_\parallel}}d\tilde{u}_\perp\,\frac{1}{(1+|v+\tilde{u}_{\perp}|^2)^{\delta}}
            \mathbf{1}_{2^{-n-1}|v|^2\leq |v+\tilde{u}_\perp+\tilde{u}_\parallel|^2\leq 2^{-n}|v|^2}\\
            &\leq\int_{E_{\tilde{u}_\parallel}}d\tilde{u}_\perp\,\frac{1}{(1+|v+\tilde{u}_{\perp}|^2)^{\delta}}
            \mathbf{1}_{|v+\tilde{u}_\perp|^2\leq 2^{-n+5}|v|^2}\\
            &\leq C_N\times
            \begin{dcases}
                \int_0^{2^{-(n-5)/2}|v|}dr\,r^{N-2} & 2^{-n/2}|v|\leq 1\\
                \int_0^{2^{5/2}} dr\,r^{N-2} + \int_{2^{5/2}}^\infty dr\,r^{N-2 - 2\delta} & 2^{-n/2}|v|\geq 1
            \end{dcases}\\
            &\leq C_{N, \delta}\times
            \begin{cases}
                (2^{-n/2}|v|)^{N-1} & 2^{-n/2}|v|\leq 1,\\
                1 & 2^{-n/2}|v|\geq 1.
            \end{cases}
        \end{split}
    \end{align*}
    From the second to the third line, we used the fact that the integral is maximized when the integrating plane $E_{\tilde{u}_\parallel}$ passes through $v+\tilde{u}_\perp = 0$, so we can use the spherical coordinate about $v+\tilde{u}_\perp$ on $E_{\tilde{u}_\parallel}$. In fact, it is an application of the Hardy-Littlewood inequality. As $2\delta>N$, we bound the last integral by a constant for $2^{-n/2}|v|\geq 1$.
    
    Using this bound, we obtain
    \begin{align*}
        \begin{split}
            &\frac{e^{-\frac{\alpha \beta/4}{2^{n+1}}|v|^\beta}}{|v|^{1-\kappa}}\int_{|\tilde{u}_\parallel|\leq 2^{-n/2+2}|v|}d\tilde{u}_\parallel\,\int_{E_{\tilde{u}_\parallel}}d\tilde{u}_\perp\,\frac{1}{|\tilde{u}_\parallel|^{N-2}}\frac{1}{(1+|v+\tilde{u}_{\perp}|^2)^{\delta}}
            \mathbf{1}_{2^{-n-1}|v|^2\leq |v+\tilde{u}_\perp+\tilde{u}_\parallel|^2\leq 2^{-n}|v|^2}\\
            &\leq \frac{C_{N, \delta}}{|v|^{1-\kappa}}e^{-\frac{\alpha \beta/4}{2^{n+1}}|v|^\beta}\int_{\{|\tilde{u}_\parallel|\leq 2^{-n/2+2}|v|\}\cap C_{v, 2^{-n/2}|v|}}d\tilde{u}_\parallel\,\frac{1}{|\tilde{u}_\parallel|^{N-2}}\times \begin{cases}
                (2^{-n/2}|v|)^{N-1}) & 2^{-n/2}|v|\leq 1,\\
                1 & 2^{-n/2}|v|\geq 1.
            \end{cases}
        \end{split}
    \end{align*}
    In the final step, we applied Lemma \ref{lem:cone} for the exact same reason in the previous case. As
    \begin{align*}
        \int_{\{|\tilde{u}_\parallel|\leq 2^{-n/2+2}|v|\}\cap C_{v, 2^{-n/2}|v|}}d\tilde{u}_\parallel\,\frac{1}{|\tilde{u}_\parallel|^{N-2}}\leq C_N\frac{2^{-n/2}|v|}{|v|}\int_0^{2^{-n/2+2}|v|}dr\,r\leq \frac{C_N}{|v|}(2^{-n/2+2}|v|)^3,
    \end{align*}
    we finally obtain
    \begin{align}\label{5_5:I_22_7}
        \begin{split}
            &\frac{e^{-\frac{\alpha \beta/4}{2^{n+1}}|v|^\beta}}{|v|^{1-\kappa}}\int_{\{|\tilde{u}_\parallel|\leq 2^{-n/2+2}|v|\}\cap C_{v, 2^{-n/2}|v|}}d\tilde{u}_\parallel\,\int_{E_{\tilde{u}_\parallel}}d\tilde{u}_\perp\,\frac{1}{|\tilde{u}_\parallel|^{N-2}}\\
            &\qquad \times \frac{1}{(1+|v+\tilde{u}_{\perp}|^2)^{\delta}}
            \mathbf{1}_{2^{-n-1}|v|^2\leq |v+\tilde{u}_\perp+\tilde{u}_\parallel|^2\leq 2^{-n}|v|^2}\\
            &\leq \frac{C_{N, \delta} e^{-\frac{\alpha \beta/4}{2^{n+1}}|v|^\beta}}{|v|^{2-\kappa}}(2^{-n/2+2}|v|)^3\times
            \begin{cases}
            (2^{-n/2}|v|)^{N-1} & 2^{-n/2}|v|\leq 1,\\
            1 & 2^{-n/2}|v|\geq 1.
            \end{cases}
        \end{split}
    \end{align}
    Now, we assume $2^M\leq |v|\leq 2^{M+1}$ for $M\in \mathbb{N}\cup\{0\}$. Indeed, if $|v|\leq 1$,
    \begin{align*}
        \sum_{n=8}^\infty \frac{e^{-\frac{\alpha \beta/4}{2^{n+1}}|v|^\beta}}{|v|^{2-\kappa}}(2^{-n/2}|v|)^3(2^{-n/2}|v|)^{N -1}\leq C|v|^{N+\kappa}\sum_{n=8}^\infty (2^{-n/2})^{N+2}<\infty.
    \end{align*}
    
    Let us divide the summation by $\sum_{n=8}^{\lfloor\beta M\rfloor}$, $\sum_{n=\lfloor\beta M\rfloor + 1}^{2M+1}$, and $\sum_{2M+2}^\infty$. For the first case, since $2^{-n/2}|v|\geq 1$,
    \begin{align}\label{5_5:I_22_8}
        \begin{split}
            \sum_{n=8}^{\lfloor\beta M\rfloor} \frac{e^{-\frac{\alpha \beta/4}{2^{n+1}}|v|^\beta}}{|v|^{2-\kappa}}(2^{-n/2}|v|)^3&\leq C\sum_{n=8}^{\lfloor\beta M\rfloor}\left(\frac{2^{n+1}}{(\alpha\beta/4)|v|^\beta}\right)\left(\frac{2^{n+1}}{(\alpha\beta/4)|v|^\beta}\right)^{3/2}\frac{1}{|v|^{2-\kappa}}(2^{-n/2}|v|)^3\\
            &\leq C_{\alpha,\beta}\sum_{n=8}^{\lfloor\beta M\rfloor} 2^{n-\beta M}|v|^{\kappa + 1 -\frac{3}{2}\beta}\\
            &\leq C_{\alpha,\beta}|v|^{\kappa + 1 -\frac{3}{2}\beta},
        \end{split}
    \end{align}
    where we used $x^{5/2} e^{-x}\leq C$ for some constant $C$ for $x\geq 0$. Since $\kappa + 1 - \frac{3}{2}\beta<0$, it gives $|v|$ decay.
    
    If $\lfloor\beta M\rfloor + 1 \leq n\leq 2M+1$, then
    \begin{align*}
        &\sum_{n=\lfloor\beta M\rfloor + 1}^{2M+1} \frac{e^{-\frac{\alpha \beta/4}{2^{n+1}}|v|^\beta}}{|v|^{2-\kappa}}(2^{-n/2}|v|)^3\leq \sum_{n=\lfloor\beta M\rfloor + 1}^{2M + 1} |v|^{\kappa + 1}2^{-\frac{3}{2}n}\\
        &= |v|^{\kappa + 1} 2^{-\frac{3}{2}(\lfloor\beta M\rfloor + 1)}\frac{2^{-\frac{3}{2}(2M + 1 - \lfloor\beta M\rfloor)} - 1}{2^{-\frac{3}{2}} - 1}\leq C|v|^{\kappa + 1 - \frac{3}{2}\beta}.
    \end{align*}
    Since $\kappa + 1 - \frac{3}{2}\beta<0$, it again gives $|v|$ decay.

    If $n\geq 2M+2$, as $2^{-n/2}|v|\leq 1$ in \eqref{5_5:I_22_7}, we have
    \begin{align*}
        \sum_{n=2M+2}^\infty \frac{e^{-\frac{\alpha \beta/4}{2^{n+1}}|v|^\beta}}{|v|^{2-\kappa}}(2^{-n/2}|v|)^{N -1}(2^{-n/2+2}|v|)^3\leq C\frac{1}{|v|^{2-\kappa}}\sum_{n=2M+2}^\infty (2^{-n/2 + (M+1)})^{N + 2}\leq \frac{C_N}{|v|^{2-\kappa}}.
    \end{align*}
    Therefore,
    \begin{align}\label{5_5:I_22_9}
        \begin{split}
            &\sum_{n=8}^\infty e^{-\frac{\alpha \beta/4}{2^{n+1}}|v|^\beta} \int_{|\tilde{u}_\parallel|\leq 2^{-n/2+2}|v|}d\tilde{u}_\parallel\,\int_{E_{\tilde{u}_\parallel}}d\tilde{u}_\perp\,\frac{\left(|\tilde{u}_\parallel|^2+|\tilde{u}_\perp|^2\right)^{\frac{\kappa-1}{2}}}{|\tilde{u}_\parallel|^{N-2}}\\
            &\qquad \times \frac{1}{(1+|v+\tilde{u}_{\perp}|^2)^{\delta}}\mathbf{1}_{2^{-n-1}|v|^2\leq |v+\tilde{u}_\perp+\tilde{u}_\parallel|^2\leq 2^{-n}|v|^2}\\
            &\leq \left(\sum_{n=8}^{\lfloor\beta M\rfloor} + \sum_{\lceil\beta M\rceil}^{2M+1} + \sum_{n=2M+2}^\infty\right)\left(\cdots\right)\leq C_{\alpha,\beta,N,\delta}\left(\frac{1}{(1+|v|)^{\frac{3}{2}\beta - 1 - \kappa}} + \frac{1}{(1+|v|)^{2 - \kappa}}\right)
        \end{split}
    \end{align}
    for some constant $C_{\alpha,\beta,N,\delta}$.
    
    Combining \eqref{5_5:I_22_1}, \eqref{5_5:I_22_6}, and \eqref{5_5:I_22_9}, we get
    \begin{align}\label{5_5:I_22}
        I_{2,2}\leq C_{\alpha,\beta,N,\delta}\frac{1}{(1+|v|)^{\frac{3}{2}\beta - 1 - \kappa}}.
    \end{align}
    Note that $\frac{3}{2}\beta - 1 - \kappa\leq \beta-\kappa$ as $\beta\leq 2$. Therefore, the dominating decaying term is $\frac{1}{(1+|v|)^{\frac{3}{2}\beta - 1 - \kappa}}$.

    Finally, combining $I_{1,1}$ (\eqref{5_5:I_11}), $I_{1,2}$ (\eqref{5_5:I_12}), $I_{2,1}$ (\eqref{5_5:I_21}), and $I_{2,2}$ (\eqref{5_5:I_22}) bounds, we get
    \begin{align*}
        I \leq I_{1,1}+I_{1,2}+I_{2,1}+I_{2,2}\leq C_{\alpha,\beta,N,\delta, \kappa}\left(\frac{1}{(1+|v|)^{\frac{N-1}{2}\beta - \kappa}}+\frac{1}{(1+|v|)^{\frac{3}{2}\beta - 1 - \kappa}}\right).
    \end{align*}
\end{proof}

Let us recall the Boltzmann collision operator $Q$ in \eqref{eq:Boltzmann}. Since the angular part of collision kernel $b(\cos\theta)$ is locally integrable, we split $Q(f,f)$ into two parts: for later use, let us define it by bi-linear form
\begin{equation} \notag
    \begin{split}
        Q_{\text{gain}}(f,g)(t,x,v) &:= \iint_{\mathbb{R}^{N}\times  \mathbb{S}^{N-1}} B(v-u, \omega)f(t,x,u')g(t,x,v') d\omega du,\\
        Q_{\text{loss}}(f,g)(t,x,v) &:= \iint_{\mathbb{R}^{N}\times  \mathbb{S}^{N-1}} B(v-u, \omega)f(t,x,u)g(t,x,v) d\omega du.
    \end{split}
\end{equation}
Recall the weight $w_{\delta}(v) = w_{\alpha,\beta,\delta}(v)$ defined in \eqref{def_w}. We keep the choices of $N, \kappa, \alpha,\beta,\delta$ in \eqref{index:weight} throughout this section, unless there is an additional statement.\\

The following lemma is the heart of the Boltzmann well-posedness theory.
\begin{lemma} \label{lem:Q_gain}
    For the choice of $\alpha$, $\beta$, $N$ in \eqref{index:weight}, $-N< \kappa\leq 1$, and $2\delta>N$, let $w_\delta(v)f(x,v), w_\delta(v)g(x,v)\in L^\infty_vL^p_x\cap L^\infty_{x,v}$. Then $Q_{\text{gain}}(f,g)$ satisfies the following estimate
    \begin{align*}
        &\|w_{\delta}(v)Q_{\text{gain}}(f,g)(x, v)\|_{L^p_x} \\
        &\leq C_{gain}\left(\frac{1}{(1+|v|)^{\frac{N-1}{2}\beta - \kappa}}+\frac{1}{(1+|v|)^{\frac{3}{2}\beta - 1 - \kappa}}\right)\|w_{\delta}(v)f(x, v)\|_{L^\infty_vL^p_x}\|w_{\delta}(v)g(x, v)\|_{L^\infty_{x,v}},
        \intertext{and}
        &\|w_{\delta}(v)Q_{\text{gain}}(f,g)(x, v)\|_{L^p_x} \\
        &\leq C_{gain}\left(\frac{1}{(1+|v|)^{\frac{N-1}{2}\beta - \kappa}}+\frac{1}{(1+|v|)^{\frac{3}{2}\beta - 1 - \kappa}}\right)\|w_{\delta}(v)f(x, v)\|_{L^\infty_{x,v}}\|w_{\delta}(v)g(x, v)\|_{L^\infty_vL^p_x}
    \end{align*}
    for a.e. $v$ for all $1\leq p\leq \infty$, where $C_{gain}$ depends on $\alpha,\beta,N,\delta$, and $\kappa$.
\end{lemma}

\begin{proof}
    Suppose $1\leq p<\infty$ and $A\subset \mathbb{R}^N$ be a measurable set. Using Minkowski's integral inequality and Lemma \ref{lem:int_bound4}, we have
    \begin{equation*}
        \begin{split}
            &\int \left(\int \left(w_{\delta}(v)Q_{\text{gain}}(f,g)(\tau, x, v)\right)^p \,dx\right)^{1/p}\mathbf{1}_{v\in A}\,dv \\
            &\leq \int dv\,\mathbf{1}_{v\in A}\iint_{\mathbb{R}^{N}\times   \mathbb{S}^{N-1}}dud\omega \left(\int_{\Omega}dx\, \left(B(v-u, \omega)\frac{w_{\delta}(v)}{w_{\delta}(v')w(u')} w_{\delta}(v')f(\tau, x, v')w(u')g(\tau, x, u')\right)^p\right)^{1/p}\\
            &=\iiint_{\mathbb{R}^N\times \mathbb{R}^{N}\times \mathbb{S}^{N-1}}dvdud\omega \,\mathbf{1}_{v'\in A}\left(\int_{\Omega}dx\, \left(B(v-u, \omega)\frac{w_{\delta}(v')}{w_{\delta}(v)w(u)} w_{\delta}(v)f(\tau, x, v)w(u)g(\tau, x, u)\right)^p\right)^{1/p}\\
            &\leq \min\{\left\|w_{\delta}(v)f(\tau, x, v)\right\|_{L^\infty_{x,v}}\left\|w_{\delta}(v)g(\tau, x, v)\right\|_{L^\infty_vL^p_x}, \left\|w_{\delta}(v)f(\tau, x, v)\right\|_{L^\infty_vL^p_x}\left\|w_{\delta}(v)g(\tau, x, v)\right\|_{L^\infty_{x,v}}\}\\
            &\qquad\times \iiint_{\mathbb{R}^{N}\times \mathbb{R}^{N}\times   \mathbb{S}^{N-1}}dv dud\omega \, \mathbf{1}_{v'\in A}B(v-u, \omega) \frac{w_{\delta}(v')}{w_{\delta}(v)w_{\delta}(u)}\\
            &\leq \min\{\left\|w_{\delta}(v)f(\tau, x, v)\right\|_{L^\infty_{x,v}}\left\|w_{\delta}(v)g(\tau, x, v)\right\|_{L^\infty_vL^p_x}, \left\|w_{\delta}(v)f(\tau, x, v)\right\|_{L^\infty_vL^p_x}\left\|w_{\delta}(v)g(\tau, x, v)\right\|_{L^\infty_{x,v}}\}\\
            &\qquad\times \int_{\mathbb{R}^N}dv\,\mathbf{1}_{v\in A}\iint_{\mathbb{R}^{N}\times   \mathbb{S}^{N-1}}dv dud\omega \, B(v-u, \omega) \frac{w_{\delta}(v)}{w_{\delta}(v')w_{\delta}(u')}\\
            & \leq \min\{\left\|w_{\delta}(v)f(\tau, x, v)\right\|_{L^\infty_{x,v}}\left\|w_{\delta}(v)g(\tau, x, v)\right\|_{L^\infty_vL^p_x}, \left\|w_{\delta}(v)f(\tau, x, v)\right\|_{L^\infty_vL^p_x}\left\|w_{\delta}(v)g(\tau, x, v)\right\|_{L^\infty_{x,v}}\}\\
            &\qquad \times \int_{\mathbb{R}^N}C_{\alpha,\beta,N,\delta, \kappa}\left(\frac{1}{(1+|v|)^{\frac{N-1}{2}\beta - \kappa}}+\frac{1}{(1+|v|)^{\frac{3}{2}\beta - 1 - \kappa}}\right)\mathbf{1}_{v\in A}\,dv.
        \end{split}
    \end{equation*}
    It shows that the lemma holds a.e. $v$ for $1\leq p<\infty$.
    
    If $p=\infty$, we also choose a measurable set $B\subset \mathbb{R}^N$. Repeating the above argument, we obtain
    \begin{equation*}
        \begin{split}
            &\iint_{\Omega\times \mathbb{R}^N} w_\delta(v)Q_{\text{gain}}(f,g)(x,v)\mathbf{1}_{v\in A}\mathbf{1}_{x\in B}\,dxdv\\
            & = \iiiint_{\Omega\times\mathbb{R}^{N}\times\mathbb{R}^{N}\times \mathbb{S}^{N-1}}dxdvdud\omega \,B(v-u, \omega)\frac{w_{\delta}(v)}{w_{\delta}(v')w(u')} w_{\delta}(v')f(\tau, x, v')w(u')g(\tau, x, u')\mathbf{1}_{v\in A}\mathbf{1}_{x\in B}\\
            & =\iiiint_{ \mathbb{S}^{N-1}\times \Omega\times\mathbb{R}^{N}\times\mathbb{R}^{N}}d\omega dxdvdu \,B(v-u, \omega)\frac{w_{\delta}(v')}{w_{\delta}(v)w(u)} w_{\delta}(v)f(\tau, x, v)w(u)g(\tau, x, u)\mathbf{1}_{v'\in A}\mathbf{1}_{x\in B}\\
            & \leq \|w_{\delta}(v)f(x, v)\|_{L^\infty_{x,v}}\|w_{\delta}(v)g(x, v)\|_{L^\infty_{x,v}}\iiiint_{ \mathbb{S}^{N-1}\times \Omega\times\mathbb{R}^{N}\times\mathbb{R}^{N}}d\omega dxdvdu \,B(v-u, \omega)\frac{w_{\delta}(v')}{w_{\delta}(v)w(u)} \mathbf{1}_{v'\in A}\mathbf{1}_{x\in B}\\
            & \leq \|w_{\delta}(v)f(x, v)\|_{L^\infty_{x,v}}\|w_{\delta}(v)g(x, v)\|_{L^\infty_{x,v}}\iiiint_{ \Omega\times\mathbb{R}^{N}\times\mathbb{R}^{N}\times  \mathbb{S}^{N-1}}dxdvdud\omega \,B(v-u, \omega)\frac{w_{\delta}(v)}{w_{\delta}(v')w(u')} \mathbf{1}_{v\in A}\mathbf{1}_{x\in B}\\
            & \leq \|w_{\delta}(v)f(x, v)\|_{L^\infty_{x,v}}\|w_{\delta}(v)g(x, v)\|_{L^\infty_{x,v}}\\
            &\qquad \times \iint_{\Omega\times\mathbb{R}^{N}}dxdv\,C_{\alpha,\beta,N,\delta, \kappa}\left(\frac{1}{(1+|v|)^{\frac{N-1}{2}\beta - \kappa}}+\frac{1}{(1+|v|)^{\frac{3}{2}\beta - 1 - \kappa}}\right)\mathbf{1}_{v\in A}\mathbf{1}_{x\in B}.
        \end{split}
    \end{equation*}
    It shows that the lemma holds for $p=\infty$.
\end{proof}

Using the $Q_{\text{gain}}$ estimate, we can directly deduce an \textit{a priori} estimate.
\begin{proposition}
    (\textit{A priori} estimate) Any solution $f(t,x,v)$ of the inhomogeneous Boltzmann equation with cutoff collision kernel \eqref{Collision_kernel} satisfies the following: if $\alpha,\beta, N$ satisfy \eqref{index:weight}, $-N<\kappa\leq 1$, $2\delta>N$, and it satisfies $\sup_{0\leq t\leq T}\left\|w_{\delta}(v)f(t,x,v)\right\|_{L^\infty_{x,v} }<\infty$ for some $T$, then
    \begin{align*}
        \left\|w_{\delta}(v)f(t, x, v)\right\|_{L^\infty_{x,v}}\leq \frac{1}{\left\|w_{\delta}(v)f_0(x,v)\right\|_{L^\infty_{x,v}}^{-1} - C_{\text{gain}}t},
    \end{align*}
    where $w_\delta(v)$ is the velocity weight \eqref{def_w} and $C_{\text{gain}}$ is the constant in Lemma \ref{lem:Q_gain}.
\end{proposition}
\begin{proof}
    Let $f(t,x,v)$ be a solution of the Boltzmann equation with cutoff collision kernel \eqref{Collision_kernel}. Multiplying by weight $w_\delta(v)$ on both sides and using Lemma \ref{lem:Q_gain}, we get
    \begin{align*}
        w_{\delta}(v)f(t,x, v) &\leq w_{\delta}(v)f_0(x-vt, v) + \int_0^t \left\|w_{\delta}(v)Q_{\text{gain}}(f,f)(\tau, x-v(t-\tau), v)\right\|_{L^\infty_{x,v} }^2\,d\tau\\
        &\leq w_{\delta}(v)f_0(x-vt, v) + \int_0^t C_{\text{gain}}\left\|w_{\delta}(v)f(\tau, x, v)\right\|_{L^\infty_{x,v} }^2\,d\tau.
    \end{align*}
    Therefore,
    \begin{align*}
        \left\|w_{\delta}(v)f(t, x, v)\right\|_{L^\infty_{x,v}} &\leq \left\|w_{\delta}(v)f_0(x,v)\right\|_{L^\infty_{x,v}} + \int_0^t C_{\text{gain}}\left\|w_{\delta}(v)f(\tau, x,v)\right\|_{L^\infty_{x,v}}^2\,d\tau.
    \end{align*}
    Let us define $\varphi(t) := \frac{1}{\left\|w_\delta(v)f_0(x,v)\right\|_{L^\infty_{x,v} }^{-1} - C_{\text{gain}}t}$, then it satisfies
    \begin{align*}
        \varphi(t) = \left\|w_\delta(v)f_0(x,v)\right\|_{L^\infty_{x,v}} + \int_0^t \varphi'(\tau)\,d\tau = \left\|w_{\delta}(v)f_0(x,v)\right\|_{L^\infty_{x,v}} + \int_0^t C_{\text{gain}}\left(\varphi(\tau)\right)^2\,d\tau,
    \end{align*}
    so
    \begin{equation*}
        \begin{split}
            &\left\|w_{\delta}(v)f(t, x-vt, v)\right\|_{L^\infty_{x,v}} - \varphi(t) \\
            &\leq C_{\text{gain}}\int_0^t \left\|w_{\delta}(v)f(\tau, x, v)\right\|_{L^\infty_{x,v}}^2 - \left(\varphi(\tau)\right)^2\,d\tau \\
            &\leq C_{\text{gain}}\int_0^t \left(\left\|w_{\delta}(v)f(\tau, x, v)\right\|_{L^\infty_{x,v}} +\varphi(\tau)\right)\left(\left\|w_{\delta}(v)f(\tau, x, v)\right\|_{L^\infty_{x,v}} - \varphi(\tau)\right)\,d\tau.
        \end{split}
    \end{equation*}
    By the assumption, $\left\|w_{\delta}(v)f(\tau, x, v)\right\|_{L^\infty_{x,v}} +\varphi(\tau)<\infty$ for all $0\leq t\leq T$. Using Gr{\"o}nwall's inequality, we get the result that $\left\|w_{\delta}(v)f(\tau, x, v)\right\|_{L^\infty_{x,v}}\leq \varphi(t).$
\end{proof}

We now establish a well-posedness theory for the Boltzmann equation \eqref{eq:Boltzmann}. 
In the next proposition, we impose the condition $2\delta > N+1$, where the additional $+1$ ensures the construction of solutions for $0 < \kappa \le 1$.
\begin{proposition} \label{prop:wellposed}
    Let $f_0 $ be a function satisfying $\left\|w_{\delta}(v)f_0\right\|_{L^\infty_{x,v}}<\infty$ for $2\delta>N+1$, $-N<\kappa\leq 1$, and \eqref{index:weight}. Then there exists $T^*>0$ such that there exists a unique solution $f$ of the cutoff Boltzmann equation \eqref{eq:Boltzmann} \eqref{Collision_kernel} in 
    \begin{align*}
        \mathcal{A}\coloneqq \left\{f\geq 0: \sup_{0\leq t\leq T^*}\left\|w_{\delta}(v)f(t,x,v)\right\|_{L^\infty_{x,v}}<\infty\right\}
    \end{align*}
    having the initial data $f_0$. Furthermore, $T^*$ can be chosen so that it satisfies
    \begin{align*}
        \sup_{0\leq t\leq T^*} \left\|w_{\delta}(v)f(t,x,v)\right\|_{L^\infty_{x,v}}\leq 3\left\|w_{\delta}(v)f_0(x,v)\right\|_{L^\infty_{x,v}}.
    \end{align*}
\end{proposition}

\begin{proof}
    The proof follows a similar argument in \cite{DHWY}. We start with the differential version of the Boltzmann equation:
    \begin{align*}
        \partial_t f + v\cdot \nabla_x f = Q_{\text{gain}}(f,f) - f(t,x,v)\iint_{\mathbb{R}^N\times   \mathbb{S}^{N-1}}B(v-u,\omega) f(t,x,u)\,dud\omega.
    \end{align*}
    Defining
    \begin{align*}
        g(t,x,v) = \iint_{\mathbb{R}^N\times   \mathbb{S}^{N-1}}B(v-u,\omega) f(t,x,u)\,dud\omega,
    \end{align*}
    we can rewrite the solution of the Boltzmann equation by
    \begin{align*}
        f(t,x+vt, v) &= \exp\left(-\int_0^t g(\tau, x+v\tau, v)\,d\tau\right)f_0(x,v) \\
        &\quad + \int_0^t \exp\left(-\int_\tau^t g(s, x+vs, v)\,ds\right)Q_{\text{gain}}(f,f)(\tau, x+v\tau, v)\,d\tau.
    \end{align*}

    We first prove the existence statement. We define a map
    \begin{align*}
        \mathcal{T}f(t,x+vt,v) &= \exp\left(-\int_0^t g(\tau, x+v\tau, v)\,d\tau\right)f_0(x,v)  \\
        &\quad + \int_0^t \exp\left(-\int_\tau^t g(s, x+vs, v)\,ds\right)Q_{\text{gain}}(f,f)(\tau, x+v\tau, v)\,d\tau
    \end{align*}
    and construct $f_n = \mathcal{T}f_{n-1}$ with $f_0(t,x,v) = f_0(x,v)$. If $f\geq 0$, then $\mathcal{T}f\geq 0$ by the construction of $\mathcal{T}$, so $f_n\geq 0$ for all $n$. We show that $f_n\in \mathcal{A}$ for all $n$ by choosing small enough $T_1$; in fact, we claim 
    \begin{align*}
        \left\|w_{\delta}(v)f_n(t, x+vt, v)\right\|_{L^\infty_{x,v} }\leq \frac{1}{\left\|w_{\delta}(v)f_0(x,v)\right\|_{L^\infty_{x,v} }^{-1}-C_{\text{gain}}t}
    \end{align*}
    until RHS is well-defined. We use the induction hypothesis. Obviously, $f_0\in \mathcal{A}$. Assuming the hypothesis is true for $f_{n-1}$, by Lemma \ref{lem:Q_gain},
    \begin{align*}
        &\left\|w_{\delta}(v)f_n(t, x+vt, v)\right\|_{L^\infty_{x,v} } \\
        &\leq \left\|w_{\delta}(v)f_0(x,v)\right\|_{L^\infty_{x,v}} + \int_0^t w_{\delta}(v)Q_{\text{gain}}(f_{n-1},f_{n-1})(\tau, x+v\tau, v)\,d\tau\\
        &\leq \left\|w_{\delta}(v)f_0(x,v)\right\|_{L^\infty_{x,v}} + \int_0^t C_{\text{gain}}\left\|w_{\delta}(v)f_{n-1}(\tau, x+v\tau, v)\right\|_{L^\infty_{x,v} }^2\,d\tau\\
        &\leq \left\|w_{\delta}(v)f_0(x,v)\right\|_{L^\infty_{x,v} }+\int_0^t C_{\text{gain}}\frac{1}{\left(\left\|wf_0\right\|_{L^\infty_{x,v} }^{-1}-C_{\text{gain}}\tau\right)^2}\,d\tau=\frac{1}{\left\|wf_0\right\|_{L^\infty_{x,v} }^{-1}-C_{\text{gain}}t}.
    \end{align*}
    Therefore, $f_{n}\in \mathcal{A}$ for all $n$.\\

    We show that $f_n(t,x,v)$ converges to a function $f(t,x,v)$. In this step, we divide the cases $\kappa\leq 0$ and $\kappa>0$ since $\kappa>0$ requires an additional technical step.\\

    \noindent \textit{(i) The case $\kappa>0$.}
    
    In this case, we will modify the weight function by $w_{\delta-\kappa/2}(v)$. Note that it still satisfies the conditions for Lemma \ref{lem:Q_gain} as $2(\delta-\kappa/2)>N$.
    
    We fix $T_1>0$, which will be chosen later, and consider time domain $0\leq t\leq T_1$. The difference $w_{\delta-\kappa/2}(v)(f_n-f_{n-1})$ is written as follows.
    \begin{align*}
        &w_{\delta-\kappa/2}(v)(f_n-f_{n-1})(t,x+vt,v) \\
        & = \left\{\exp\left(-\int_0^t g_{n-1}(\tau, x+v\tau, v)\,d\tau\right) - \exp\left(-\int_0^t g_{n-2}(\tau, x+v\tau, v)\,d\tau\right)\right\} w_{\delta-\kappa/2}(v)f_0(x,v)\\
        &\quad + \int_0^t \left\{\exp\left(-\int_\tau^t g_{n-1}(s, x+vs, v)\,d\tau\right) - \exp\left(-\int_\tau^t g_{n-2}(s, x+vs, v)\,ds\right)\right\}\\
        &\quad\qquad \times w_{\delta-\kappa/2}(v)Q_{\text{gain}}(f_{n-1},f_{n-1})(\tau, x+v\tau, v)\,d\tau\\
        &\quad + \int_0^t \exp\left(-\int_\tau^t g_{n-2}(s, x+vs, v)\,ds\right) w_{\delta-\kappa/2}(v)\left(Q_{\text{gain}}(f_{n-1},f_{n-1}) - Q_{\text{gain}}(f_{n-2},f_{n-2})\right)(\tau, x+v\tau, v)\,d\tau\\
        &\eqqcolon I_1+I_2+I_3.
    \end{align*}

    We first bound $I_1$ and $I_2$. By the definition of $g$ and using
    \begin{align}\label{5_2:Bbound}
        B(v-u, \omega)\leq |v-u|^\kappa b(\cos\theta)\leq 2(1+|v|^2)^{\kappa/2}(1+|u|^2)^{\kappa/2} b(\cos\theta),
    \end{align}
    we have
    \begin{align*}
        &|g_{n-1}- g_{n-2}|(\tau, x+v\tau, v) \\
        & \leq 2\iint_{\mathbb{R}^N\times \mathbb{S}^{N-1}} (1+|v|^2)^{\kappa/2}(1+|u|^2)^{\kappa/2} b(\cos\theta) |f_{n-1}-f_{n-2}|(\tau, x+v\tau, u)\,du d\omega\\
        & \leq 2(1+|v|^2)^{\kappa/2}\left\|w_{\delta-\kappa/2}(u)(f_{n-1}-f_{n-2})(\tau, x+v\tau, u)\right\|_{L^\infty_{x,u}}\iint_{\mathbb{R}^{N}\times  \mathbb{S}^{N-1}} \frac{(1+|u|^2)^{\kappa/2}}{w_{\delta-\kappa/2}(u)}b(\cos\theta)\,du d\omega\\
        & \leq C'_{N, \delta, \kappa}(1+|v|^2)^{\kappa/2}\left\|w_{\delta-\kappa/2}(v)(f_{n-1}-f_{n-2})(\tau, x, v)\right\|_{L^\infty_{x,v}}
    \end{align*}
    for some constant $C'_{N, \delta, \kappa}$ as $2\delta-\kappa>N$. As $|e^{-x}-e^{-y}|\leq |x-y|$ for $x,y\geq 0$, we get
    \begin{equation}\label{5_2:I1est}
        \begin{split}
            |I_1|&\leq w_{\delta-\kappa/2}(v)f_0(x,v)\int_0^t |g_{n-1}- g_{n-2}|(\tau, x+v\tau, v)\,d\tau\\
            &\leq C'(1+|v|^2)^{\kappa/2} w_{\delta-\kappa/2}(v)f_0(x,v)\int_0^t \left\|w_{\delta-\kappa/2}(u)(f_{n-1}-f_{n-2})(\tau, x, u)\right\|_{L^\infty_{x,v}}\,d\tau\\
            &\leq C'T_1\left\|w_{\delta}(v)f_0(x,v)\right\|_{L^\infty_{x,v}} \sup_{0\leq t\leq T_1}\left\|w_{\delta-\kappa/2}(v)(f_{n-1}-f_{n-2})(t, x, v)\right\|_{L^\infty_{x,v}}
        \end{split}
    \end{equation}
    and
    \begin{equation}\label{5_2:I2est}
        \begin{split}
            |I_2|&\leq \int_0^t \left(\int_\tau^t |g_n- g_{n-1}|(s, x+vs, v)\,ds\right) w_{\delta-\kappa/2}(v)Q_{\text{gain}}(f_{n-1}, f_{n-1})(\tau, x+v\tau, v)\,d\tau\\
            &\leq C' \sup_{0\leq t\leq T_1} \left\|w_{\delta-\kappa/2}(v)(f_{n-1}-f_{n-2})(t, x, v)\right\|_{L^\infty_{x,v}}\\
            &\qquad \times\int_0^t (1+|v|^2)^{\kappa/2} (t-\tau) w_{\delta-\kappa/2}(v)Q_{\text{gain}}(f_{n-1}, f_{n-1})(\tau, x+v\tau, v)\,d\tau\\
            &\leq C' \sup_{0\leq t\leq T_1} \left\|w_{\delta-\kappa/2}(v)(f_{n-1}-f_{n-2})(t, x, v)\right\|_{L^\infty_{x,v}}\int_0^t (t-\tau)w_{\delta}(v)Q_{\text{gain}}(f_{n-1}, f_{n-1})(\tau, x+v\tau, v)\,d\tau\\
            &\leq C_{\text{gain}} C' \sup_{0\leq t\leq T_1} \left\|w_{\delta}(v)f_{n-1}(t, x, v)\right\|^2_{L^\infty_{x,v}}\sup_{0\leq t\leq T_1} \left\|w_{\delta-\kappa/2}(v)(f_{n-1}-f_{n-2})(t, x, v)\right\|_{L^\infty_{x,v}}\int_0^t (t-\tau)\,d\tau\\
            &= \frac{C_{\text{gain}} C'T_1^2}{2} \sup_{0\leq t\leq T_1} \left\|w_{\delta}(v)f_{n-1}(t, x, v)\right\|^2_{L^\infty_{x,v}}\sup_{0\leq t\leq T_1} \left\|w_{\delta-\kappa/2}(v)(f_{n-1}-f_{n-2})(t, x, v)\right\|_{L^\infty_{x,v}}.
        \end{split}
    \end{equation}
    
    For $I_3$, we can use Lemma \ref{lem:Q_gain} as follows.
    \begin{equation}\label{5_2:I3est}
        \begin{split}
            |I_3|&\leq \int_0^t \exp\left(-\int_\tau^t g_{n-2}(s, x+vs, v)\,ds\right) w_{\delta-\kappa/2}(v)\\
            &\qquad\times\left|Q_{\text{gain}}(f_{n-1},f_{n-1}-f_{n-2}) + Q_{\text{gain}}(f_{n-1}-f_{n-2},f_{n-2})\right|(\tau, x+v\tau, v)\,d\tau\\
            &\leq \int_0^t w_{\delta-\kappa/2}(v)\left|Q_{\text{gain}}(f_{n-1},f_{n-1}-f_{n-2})\right| + w_{\delta-\kappa/2}(v)\left|Q_{\text{gain}}(f_{n-1}-f_{n-2},f_{n-2})\right|(\tau, x+v\tau, v)\,d\tau\\
            &\leq C_{\text{gain}}T_1\left(\sup_{0\leq t\leq T_1} \left\|w_{\delta}(v)f_{n-1}(t, x, v)\right\|_{L^\infty_{x,v}} + \sup_{0\leq t\leq T_1} \left\|w_{\delta}(v)f_{n-2}(t, x, v)\right\|_{L^\infty_{x,v}}\right)\\
            &\qquad\times \sup_{0\leq t\leq T_1} \left\|w_{\delta-\kappa/2}(v)(f_{n-1}-f_{n-2})(t, x, v)\right\|_{L^\infty_{x,v}}.
        \end{split}
    \end{equation}

    Now, we choose
    \begin{align*}
        T_1\leq \min\{C_{\text{gain}}^{-1},(C')^{-1}\}\frac{\left\|w_{\delta}(v)f_0(t, x, v)\right\|^{-1}_{L^\infty_{x,v}}}{4},
    \end{align*}
    then it satisfies the following estimates: 
    \begin{align*}
        &\sup_{0\leq t\leq T_1} \left\|w_{\delta}(v)f_n(t, x, v)\right\|_{L^\infty_{x,v}}\leq \frac{1}{\left\|w_{\delta}(v)f_0(t, x, v)\right\|^{-1}_{L^\infty_{x,v}} - C_{\text{gain}}T_1}\leq \frac{4}{3}\left\|w_{\delta}(v)f_0(t, x, v)\right\|^{-1}_{L^\infty_{x,v}},\\
        &C'T_1\left\|w_{\delta}(v)f_0(x,v)\right\|_{L^\infty_{x,v}}\leq \frac{1}{4},\\
        &C_{\text{gain}}T_1\left(\sup_{0\leq t\leq T_1} \left\|w_{\delta}(v)f_{n-1}(t, x, v)\right\|_{L^\infty_{x,v}} + \sup_{0\leq t\leq T_1} \left\|w_{\delta}(v)f_{n-2}(t, x, v)\right\|_{L^\infty_{x,v}}\right)\leq \frac{2}{3},\\
        &\frac{C_{\text{gain}} C'T_1^2}{2} \sup_{0\leq t\leq T_1} \left\|w_{\delta}(v)f_{n-1}(t, x, v)\right\|^2_{L^\infty_{x,v}}\leq \frac{1}{18}.
    \end{align*}
    
    Combining this bounds under the condition $t\leq T_1$ with estimate of $I_i$, we have
    \begin{align*}
        \left\|w_{\delta-\kappa/2}(v)(f_n-f_{n-1})(t,x+vt,v)\right\|_{L^\infty_{x,v}}
        \leq\frac{35}{36}\sup_{0\leq t\leq T_1}\left\|w_{\delta-\kappa/2}(v)(f_{n-1}-f_{n-2})(t, x, v)\right\|_{L^\infty_{x,v}}.
    \end{align*}
    Therefore, there exists a function $f(t,x,v)\in \mathcal{A}$ satisfying $f = \mathcal{T}f$, implying $f$ satisfies the Boltzmann equation. Since $f_n\geq 0$ for all $t$, $x$, and $v$, $f(t,x,v)\geq 0$ for a.e. $x, v$ for all $t$.

    Since constant $C_{\text{gain}}$ and $C'$ in \eqref{5_2:I1est}, \eqref{5_2:I2est}, and \eqref{5_2:I3est} only depends on $\alpha,\beta,\delta, \kappa$, and $N$, it is a standard argument that extends the solution in time by replacing initial data $f(T_1/2, x, v)$ for the first step for example and repeating it. Thus, we can check that the solution exists until $t\leq T^*$ satisfying 
    \begin{align*}
        \sup_{0\leq t\leq T^*} \left\|w_{\delta}(v)f(t,x,v)\right\|_{L^\infty_{x,v}}\leq 3\left\|w_{\delta}(v)f_0(x,v)\right\|_{L^\infty_{x,v}}.
    \end{align*}
    It ends the proof of existence part.

    To prove the uniqueness part, let us assume there exist two solutions of Boltzmann equation $f,\tilde{f}\in\mathcal{A}$ with the same initial $f_0$. Then, following the analysis for existence part, we can derive that
    \begin{align*}
        &\sup_{0\leq t\leq T_1}\left\|w_{\delta-\kappa/2}(v)\left(f- \tilde{f}\right)(t, x, v)\right\|_{L^\infty_{x,v}} \\
        & \leq \left(C'T_1\left\|w_{\delta}(v)f_0(x,v)\right\|_{L^\infty_{x,v}} + \frac{C_{\text{gain}} C'T_1^2}{2} \sup_{0\leq t\leq T_1} \left\|w_{\delta}(v)f(t, x, v)\right\|^2_{L^\infty_{x,v}} \right.\\
        &\quad\quad +\left. C_{\text{gain}}T_1\left(\sup_{0\leq t\leq T_1} \left\|w_{\delta}(v)f(t, x, v)\right\|_{L^\infty_{x,v}} + \sup_{0\leq t\leq T_1} \left\|w_{\delta}(v)\tilde{f}(t, x, v)\right\|_{L^\infty_{x,v}}\right)\right)\\
        &\quad\qquad \times \sup_{0\leq t\leq T_1} \left\|w_{\delta-\kappa/2}(v)(f-\tilde{f})(t, x, v)\right\|_{L^\infty_{x,v}}
    \end{align*}
    for some $T_1$, which is just to replace $f_{n-1}$ and $f_{n-2}$ in $I_i$ estimates \eqref{5_2:I1est}, \eqref{5_2:I2est}, and \eqref{5_2:I3est} with $f$ and $\tilde{f}$. Now, choose small enough $T_1$ so that the value in parenthesis in the RHS is smaller than $1$. In this case, we have
    \begin{align*}
        \sup_{0\leq t\leq T_1}\left\|w_{\delta-\kappa/2}(v)\left(f- \tilde{f}\right)(t, x, v)\right\|_{L^\infty_{x,v}} = 0.
    \end{align*}
    We can again repeat this calculation as
    \begin{align*}
        \sup_{0\leq t\leq T^*} \left\|w_{\delta}(v)f(t,x,v)\right\|_{L^\infty_{x,v}}, \left\|w_{\delta}(v)\tilde{f}(t,x,v)\right\|_{L^\infty_{x,v}}\leq 3\left\|w_{\delta}(v)f_0(x,v)\right\|_{L^\infty_{x,v}},
    \end{align*}
    so we finally obtain uniqueness for the whole time interval $[0, T^*]$.\\

    \noindent \textit{(ii) The case $\kappa\leq 0$.}
    
    In this case, we use the original weight $w_\delta(v)$. The difference $w_{\delta}(v)(f_n-f_{n-1})$ is written by
    \begin{align*}
        &w_\delta(v)(f_n-f_{n-1})(t,x+vt,v) \\
        & = \left\{\exp\left(-\int_0^t g_{n-1}(\tau, x+v\tau, v)\,d\tau\right) - \exp\left(-\int_0^t g_{n-2}(\tau, x+v\tau, v)\,d\tau\right)\right\} w_\delta(v)f_0(x,v)\\
        &\quad + \int_0^t \left\{\exp\left(-\int_\tau^t g_{n-1}(s, x+vs, v)\,d\tau\right) - \exp\left(-\int_\tau^t g_{n-2}(s, x+vs, v)\,ds\right)\right\}\\
        &\quad\qquad \times w_\delta(v)Q_{\text{gain}}(f_{n-1},f_{n-1})(\tau, x+v\tau, v)\,d\tau\\
        &\quad + \int_0^t \exp\left(-\int_\tau^t g_{n-2}(s, x+vs, v)\,ds\right) w_\delta(v)\left(Q_{\text{gain}}(f_{n-1},f_{n-1}) - Q_{\text{gain}}(f_{n-2},f_{n-2})\right)(\tau, x+v\tau, v)\,d\tau\\
        &\eqqcolon I_4+I_5+I_6.
    \end{align*}
    
    All the steps to bound $I_4$, $I_5$, and $I_6$ are almost same except the bound of $B(v-u,\omega)$ \eqref{5_2:Bbound}, requiring different bound of $|g_{n-1}-g_{n-2}|$. The difference $|g_{n-1}-g_{n-2}|$ is bounded by
    \begin{align*}
        &|g_{n-1}- g_{n-2}|(\tau, x+v\tau, v) \\
        & \leq \iint_{\mathbb{R}^N\times \mathbb{S}^{N-1}} |v-u|^\kappa b(\cos\theta)|f_{n-1}-f_{n-2}|(\tau, x+v\tau, u)\,du d\omega\\
        & \leq \left\|w_\delta(u)(f_{n-1}-f_{n-2})(\tau, x+v\tau, u)\right\|_{L^\infty_{x,u}}\iint_{\mathbb{R}^{N}\times  \mathbb{S}^{N-1}}\frac{1}{|v-u|^{-\kappa}w_\delta(u)}b(\cos\theta)\,du d\omega\\
        &\leq C_N\left\|w_\delta(u)(f_{n-1}-f_{n-2})(\tau, x+v\tau, u)\right\|_{L^\infty_{x,u}}\\
        &\qquad \times \left(\int_{|v-u|\leq 1}\frac{1}{|v-u|^{-\kappa}}\,du + \int_{|v-u|\geq 1}\frac{1}{w_\delta(u)}\,du\right)\\
        & \leq C_{N, \delta, \kappa}\left\|w_\delta(v)(f_{n-1}-f_{n-2})(\tau, x, v)\right\|_{L^\infty_{x,v}}.
    \end{align*}
    Therefore, following the same steps in the previous case, we get the existence and uniqueness of the solution.
\end{proof}

\begin{remark}
    We compare Theorem \ref{thm:2} and Proposition \ref{prop:wellposed}. The velocity weight $w_{\alpha,\beta, \delta}(v)$ under the condition \eqref{index:weight} also fulfills the condition of Theorem \ref{thm:2}. Therefore, $L^{\infty}_{x,v}((1+|v|^2)^{\delta} e^{\alpha|v|^{\beta}})$ with \eqref{index:weight} is a function space in which the BGK equation \eqref{BGK} is ill-posed while the Boltzmann equation \ref{eq:Boltzmann} is well-posed at least local in time.
\end{remark}

From now on, we will extend the $L^\infty_{x,v}$ well-posed result to $L^p_xL^q_v$ or $L^p_xL^q_v$ space, i.e. $\sup_{0\leq t\leq T}\|f(t,x,v)\|_{L^p_xL^q_v}<\infty$ or $\sup_{0\leq t\leq T}\|f(t,x,v)\|_{L^q_vL^p_x}<\infty$. Calculating the $v$ integral first, however, generates a problem; we need to bound $\int_{\mathbb{R}^N} f(t,x-vt, v)\,dv$, adding spatial effect to velocity integration. To avoid this issue, we initially impose $f_0(x,v)\in L^q_vL^p_x$ and then show that $\sup_{0\leq t\leq T}\|f(t,x,v)\|_{L^q_vL^p_x}<\infty$. The $L^q_vL^p_x$ case will be treated as a subspace of $L^p_xL^q_v$ if $q\leq p$ using Minkowski's integral inequality.\\

The result will depend on the validity of Proposition \ref{prop:wellposed}; we suppose that there exists a unique solution $f(t,x,v)$ satisfying $\sup_{0\leq t\leq T^*}\|f(t,x,v)\|_{L^\infty_{x,v}}<\infty$. Complex interpolation technique is the key technique in extending the solution space to $L^p_vL^q_x$. For the reader's convenience, we write some facts about complex interpolation method referring to \cite{Bergh J}. Let $V$ be a vector space, and $L^\infty_{0, v}(V)$ be a closed subspace of $L^\infty_v(V)$ constructed by taking completion of step functions $\sum_{i=1}^n a_i\mathbf{1}_{E_i}$ such that $a_i\in V$ and $E_i\subset \mathbb{R}^N$ with $|E_i|<\infty$.
\begin{lemma} [Riesz-Thorin] \label{Riesz-Thorin}
     Suppose $1\leq p_0,p_1,p_2,p_3\leq \infty$ and $1\leq q_0,q_1,q_2,q_3< \infty$, and let a linear function $T$ is continuous as an operator on each $L^{q_0}_vL^{p_0}_x\rightarrow L^{q_2}_vL^{p_2}_x$ and $L^{q_1}_vL^{p_1}_x\rightarrow L^{q_3}_vL^{p_3}_x$. For $0<\theta<1$, let us define
     \begin{align*}
         \frac{1}{p_{0,\theta}} = \frac{\theta}{p_0} + \frac{1-\theta}{p_1},\quad \frac{1}{p_{1,\theta}} = \frac{\theta}{p_2} + \frac{1-\theta}{p_3},\quad \frac{1}{q_{0,\theta}} = \frac{\theta}{q_0} + \frac{1-\theta}{q_1},\quad \frac{1}{q_{1,\theta}} = \frac{\theta}{q_2} + \frac{1-\theta}{q_3}.
     \end{align*}
     Then $T$ can be defined as an continuous operator from $L^{q_{0,\theta}}_vL^{p_{0,\theta}}_x\rightarrow L^{q_{1,\theta}}_vL^{p_{1,\theta}}_x$ satisfying
     \begin{align*}
         \|T\|_{L^{q_{0,\theta}}_vL^{p_{0,\theta}}_x\rightarrow L^{q_{1,\theta}}_vL^{p_{1,\theta}}_x} \leq          \|T\|^{\theta}_{L^{q_{0}}_vL^{p_{0}}_x\rightarrow L^{q_{2}}_vL^{p_{2}}_x}\|T\|^{1-\theta}_{L^{q_{1}}_vL^{p_{1}}_x\rightarrow L^{q_{3}}_vL^{p_{3}}_x}.
     \end{align*}
     Let $1\leq q_0, q_2<\infty$, and $T$ is continuous as an operator on each $L^{q_0}_vL^{p_0}_x\rightarrow L^{q_2}_vL^{p_2}_x$ and $L^{\infty}_{0,v}L^{p_1}_x\rightarrow L^{\infty}_{0,v}L^{p_3}_x$. For $0<\theta<1$
     \begin{align*}
         \|T\|_{L^{q_{0,\theta}}_vL^{p_{0,\theta}}_x\rightarrow L^{q_{1,\theta}}_vL^{p_{1,\theta}}_x} \leq          \|T\|^{\theta}_{L^{q_{0}}_vL^{p_{0}}_x\rightarrow L^{q_{2}}_vL^{p_{2}}_x}\|T\|^{1-\theta}_{L^\infty_{0,v}L^{p_{1}}_x\rightarrow L^{\infty}_{0,v}L^{p_{3}}_x},
     \end{align*}
     where all are same but $\frac{1}{q_{0,\theta}} = \frac{\theta}{q_0}$ and $\frac{1}{q_{1,\theta}} = \frac{\theta}{q_2}$.
\end{lemma}
\begin{proof}
    We refer to Theorem 5.1.1 and 5.1.2 in \cite{Bergh J}.
\end{proof}

To use the Riesz-Thorin interpolation, we split $Q_{gain}(f,g)$ as follows. Starting from
\begin{align*}
    &\iint_{\mathbb{R}^N\times\mathbb{S}^{N-1}}B(v-u,\omega)w_\delta(v)f(x,v')g(x,u')\,dud\omega\\
    &=\iint_{\mathbb{R}^N\times\mathbb{S}^{N-1}}B(v-u,\omega)\frac{w_\delta(v)}{w_\delta(v')w_\delta(u')}w_\delta(v')f(x,v')w_\delta(u')g(x,u')\,dud\omega\\
    &\leq C\iint_{\mathbb{R}^N\times\mathbb{S}^{N-1}}B(v-u,\omega)\Big(\frac{1}{(1+|v'|^2)^{\delta}} + \frac{1}{(1+|u'|^2)^{\delta}} \Big)\\
    &\quad \times \left(e^{-\alpha (2^{\beta/2}-1)|u|^\beta}\mathbf{1}_{|u|\geq |v|} + \sum_{n=0}^\infty e^{-\frac{\alpha \beta/4}{2^{(1-\beta/2)(n+1)}}|u|^\beta}\mathbf{1}_{2^{-n-1}|v|^2\leq |u|^2\leq 2^{-n}|v|^2}\right)w_\delta(v')f(x,v')w_\delta(u')g(x,u')\,dud\omega,
\end{align*}
we define
\begin{align*}
    W_1(v,u,\omega) &= B(v-u,\omega)\frac{1}{(1+|v'|^2)^{\delta}}\left(e^{-\alpha (2^{\beta/2}-1)|u|^\beta}\mathbf{1}_{|u|\geq |v|} + \sum_{n=0}^\infty e^{-\frac{\alpha \beta/4}{2^{(1-\beta/2)(n+1)}}|u|^\beta}\mathbf{1}_{2^{-n-1}|v|^2\leq |u|^2\leq 2^{-n}|v|^2}\right),\\
    W_2(v,u,\omega) &= B(v-u,\omega)\frac{1}{(1+|u'|^2)^{\delta}}\left(e^{-\alpha (2^{\beta/2}-1)|u|^\beta}\mathbf{1}_{|u|\geq |v|} + \sum_{n=0}^\infty e^{-\frac{\alpha \beta/4}{2^{(1-\beta/2)(n+1)}}|u|^\beta}\mathbf{1}_{2^{-n-1}|v|^2\leq |u|^2\leq 2^{-n}|v|^2}\right),\\
    Q_1(f,g) &= \iint_{\mathbb{R}^N\times\mathbb{S}^{N-1}}W_1(v,u,\omega)w_\delta(v')f(x,v')w_\delta(u')g(x,u')\,dud\omega,\\
    Q_2(f,g) &= \iint_{\mathbb{R}^N\times\mathbb{S}^{N-1}}W_2(v,u,\omega)w_\delta(v')f(x,v')w_\delta(u')g(x,u')\,dud\omega.\\
\end{align*}

The following lemma shows that $Q_1(f,g)$ (resp., $Q_2(f,g)$) is a bounded operator in $L^1_vL^p_x$ norm in $f$ (resp., $g$) when $g$ (resp., $f$) is in $L^\infty_{x,v}$ for all $1\leq p\leq \infty$. The proof is similar to Lemma \ref{lem:int_bound4}.
\begin{lemma}\label{lem:Q_gainL1}
    Assume that $\alpha$, $\beta$, and $N$ satisfy \eqref{index:weight}, $-N<\kappa\leq 1$, and $2\delta>N+1$. If $N=2$, we impose one more condition $1+\kappa\leq \beta$. Then, we have
    \begin{align*}
        \left\|w_\delta(v) Q_1(f,g)\right\|_{L^1_vL^p_x}&\leq C_{\alpha,\beta,N, \delta, \kappa}\|w_\delta(v)f\|_{L^\infty_{x,v}}\|w_\delta(v)g\|_{L^1_vL^p_x},\\
        \left\|w_\delta(v) Q_2(f,g)\right\|_{L^1_vL^p_x}&\leq C_{\alpha,\beta,N, \delta, \kappa}\|w_\delta(v)g\|_{L^\infty_{x,v}}\|w_\delta(v)f\|_{L^1_vL^p_x}
    \end{align*}
    for all $1\leq p\leq \infty$.
\end{lemma}

\begin{proof}
    By Minkowski's integral inequality, for $1\leq p\leq \infty$ and $i=1,2$,
    \begin{align*}
        \left\|w_\delta(v) Q_i(f,g)\right\|_{L^1_vL^p_x} &= \int_{\mathbb{R}^N}\,\left[\int_{\Omega}\left(\iint_{\mathbb{R}^N\times\mathbb{S}^{N-1}}W_i(v,u,\omega)w_\delta(v')f(x,v')w_\delta(u')g(x,u')\,dud\omega\right)^p\,dx\right]^{1/p}\,dv\\
        &\leq \int_{\mathbb{R}^N}\,\iint_{\mathbb{R}^N\times\mathbb{S}^{N-1}}\left(\int_{\Omega}\left[W_i(v,u,\omega)w_\delta(v')f(x,v')w_\delta(u')g(x,u')\right]^p\,dx\right)^{1/p}\,dud\omega\,dv\\
        &\leq \int_{\mathbb{R}^N}\,\iint_{\mathbb{R}^N\times\mathbb{S}^{N-1}}W_i(v,u,\omega)\left(\int_{\Omega}[w_\delta(v')f(x,v')w_\delta(u')g(x,u')]^p\,dx\right)^{1/p}\,dud\omega\,dv.
    \end{align*}
    Now, we split the integral by
    \begin{align*}
        I_{i,1} &= \int_{\mathbb{R}^N}\,\iint_{\mathbb{R}^N\times\mathbb{S}^{N-1}}\mathbf{1}_{|u|\geq |v|/256} W_i(v,u,\omega)\left(\int_{\Omega}\left[w_\delta(v')f(x,v')w_\delta(u')g(x,u')\right]^p\,dx\right)^{1/p}\,dud\omega\,dv,\\
        I_{i, 2} &= \int_{\mathbb{R}^N}\,\iint_{\mathbb{R}^N\times\mathbb{S}^{N-1}}\mathbf{1}_{|u|\leq |v|/256} W_i(v,u,\omega)\left(\int_{\Omega}\left[w_\delta(v')f(x,v')w_\delta(u')g(x,u')\right]^p\,dx\right)^{1/p}\,dud\omega\,dv
    \end{align*}
    for $i=1, 2$. It is equivalent to divide $n\leq 15$ and $n\geq 16$ cases in the summation of $n$.\\
    
    \noindent \textit{(i) Estimate of $I_{1, 1}$ and $I_{2, 1}$.}
    
    In this case, we will demonstrate the bound of $I_{1,1}$ since $I_{2,1}$ can be bounded using a similar method.
    
    For $|u|\geq |v|/256$, as $n\leq 15$, we have
    \begin{align*}
        \begin{split}
            I_{1, 1}&\leq\int_{\mathbb{R}^N}\,\iint_{\mathbb{R}^N\times\mathbb{S}^{N-1}}\mathbf{1}_{|u|\geq |v|/256} B(v-u,\omega)\frac{1}{(1+|v'|^2)^\delta}e^{-\frac{\alpha\beta}{2^{(18-8\beta)}}|u|^\beta}\\
            &\qquad \times \left(\int_{\Omega}\left[w_\delta(v')f(x,v')w_\delta(u')g(x,u')\right]^p\,dx\right)^{1/p}\,dud\omega\,dv\\
            &\leq\int_{\mathbb{R}^N}\,\iint_{\mathbb{R}^N\times\mathbb{S}^{N-1}}\mathbf{1}_{|u'|\geq |v'|/256} B(v-u,\omega)\frac{1}{(1+|v|^2)^\delta}e^{-\frac{\alpha\beta}{2^{(18-8\beta)}}|u'|^\beta}\|w_\delta(v)f(x,v)\|_{L^\infty_x}\|w_\delta(u)g(x,u)\|_{L^p_x}\,dud\omega\,dv\\
            &\leq \|w_\delta(v)f(x,v)\|_{L^\infty_{x,v}}\|w_\delta(v)g(x,v)\|_{L^1_vL^p_x}\\
            &\qquad \times \sup_u\,\iint_{\mathbb{R}^N\times\mathbb{S}^{N-1}}\mathbf{1}_{|u'|\geq |v'|/256} B(v-u,\omega)\frac{1}{(1+|v|^2)^\delta}e^{-\frac{\alpha\beta}{2^{(18-8\beta)}}|u'|^\beta}\,dvd\omega
        \end{split}
    \end{align*}
    In the middle step, we used the coordinate change $(v,u)\mapsto (v', u')$. In the final step, we applied $u$ integral to $w_\delta(u)g(x,u)$.
    
    By energy conservation and $|u'|\geq |v'|/256$ condition, 
    \begin{align*}
        &\max\{|v|^2,|u|^2\}\leq |v|^2+|u|^2=|v'|^2+|u'|^2\leq (256^2+1)|u'|^2.
    \end{align*}
    Therefore, if $\kappa\geq 0$, using \eqref{5_2:Bbound},
    \begin{align*}
        &\iint_{\mathbb{R}^N\times\mathbb{S}^{N-1}}\mathbf{1}_{|u'|\geq |v'|/256} B(v-u,\omega)\frac{1}{(1+|v|^2)^\delta}e^{-\frac{\alpha\beta}{2^{(18-8\beta)}}|u'|^\beta}\,dvd\omega\\
        & \leq C\int (1+|v|^2)^{\kappa/2}(1+|u|^2)^{\kappa/2} \frac{e^{-\frac{1}{2}\frac{\alpha\beta}{2^{(18-8\beta)}}(256^2+1)^{-\beta/2}|u|^\beta}e^{-\frac{1}{2}\frac{\alpha\beta}{2^{(18-8\beta)}}(256^2+1)^{-\beta/2}|v|^\beta}}{(1+|v|^2)^{\delta}}\,dvd\omega.
    \end{align*}
    The right-hand side is bounded uniformly about $u$.
    
    If $\kappa<0$, then
    \begin{align*}
        &\iint_{\mathbb{R}^N\times\mathbb{S}^{N-1}}\mathbf{1}_{|u'|\geq |v'|/256} B(v-u,\omega)\frac{1}{(1+|v|^2)^\delta}e^{-\frac{\alpha\beta}{2^{(18-8\beta)}}|u'|^\beta}\,dvd\omega\\
        &\leq C\int \frac{1}{|v-u|^{-\kappa}} \frac{e^{-\frac{1}{2}\frac{\alpha\beta}{2^{(18-8\beta)}}(256^2+1)^{-\beta/2}|u|^\beta}e^{-\frac{1}{2}\frac{\alpha\beta}{2^{(18-8\beta)}}(256^2+1)^{-\beta/2}|v|^\beta}}{(1+|v|^2)^{\delta}}\,dvd\omega.
    \end{align*}
    Since $-\kappa<N$, the right-hand side is bounded uniformly about $u$. Therefore, we get
    \begin{align}\label{5_9:I_11}
        I_{1,1}&\leq C_{\alpha,\beta, N,\delta,\kappa}\|w_\delta(v)f(x,v)\|_{L^\infty_{x,v}}\|w_\delta(v)g(x,v)\|_{L^1_vL^p_x},
    \end{align}
    and
    \begin{align}\label{5_9:I_21}
        I_{2,1}&\leq C_{\alpha,\beta, N,\delta,\kappa}\|w_\delta(v)g(x,v)\|_{L^\infty_{x,v}}\|w_\delta(v)f(x,v)\|_{L^1_vL^p_x}
    \end{align}
    for the same reason.\\

    \noindent \textit{(ii) Estimate of $I_{1,2}$.}
    
    For $|u|\leq |v|/256$, we use Carlemann type representation as in the proof of Lemma \ref{lem:int_bound4}.

    \begin{align}\label{5_9:J_1_2}
        \begin{split}
            I_{1, 2}&= \sum_{n=16}^\infty \int_{\mathbb{R}^N}dv\,\int_{\mathbb{R}^N} d\tilde{u}_\parallel\,\int_{E_{\tilde{u}_\parallel}} d\tilde{u}_\perp\,\mathbf{1}_{2^{-n-1}|v|^2\leq |v+\tilde{u}_\parallel+\tilde{u}_\perp|^2\leq 2^{-n}|v|^2}\\
            &\qquad \times \left(\int_{\Omega}\left[w_\delta(v+\tilde{u}_\parallel)f(x,v+\tilde{u}_\parallel)w_\delta(v+\tilde{u}_\perp)g(x,v+\tilde{u}_\perp)\right]^p\,dx\right)^{1/p}\\
            &\qquad\times\frac{\left(|\tilde{u}_\parallel|^2+|\tilde{u}_\perp|^2\right)^{\frac{\kappa-1}{2}}}{|\tilde{u}_\parallel|^{N-2}}\frac{1}{\left(1+|v+\tilde{u}_\parallel|^2\right)^\delta}e^{-\frac{\alpha \beta/4}{2^{(1-\beta/2)(n+1)}}|v+\tilde{u}_\parallel+\tilde{u}_\perp|^\beta}\mathbf{1}_{2^{-n-1}|v|^2\leq |v+\tilde{u}_\perp+\tilde{u}_\parallel|^2\leq 2^{-n}|v|^2}\\
            &\leq \sum_{n=16}^\infty \|w_\delta(v)f(x,v)\|_{L^\infty_{x,v}}\int_{\mathbb{R}^N} dv\,\int_{\mathbb{R}^N} d\tilde{u}_\parallel\,\int_{E_{\tilde{u}_\parallel}} d\tilde{u}_\perp\,\|w_\delta(v+\tilde{u}_\perp)g(x,v+\tilde{u}_\perp)\|_{L^p_x}\\
            &\qquad\times\frac{\left(|\tilde{u}_\parallel|^2+|\tilde{u}_\perp|^2\right)^{\frac{\kappa-1}{2}}}{|\tilde{u}_\parallel|^{N-2}}\frac{1}{\left(1+|v+\tilde{u}_\parallel|^2\right)^\delta} e^{-\frac{\alpha \beta/4}{2^{(1-\beta/2)(n+1)}}|v+\tilde{u}_\perp+\tilde{u}_\parallel|^\beta}\mathbf{1}_{2^{-n-1}|v|^2\leq |v+\tilde{u}_\perp+\tilde{u}_\parallel|^2\leq 2^{-n}|v|^2}\\
            &= \sum_{n=16}^\infty \|w_\delta(v)f(x,v)\|_{L^\infty_{x,v}}\int_{\mathbb{R}^N} dv\,\int_{\mathbb{R}^N} d\tilde{u}_\parallel\,\int_{E_{\tilde{u}_\parallel}} d\tilde{u}_\perp\,\|w_\delta(v)g(x,v)\|_{L^p_x}\\
            &\qquad\times\frac{\left(|\tilde{u}_\parallel|^2+|\tilde{u}_\perp|^2\right)^{\frac{\kappa-1}{2}}}{|\tilde{u}_\parallel|^{N-2}}\frac{1}{\left(1+|v-\tilde{u}_\perp+\tilde{u}_\parallel|^2\right)^\delta} e^{-\frac{\alpha \beta/4}{2^{(1-\beta/2)(n+1)}}|v+\tilde{u}_\parallel|^\beta}\mathbf{1}_{2^{-n-1}|v-\tilde{u}_\perp|^2\leq |v+\tilde{u}_\parallel|^2\leq 2^{-n}|v-\tilde{u}_\perp|^2}\\
            &\leq\sum_{n=16}^\infty \|w_\delta(v)f(x,v)\|_{L^\infty_{x,v}}\|w_\delta(v)g(x,v)\|_{L^1_vL^p_x}\sup_v\int_{\mathbb{R}^N} d\tilde{u}_\parallel\,\int_{E_{\tilde{u}_\parallel}} d\tilde{u}_\perp\,\\
            &\qquad\times\frac{\left(|\tilde{u}_\parallel|^2+|\tilde{u}_\perp|^2\right)^{\frac{\kappa-1}{2}}}{|\tilde{u}_\parallel|^{N-2}}\frac{1}{\left(1+|v-\tilde{u}_\perp+\tilde{u}_\parallel|^2\right)^\delta} e^{-\frac{\alpha \beta/4}{2^{(1-\beta/2)(n+1)}}|v+\tilde{u}_\parallel|^\beta}\mathbf{1}_{2^{-n-1}|v-\tilde{u}_\perp|^2\leq |v+\tilde{u}_\parallel|^2\leq 2^{-n}|v-\tilde{u}_\perp|^2}\\
            &=\sum_{n=16}^\infty \|w_\delta(v)f(x,v)\|_{L^\infty_{x,v}}\|w_\delta(v)g(x,v)\|_{L^1_vL^p_x}\sup_v\int_{\mathbb{R}^N} d\tilde{u}_\perp\,\int_{E_{\tilde{u}_\perp}} d\tilde{u}_\parallel\,\\
            &\qquad\times\frac{\left(|\tilde{u}_\parallel|^2+|\tilde{u}_\perp|^2\right)^{\frac{\kappa-1}{2}}}{|\tilde{u}_\perp||\tilde{u}_\parallel|^{N-3}}\frac{1}{\left(1+|v-\tilde{u}_\perp+\tilde{u}_\parallel|^2\right)^\delta} e^{-\frac{\alpha \beta/4}{2^{(1-\beta/2)(n+1)}}|v+\tilde{u}_\parallel|^\beta}\mathbf{1}_{2^{-n-1}|v-\tilde{u}_\perp|^2\leq |v+\tilde{u}_\parallel|^2\leq 2^{-n}|v-\tilde{u}_\perp|^2},
        \end{split}
    \end{align}
    where $E_{\tilde{u}_\parallel}$ (resp., $E_{\tilde{u}_\perp}$) is $N-1$ dimensional plane through $0$ perpendicular to $\tilde{u}_\parallel$ (resp., $\tilde{u}_\perp$). In the final step, we exchanged $d\tilde{u}_\perp$ and $d\tilde{u}_\parallel$ with additional factor $\frac{|\tilde{u}_\parallel|}{|\tilde{u}_\perp|}$. We define
    \begin{align*}
        &J_{1,2} =\sum_{n=16}^\infty \int_{\mathbb{R}^N} d\tilde{u}_\perp\,\int_{E_{\tilde{u}_\perp}} d\tilde{u}_\parallel\,\\
        &\qquad\times\frac{\left(|\tilde{u}_\parallel|^2+|\tilde{u}_\perp|^2\right)^{\frac{\kappa-1}{2}}}{|\tilde{u}_\perp||\tilde{u}_\parallel|^{N-3}}\frac{1}{\left(1+|v-\tilde{u}_\perp+\tilde{u}_\parallel|^2\right)^\delta} e^{-\frac{\alpha \beta/4}{2^{(1-\beta/2)(n+1)}}|v+\tilde{u}_\parallel|^\beta}\mathbf{1}_{2^{-n-1}|v-\tilde{u}_\perp|^2\leq |v+\tilde{u}_\parallel|^2\leq 2^{-n}|v-\tilde{u}_\perp|^2}.
    \end{align*}
    The remaining goal is bound $J_{1,2}$ uniformly about $v$.\\

    \noindent \textit{(ii)-1 The case $|\tilde{u}_\perp|\geq |v|/4$.}
    
    In this case, as $n\geq 16$, $|v+\tilde{u}_\parallel|\leq 2^{-8}|v-\tilde{u}_\perp|\leq 2^{-8}(|\tilde{u}_\perp|+|v|)\leq |\tilde{u}_\perp|/4$. Therefore,
    \begin{align*}
        \frac{1}{1+|v-\tilde{u}_\perp+\tilde{u}_\parallel|^2}\leq \frac{C}{1+|\tilde{u}_\perp|^2}.
    \end{align*}
    Also, we use a similar argument to \eqref{5_5:kappa_split} so that
    \begin{align*}
        \frac{\left(|\tilde{u}_\parallel|^2+|\tilde{u}_\perp|^2\right)^{\frac{\kappa-1}{2}}}{|\tilde{u}_\perp||\tilde{u}_\parallel|^{N-3}} &\leq \frac{1}{|\tilde{u}_\perp|^{1 + \frac{N-1}{N+1}(1-\kappa)}}\frac{1}{|\tilde{u}_\parallel|^{N-3 + \frac{2}{N+1}(1-\kappa)}}.
    \end{align*}
    Using these bounds, we have
    \begin{align*}
        &\int_{|\tilde{u}_\perp|\geq |v|/4} d\tilde{u}_\perp\,\int_{E_{\tilde{u}_\perp}} d\tilde{u}_\parallel\,\frac{\left(|\tilde{u}_\parallel|^2+|\tilde{u}_\perp|^2\right)^{\frac{\kappa-1}{2}}}{|\tilde{u}_\perp||\tilde{u}_\parallel|^{N-3}}\frac{1}{\left(1+|v-\tilde{u}_\perp+\tilde{u}_\parallel|^2\right)^\delta} \\
        &\qquad\times e^{-\frac{\alpha \beta/4}{2^{(1-\beta/2)(n+1)}}|v+\tilde{u}_\parallel|^\beta}\mathbf{1}_{2^{-n-1}|v-\tilde{u}_\perp|^2\leq |v+\tilde{u}_\parallel|^2\leq 2^{-n}|v-\tilde{u}_\perp|^2}\\
        &\leq C_\delta \int_{|\tilde{u}_\perp|\geq |v|/4} d\tilde{u}_\perp\,\frac{1}{|\tilde{u}_\perp|^{1 + \frac{N-1}{N+1}(1-\kappa)}}\frac{1}{\left(1+|\tilde{u}_\perp|^2\right)^\delta} \int_{E_{\tilde{u}_\perp}} d\tilde{u}_\parallel\,\frac{1}{|\tilde{u}_\parallel|^{N-3 + \frac{2}{N+1}(1-\kappa)}}\mathbf{1}_{|v+\tilde{u}_\parallel|^2\leq 2^{-n}|v-\tilde{u}_\perp|^2}.
    \end{align*}
    Bounding the integral can be a little tricky compared to the other cases since $N-3 + \frac{2}{N+1}(1-\kappa)$ can be negative when $N=2$, which obstructs the use of Hardy-Littlewood inequality. Therefore, we use a different method.

    By summing up the integral about $n\geq 16$, we write
    \begin{align*}
        &\sum_{n=16}^\infty\int_{|\tilde{u}_\perp|\geq |v|/4} d\tilde{u}_\perp\,\frac{1}{|\tilde{u}_\perp|^{1 + \frac{N-1}{N+1}(1-\kappa)}}\frac{1}{\left(1+|\tilde{u}_\perp|^2\right)^\delta} \int_{E_{\tilde{u}_\perp}} d\tilde{u}_\parallel\,\frac{1}{|\tilde{u}_\parallel|^{N-3 + \frac{2}{N+1}(1-\kappa)}}\mathbf{1}_{2^{-n-1}|v-\tilde{u}_\perp|^2\leq |v+\tilde{u}_\parallel|^2\leq 2^{-n}|v-\tilde{u}_\perp|^2}\\
        &=\int_{|\tilde{u}_\perp|\geq |v|/4} d\tilde{u}_\perp\,\frac{1}{|\tilde{u}_\perp|^{1 + \frac{N-1}{N+1}(1-\kappa)}}\frac{1}{\left(1+|\tilde{u}_\perp|^2\right)^\delta} \int_{E_{\tilde{u}_\perp}} d\tilde{u}_\parallel\,\frac{1}{|\tilde{u}_\parallel|^{N-3 + \frac{2}{N+1}(1-\kappa)}}\mathbf{1}_{|v+\tilde{u}_\parallel|\leq 2^{-8}|v-\tilde{u}_\perp|}.
    \end{align*}
    For $|\tilde{u}_\perp|\geq |v|/4$, it satisfies $|\tilde{u}_\parallel|\leq |v| + 2^{-8}|v-\tilde{u}_\perp|\leq 5|\tilde{u}_\perp|$. So, we have
    \begin{align*}
        \int_{E_{\tilde{u}_\perp}} d\tilde{u}_\parallel\,\frac{1}{|\tilde{u}_\parallel|^{N-3 + \frac{2}{N+1}(1-\kappa)}}\mathbf{1}_{|v+\tilde{u}_\parallel|\leq 2^{-8}|v-\tilde{u}_\perp|}&\leq \int_{E_{\tilde{u}_\perp}} d\tilde{u}_\parallel\,\frac{1}{|\tilde{u}_\parallel|^{N-3 + \frac{2}{N+1}(1-\kappa)}}\mathbf{1}_{|\tilde{u}_\parallel|\leq 5|\tilde{u}_\perp|}\\
        &\leq C_{N,\kappa} (5|\tilde{u}_\perp|)^{2-\frac{2}{N+1}(1-\kappa)}.
    \end{align*}
    As
    \begin{align}\label{5_9:J_12_1}
        J_{1,2}\leq C_{N,\kappa}\int_{|\tilde{u}_\perp|\geq |v|/4} d\tilde{u}_\perp\,\frac{|\tilde{u}_\perp|^{2-\frac{2}{N+1}(1-\kappa)}}{|\tilde{u}_\perp|^{1 + \frac{N-1}{N+1}(1-\kappa)}}\frac{1}{\left(1+|\tilde{u}_\perp|^2\right)^\delta} \leq C_{N,\kappa}\int_{\mathbb{R}^N} d\tilde{u}_\perp\,\frac{1}{|\tilde{u}_\perp|^{-\kappa}}\frac{1}{\left(1+|\tilde{u}_\perp|^2\right)^\delta}<\infty
    \end{align}
    for $-N<\kappa\leq 1$ and $2\delta>N+1$, by Lemma \ref{lem:int_bound3}, we get a uniform estimate of $J_{1,2}$ in this case.\\
    
    \noindent \textit{(ii)-2 The case $2^{-n/2+2}|v|\leq |\tilde{u}_\perp|\leq |v|/4$.}

    We first compute lower and upper bounds of variables. Since $|v+\tilde{u}_\parallel|\leq 2^{-8}|v-\tilde{u}_\perp|\leq 2^{-7}|v|$, we have $|v|/2\leq |\tilde{u}_\parallel|\leq 3|v|/2$. It implies
    \begin{align}\label{5_9:J_12_2_1}
        \left(|\tilde{u}_\parallel|^2+|\tilde{u}_\perp|^2\right)^{\frac{\kappa-1}{2}}\leq \frac{C_\kappa}{|v|^{1-\kappa}},\text{ and}\quad \frac{1}{|\tilde{u}_\parallel|^{N-3}}\leq \frac{C_N}{|v|^{N-3}}.
    \end{align}
    Employing the conditions $2^{-\frac{n+1}{2}}|v-\tilde{u}_\perp|\leq |v+\tilde{u}_\parallel|$ and $|\tilde{u}_\perp|\leq |v|/4$, we deduce
    \begin{align}\label{5_9:J_12_2_2}
        e^{-\frac{\alpha \beta/4}{2^{(1-\beta/2)(n+1)}}|v+\tilde{u}_\parallel|^\beta}\leq e^{-\frac{\alpha \beta/4}{2^{n+1}}|v-\tilde{u}_\perp|^\beta}\leq e^{-\frac{\alpha \beta/4}{2^{n+1+\beta}}|v|^\beta}.
    \end{align}
    Furthermore, as $2^{-n/2}|v-\tilde{u}_\perp|\leq 2^{-n/2+1}|v|\leq \frac{1}{2}|\tilde{u}_\perp|$, we have
    \begin{align*}
        |v-\tilde{u}_\perp + \tilde{u}_\parallel|\geq |\tilde{u}_\perp| - |v+ \tilde{u}_\parallel|\geq |\tilde{u}_\perp| - 2^{-n/2}|v-\tilde{u}_\perp|\geq \frac{1}{2}|\tilde{u}_\perp|.
    \end{align*}
    
    Applying all these bounds, we have
    \begin{align}\label{5_9:J_12_2}
        \begin{split}
            &\int_{2^{-n/2+2}|v|\leq |\tilde{u}_\perp|\leq |v|/4} d\tilde{u}_\perp\,\int_{E_{\tilde{u}_\perp}} d\tilde{u}_\parallel\,\frac{\left(|\tilde{u}_\parallel|^2+|\tilde{u}_\perp|^2\right)^{\frac{\kappa-1}{2}}}{|\tilde{u}_\perp||\tilde{u}_\parallel|^{N-3}}\\
            &\qquad\times \frac{1}{\left(1+|v-\tilde{u}_\perp+\tilde{u}_\parallel|^2\right)^\delta} e^{-\frac{\alpha \beta/4}{2^{(1-\beta/2)(n+1)}}|v+\tilde{u}_\parallel|^\beta}\mathbf{1}_{2^{-n-1}|v-\tilde{u}_\perp|^2\leq |v+\tilde{u}_\parallel|^2\leq 2^{-n}|v-\tilde{u}_\perp|^2}\\
            &\leq \frac{C_{\delta, \kappa}e^{-\frac{\alpha \beta/4}{2^{n+1+\beta}}|v|^\beta}}{|v|^{1-\kappa}|v|^{N-3}}\int_{2^{-n/2+2}|v|\leq |\tilde{u}_\perp|\leq |v|/4} d\tilde{u}_\perp\,\frac{1}{|\tilde{u}_\perp|\left(1+|\tilde{u}_\perp|^2\right)^\delta} \int_{E_{\tilde{u}_\perp}} d\tilde{u}_\parallel\,\mathbf{1}_{|v+\tilde{u}_\parallel|^2\leq 2^{-n+2}|v|^2}\\
            &\leq \frac{C_{N, \delta, \kappa}e^{-\frac{\alpha \beta/4}{2^{n+1+\beta}}|v|^\beta}}{|v|^{N-2-\kappa}}(2^{-n/2}|v|)^{N-1}\int_{\{2^{-n/2+2}|v|\leq |\tilde{u}_\perp|\leq |v|/4\}\cap C_{v, 2^{-n/2+1}|v|}} d\tilde{u}_\perp\,\frac{1}{|\tilde{u}_\perp|\left(1+|\tilde{u}_\perp|^2\right)^\delta}\\
            &\leq \frac{C_{N, \delta, \kappa}e^{-\frac{\alpha \beta/4}{2^{n+1+\beta}}|v|^\beta}}{|v|^{N-1-\kappa}}(2^{-n/2}|v|)^N\int_{2^{-n/2+2}|v|}^{|v|/4} dr\,\frac{r^{N-2}}{\left(1+r^2\right)^\delta}.
        \end{split}
    \end{align}
    In the last two lines, we used Lemma \ref{lem:cone2} for $a=v$, $x = \tilde{u}_\perp$, $y = \tilde{u}_\parallel$, and $R = 2^{-n/2+1}|v|$. For $2\delta>N+2$, the last integral is bounded by
    \begin{align*}
        \int_{2^{-n/2+2}|v|}^{|v|/4} dr\,\frac{r^{N-2}}{\left(1+r^2\right)^\delta}&\leq C_{N,\delta}\times \begin{dcases}
            1 & 2^{-n/2}|v|\leq 1,\\
            (2^{-n/2}|v|)^{N-1-2\delta} & 2^{-n/2}|v|\geq 1.
        \end{dcases}
    \end{align*}
    
    If $|v|\leq 1$, the last line of \eqref{5_9:J_12_2} is bounded as
    \begin{align*}
        \sum_{n=16}^\infty\frac{1}{|v|^{N-1-\kappa}}(2^{-n/2}|v|)^N\int_{2^{-n/2+2}|v|}^{|v|/4} dr\,\frac{r^{N-2}}{\left(1+r^2\right)^\delta}&\leq C_{N,\delta}\sum_{n=16}^\infty\frac{1}{|v|^{N-1-\kappa}}(2^{-n/2}|v|)^N\\
        &\leq C_{N,\delta} |v|^{N+\kappa}\sum_{n=16}^\infty 2^{-Nn/2}.
    \end{align*}
    Since $-N<\kappa\leq 1$, the summation is well-bounded. Therefore, we can assume $|v|\geq 1$ and choose $M\in \mathbb{N}\cup\{0\}$ such that $2^M\leq |v|\leq 2^{M+1}$. In contrast to the other cases, the summation $\sum_{n=16}^\infty$ of \eqref{5_9:J_12_2} requires some tricky steps. For this reason, we split $N=2$ and $N\geq 3$ cases.
    
    When $N=2$, we divide into two cases for $n$: $16\leq n\leq 2M$ and $n\geq 2M+1$. When $16\leq n\leq 2M$, as $2^{-n/2}|v|\geq 1$, we have
    \begin{align*}
        \sum_{n=16}^{2M}\frac{1}{|v|^{1-\kappa}}(2^{-n/2}|v|)^2\int_{2^{-n/2+2}|v|}^{|v|/4} dr\,\frac{1}{\left(1+r^2\right)^\delta}\leq C_\delta\sum_{n=16}^{2M}\frac{1}{|v|^{1-\kappa}}(2^{-n/2}|v|)^{3-2\delta}.
    \end{align*}
    As $2\delta>N+1=3$, the summation is given by
    \begin{align*}
        |v|^{\kappa +2-2\delta}\sum_{n=16}^{2M}2^{\frac{2\delta-3}{2}n}\leq C_\delta |v|^{\kappa +2-2\delta}|v|^{2\delta-3} = C_\delta |v|^{\kappa -1}.
    \end{align*}
    So, we get
    \begin{align*}
        \sum_{n=16}^{2M}\frac{1}{|v|^{1-\kappa}}(2^{-n/2}|v|)^2\int_{2^{-n/2+2}|v|}^{|v|/4} dr\,\frac{1}{\left(1+r^2\right)^\delta}\leq \frac{C_\delta}{|v|^{1-\kappa}}.
    \end{align*}
    If $n\geq 2M+1$, then
    \begin{align}\label{5_9:J_12_3_N=2}
        \sum_{n=2M+1}^\infty\frac{1}{|v|^{1-\kappa}}(2^{-n/2}|v|)^2\int_{2^{-n/2+2}|v|}^{|v|/4} dr\,\frac{1}{\left(1+r^2\right)^\delta}\leq C_\delta\sum_{n=2M+1}^\infty\frac{1}{|v|^{1-\kappa}}(2^{-n/2}|v|)^2\leq \frac{C_\delta}{|v|^{1-\kappa}}.
    \end{align}
    Thus, we have
    \begin{align*}
        \sum_{n=16}^\infty\frac{1}{|v|^{1-\kappa}}(2^{-n/2}|v|)^2\int_{2^{-n/2+2}|v|}^{|v|/4} dr\,\frac{1}{\left(1+r^2\right)^\delta}\leq \frac{C_\delta}{(1+|v|)^{1-\kappa}}.
    \end{align*}
    for $N=2$.

    Next, we choose $N\geq 3$. In this case, we divide into three cases for $n$: $16\leq n\leq \lfloor \beta M\rfloor$, $\lfloor \beta M\rfloor+1\leq n\leq 2M$, and $n\geq 2M+1$. When $16\leq n\leq \lfloor \beta M\rfloor$, we write
    \begin{align*}
        &\sum_{n=16}^{\lfloor \beta M\rfloor}\frac{(2^{-n/2}|v|)^N}{|v|^{N-1-\kappa}}e^{-\frac{\alpha \beta/4}{2^{n+1+\beta}}|v|^\beta}\int_{2^{-n/2+2}|v|}^{|v|/4} dr\,\frac{r^{N-2}}{\left(1+r^2\right)^\delta}\\
        &\leq C_{N,\delta}\sum_{n=16}^{\lfloor \beta M\rfloor}\frac{(2^{-n/2}|v|)^{2N-1-2\delta}}{|v|^{N-1-\kappa}}e^{-\frac{\alpha \beta/4}{2^{n+1+\beta}}|v|^\beta}.
    \end{align*}
    In this case, we use the exponential term. From $x^{2N} e^{-x}\leq C_N$ for $x>0$, we have the summation is attained at $n = \lfloor \beta M\rfloor$, so
    \begin{align*}
        \sum_{n=16}^{\lfloor \beta M\rfloor}\frac{(2^{-n/2}|v|)^{2N-1-2\delta}}{|v|^{N-1-\kappa}}e^{-\frac{\alpha \beta/4}{2^{n+1+\beta}}|v|^\beta}&\leq \sum_{n=16}^{\lfloor \beta M\rfloor}\frac{(2^{-n/2}|v|)^{2N-1-2\delta}}{|v|^{N-1-\kappa}}\left(\frac{2^{n+1+\beta}}{(\alpha \beta/4)|v|^\beta}\right)^{2N}\\
        &\leq C_{\alpha,\beta, N}\sum_{n=16}^{\lfloor \beta M\rfloor}|v|^{N+\kappa-2\delta - 2\beta N} 2^{\frac{1+2\delta}{2}n}.
    \end{align*}
    The summation is attained at $n=\lfloor \beta M\rfloor$, so it is bounded by
    \begin{align*}
        \sum_{n=16}^{\lfloor \beta M\rfloor}\frac{(2^{-n/2}|v|)^{2N-1-2\delta}}{|v|^{N-1-\kappa}}e^{-\frac{\alpha \beta/4}{2^{n+1+\beta}}|v|^\beta}\leq C_{\alpha,\beta,\delta}|v|^{(\delta-\frac{N+1}{2})(\beta-2) - \frac{\beta}{2}N -1+\kappa + \beta}.
    \end{align*}
    Since $2\delta>N+1$ and $N\geq 3$, we have
    \begin{align}\label{5_9:J_12_3_N=3_1}
        \sum_{n=16}^{\lfloor \beta M\rfloor}\frac{(2^{-n/2}|v|)^{2N-1-2\delta}}{|v|^{N-1-\kappa}}e^{-\frac{\alpha \beta/4}{2^{n+1+\beta}}|v|^\beta}\leq C_{N,\delta}|v|^{-\frac{\beta}{2}N + \beta - 1+\kappa}\leq C_{N,\delta}|v|^{-1+\kappa}.
    \end{align}
    
    When $\lfloor \beta M\rfloor+1\leq n\leq 2M$, the summation is given by
    \begin{align*}
        \sum_{n=\lfloor \beta M\rfloor+1}^{2M}\frac{(2^{-n/2}|v|)^N}{|v|^{N-1-\kappa}}e^{-\frac{\alpha \beta/4}{2^{n+1+\beta}}|v|^\beta}\int_{2^{-n/2+2}|v|}^{|v|/4} dr\,\frac{r^{N-2}}{\left(1+r^2\right)^\delta}\leq C_{N,\delta}\sum_{n=\lfloor \beta M\rfloor+1}^{2M}\frac{(2^{-n/2}|v|)^{2N-1-2\delta}}{|v|^{N-1-\kappa}}.
    \end{align*}
    If $2N-1-2\delta<0$, then the summation is achieved at $n = \lfloor \beta M\rfloor+1$, so
    \begin{align*}
        \sum_{n=\lfloor \beta M\rfloor+1}^{2M}\frac{(2^{-n/2}|v|)^{2N-1-2\delta}}{|v|^{N-1-\kappa}}\leq C_{N,\delta}|v|^{(\delta-\frac{N+1}{2})(\beta-2) - \frac{\beta}{2}N -1+\kappa + \beta}\leq C_{N,\delta}|v|^{-\frac{\beta}{2}N + \beta - 1+\kappa}.
    \end{align*}
    If $2N-1-2\delta>0$, then the summation is attained at $n = 2M$, so
    \begin{align*}
        \sum_{n=\lfloor \beta M\rfloor+1}^{2M}\frac{(2^{-n/2}|v|)^{2N-1-2\delta}}{|v|^{N-1-\kappa}}\leq \frac{C_{N,\delta}}{|v|^{N-1-\kappa}}.
    \end{align*}
    Finally, when $2N-1-2\delta = 0$, then
    \begin{align*}
        \sum_{n=\lfloor \beta M\rfloor+1}^{2M}\frac{(2^{-n/2}|v|)^{2N-1-2\delta}}{|v|^{N-1-\kappa}}\leq (2M-\lfloor \beta M\rfloor)|v|^{-N+1+\kappa}\leq C(2-\beta)|v|^{-N+1+\kappa}\ln(|v|+1).
    \end{align*}
    Merging all the possible cases for $\delta$, we have
    \begin{align}\label{5_9:J_12_3_N=3_2}
        \begin{split}
            \sum_{n=\lfloor \beta M\rfloor+1}^{2M}\frac{(2^{-n/2}|v|)^{2N-1-2\delta}}{|v|^{N-1-\kappa}}&\leq C_{N,\delta}\max\left\{|v|^{-\frac{\beta}{2}N + \beta - 1+\kappa}, (2-\beta)|v|^{-N+1+\kappa}\ln(|v|+1)\right\}\\
            &\leq \frac{C_{N,\delta}}{(1+|v|)^{1-\kappa}}.
        \end{split}
    \end{align}
    as $N\geq 3$.

    The remaining case corresponds to $n\geq 2M+1$. Since the integral is bounded by a constant, we estimate 
    \begin{align}\label{5_9:J_12_3_N=3_3}
        \begin{split}
            &\sum_{n=2M+1}^\infty\frac{(2^{-n/2}|v|)^N}{|v|^{N-1-\kappa}}\int_{2^{-n/2+2}|v|}^{|v|/4} dr\,\frac{r^{N-2}}{\left(1+r^2\right)^\delta}\\
            &\leq C_{N,\delta}\sum_{n=2M+1}^\infty 2^{-\frac{N}{2}n}|v|^{1+\kappa}\\
            &\leq C_{N,\delta} 2^{-NM}|v|^{1+\kappa}\leq C_{N,\delta} |v|^{-N+1+\kappa}.
        \end{split}
    \end{align}
    
    Combining all the estimates \eqref{5_9:J_12_3_N=2} for $N=2$ and \eqref{5_9:J_12_3_N=3_1}, \eqref{5_9:J_12_3_N=3_2}, and \eqref{5_9:J_12_3_N=3_3} for $N\geq 3$, we have
    \begin{align}\label{5_9:J_12_5}
        \begin{split}
            &\sum_{n=16}^\infty\int_{2^{-n/2+2}|v|\leq |\tilde{u}_\perp|\leq |v|/4} d\tilde{u}_\perp\,\int_{E_{\tilde{u}_\perp}} d\tilde{u}_\parallel\,\frac{\left(|\tilde{u}_\parallel|^2+|\tilde{u}_\perp|^2\right)^{\frac{\kappa-1}{2}}}{|\tilde{u}_\perp||\tilde{u}_\parallel|^{N-3}}\\
            &\qquad\times \frac{1}{\left(1+|v-\tilde{u}_\perp+\tilde{u}_\parallel|^2\right)^\delta} e^{-\frac{\alpha \beta/4}{2^{(1-\beta/2)(n+1)}}|v+\tilde{u}_\parallel|^\beta}\mathbf{1}_{2^{-n-1}|v-\tilde{u}_\perp|^2\leq |v+\tilde{u}_\parallel|^2\leq 2^{-n}|v-\tilde{u}_\perp|^2}\\
            &\leq C_{N, \delta, \kappa}\frac{1}{(1+|v|)^{1-\kappa}}.
        \end{split}
    \end{align}
    for $N\geq 2$. It gives a uniform upper bound.\\
    
    \noindent \textit{(ii)-3 The case $|\tilde{u}_\perp|\leq 2^{-n/2+2}|v|$.}
    
    One can readily check that we can import the estimates \eqref{5_9:J_12_2_1} and \eqref{5_9:J_12_2_2} in this case. Since $|v-\tilde{u}_\perp|\leq 2|v|$, we bound the integral as follows.
    \begin{align*}
        &\int_{E_{\tilde{u}_\perp}} d\tilde{u}_\parallel\,\frac{1}{|\tilde{u}_\parallel|^{N-3}}\frac{1}{\left(1+|v-\tilde{u}_\perp+\tilde{u}_\parallel|^2\right)^\delta} \mathbf{1}_{2^{-n-1}|v-\tilde{u}_\perp|^2\leq |v+\tilde{u}_\parallel|^2\leq 2^{-n}|v-\tilde{u}_\perp|^2}\\
        &\leq \frac{C_N}{|v|^{N-3}}\int_{E_{\tilde{u}_\perp}} d\tilde{u}_\parallel\,\frac{1}{\left(1+|v-\tilde{u}_\perp+\tilde{u}_\parallel|^2\right)^\delta} \mathbf{1}_{|v+\tilde{u}_\parallel|^2\leq 2^{-n+2}|v-\tilde{u}_\perp|^2}\\
        &\leq \frac{C_N}{|v|^{N-3}}\int_{E_{\tilde{u}_\perp}} d\tilde{u}_\parallel\,\frac{1}{\left(1+|\tilde{u}_\parallel|^2\right)^\delta}\mathbf{1}_{|\tilde{u}_\parallel|^2\leq 2^{-n+2}|v|^2}\\
        &\leq \frac{C_N}{|v|^{N-3}}\int_0^{2^{-n/2+1}|v|} dr\,\frac{r^{N-2}}{\left(1+r^2\right)^\delta}\\
        &\leq \frac{C_N}{|v|^{N-3}}\times\begin{dcases}
            \int_0^2 dr\,r^{N-2} + \int_2^\infty dr\,r^{N-2-2\delta} & 2^{-n/2}|v|\geq 1\\
            \int_0^{2^{-n/2+1}|v|} dr\,r^{N-2} & 2^{-n/2}|v|\leq 1
        \end{dcases}\\
        &\leq \frac{C_{N,\delta}}{|v|^{N-3}}\times\begin{dcases}
            1 & 2^{-n/2}|v|\geq 1,\\
            (2^{-n/2}|v|)^{N-1} & 2^{-n/2}|v|\leq 1.
        \end{dcases}
    \end{align*}
    From the second to the third line, we used the Hardy-Littlewood inequality for $\tilde{u}_\parallel$. The condition $2\delta>N+1$ implies that the final integral is bounded by a constant for $2^{-n/2}|v|\geq 1$.
    
    Therefore, using Lemma \ref{lem:cone2},
    \begin{align*}
        &\int_{|\tilde{u}_\perp|\leq 2^{-n/2+2}|v|} d\tilde{u}_\perp\,\int_{E_{\tilde{u}_\perp}} d\tilde{u}_\parallel\,\frac{\left(|\tilde{u}_\parallel|^2+|\tilde{u}_\perp|^2\right)^{\frac{\kappa-1}{2}}}{|\tilde{u}_\perp||\tilde{u}_\parallel|^{N-3}}\\
        &\qquad\times \frac{1}{\left(1+|v-\tilde{u}_\perp+\tilde{u}_\parallel|^2\right)^\delta} e^{-\frac{\alpha \beta/4}{2^{(1-\beta/2)(n+1)}}|v+\tilde{u}_\parallel|^\beta}\mathbf{1}_{2^{-n-1}|v-\tilde{u}_\perp|^2\leq |v+\tilde{u}_\parallel|^2\leq 2^{-n}|v-\tilde{u}_\perp|^2}\\
        &\leq \frac{C_{N,\delta, \kappa} e^{-\frac{\alpha \beta/4}{2^{n+1+\beta}}|v|^\beta}}{|v|^{N-2-\kappa}}\int_{\{|\tilde{u}_\perp|\leq 2^{-n/2+2}|v|\}\cap C_{v, 2^{-n/2+1}|v|}} d\tilde{u}_\perp\,\frac{1}{|\tilde{u}_\perp|}\times\begin{dcases}
            1 & 2^{-n/2}|v|\geq 1\\
            (2^{-n/2}|v|)^{N-1} & 2^{-n/2}|v|\leq 1
        \end{dcases}\\
        &\leq \frac{C_{N,\delta, \kappa} e^{-\frac{\alpha \beta/4}{2^{n+1+\beta}}|v|^\beta}}{|v|^{N-2-\kappa}}2^{-n/2} (2^{-n/2}|v|)^{N-1}\times\begin{dcases}
            1 & 2^{-n/2}|v|\geq 1,\\
            (2^{-n/2}|v|)^{N-1} & 2^{-n/2}|v|\leq 1.
        \end{dcases}
    \end{align*}
    
    If $|v|\leq 1$, then
    \begin{align}\label{5_9:J_12_6}
        \begin{split}
            &\sum_{n=16}^\infty \frac{e^{-\frac{\alpha \beta/4}{2^{n+1+\beta}}|v|^\beta}}{|v|^{N-2-\kappa}}2^{-n/2} (2^{-n/2}|v|)^{N-1}(2^{-n/2}|v|)^{N-1}\\
            &\leq \sum_{n=16}^\infty (2^{-n/2})^{2N-1}|v|^{N+\kappa}<\infty
        \end{split}
    \end{align}
    as $N+\kappa>0$.
    
    If $|v|\geq 1$, we follow the argument from $I_{2,2}$ bound of Lemma \ref{lem:int_bound4}. First, choose $M\in \mathbb{N}\cup\{0\}$ satisfying $2^M\leq |v|\leq 2^{M+1}$. Following \eqref{5_5:I_22_8}, we have
    \begin{align}\label{5_9:J_12_7}
        \begin{split}
            \sum_{n=16}^{\lfloor \beta M\rfloor} \frac{e^{-\frac{\alpha \beta/4}{2^{n+1+\beta}}|v|^\beta}}{|v|^{N-2-\kappa}}2^{-n/2}(2^{-n/2}|v|)^{N-1}&= |v|^{1+\kappa}\sum_{n=16}^{\lfloor \beta M\rfloor} e^{-\frac{\alpha \beta/4}{2^{n+1+\beta}}|v|^\beta} (2^{-n/2})^N\\
            &\leq C_{\alpha,\beta,N}|v|^{1+\kappa - \beta N/2}.
        \end{split}
    \end{align}
    For the remaining summation, we have
    \begin{align}\label{5_9:J_12_8}
        \sum_{n=\lfloor \beta M\rfloor+1}^\infty \frac{e^{-\frac{\alpha \beta/4}{2^{n+1+\beta}}|v|^\beta}}{|v|^{N-2-\kappa}}2^{-n/2}(2^{-n/2}|v|)^{N-1}\leq |v|^{1+\kappa}\sum_{n=\lfloor \beta M\rfloor+1}^\infty (2^{-n/2})^N\leq C_N |v|^{1+\kappa - \beta N/2}.
    \end{align}
    \eqref{5_9:J_12_7} and \eqref{5_9:J_12_8} are bounded if $\beta \geq \frac{2(1+\kappa)}{N}$. Note that it is satisfied by \eqref{index:weight}.
    
    Finally, from \eqref{5_9:J_12_6}, \eqref{5_9:J_12_7}, and \eqref{5_9:J_12_8}, we obtain
    \begin{align}\label{5_9:J_12_9}
        \begin{split}
            &\int_{|\tilde{u}_\perp|\leq 2^{-n/2+2}|v|} d\tilde{u}_\perp\,\int_{E_{\tilde{u}_\perp}} d\tilde{u}_\parallel\,\frac{\left(|\tilde{u}_\parallel|^2+|\tilde{u}_\perp|^2\right)^{\frac{\kappa-1}{2}}}{|\tilde{u}_\perp||\tilde{u}_\parallel|^{N-3}}\\
            &\qquad\times \frac{1}{\left(1+|v-\tilde{u}_\perp+\tilde{u}_\parallel|^2\right)^\delta} e^{-\frac{\alpha \beta/4}{2^{(1-\beta/2)(n+1)}}|v+\tilde{u}_\parallel|^\beta}\mathbf{1}_{2^{-n-1}|v-\tilde{u}_\perp|^2\leq |v+\tilde{u}_\parallel|^2\leq 2^{-n}|v-\tilde{u}_\perp|^2}\\
            &\leq C_{\alpha,\beta, N, \delta, \kappa}\frac{1}{(1+|v|)^{1+\kappa -\beta N/2}}.
        \end{split}
    \end{align}
    (For this case, we additionally imposed the condition $1+\kappa\leq \beta$ when $N=2$.)
    
    Combining \eqref{5_9:J_12_1}, \eqref{5_9:J_12_5}, and \eqref{5_9:J_12_9} under \eqref{index:weight}, we conclude that there exists a constant $C_{\alpha,\beta, N, \delta, \kappa}$ such that
    \begin{align*}
        J_{1,2}\leq C_{\alpha,\beta,N, \delta, \kappa},
    \end{align*}
    and it implies
    \begin{align}\label{5_9:I_12}
        I_{1,2}\leq C_{\alpha,\beta,N, \delta, \kappa}\|w_\delta(v)f(x,v)\|_{L^\infty_{x,v}}\|w_\delta(v)g(x,v)\|_{L^1_vL^p_x}.
    \end{align}\\
    
    \noindent \textit{(iii) Estimate of $I_{2,2}$.}
    
    As in the \eqref{5_9:J_1_2}, we calculate $I_{2,2}$ as follows:
    \begin{align*}
        I_{2,2}&=\sum_{n=16}^\infty \int_{\mathbb{R}^N}dv\,\int d\tilde{u}_\parallel\,\int_{E_{\tilde{u}_\parallel}} d\tilde{u}_\perp\,\mathbf{1}_{2^{-n-1}|v|^2\leq |v+\tilde{u}_\parallel+\tilde{u}_\perp|^2\leq 2^{-n}|v|^2}\\
        &\qquad \times \left(\int_{\Omega}\left[w_\delta(v+\tilde{u}_\parallel)f(x,v+\tilde{u}_\parallel)w_\delta(v+\tilde{u}_\perp)g(x,v+\tilde{u}_\perp)\right]^p\,dx\right)^{1/p}\\
        &\qquad\times\frac{\left(|\tilde{u}_\parallel|^2+|\tilde{u}_\perp|^2\right)^{\frac{\kappa-1}{2}}}{|\tilde{u}_\parallel|^{N-2}}\frac{1}{\left(1+|v+\tilde{u}_\perp|^2\right)^\delta} e^{-\frac{\alpha \beta/4}{2^{(1-\beta/2)(n+1)}}|v+\tilde{u}_\perp+\tilde{u}_\parallel|^\beta}\mathbf{1}_{2^{-n-1}|v|^2\leq |v+\tilde{u}_\perp+\tilde{u}_\parallel|^2\leq 2^{-n}|v|^2}\\
        &\leq\sum_{n=16}^\infty \|w_\delta(v)g(x,v)\|_{L^\infty_{x,v}}\|w_\delta(v)f(x,v)\|_{L^1_vL^p_x}\sup_v\int d\tilde{u}_\parallel\,\int_{E_{\tilde{u}_\parallel}} d\tilde{u}_\perp\,\\
        &\qquad\times\frac{\left(|\tilde{u}_\parallel|^2+|\tilde{u}_\perp|^2\right)^{\frac{\kappa-1}{2}}}{|\tilde{u}_\parallel|^{N-2}}\frac{1}{\left(1+|v+\tilde{u}_\perp-\tilde{u}_\parallel|^2\right)^\delta}e^{-\frac{\alpha \beta/4}{2^{(1-\beta/2)(n+1)}}|v+\tilde{u}_\perp|^\beta}\mathbf{1}_{2^{-n-1}|v-\tilde{u}_\parallel|^2\leq |v+\tilde{u}_\perp|^2\leq 2^{-n}|v-\tilde{u}_\parallel|^2}\\
        &=\sum_{n=16}^\infty \|w_\delta(v)g(x,v)\|_{L^\infty_{x,v}}\|w_\delta(v)f(x,v)\|_{L^1_vL^p_x}\sup_v\int d\tilde{u}_\parallel\,\int_{E_{\tilde{u}_\parallel}} d\tilde{u}_\perp\,\\
        &\qquad\times\frac{\left(|\tilde{u}_\parallel|^2+|\tilde{u}_\perp|^2\right)^{\frac{\kappa-1}{2}}}{|\tilde{u}_\parallel|^{N-2}}\frac{1}{\left(1+|v+\tilde{u}_\perp-\tilde{u}_\parallel|^2\right)^\delta}e^{-\frac{\alpha \beta/4}{2^{(1-\beta/2)(n+1)}}|v+\tilde{u}_\perp|^\beta}\mathbf{1}_{2^{-n-1}|v-\tilde{u}_\parallel|^2\leq |v+\tilde{u}_\perp|^2\leq 2^{-n}|v-\tilde{u}_\parallel|^2}.
    \end{align*}
    
    We define
    \begin{align*}
        J_{2,2}&=\sum_{n=16}^\infty \int d\tilde{u}_\parallel\,\int_{E_{\tilde{u}_\parallel}} d\tilde{u}_\perp\,\\
        &\qquad\times\frac{\left(|\tilde{u}_\parallel|^2+|\tilde{u}_\perp|^2\right)^{\frac{\kappa-1}{2}}}{|\tilde{u}_\parallel|^{N-2}}\frac{1}{\left(1+|v+\tilde{u}_\perp-\tilde{u}_\parallel|^2\right)^\delta}e^{-\frac{\alpha \beta/4}{2^{(1-\beta/2)(n+1)}}|v+\tilde{u}_\perp|^\beta}\mathbf{1}_{2^{-n-1}|v-\tilde{u}_\parallel|^2\leq |v+\tilde{u}_\perp|^2\leq 2^{-n}|v-\tilde{u}_\parallel|^2}.
    \end{align*}
    As before, we divide cases to bound $J_{2,2}$ uniformly about $v$.\\
    
    \noindent \textit{(iii)-1 The case $|\tilde{u}_\parallel|\geq |v|/4$.}
    
    In this case, as $n\geq 16$, $|v+\tilde{u}_\perp|\leq 2^{-8}|v-\tilde{u}_\parallel|\leq 2^{-8}\left(|\tilde{u}_\parallel|+|v|\right)\leq |\tilde{u}_\parallel|/4$.
    
    We use similar argument to \eqref{5_5:kappa_split} so that
    \begin{align*}
        \frac{\left(|\tilde{u}_\parallel|^2+|\tilde{u}_\perp|^2\right)^{\frac{\kappa-1}{2}}}{|\tilde{u}_\perp|^{N-2}} &\leq \frac{1}{|\tilde{u}_\perp|^{N-2 + \frac{2}{N+1}(1-\kappa)}}\frac{1}{|\tilde{u}_\parallel|^{\frac{N-1}{N+1}(1-\kappa)}}.
    \end{align*}
    Using these bounds, we have
    
    \begin{align*}
        &\int_{|\tilde{u}_\parallel| \geq |v|/4}d\tilde{u}_\parallel\,\int_{E_{\tilde{u}_\parallel}} d\tilde{u}_\perp\,\frac{\left(|\tilde{u}_\parallel|^2+|\tilde{u}_\perp|^2\right)^{\frac{\kappa-1}{2}}}{|\tilde{u}_\parallel|^{N-2}\left(1+|v+\tilde{u}_\perp-\tilde{u}_\parallel|^2\right)^\delta}\\
        &\qquad\times e^{-\frac{\alpha \beta/4}{2^{(1-\beta/2)(n+1)}}|v+\tilde{u}_\perp|^\beta}\mathbf{1}_{2^{-n-1}|v-\tilde{u}_\parallel|^2\leq |v+\tilde{u}_\perp|^2\leq 2^{-n}|v-\tilde{u}_\parallel|^2}\\
        &\leq \int_{|\tilde{u}_\parallel| \geq |v|/4}d\tilde{u}_\parallel\,\frac{1}{|\tilde{u}_\perp|^{N-2 + \frac{2}{N+1}(1-\kappa)}\left(1+(|\tilde{u}_\parallel| - |\tilde{u}_\parallel|/4)^2\right)^\delta}\\
        &\qquad\times \int_{E_{\tilde{u}_\parallel}} d\tilde{u}_\perp\,\frac{1}{|\tilde{u}_\parallel|^{\frac{N-1}{N+1}(1-\kappa)}} \mathbf{1}_{2^{-n-1}|v-\tilde{u}_\parallel|^2\leq |v+\tilde{u}_\perp|^2\leq 2^{-n}|v-\tilde{u}_\parallel|^2}\\
        &\leq C_\delta \int_{|\tilde{u}_\parallel| \geq |v|/4} d\tilde{u}_\parallel\,\frac{1}{|\tilde{u}_\parallel|^{N-2 + \frac{2}{N+1}(1-\kappa)}\left(1+|\tilde{u}_\parallel|^2\right)^\delta} \int_{E_{\tilde{u}_\parallel}} d\tilde{u}_\perp\,\frac{1}{|\tilde{u}_\perp|^{\frac{N-1}{N+1}(1-\kappa)}}\mathbf{1}_{|v+\tilde{u}_\perp|^2\leq 2^{-n}|v-\tilde{u}_\parallel|^2}.
    \end{align*}
    
    By the Hardy-Littlewood inequality, the last integral is bounded by
    \begin{align*}
        &\int_{E_{\tilde{u}_\parallel}} d\tilde{u}_\perp\,\frac{1}{|\tilde{u}_\perp|^{\frac{N-1}{N+1}(1-\kappa)}}\mathbf{1}_{|v+\tilde{u}_\perp|^2\leq 2^{-n}|v-\tilde{u}_\parallel|^2}\\
        &\leq \int_{E_{\tilde{u}_\parallel}} d\tilde{u}_\perp\,\frac{1}{|\tilde{u}_\perp|^{\frac{N-1}{N+1}(1-\kappa)}}\mathbf{1}_{|\tilde{u}_\perp|^2\leq 2^{-n}|v-\tilde{u}_\parallel|^2}\\
        &\leq C_N(2^{-n/2}|v-\tilde{u}_\parallel|)^{N-1 - \frac{N-1}{N+1}(1-\kappa)}\leq C_{N, \kappa}(2^{-n/2}|\tilde{u}_\parallel|)^{N-1 - \frac{N-1}{N+1}(1-\kappa)}.
    \end{align*}
    In the final line, we used $|\tilde{u}_\parallel|\geq |v|/4$.
    
    Therefore,
    \begin{align*}
        &\sum_{n=16}^\infty \int_{|\tilde{u}_\parallel| \geq |v|/4} d\tilde{u}_\parallel\,\frac{1}{|\tilde{u}_\parallel|^{N-2 + \frac{2}{N+1}(1-\kappa)}\left(1+|\tilde{u}_\parallel|^2\right)^\delta} \int_{E_{\tilde{u}_\parallel}} d\tilde{u}_\perp\,\frac{1}{|\tilde{u}_\perp|^{\frac{N-1}{N+1}(1-\kappa)}}\mathbf{1}_{|v+\tilde{u}_\perp|^2\leq 2^{-n}|v-\tilde{u}_\parallel|^2}\\
        &\leq C_{N, \kappa}\int_{|\tilde{u}_\parallel| \geq |v|/4} d\tilde{u}_\parallel\,\frac{1}{|\tilde{u}_\parallel|^{-\kappa}\left(1+|\tilde{u}_\parallel|^2\right)^\delta}\sum_{n=16}^\infty (2^{-n/2})^{N-1 - \frac{N-1}{N+1}(1-\kappa)}.
    \end{align*}
    Since $N-1 - \frac{N-1}{N+1}(1-\kappa)>0$, the summation is finite. Also, the final integral is finite as $2\delta>N + 1$ and $-N<\kappa\leq 1$; see Lemma \ref{lem:int_bound3}. Therefore, we get
    \begin{align}\label{5_9:J_22_1}
        \begin{split}
            &\sum_{n=16}^\infty \int_{|\tilde{u}_\parallel| \geq |v|/4}d\tilde{u}_\parallel\,\int_{E_{\tilde{u}_\parallel}} d\tilde{u}_\perp\,\frac{\left(|\tilde{u}_\parallel|^2+|\tilde{u}_\perp|^2\right)^{\frac{\kappa-1}{2}}}{|\tilde{u}_\parallel|^{N-2}\left(1+|v+\tilde{u}_\perp-\tilde{u}_\parallel|^2\right)^\delta}\\
            &\qquad\times e^{-\frac{\alpha \beta/4}{2^{(1-\beta/2)(n+1)}}|v+\tilde{u}_\perp|^\beta}\mathbf{1}_{2^{-n-1}|v-\tilde{u}_\parallel|^2\leq |v+\tilde{u}_\perp|^2\leq 2^{-n}|v-\tilde{u}_\parallel|^2}\\
            &\leq C_{N,\delta,\kappa}.
        \end{split}
    \end{align}\\
    
    \noindent \textit{(iii)-2 The case $2^{-n/2+2}|v|\leq |\tilde{u}_\parallel|\leq |v|/4$.}

    As in $J_{1,2}$, we first estimate lower and upper bounds of variables. First, $|v+\tilde{u}_\perp|\leq 2^{-n/2}|v-\tilde{u}_\parallel|$ and $|\tilde{u}_\parallel|\leq |v|/4$ imply
    \begin{align*}
        |v+\tilde{u}_\perp|\leq 2^{-n/2}|v-\tilde{u}_\parallel|\leq 2^{-n/2+1}|v|\leq |v|/2,
    \end{align*}
    so $|\tilde{u}_\perp|\geq |v|/2$. Therefore, we get
    \begin{align}\label{5_9:J_22_2_1}
        \left(|\tilde{u}_\parallel|^2+|\tilde{u}_\perp|^2\right)^{\frac{\kappa-1}{2}}\leq \frac{C_\kappa}{|v|^{1-\kappa}}.
    \end{align}
    Conversely, from $2^{-\frac{n+1}{2}}|v-\tilde{u}_\parallel|\leq |v+\tilde{u}_\perp|$ and $|\tilde{u}_\parallel|\leq |v|/4$, we have
    \begin{align}\label{5_9:J_22_2_2}
        e^{-\frac{\alpha \beta/4}{2^{(1-\beta/2)(n+1)}}|v+\tilde{u}_\perp|^\beta}\leq e^{-\frac{\alpha \beta/4}{2^{n+1}}|v-\tilde{u}_\parallel|^\beta}\leq e^{-\frac{\alpha \beta/4}{2^{n+1+\beta}}|v|^\beta}.
    \end{align}
    As $2^{-n/2}|v-\tilde{u}_\parallel|\leq 2^{-n/2+1}|v|\leq \frac{1}{2}|\tilde{u}_\parallel|$,
    \begin{align*}
        |v+\tilde{u}_\perp-\tilde{u}_\parallel|\geq |\tilde{u}_\parallel| - |v+\tilde{u}_\perp|\geq |\tilde{u}_\parallel| - 2^{-n/2}|v-\tilde{u}_\parallel|\geq \frac{1}{2}|\tilde{u}_\parallel|.
    \end{align*}
    
    Now, we obtain
    \begin{align*}
        &\int_{2^{-n/2+2}|v|\leq |\tilde{u}_\parallel|\leq |v|/4} d\tilde{u}_\parallel\,\int_{E_{\tilde{u}_\parallel}} d\tilde{u}_\perp\,\frac{\left(|\tilde{u}_\parallel|^2+|\tilde{u}_\perp|^2\right)^{\frac{\kappa-1}{2}}}{|\tilde{u}_\parallel|^{N-2}\left(1+|v+\tilde{u}_\perp-\tilde{u}_\parallel|^2\right)^\delta}\\
        &\qquad \times e^{-\frac{\alpha \beta/4}{2^{(1-\beta/2)(n+1)}}|v+\tilde{u}_\perp|^\beta}\mathbf{1}_{2^{-n-1}|v-\tilde{u}_\parallel|^2\leq |v+\tilde{u}_\perp|^2\leq 2^{-n}|v-\tilde{u}_\parallel|^2}\\
        &\leq \frac{C_{\delta, \kappa}}{|v|^{1-\kappa}}e^{-\frac{\alpha \beta/4}{2^{n+1+\beta}}|v|^\beta}\int_{2^{-n/2+2}|v|\leq |\tilde{u}_\parallel|\leq |v|/4} d\tilde{u}_\parallel\,\frac{1}{|\tilde{u}_\parallel|^{N-2}\left(1+|\tilde{u}_\parallel|^2\right)^\delta} \int_{E_{\tilde{u}_\parallel}} d\tilde{u}_\perp\,\mathbf{1}_{|v+\tilde{u}_\perp|^2\leq 2^{-n+2}|v|^2}\\
        &\leq \frac{C_{N, \delta, \kappa}}{|v|^{1-\kappa}}e^{-\frac{\alpha \beta/4}{2^{n+1+\beta}}|v|^\beta}(2^{-n/2}|v|)^{N-1}\int_{\{2^{-n/2+2}|v|\leq |\tilde{u}_\parallel|\leq |v|/4\}\cap C_{v, 2^{-n/2+1}|v|}} d\tilde{u}_\parallel\,\frac{1}{|\tilde{u}_\parallel|^{N-2}\left(1+|\tilde{u}_\parallel|^2\right)^\delta}\\
        &\leq \frac{C_{N, \delta, \kappa}}{|v|^{2-\kappa}}e^{-\frac{\alpha \beta/4}{2^{n+1+\beta}}|v|^\beta}(2^{-n/2}|v|)^{N}\int_{2^{-n/2+2}|v|}^{|v|/4} dr\,\frac{r}{\left(1+r^2\right)^\delta}.
    \end{align*}
    In the fourth line, we used Lemma \ref{lem:cone2} and similar logic in \eqref{int_perp}.
    
    This is the same as \eqref{5_5:I_22_2}. Therefore, we bound it by \eqref{5_5:I_22_6}:
    \begin{align}\label{5_9:J_22_2}
        \begin{split}
            &\sum_{n=16}^\infty \int_{2^{-n/2+2}|v|\leq |\tilde{u}_\parallel|\leq |v|/4} d\tilde{u}_\parallel\,\int_{E_{\tilde{u}_\parallel}} d\tilde{u}_\perp\,\frac{\left(|\tilde{u}_\parallel|^2+|\tilde{u}_\perp|^2\right)^{\frac{\kappa-1}{2}}}{|\tilde{u}_\parallel|^{N-2}\left(1+|v+\tilde{u}_\perp-\tilde{u}_\parallel|^2\right)^\delta}\\
            &\qquad \times e^{-\frac{\alpha \beta/4}{2^{(1-\beta/2)(n+1)}}|v+\tilde{u}_\perp|^\beta}\mathbf{1}_{2^{-n-1}|v-\tilde{u}_\parallel|^2\leq |v+\tilde{u}_\perp|^2\leq 2^{-n}|v-\tilde{u}_\parallel|^2}\\
            &\leq \frac{C_{\alpha,\beta, N,\delta, \kappa}}{(1+|v|)^{\beta-\kappa}}.
        \end{split}
    \end{align}\\
    
    \noindent \textit{(iii)-3 The case $|\tilde{u}_\parallel|\leq 2^{-n/2+2}|v|$.}

    We can again import \eqref{5_9:J_22_2_1} and \eqref{5_9:J_22_2_2} since $|\tilde{u}_\parallel|\leq |v|/4$. We also bound the step function by
    \begin{align*}
        \mathbf{1}_{2^{-n-1}|v-\tilde{u}_\parallel|^2\leq |v+\tilde{u}_\perp|^2\leq 2^{-n}|v-\tilde{u}_\parallel|^2}\leq \mathbf{1}_{|v+\tilde{u}_\perp|^2\leq 2^{-n+2}|v|^2}.
    \end{align*}
    Therefore, we first obtain
    \begin{align*}
        &\int_{|\tilde{u}_\parallel|\leq 2^{-n/2+2}|v|} d\tilde{u}_\parallel\,\int_{E_{\tilde{u}_\parallel}} d\tilde{u}_\perp\,\frac{\left(|\tilde{u}_\parallel|^2+|\tilde{u}_\perp|^2\right)^{\frac{\kappa-1}{2}}}{|\tilde{u}_\parallel|^{N-2}\left(1+|v+\tilde{u}_\perp-\tilde{u}_\parallel|^2\right)^\delta}\\
        &\qquad \times e^{-\frac{\alpha \beta/4}{2^{(1-\beta/2)(n+1)}}|v+\tilde{u}_\perp|^\beta}\mathbf{1}_{2^{-n-1}|v-\tilde{u}_\parallel|^2\leq |v+\tilde{u}_\perp|^2\leq 2^{-n}|v-\tilde{u}_\parallel|^2}\\
        &\leq \frac{C_\kappa e^{-\frac{\alpha \beta/4}{2^{n+1+\beta}}|v|^\beta}}{|v|^{1-\kappa}}\int_{|\tilde{u}_\parallel|\leq 2^{-n/2+2}|v|} d\tilde{u}_\parallel\,\frac{1}{|\tilde{u}_\parallel|^{N-2}}\int_{E_{\tilde{u}_\parallel}} d\tilde{u}_\perp\,\frac{1}{\left(1+|v+\tilde{u}_\perp-\tilde{u}_\parallel|^2\right)^\delta}\mathbf{1}_{|v+\tilde{u}_\perp|^2\leq 2^{-n+2}|v|^2}.
    \end{align*}
    
    We first bound the $\tilde{u}_\perp$ integral. By an application of the Hardy-Littlewood inequality, we have
    \begin{align*}
        &\int_{E_{\tilde{u}_\parallel}} d\tilde{u}_\perp\,\frac{1}{\left(1+|v+\tilde{u}_\perp-\tilde{u}_\parallel|^2\right)^\delta}\mathbf{1}_{|v+\tilde{u}_\perp|^2\leq 2^{-n+2}|v|^2}\\
        &\leq\int_{E_{\tilde{u}_\parallel}} d\tilde{u}_\perp\,\frac{1}{\left(1+|\tilde{u}_\perp|^2\right)^\delta}\mathbf{1}_{|\tilde{u}_\perp|^2\leq 2^{-n+2}|v|^2}=|\mathbb{S}^{N-2}|\int_0^{2^{-n/2+1}|v|} dr\,\frac{r^{N-2}}{\left(1+r^2\right)^\delta}\\
        &\leq |\mathbb{S}^{N-2}|\times \begin{dcases}
            \int_0^{2^{-n/2+1}|v|} dr\,r^{N-2} & 2^{-n/2}|v|\leq 1\\
            \left(\int_0^2 dr\,r^{N-2} + \int_2^\infty dr\,r^{N-2 - 2\delta}\right) & 2^{-n/2}|v|\geq 1
        \end{dcases}\\
        &\leq C_{N,\delta}\times \begin{dcases}
            (2^{-n/2+1}|v|)^{N-1} & 2^{-n/2}|v|\leq 1,\\
            1 & 2^{-n/2}|v|\geq 1.
        \end{dcases}
    \end{align*}
    In the last line, the integral is bounded by a constant as $2\delta>N+1$.
    
    Combining these estimates and Lemma \ref{lem:cone2}, we obtain
    \begin{align*}
        &\int_{|\tilde{u}_\parallel|\leq 2^{-n/2+2}|v|} d\tilde{u}_\parallel\,\int_{E_{\tilde{u}_\parallel}} d\tilde{u}_\perp\,\frac{\left(|\tilde{u}_\parallel|^2+|\tilde{u}_\perp|^2\right)^{\frac{\kappa-1}{2}}}{|\tilde{u}_\parallel|^{N-2}\left(1+|v+\tilde{u}_\perp-\tilde{u}_\parallel|^2\right)^\delta}\\
        &\qquad \times e^{-\frac{\alpha \beta/4}{2^{(1-\beta/2)(n+1)}}|v+\tilde{u}_\perp|^\beta}\mathbf{1}_{2^{-n-1}|v-\tilde{u}_\parallel|^2\leq |v+\tilde{u}_\perp|^2\leq 2^{-n}|v-\tilde{u}_\parallel|^2}\\
        &\leq \frac{C_{N,\delta, \kappa} e^{-\frac{\alpha \beta/4}{2^{n+1+\beta}}|v|^\beta}}{|v|^{1-\kappa}}\int_{\{|\tilde{u}_\parallel|\leq 2^{-n/2+2}|v|\}\cap C_{v, 2^{-n/2+1}|v|}} d\tilde{u}_\parallel\,\frac{1}{|\tilde{u}_\parallel|^{N-2}}\times \begin{dcases}
            (2^{-n/2+1}|v|)^{N-1} & 2^{-n/2}|v|\leq 1\\
            1 & 2^{-n/2}|v|\geq 1
        \end{dcases}\\
        &\leq \frac{C_{N,\kappa} e^{-\frac{\alpha \beta/4}{2^{n+1+\beta}}|v|^\beta}}{|v|^{2-\kappa}}(2^{-n/2}|v|)^3\times \begin{dcases}
            (2^{-n/2+1}|v|)^{N-1} & 2^{-n/2}|v|\leq 1,\\
            1 & 2^{-n/2}|v|\geq 1.
        \end{dcases}
    \end{align*}
    The final form is same as \eqref{5_5:I_22_7}, so we finally bound it by \eqref{5_5:I_22}:
    \begin{align}\label{5_9:J_22_3}
        \begin{split}
            &\int_{|\tilde{u}_\parallel|\leq 2^{-n/2+2}|v|} d\tilde{u}_\parallel\,\int_{E_{\tilde{u}_\parallel}} d\tilde{u}_\perp\,\frac{\left(|\tilde{u}_\parallel|^2+|\tilde{u}_\perp|^2\right)^{\frac{\kappa-1}{2}}}{|\tilde{u}_\parallel|^{N-2}\left(1+|v+\tilde{u}_\perp-\tilde{u}_\parallel|^2\right)^\delta}\\
            &\qquad \times e^{-\frac{\alpha \beta/4}{2^{(1-\beta/2)(n+1)}}|v+\tilde{u}_\perp|^\beta}\mathbf{1}_{2^{-n-1}|v-\tilde{u}_\parallel|^2\leq |v+\tilde{u}_\perp|^2\leq 2^{-n}|v-\tilde{u}_\parallel|^2}\\
            &\leq C_{\alpha,\beta,N,\delta, \kappa}\frac{1}{(1+|v|)^{\frac{3}{2}\beta - 1 - \kappa}}.
        \end{split}
    \end{align}
    \\
    
    Combining \eqref{5_9:J_22_1}, \eqref{5_9:J_22_2}, and \eqref{5_9:J_22_3} under \eqref{index:weight}, we conclude that there exists a constant $C_{\alpha,\beta,N, \delta, \kappa}$ such that
    \begin{align}\label{5_9:I_22}
        \begin{split}
            J_{2,2}&\leq C_{\alpha,\beta,N, \delta, \kappa},\\
            I_{2,2}&\leq C_{\alpha,\beta,N, \delta, \kappa}\|w_\delta(v)g(x,v)\|_{L^\infty_{x,v}}\|w_\delta(v)f(x,v)\|_{L^1_vL^p_x}.
        \end{split} 
    \end{align}
    
    Collecting all bounds \eqref{5_9:I_11}, \eqref{5_9:I_21}, \eqref{5_9:I_12}, and \eqref{5_9:I_22}, we finally reach
    \begin{align*}
        I_{1,1}+I_{1,2}&\leq C_{\alpha,\beta,N, \delta, \kappa}\|w_\delta(v)f(x,v)\|_{L^\infty_{x,v}}\|w_\delta(v)g(x,v)\|_{L^1_vL^p_x},\\
        I_{2,1}+I_{2,2}&\leq C_{\alpha,\beta,N, \delta, \kappa}\|w_\delta(v)f(x,v)\|_{L^1_vL^p_x}\|w_\delta(v)g(x,v)\|_{L^\infty_{x,v}}
    \end{align*}
    and get Lemma \ref{lem:Q_gainL1}.
\end{proof}

Using $L^1_v$ bound of $Q_{\text{gain}}$, we get $L^q_vL^p_x$ bound for all $1\leq q,p\leq \infty$.

\begin{lemma}\label{lem:Q_gainLp}
    Let us recall \eqref{def_w} and assume \eqref{index:weight}. When $N=2$, we impose one more condition $1+\kappa\leq \beta$. If $w_\delta(v)f(x,v)\in L^q_vL^p_x \cap L^\infty_{x,v}$, then we have
    \begin{align*}
        \left\|w_\delta(v) Q_{\text{gain}}(f,f)\right\|_{L^q_vL^p_x}\leq C_{\alpha,\beta,N, \delta, \kappa}\|w_\delta(v)f(x,v)\|_{L^\infty_{x,v}}\|w_\delta(v)f(x,v)\|_{L^q_vL^p_x}
    \end{align*}
    for all $1\leq p,q\leq \infty$.
\end{lemma}
\begin{proof}
    Lemma \ref{lem:Q_gain} proves
    \begin{align*}
        \|w_\delta(v)Q_1(f,g)\|_{L^\infty_v L^p_x}&\leq \|w_\delta(v)f\|_{L^\infty_{x,v}}\|w_\delta(v)g\|_{L^\infty_v L^p_x},\\
        \|w_\delta(v)Q_2(f,g)\|_{L^\infty_v L^p_x}&\leq \|w_\delta(v)f\|_{L^\infty_vL^p_x}\|w_\delta(v)g\|_{L^\infty_{x,v}},
    \end{align*}
    and $w_\delta(v)Q_i(f,g)\in L^\infty_{0,v}L^p_x$ for all $1\leq p\leq \infty$ when $\max\Big\{0, \frac{2}{N-1}\kappa, \frac{2}{3}(1+\kappa)\Big\}<\beta\leq 2$.
    Lemma \ref{lem:Q_gainL1} proves
    \begin{align*}
        \|w_\delta(v)Q_1(f,g)\|_{L^1_v L^p_x}&\leq \|w_\delta(v)f\|_{L^\infty_{x,v}}\|w_\delta(v)g\|_{L^1_v L^p_x},\\
        \|w_\delta(v)Q_2(f,g)\|_{L^1_v L^p_x}&\leq \|w_\delta(v)f\|_{L^1_vL^p_x}\|w_\delta(v)g\|_{L^\infty_{x,v}}
    \end{align*}
    for all $1\leq p\leq \infty$. Therefore, using Lemma \ref{Riesz-Thorin} (Riesz-Thorin interpolation) for each $Q_1$ and $Q_2$, we can extend these bounds for all $1\leq p,q\leq \infty$. Finally, putting $f = g \in L^\infty_{x,v}\cap L^q_vL^p_x$, we have
    \begin{align*}
        \|w_\delta(v)Q_{gain}(f,f)\|_{L^q_v L^p_x}&\leq \|w_\delta(v)Q_1(f,f)\|_{L^q_v L^p_x} + \|w_\delta(v)Q_2(f,f)\|_{L^q_v L^p_x}\\
        &\leq C\|w_\delta(v)f\|_{L^\infty_{x,v}} \|w_\delta(v)f\|_{L^q_v L^p_x}.
    \end{align*}
\end{proof}

Using the $Q_{\text{gain}}$ estimate for $L^q_vL^p_x$, we get the following \textit{a priori} estimate.

\begin{proposition}\label{prop:Q_gainLp}
    (\textit{A priori} estimate) Suppose $f$ satisfies
    \begin{align*}
        f(t,x,v)\leq f_0(x-vt,v) + \int_0^t Q_{\text{gain}}(f,f)(\tau,x-v(t-\tau),v)\,d\tau.
    \end{align*}
    If $w_\delta(v)f(t,x,v)\in L^q_vL^p_x\cap L^\infty_{x,v}$, then
    \begin{align*}
        \left\|w_\delta(v)f(t,x,v)\right\|_{L^q_vL^p_x}\leq \left\|w_\delta(v)f_0(x,v)\right\|_{L^q_vL^p_x} \exp\left(C\int_0^t \left\|w_\delta(v)f(\tau, x,v)\right\|_{L^\infty_{x,v}} \,d\tau\right).
    \end{align*}
\end{proposition}
\begin{proof}
    Having $L^\infty_{x,v}$ bound for $f(t,x,v)$,
    \begin{align*}
        \left\|w_\delta(v)f(t,x,v)\right\|_{L^q_vL^p_x}&\leq \left\|w_\delta(v)f_0(x-vt,v)\right\|_{L^q_vL^p_x} + \left\|\int_0^t w_\delta(v)Q_{\text{gain}}(f,f)(\tau, x-v(t-\tau), v)\,d\tau\right\|_{L^q_vL^p_x}\\
        &\leq \left\|w_\delta(v)f_0(x,v)\right\|_{L^q_vL^p_x} + \int_0^t \left\|w_\delta(v)Q_{\text{gain}}(f,f)(\tau, x-v(t-\tau), v)\right\|_{L^q_vL^p_x}\,d\tau\\
        &= \left\|w_\delta(v)f_0(x,v)\right\|_{L^q_vL^p_x} + \int_0^t \left\|w_\delta(v)Q_{\text{gain}}(f,f)(\tau, x, v)\right\|_{L^q_vL^p_x}\,d\tau\\
        &\leq \left\|w_\delta(v)f_0(x,v)\right\|_{L^q_vL^p_x} + C\int_0^t \|w_\delta(v)f\|_{L^\infty_{x,v}}(\tau)\|w_\delta(v)f\|_{L^q_vL^p_x}(\tau)\,d\tau.
    \end{align*}
    In the first to second line, we used Minkowski's integral inequality to interchange $d\tau$ integral with $dx$ and $dv$ integral. The final step uses Lemma \ref{lem:Q_gainLp}. By the Gr{\"o}nwall's inequality, we obtain the result.
\end{proof}

Using the \textit{a priori} estimate, we extend the result of Proposition \ref{prop:wellposed} into the general $L^q_vL^p_x$ case.
\begin{proof} [\textbf{Proof of Theorem \ref{thm:boltz}}]
    Proof of $L^{\infty}_{x,v}$ estimate is already obtained in Proposition \ref{prop:wellposed}. Applying Gr{\"o}nwall's inequality in the proof of Proposition \ref{prop:Q_gainLp} at the iteration step of Proposition \ref{prop:wellposed}, we can check that the solution $f$ satisfies $w_\delta(v)f(t,x,v)\in L^q_vL^p_x\cap L^\infty_{x,v}$ local in time and
    \begin{align*}
        \left\|w_\delta(v)f(t,x,v)\right\|_{L^q_vL^p_x}\leq \left\|w_\delta(v)f_0(x,v)\right\|_{L^q_vL^p_x} \exp\left(C\int_0^t \left\|w_\delta(v)f(\tau, x,v)\right\|_{L^\infty_{x,v}} \,d\tau\right).
    \end{align*}
    Using this \textit{a priori} estimate, we can extend $L^q_vL^p_x$ estimate until $\left\|w_\delta(v)f(\tau, x,v)\right\|_{L^\infty_{x,v}}$ is finite. Since we chose $T^*$ satisfying
    \begin{align*}
        \sup_{0\leq t\leq T^*}\left\|w_\delta(v)f(\tau, x,v)\right\|_{L^\infty_{x,v}}\leq 3\left\|w_\delta(v)f_0(x,v)\right\|_{L^\infty_{x,v}},
    \end{align*}
    we get
    \begin{align*}
        \left\|w_\delta(v)f(t,x,v)\right\|_{L^q_vL^p_x}\leq \left\|w_\delta(v)f_0(x,v)\right\|_{L^q_vL^p_x} \exp\left(3C\left\|w_\delta(v)f(0, x,v)\right\|_{L^\infty_{x,v}} t\right)
    \end{align*}
    for $t\in [0,T^*]$.

    The last statement of Theorem \ref{thm:boltz} is due to Minkowski's integral inequality.
\end{proof}

\section{Asymptotic blow-up of the Boltzmann equation}
In the previous section, we proved that any solution $f(t,x,v)$ of the Boltzmann equation has finite $\|w_{\alpha,\beta,\delta}(v)f(t,x,v)\|_{L^q_vL^p_x}$ at least for a finite time. However, it does not guarantee that the weighted norm is finite for all time. In fact, it is more plausible to say that the weighted norm blows up as $t\rightarrow\infty$ if the solution converges to a Maxwellian with low temperature as in the BGK case. In this section, we explicitly construct a simple example to demonstrate this idea.

We consider spatially homogeneous Boltzmann equation
\begin{align*}
    \begin{split}
        \partial_t f&=Q(f,f) \\
    \end{split}
\end{align*}
where $Q(f,f)$ is defined in \eqref{eq:Boltzmann}, and collision kernel is given by
\begin{align*}
    B(|v-u|,\cos\theta) = |v-u|^\kappa\cos\theta,
\end{align*}
where $\kappa \in (0,1]$. We further restrict $N=3$ for $v\in \mathbb{R}^N$. We provide the following long-time blow-up scenario of the Boltzmann equation in the exponential class.
\begin{theorem} \label{thm:3}
    For any $0<\alpha'\leq \alpha$ and $1\leq p\leq \infty$, there exists an initial data $f_0$ such that 
    \[
        \|e^{\alpha|v|^{2}}f_0 \|_{L^p_{v}} < \infty,
    \]
    but a unique global solution $f(t)$ satisfies
    \[
        \lim_{t\rightarrow\infty}\|e^{\alpha'|v|^{2}}f(t)\|_{L^p_{v}} = \infty.
    \]
\end{theorem}
\begin{proof}
    We choose the initial data
    \begin{align*}
        f^n_0(v) &= A_{n,p} e^{-\alpha|v|^2} \mathbf{1}_{n\leq |v|\leq n^2},
    \end{align*}
    where $A_{n,p}$ was defined in \eqref{2_2:def_Anp}; it was used in Proposition \ref{prop:2_2}. It is obvious that $f^n_0 \in L_v^1(\langle v^2 \rangle)\cap L_v^2$ for all $n$. From \cite{L1Mouhot} and \cite{MisWen}, we have a unique global in time solution $f^n(t,v)\in C([0,\infty);L_v^1(\langle v^2 \rangle))$ which converges to give Maxwellian exponentially, i.e.,
    \begin{align*}
        \left\|f^n - \mathcal{M}(f^n)\right\|_{L^1}\leq C_n e^{-\lambda t}
    \end{align*}
    for small enough fixed $\lambda$ and a constant $C_n$ depending on the collision kernel $B$, the mass, energy and $L^2$ norm of $f^n_0$, and $\lambda$.
    Applying Lemma \ref{lem:Tinfty_asym}, we obtain $T_n\rightarrow \infty$ as $n\rightarrow \infty$. 
    
    We define
    \begin{align*}
        g^n := f^n - \mathcal{M}(f^n).
    \end{align*}
    Since $g^n$ satisfies
    \[
        \|g^n(t,v)\|_{L_v^1}\leq C_n e^{-\lambda t},
    \]
    thanks to Chebyshev's inequality, we have
    \begin{align*}
        \big|\{ v\in\mathbb{R}^3\,\big|\, |g^n(t,v)|\geq C_ne^{-\lambda t}\}\big|\leq \frac{\|g^n(t)\|_{L^1}}{C_n e^{-\lambda t}}\leq 1.
    \end{align*}
    Let us denote the set
    \[
        A^n_t=\{v\in\mathbb{R}^3\,\big|\, |g(v)|\geq C_ne^{-\lambda t}\}
    \]
    for $t>t_0$, where $t_0$ satisfies $\left|\{v:e^{\alpha'|v|^2}\leq e^{\lambda t_0}\}\right|>1$. Now, we have
    \begin{align}
        &\left\|e^{\alpha'|v|^2}f^n(t)\right\|_{L^p_v}
        =\left\|e^{\alpha'|v|^2}\left\{\mathcal{M}(f^n) + g^n(t)\right\}\right\|_{L^p_v} \notag \\
        &\geq\left\|e^{\alpha'|v|^2}\left\{\mathcal{M}(f^n) + g^n(t)\right\}\mathbf{1}_{\left\{v:e^{\alpha'|v|^2}\leq e^{\lambda t}\right\}\setminus A^n_t} \right\|_{L^p_v}\notag \\
        &\geq\left\|\frac{\rho_n}{(2\pi T_n)^{3/2}} e^{(\alpha' -\frac{1}{2T_n})|v|^{2}}\mathbf{1}_{\left\{|v|\leq \sqrt{\frac{\lambda t}{\alpha'}}\right\}\setminus A^n_t}\right\|_{L^p_v}-\left\|e^{\alpha'|v|^2} g^n(t)\mathbf{1}_{\left\{|v|\leq \sqrt{\frac{\lambda t}{\alpha'}}\right\}\setminus A^n_t} \right\|_{L^p_v},\label{last}  
    \end{align}
    where $\rho_n$, $T_n$ are the mass and energy of $f^n_0$.
    
    Since $|g^n(t)|\leq C_n e^{-\lambda t}$ on $A^c_t$, we can compute the second term of \eqref{last} as
    \begin{align*}
        &\left\|e^{\alpha'|v|^2} g^n(t)\mathbf{1}_{\left\{|v|\leq \sqrt{\frac{\lambda t}{\alpha'}}\right\}\setminus A^n_t} \right\|_{L^p_v} \\
        &\leq \left(\left\|e^{\alpha'|v|^2} g^n(t)\mathbf{1}_{\left\{|v|\leq \sqrt{\frac{\lambda t}{\alpha'}}\right\}\setminus A^n_t} \right\|_{L^1_v}\right)^{1/p}\left(\left\|e^{\alpha'|v|^2} g^n(t)\mathbf{1}_{\left\{|v|\leq \sqrt{\frac{\lambda t}{\alpha'}}\right\}\setminus A^n_t}\right\|_{L^\infty_v}\right)^{1-1/p}\\
        &\leq \left(\left\|e^{\lambda t} g^n(t)\mathbf{1}_{\left\{|v|\leq \sqrt{\frac{\lambda t}{\alpha'}}\right\}\setminus A^n_t} \right\|_{L^1_v}\right)^{1/p}\left(C_n e^{\lambda t}e^{-\lambda t}\right)^{1-1/p}\\
        &\leq \left(C_n e^{\lambda t} e^{-\lambda t}\right)^{1/p}\left(C_n\right)^{1-1/p} = C_n
    \end{align*}
    for $1\leq p<\infty$. If we define $1/\infty = 0$, then this equation covers the $p=\infty$ case.
    
    On the other hand, we bound the first term of \eqref{last} from below as follows. First, choose $\sqrt{\frac{\lambda t}{\alpha'}}>2$. Since $|A^n_t|\leq 1$, \begin{align*}
        \left|\left\{v:\sqrt{\lambda t/\alpha'}-1\leq |v|\leq \sqrt{\lambda t/\alpha'}\right\}\setminus A_t^n\right|>0
    \end{align*}
    for all $t$ satisfying $\sqrt{\lambda t/\alpha'}>2$. It means that
    \begin{align} \label{last2}
        \left\|\frac{\rho_n}{(2\pi T_n)^{3/2}} e^{(\alpha' -\frac{1}{2T_n})|v|^{2}}\mathbf{1}_{\left\{|v|\leq \sqrt{\frac{\lambda t}{\alpha'}}\right\}\setminus A^n_t}\right\|_{L^p_v}&\geq \left\|\frac{\rho_n}{(2\pi T_n)^{3/2}} e^{(\alpha' -\frac{1}{2T_n})|v|^{2}}\mathbf{1}_{|v|\leq \sqrt{\frac{\lambda t}{\alpha'}} - 1}\right\|_{L^p_v}.
    \end{align}
    
    From \eqref{last} and \eqref{last2},
    \begin{align*}
        \|e^{\alpha'|v|^2}f^n(t)\big\|_{L^p_v}
        \geq \left\|\frac{\rho_n}{(2\pi T_n)^{3/2}} e^{(\alpha' -\frac{1}{2T_n})|v|^{2}}\mathbf{1}_{|v|\leq \sqrt{\frac{\lambda t}{\alpha'}} - 1}\right\|_{L^p_v} - C_n
    \end{align*}
    for $t>\max\{t_0, 4\alpha'/\lambda\}$. Since the temperature $T_n$ goes to $\infty$ as $n\rightarrow \infty$, choose large $n$ so that $\alpha' > \frac{1}{2T_n}$. Now, we obtain
    \begin{align*}
        \lim_{t\rightarrow \infty}\|e^{\alpha'|v|^2}f^n(t)\big\|_{L^p_{v}}=\infty.
    \end{align*}
\end{proof}
Theorem \ref{thm:3} again states that we can construct an ill-posedness result in the space $L^p_v(e^{\alpha'|v|^2})$ if there is no temperature ceiling for the initial $f_0$.\\

\noindent{\bf Data availability:} No data was used for the research described in the article.
\newline

\noindent{\bf Conflict of interest:} The authors declare that they have no conflict of interest.\newline

\noindent{\bf Acknowledgement}
The authors would like to thank Professor Xuwen Chen for the advice on references.
D. Lee and S. Park are supported by the National Research Foundation of Korea(NRF) grant funded by the Korea government(MSIT)(No.RS-2023-00212304 and No.RS-2023-00219980). Yun was supported by the National Research Foundation of Korea(NRF) grant funded by the Korea government(MSIT). (No.RS-2023-NR076676)

%
%
%
%

\bibliographystyle{amsplain}

\begin{thebibliography}{10}

\bibitem{ACGM} Alonso, Ricardo; Ca{\~n}izo, Jos{\'e} A.; Gamba, Irene; Mouhot, Cl{\'e}ment: A new approach to the creation and propagation of exponential moments in the Boltzmann equation. Comm. Partial Differential Equations {\bf38} (2013), no. 1, 155–169.

\bibitem{A-G} Alonso, Ricardo J.; Gamba, Irene M.: Propagation of $L^1$ and $L^\infty$ Maxwellian weighted bounds for derivatives of solutions to the homogeneous elastic Boltzmann equation. J. Math. Pures Appl. (9) {\bf89} (2008), no. 6, 575–595.

\bibitem{ACI} Andr\'easson, H., Calogero, S. and Illner, R. On blow-up for gain-term-only classical and relativistic Boltzmann equations. MATH METHOD APPL SCI. 27 (2004), Issue 18, 2231-2240.


\bibitem{A-P} Andries, P.,  Le Tallec, P., Perlat, J.-P., Perthame, B.: The Gaussian-BGK model of Boltzmann
equation with small Prandtl number. Eur. J. Mech. B Fluids {\bf19} (2000), no. 6, 813-830.	


\bibitem{Arsenio} Arsénio, Diogo: On the global existence of mild solutions to the Boltzmann equation for small data in  $L^D$.
Comm. Math. Phys. 302 (2011), no. 2, 453–476.

\bibitem{BGS} Bedrossian, J., Gualdani, M. and Snelson, S.: Non-existence of some approximately self-similar singularities for the Landau, Vlasov-Poisson-Landau, and Boltzmann equations. Transactions of the AMS. 375(3) 2187–2216, 2022.


\bibitem{Bobylev} Bobylev, A. V.:
Moment inequalities for the Boltzmann equation and applications to spatially homogeneous problems. J. Statist. Phys. {\bf88}(1997), no.5-6, 1183–1214.


\bibitem{BCRY}Boscarino, S., Cho, S.-Y., Russo, G., Yun, S.-B.:
Convergence estimates of a semi-Lagrangian scheme for the ellipsoidal BGK model for polyatomic
molecules. ESAIM: Mathematical Modelling and Numerical Analysis {\bf56} no.3, 893--942

\bibitem{BKLY} Bae, G.-C., Ko, G. Lee, D., Yun, S.-B.:
Large amplitude problem of BGK model: Relaxation to quadratic nonlinearity. To appear in SIAM J. Math. Anal.


\bibitem{Bello} Bellouquid, A.: Global existence and large-time behavior for BGK model for a gas with non-constant cross section. Transport Theory Statist. Phys. {\bf 32 } (2003) no. 2, 157-185.

\bibitem{BGCY} Bae, G.-C., Yun, S.-B.: Stationary quantum BGK model for bosons and fermions in a bounded interval. J. Stat. Phys. {\bf178} (2020), no. 4, 845–868.

\bibitem{Bang Y} Bang, J. and Yun, S.-B.: Stationary solution for the ellipsoidal BGK model in slab. J. Differential Equations, {\bf261} (2016), 5803--5828

\bibitem{Bergh J} Bergh, J{\"o}ran; L{\"o}fstr{\"o}m, J{\"o}rgen: Interpolation spaces. An introduction, Springer-Verlag, 1976.


\bibitem{BGK} Bhatnagar, P. L., Gross, E. P. and Krook, M.: A model for collision processes in gases. Small amplitude process in charged and neutral one-component systems. Phys. Rev., {\bf 94} (1954), 511-525.

\bibitem{Brull-Yun} Brull, S., Yun, S.-B.: Stationary Flows of the ES-BGK model with the correct Prandtl number. SIAM J. Math. Anal. {\bf56} (2024), no. 5, 6361–6397.

\bibitem{CSS} Cameron, Stephen; Silvestre, Luis; Snelson, Stanley: Global a priori estimates for the inhomogeneous Landau equation with moderately soft potentials. Ann. Inst. H. Poincar{\'e} C Anal. Non Lin{\'e}aire {\bf35} (2018), no.3, 625–642.

\bibitem{CS} Cameron, Stephen; Snelson, Stanley: Velocity decay estimates for Boltzmann equation with hard potentials. Nonlinearity {\bf33} (2020), no. 6, 2941–2958.

\bibitem{CHJ} Cao, Chuqi; He, Ling-Bing; Ji, Jie: Propagation of moments and sharp convergence rate for inhomogeneous noncutoff Boltzmann equation with soft potentials. SIAM J. Math. Anal. {\bf56} (2024), no. 1, 1321–1426.

\bibitem{Carlemann} Carlemann, I.: Sur la th\'eorie de l'\'equation int\'egrodifferentielle de Boltzmann, Acta Math. 60 (1933), no. 1, 91–146, doi:10.1007/BF02398270.

\bibitem{TRN1} Chen, T., Denlinger, R., Pavlovi\'c, N.: Local well-posedness for Boltzmann’s equation and the Boltzmann hierarchy via Wigner transform. Comm. Math. Phys. 368, 427–465 (2019)

\bibitem{TRN2} Chen, T., Denlinger, R., Pavlovi\'c, N.: Small data global well-posedness for a Boltzmann equation via bilinear spacetime estimates. Arch. Ration. Mech. Anal. 240, 327–381 (2021)

\bibitem {CH} Chen, X., Holmer, J.:
Well/Ill-Posedness Bifurcation for the Boltzmann Equation with Constant Collision Kernel.Ann. PDE10(2024), no.2, 14.

\bibitem {CSZ} Chen, X., Shen, S. and Zhang Z.:
Well/Ill-posedness of the Boltzmann Equation with Soft Potential, Comm. Math. Phys. {\bf405} (2024), no. 12, Paper No. 283, 51 pp.

\bibitem{C-Z} Chen, Z., Zhang, X.: Global existence and uniqueness to the Cauchy problem of the BGK equation with infinite energy.
Math. Methods Appl. Sci. {\bf 39}, no.11, 3116--3135 (2016)

\bibitem{Cercignani} Cercignani, C.: The Boltzmann Equation and Its Application. Springer-Verlag, 1988

\bibitem{CC} Chapman, C. and Cowling, T. G.: The mathematical theory of non-uniform gases, Cambridge University Press, 1970.

\bibitem{Desvillettes3} Desvillettes, L.: Entropy dissipation estimates for the Landau equation in the Coulomb case and applications. J. Funct. Anal. {\bf269} (2015), no. 5, 1359–1403.

\bibitem{Desvillettes} Desvillettes, L.: Some applications of the method of moments for the homogeneous Boltzmann and Kac equations. Arch. Rational Mech. Anal. {\bf123} (1993), no. 4, 387–404.

\bibitem{Desvillettes2} Desvillettes, L.; Villani, C.: On the spatially homogeneous Landau equation for hard potentials. I. Existence, uniqueness and smoothness. Comm. Partial Differential Equations {\bf25} (2000), no. 1-2, 179–259.

\bibitem{DL} DiPerna, R. and Lions, P-L.: On the Cauchy problem for Boltzmann equations: global existence and weak stability, Ann. of Math. (2) 130 (1989), no. 2, 321–366, doi:10.2307/1971423.

\bibitem{DHWY} Duan, R., Huang, F., Wang, Y., Yang, T.: Global well-posedness of the Boltzmann equation with large amplitude initial data, Arch. Rational Mech. Anal. {\bf 225} (2017), no. 1, 375--424.
\bibitem{DKL2019} Duan, R., Ko, G., Lee, D.: The Boltzmann equation with a class of large-amplitude initial data and specular reflection boundary condition. J. Stat. Phys. {\bf 190} (2023), no.12, Paper No. 189, 46 pp.

\bibitem{DuanAdv} Duan, R., Wang, Y.: The Boltzmann equation with large-amplitude initial data in bounded domains. Adv. Math.{\bf343}(2019), 36–109.

\bibitem{Elmroth} Elmroth, T.: Global boundedness of moments of solutions of the Boltzmann equation for forces of infinite range.Arch. Rational Mech. Anal. {\bf82} (1983), no. 1, 1–12.

\bibitem{Fournier} Fournier, Nicolas: On exponential moments of the homogeneous Boltzmann equation for hard potentials without cutoff. Comm. Math. Phys. {\bf387} (2021), no. 2, 973–994.

\bibitem{GPV} Gamba, I. M.; Panferov, V.; Villani, C.: Upper Maxwellian bounds for the spatially homogeneous Boltzmann equation. Arch. Ration. Mech. Anal. {\bf194}(2009), no. 1, 253–282.

\bibitem{GPT} Gamba, Irene M.; Pavlovi{\'c}, Nata{\v s}a; Taskovi{\'c}, Maja: On pointwise exponentially weighted estimates for the Boltzmann equation. SIAM J. Math. Anal. {\bf51} (2019), no. 5, 3921–3955.

\bibitem{Grad} Grad, H.: Principles of the kinetic theory of gases, Handbuch der Physik [Encyclopedia of Physics], vol. Bd. 12, Thermodynamik der, Springer-Verlag, Berlin-G\"ottingen-Heidelberg, 1958, Herausgegeben von S. Fl\"ugge, pp. 205–294.

\bibitem{GMM} Gualdani, M. P.; Mischler, S.; Mouhot, C.: Factorization of non-symmetric operators and exponential H-theorem. M{\'e}m. Soc. Math. Fr. (N.S.) (2017), no. 153, 137 pp.

\bibitem{NGLS} Guillen, N. and Silvestre, L.: The Landau equation does not blow up, Acta Math. {\bf234} (2025), no. 2, 315–375.

\bibitem {HJ} He, L.-B., Jiang, J.-C.: On the Cauchy problem for the cutoff Boltzmann equation with small initial data. J. Stat. Phys. 190 (2023), no. 3, Paper No. 52, 25 pp.

\bibitem{HY} Hwang, B.-H., Yun, S.-B.: Stationary solutions to the boundary value problem for the relativistic BGK model in a slab. Kinet. Relat. Models {\bf12} (2019), no. 4, 749–-764. 

\bibitem{IS} Illner, R. and Shinbrot, M.: The Boltzmann equation: global existence for a rare gas in an infinite vacuum, Comm. Math. Phys. 95 (1984), no. 2, 217–226.

\bibitem{ISN} Illner, R., Shinbrot, M. and Neunzert H. Blow-up of solutions of the gain-term-only Boltzmann equation. MATH METHOD APPL SCI. Vol 9, (1987), Issue 1, 251-259.

\bibitem{IMS} Imbert, Cyril; Mouhot, Cl{\'e}ment; Silvestre, Luis: Decay estimates for large velocities in the Boltzmann equation without cutoff. J. {\'E}c. polytech. Math. {\bf7} (2020), 143–184.

\bibitem{Issau} Issautier, D.: Convergence of a weighted particle method for solving the Boltzmann (B.G,K.) equaiton, Siam Journal on Numerical Analysis,{\bf 33}, no 6 (1996), 2099-2199.

\bibitem{KS} Kaniel, S. and Shinbrot, M.: The Boltzmann equation. I. Uniqueness and local exis- tence, Comm. Math. Phys. 58 (1978), no. 1, 65–84.

\bibitem{KLP} Ko, G., Lee, D., Park, K.: The large amplitude solution of the Boltzmann equation with soft potential, J. Differential Equations, 307 (2022), 297-347.

\bibitem{LM} Lu, Xuguang; Mouhot, Cl{\'e}ment: On measure solutions of the Boltzmann equation, part I: moment production and stability estimates. J. Differential Equations {\bf252} (2012), no. 4, 3305–3363.

\bibitem{Mischler} Mischler, S.: Uniqueness for the BGK-equation in $R^n$ and the rate of convergence for a semi-discrete scheme. Differential integral Equations {\bf 9} (1996), no.5, 1119--1138.

\bibitem{MisWen} Mischler, S. and Wennberg, B.: On the spatially homogeneous Boltzmann equation. Ann. Inst. Henri Poincar\'e. {\bf16}, no 4, pp. 467-501 (1999)

\bibitem{L1Mouhot} Mouhot, C.: Rate of Convergence to Equilibrium for the Spatially Homogeneous Boltzmann Equation with Hard Potentials. Commun. Math. Phys. 261, 629–672 (2006)

\bibitem{PY1} Park, S. J. and Yun, S.-B.:  Cauchy problem for the ellipsoidal-BGK model of the Boltzmann equation. J. Math. Phys. {\bf57} (8), 081512 (2016).

\bibitem{PY2} Park, S. J. and Yun, S.-B.: Cauchy problem for the ellipsoidal BGK model for polyatomic particles. J. Differ. Equations {\bf266} (11), 7678-–7708 (2019). 





\bibitem{Perthame} Perthame, B.: Global existence to the BGK model of Boltzmann equation. J. Differential Equations. {\bf 82} (1989), no.1, 191-205.

\bibitem{P-P} Perthame, B., Pulvirenti, M.: Weighted $L^{\infty}$ bounds and uniqueness for the Boltzmann BGK model. Arch. Rational Mech. Anal. {\bf 125}
(1993), no. 3, 289--295.

\bibitem{Povzner} Povzner, A.J.: On the Boltzmann equation in the kinetic theory of gases. Mat. Sb. (N.S.) {\bf58}(100), 65–86 (1962)


\bibitem{RSY} Russo, G., Santagati, P., and Yun, S.-B.: Convergence of a semi-Lagrangian scheme for the BGK model of the Boltzmann equation. SIAM Journal on Numerical Analysis,{\bf 50}, no 3 (2012), 1111-1135.

\bibitem{RY} Russo, G. and Yun, S.-B. : Convergence of a Semi-Lagrangian Scheme for the Ellipsoidal BGK Model of the Boltzmann Equation.  SIAM J. Numer. Anal., {\bf56} (6), (2018),  3580--3610.

\bibitem{Son-Yun} Son, S.-j.; Yun, S.-B.: The ES-BGK for the polyatomic molecules with infinite energy. J. Stat. Phys. {\bf190} (2023), no. 8, Paper No. 129, 27 pp.

\bibitem{Sone} Sone, Y.: Kinetic Theory and Fluid Mechanics. Boston: Birkh\"{a}user, 2002.
\bibitem{Sone2} Sone, Y.: Molecular Gas Dynamics: Theory, Techniques, and Applications. Boston: Brikh\"{a}user, 2006.

\bibitem{TAGP}Taskovi{\'c}, Maja; Alonso, Ricardo J.; Gamba, Irene M.; Pavlovi{\'c}, Nata{\v s}a: On Mittag-Leffler moments for the Boltzmann equation for hard potentials without cutoff. SIAM J. Math. Anal. {\bf50} (2018), no. 1, 834–869.

\bibitem{Tricomi} Tricomi, F.G.: Asymptotische eigenschaften der unvollständigen gammafunktion. Math. Z. {\bf 53} (1950), 136–148.

\bibitem{Seiji Ukai} Ukai, S.: On the existence of global solutions of mixed problem for non-linear Boltzmann equation, Proc. Japan Acad. 50 (1974), 179–184.

\bibitem{V} Villani, C.: A Review of mathematical topics in collisional kinetic theory. Handbook of mathematical fluid dynamics. Vol. I. North-Holland. Amsterdam, 2002, 71-305.

\bibitem{V2} Villani, C.: On the spatially homogeneous Landau equation for Maxwellian molecules. Math. Models Methods Appl. Sci. {\bf8} (1998), no. 6, 957–983.

\bibitem{Wal} Walender, P.: On the temperature jump in a rarefied gas, Ark, Fys. {\bf 7}  (1954), 507-553.

\bibitem{HF} Wang, Hao; Fang, Zhendong: Global existence and time decay of the non-cutoff Boltzmann equation with hard potential. Nonlinear Anal. Real World Appl. {\bf63} (2022), Paper No. 103416, 32 pp.

\bibitem{W-Z} Wei, J., Zhang, X.: The Cauchy problem for the BGK equation with an external force. J. Math. Anal. Appl. {\bf391}, no.1, 10--25, (2012)



\bibitem{Yun1} Yun, S.-B.: Cauchy problem for the Boltzmann-BGK model near a global Maxwellian. J. Math. Phy. {\bf 51} (2010), no. 12, 123514, 24pp.
\bibitem{Yun2} Yun, S.-B.: Classical solutions for the ellipsoidal BGK model with fixed collision frequency. J. Differential Equations  259  (2015),  no. 11, 6009–-6037

\bibitem{Yun22} Yun, S.-B.: Ellipsoidal BGK model for polyatomic molecules near Maxwellians: a dichotomy in the dissipation estimate. J. Differential Equations {\bf266} (2019), no. 9, 5566–-5614.
\bibitem{Yun3} Yun, S.-B.: Ellipsoidal BGK model near a global Maxwellian. SIAM J. Math. Anal. {\bf47} (2015), no. 3, 2324–-2354.

\bibitem{Z-VP} Zhang, X.: Global weak solutions to the Vlasov–Poisson–BGK system for initial data in $L^p (\rrr \times \rrr)$.
Appl. Math. Lett.{\bf26}, no.11, 1087–-1093 (2013)

\bibitem{Zhang} Zhang, X: On the Cauchy problem of the Vlasov-Poisson-BGK system: global existence of weak solutions
J Stat Phys, {\bf141}, 566-–588 (2010)

\bibitem{Z-H} Zhang, X., Hu, S.: $L^p$ solutions to the Cauchy problem of the BGK equation,
J. Math. Phys. 48, 113304 (2007)


\end{thebibliography}

\end{document}